\documentclass[a4paper]{amsart}

\RequirePackage{amsmath} \RequirePackage{amssymb}

\usepackage{amscd,latexsym,amsthm,amsfonts,amssymb,amsmath,amsxtra}
\usepackage[colorlinks=true,urlcolor=blue,citecolor=blue]{hyperref}
\usepackage{color}
\usepackage[all]{xy}
\usepackage[OT2,T1]{fontenc}
\usepackage{bm}
\usepackage{mathtools}
\usepackage{ mathrsfs }
\usepackage{xcolor}
\usepackage{comment}
\usepackage[british]{babel}
\usepackage{enumitem}
\usepackage[demo]{graphicx}
\newcommand*{\transp}[2][-3mu]{\ensuremath{\mskip1mu\prescript{\smash{\mathrm t\mkern#1}}{}{\mathstrut#2}}}
\usepackage{tikz-cd}
\usepackage{marginnote} 

\DeclareSymbolFont{cyrletters}{OT2}{wncyr}{m}{n}
\DeclareMathSymbol{\Sha}{\mathalpha}{cyrletters}{"58}

\let\Re\undefined

\DeclareMathOperator{\Re}{Re}

\DeclareMathOperator{\Tr}{Tr}

\DeclareMathOperator{\supp}{supp}

\DeclareMathOperator{\vol}{vol}
\DeclareMathOperator{\Spec}{Spec}
\DeclareMathOperator{\Hom}{Hom}

\begin{document}
\theoremstyle{plain}
	\newtheorem{thm}{Theorem}[section]
	
	\newtheorem{cor}[thm]{Corollary}
	\newtheorem{thmy}{Theorem}
	\renewcommand{\thethmy}{\Alph{thmy}}
	\newenvironment{thmx}{\stepcounter{thm}\begin{thmy}}{\end{thmy}}
	\newtheorem{cory}{Corollary}
	\renewcommand{\thecory}{\Alph{cory}}
	\newenvironment{corx}{\stepcounter{thm}\begin{cory}}{\end{cory}}
	\newtheorem*{thma}{Theorem A}
	\newtheorem*{corb}{Corollary B}
	\newtheorem*{thmc}{Theorem C}
	\newtheorem{lemma}[thm]{Lemma}  
	\newtheorem{prop}[thm]{Proposition}
	\newtheorem{conj}[thm]{Conjecture}
	\newtheorem{fact}[thm]{Fact}
	\newtheorem{claim}[thm]{Claim}
	
	\theoremstyle{definition}
	\newtheorem{defn}[thm]{Definition}
	\newtheorem{example}[thm]{Example}
	\theoremstyle{remark}
	
	\newtheorem{remark}[thm]{Remark}	
	\numberwithin{equation}{section}
		
	\newcommand{\BA}{{\mathbb {A}}} \newcommand{\BB}{{\mathbb {B}}}
	\newcommand{\BC}{{\mathbb {C}}} \newcommand{\BD}{{\mathbb {D}}}
	\newcommand{\BE}{{\mathbb {E}}} \newcommand{\BF}{{\mathbb {F}}}
	\newcommand{\BG}{{\mathbb {G}}} \newcommand{\BH}{{\mathbb {H}}}
	\newcommand{\BI}{{\mathbb {I}}} \newcommand{\BJ}{{\mathbb {J}}}
	\newcommand{\BK}{{\mathbb {K}}} \newcommand{\BL}{{\mathbb {L}}}
	\newcommand{\BM}{{\mathbb {M}}} \newcommand{\BN}{{\mathbb {N}}}
	\newcommand{\BO}{{\mathbb {O}}} \newcommand{\BP}{{\mathbb {P}}}
	\newcommand{\BQ}{{\mathbb {Q}}} \newcommand{\BR}{{\mathbb {R}}}
	\newcommand{\BS}{{\mathbb {S}}} \newcommand{\BT}{{\mathbb {T}}}
	\newcommand{\BU}{{\mathbb {U}}} \newcommand{\BV}{{\mathbb {V}}}
	\newcommand{\BW}{{\mathbb {W}}} \newcommand{\BX}{{\mathbb {X}}}
	\newcommand{\BY}{{\mathbb {Y}}} \newcommand{\BZ}{{\mathbb {Z}}}
	
	\newcommand{\CA}{{\mathcal {A}}} \newcommand{\CB}{{\mathcal {B}}}
	\newcommand{\CC}{{\mathcal {C}}} \renewcommand{\CD}{{\mathcal {D}}}
	\newcommand{\CE}{{\mathcal {E}}} \newcommand{\CF}{{\mathcal {F}}}
	\newcommand{\CG}{{\mathcal {G}}} \newcommand{\CH}{{\mathcal {H}}}
	\newcommand{\CI}{{\mathcal {I}}} \newcommand{\CJ}{{\mathcal {J}}}
	\newcommand{\CK}{{\mathcal {K}}} \newcommand{\CL}{{\mathcal {L}}}
	\newcommand{\CM}{{\mathcal {M}}} \newcommand{\CN}{{\mathcal {N}}}
	\newcommand{\CO}{{\mathcal {O}}} \newcommand{\CP}{{\mathcal {P}}}
	\newcommand{\CQ}{{\mathcal {Q}}} \newcommand{\CR}{{\mathcal {R}}}
	\newcommand{\CS}{{\mathcal {S}}} \newcommand{\CT}{{\mathcal {T}}}
	\newcommand{\CU}{{\mathcal {U}}} \newcommand{\CV}{{\mathcal {V}}}
	\newcommand{\CW}{{\mathcal {W}}} \newcommand{\CX}{{\mathcal {X}}}
	\newcommand{\CY}{{\mathcal {Y}}} \newcommand{\CZ}{{\mathcal {Z}}}
	
	\newcommand{\RA}{{\mathrm {A}}} \newcommand{\RB}{{\mathrm {B}}}
	\newcommand{\RC}{{\mathrm {C}}} \newcommand{\RD}{{\mathrm {D}}}
	\newcommand{\RE}{{\mathrm {E}}} \newcommand{\RF}{{\mathrm {F}}}
	\newcommand{\RG}{{\mathrm {G}}} \newcommand{\RH}{{\mathrm {H}}}
	\newcommand{\RI}{{\mathrm {I}}} \newcommand{\RJ}{{\mathrm {J}}}
	\newcommand{\RK}{{\mathrm {K}}} \newcommand{\RL}{{\mathrm {L}}}
	\newcommand{\RM}{{\mathrm {M}}} \newcommand{\RN}{{\mathrm {N}}}
	\newcommand{\RO}{{\mathrm {O}}} \newcommand{\RP}{{\mathrm {P}}}
	\newcommand{\RQ}{{\mathrm {Q}}} \newcommand{\RR}{{\mathrm {R}}}
	\newcommand{\RS}{{\mathrm {S}}} \newcommand{\RT}{{\mathrm {T}}}
	\newcommand{\RU}{{\mathrm {U}}} \newcommand{\RV}{{\mathrm {V}}}
	\newcommand{\RW}{{\mathrm {W}}} \newcommand{\RX}{{\mathrm {X}}}
	\newcommand{\RY}{{\mathrm {Y}}} \newcommand{\RZ}{{\mathrm {Z}}}
	
	\newcommand{\fa}{{\mathfrak{a}}} \newcommand{\fb}{{\mathfrak{b}}}
	\newcommand{\fc}{{\mathfrak{c}}} \newcommand{\fd}{{\mathfrak{d}}}
	\newcommand{\fe}{{\mathfrak{e}}} \newcommand{\ff}{{\mathfrak{f}}}
	\newcommand{\fg}{{\mathfrak{g}}} \newcommand{\fh}{{\mathfrak{h}}}
	\newcommand{\fii}{{\mathfrak{i}}} \newcommand{\fj}{{\mathfrak{j}}}
	\newcommand{\fk}{{\mathfrak{k}}} \newcommand{\fl}{{\mathfrak{l}}}
	\newcommand{\fm}{{\mathfrak{m}}} \newcommand{\fn}{{\mathfrak{n}}}
	\newcommand{\fo}{{\mathfrak{o}}} \newcommand{\fp}{{\mathfrak{p}}}
	\newcommand{\fq}{{\mathfrak{q}}} \newcommand{\fr}{{\mathfrak{r}}}
	\newcommand{\fs}{{\mathfrak{s}}} \newcommand{\ft}{{\mathfrak{t}}}
	\newcommand{\fu}{{\mathfrak{u}}} \newcommand{\fv}{{\mathbf{v}}}
	\newcommand{\fw}{{\mathfrak{w}}} \newcommand{\fx}{{\mathfrak{x}}}
	\newcommand{\fy}{{\mathfrak{y}}} \newcommand{\fz}{{\mathfrak{z}}}
	\newcommand{\fA}{{\mathfrak{A}}} \newcommand{\fB}{{\mathfrak{B}}}
	\newcommand{\fC}{{\mathfrak{C}}} \newcommand{\fD}{{\mathfrak{D}}}
	\newcommand{\fE}{{\mathfrak{E}}} \newcommand{\fF}{{\mathfrak{F}}}
	\newcommand{\fG}{{\mathfrak{G}}} \newcommand{\fH}{{\mathfrak{H}}}
	\newcommand{\fI}{{\mathfrak{I}}} \newcommand{\fJ}{{\mathfrak{J}}}
	\newcommand{\fK}{{\mathfrak{K}}} \newcommand{\fL}{{\mathfrak{L}}}
	\newcommand{\fM}{{\mathfrak{M}}} \newcommand{\fN}{{\mathfrak{N}}}
	\newcommand{\fO}{{\mathfrak{O}}} \newcommand{\fP}{{\mathfrak{P}}}
	\newcommand{\fQ}{{\mathfrak{Q}}} \newcommand{\fR}{{\mathfrak{R}}}
	\newcommand{\fS}{{\mathfrak{S}}} \newcommand{\fT}{{\mathfrak{T}}}
	\newcommand{\fU}{{\mathfrak{U}}} \newcommand{\fV}{{\mathbf{v}}}
	\newcommand{\fW}{{\mathfrak{W}}} \newcommand{\fX}{{\mathfrak{X}}}
	\newcommand{\fY}{{\mathfrak{Y}}} \newcommand{\fZ}{{\mathfrak{Z}}}
	\newcommand{\Mod}[1]{\ (\mathrm{mod}\ #1)}
	\newcommand{\Int}{\operatorname{Int}}
	\newcommand{\pdiv}{\mid\!\mid}
	\newcommand{\Res}{\operatorname{Res}}
	\newcommand{\tr}{\operatorname{tr}}
	\newcommand{\Eis}{\operatorname{Eis}}
	\newcommand{\Geo}{\operatorname{Geo}}
	\newcommand{\End}{\operatorname{End}}
	\newcommand{\K}{\operatorname{K}}
	\newcommand{\sgn}{\operatorname{sgn}}
	\newcommand{\s}{\operatorname{sp}}
	\newcommand{\Ker}{\operatorname{Ker}}
	\newcommand{\Ad}{\operatorname{Ad}}
	\newcommand{\Weyl}{\operatorname{Weyl}}
	\newcommand{\ram}{\operatorname{ram}}
	\newcommand{\Supp}{\operatorname{Supp}}
	\newcommand{\Cond}{\operatorname{Cond}}
	\newcommand{\Id}{\operatorname{Id}}
	\newcommand{\Pet}{\operatorname{Pet}}
	\newcommand{\Reg}{\operatorname{Reg}}
	\newcommand{\reg}{\operatorname{reg}}
	\newcommand{\Red}{\operatorname{Red}}
	\newcommand{\fin}{\operatorname{fin}}
	\newcommand{\diag}{\operatorname{diag}}
	\newcommand{\LHS}{\operatorname{LHS}}
	\newcommand{\RHS}{\operatorname{RHS}}
	\newcommand{\Vol}{\operatorname{Vol}}
	\newcommand{\SO}{\operatorname{SO}}
	\newcommand{\si}{\operatorname{Sing}}
	\newcommand{\St}{\operatorname{St}}
	\newcommand{\sm}{\operatorname{Small}}
	\newcommand{\du}{\operatorname{Dual}}
	\newcommand{\bi}{\operatorname{Big}}
	\newcommand{\const}{\operatorname{Const}}
	\newcommand{\Kuz}{\operatorname{Kuz}}
	\newcommand{\Kl}{\operatorname{Kl}}
	\newcommand{\ER}{\operatorname{ER}}
	\newcommand{\sph}{\operatorname{sph}}
	\newcommand{\gen}{\operatorname{gen}}
	\newcommand{\aut}{\operatorname{aut}}
	\newcommand{\Cartan}{\operatorname{Cartan}}
	\newcommand{\Semi}{\operatorname{semi-reg}}
	\newcommand{\semi}{\operatorname{semi-sing}}
	\newcommand{\sing}{\operatorname{sing}}
	\newcommand{\Ind}{\operatorname{Ind}}
	\newcommand{\Nr}{\operatorname{Nr}}
	\newcommand{\Error}{\operatorname{Error}}
	\newcommand{\Main}{\operatorname{Main}}
	\newcommand{\hol}{\operatorname{hol}}
	\newcommand{\dist}{\operatorname{dist}}
	\newcommand{\Sym}{\operatorname{Sym}}
	\newcommand{\RNum}[1]{\uppercase\expandafter{\romannumeral #1\relax}}

\title[Rankin-Selberg $L$-functions for $\mathrm{GL}(n+1)\times \mathrm{GL}(n)$]{Relative Trace Formula and $L$-functions for $\mathrm{GL}(n+1)\times \mathrm{GL}(n)$}%
	\author{Liyang Yang}
	
	\address{Fine Hall, 304 Washington Rd, Princeton, 
		NJ 08544, USA}
	\email{liyangy@princeton.edu}
	
\begin{abstract}
We establish a relative trace formula on $\mathrm{GL}(n+1)$ weighted by cusp forms on $\mathrm{GL}(n)$ over number fields. The spectral side is a weighted average of Rankin-Selberg $L$-functions for $\mathrm{GL}(n+1)\times\mathrm{GL}(n)$ over the \textit{full spectrum}, and the geometric side consists of Rankin-Selberg $L$-functions for $\mathrm{GL}(n)\times\mathrm{GL}(n),$ and certain explicit holomorphic functions. The formula yields new results towards central $L$-values for $\mathrm{GL}(n+1)\times\mathrm{GL}(n)$ (over number fields): the second moment evaluation, and  simultaneous nonvnaishing in the level aspect. 
\end{abstract}
	
	\date{\today}%
	\maketitle
	\tableofcontents
	
\section{Introduction}

Let $F$ be a number field with ring of adeles $\mathbb{A}_F$. Let $\chi$ be a quadratic Hecke character over $\mathbb{A}_{F}^{\times}.$ Let $A=\diag(\mathrm{GL}(1),1).$ Let $\K^f$ be the automorphic kernel attached to a function $f$ on $\mathrm{GL}(2,\mathbb{A}_F)$ with compact support modulo the center. Jacquet \cite{Jac86} established a relative trace formula, which is an identity between the integral 
\begin{equation}\label{1.}
\int_{A(F)\backslash A(\mathbb{A}_F)}\int_{A(F)\backslash A(\mathbb{A}_F)}\K^f(x,y)\chi(\det y)d^{\times}xd^{\times}y,
\end{equation}
and its counterpart 
\begin{equation}\label{40}
\sum_{(G^*,T^*)}\int_{A^*(F)\backslash A^*(\mathbb{A}_F)}\int_{A^*(F)\backslash A^*(\mathbb{A}_F)} \K^{f^*}(x^*,y^*)d^{\times}x^*d^{\times}y^*,
\end{equation}
where $E/F$ is the quadratic extension corresponding to $\chi$, the sum is over pairs of inner forms $G^*$ (of $\mathrm{GL}(2)$) and its maximal torus $T^*$ which is isomorphic to $E^{\times}$ over $F,$ and $f^*$ is a test function associated to $(G^*,T^*).$ Via his relative trace formula Jacquet gave a new proof of the remarkable result of Waldspurger \cite{Wal85}. 
\medskip 

Note that by spectral expansion and Rankin-Selberg convolution, the cuspidal spectrum in \eqref{1.} is a weighted sum of $L(1/2,\pi)L(1/2,\pi\times\chi)$ as $\pi$ ranges through cuspidal  representations of $\mathrm{GL}(2,\mathbb{A}_F).$ Based on Jacquet's relative trace formula, Guo \cite{Guo96a} managed to show that $L(1/2,\pi)L(1/2,\pi\times\chi)\geq 0$ for all cuspidal representations $\pi.$  
\medskip 

On the other hand, following the spirit of the Arthur-Selberg trace formula, one can also study the spectral expansion of \eqref{1.} by its geometric side, which gives another perspective to understand Rankin-Selberg $L$-functions $L(1/2,\pi)L(1/2,\pi\times\chi)$ quantitatively. Ramakrishnan and Rogawski \cite{RR05} gave the first description of the geometric side of \eqref{1.} to generalize Duke's theorem \cite{Duk95}. Such a formula has many other applications, e.g., the moment formula of $L$-values, vertical Sato-Tate distribution, nonvanishing problem of central $L$-values, subconvexity problem, distribution of low-lying zeros (cf. \cite{RR05}, \cite{FW09}, \cite{Tsu15}, \cite{AC19}). 

\medskip 

We aim to study the integral of type \eqref{1.} in higher ranks. Let $G=\mathrm{GL}(n+1)$ and $G'=\mathrm{GL}(n).$ Let $\pi'_1,$ $\pi_2'$ be cuspidal representations of $G'(\mathbb{A}_F)$. Fix cusp forms $\phi'_1\in \pi'_1$ and $\phi_2'\in \pi'_2$. A rudimentary generalization is in principle based on the equality of geometric and spectral expansions of the integral
\begin{equation}\label{39}
\int_{G'(F)\backslash G'(\mathbb{A}_F)}\int_{G'(F)\backslash G'(\mathbb{A}_F)}\K^f\left(\begin{pmatrix}
		x\\
		&1
	\end{pmatrix},\begin{pmatrix}
		y\\
		&1
	\end{pmatrix}\right)\phi_1'(x)\overline{\phi_2'(y)}dxdy.
\end{equation}
Were \eqref{39} convergent, such an equality is of significant interest since its spectral side consists of products of Rankin-Selberg periods, which, by \cite{JPSS83}, \cite{JS90}, \cite{Jac09}, represent $L(1/2,\pi\times\pi_1')L(1/2,\widetilde{\pi}\times\widetilde{\pi}_2')$ as $\pi$ ranges over automotphic representations of $G(\mathbb{A}_F).$ Hence understanding the geometric side of the integral \eqref{39} will provide a new instrumental identity to study $L$-functions for $\mathrm{GL}(n+1)\times\mathrm{GL}(n)$ directly (as apposed to the Kuznetsov formula) by choosing an adequate $f$ and computing or estimating both sides. In the $U(3)\times U(2)$ case, such a relative trace formula, together with an asymptotic form of the Ichino-Ikeda formula, is the key ingredient to obtain new nonvanishing and subconvexity results. See \cite{MRY22} for details. 
\medskip 

However, there are at least two major barriers to obtain such a desired relative trace formula from \eqref{39}.
\medskip 

Firstly, unlike the situation in \cite{DGG12}, the underlying integral \eqref{39} does not converge in general. So a regularization is required. Moreover, since we expect applications in the analytic theory of $L$-functions, the regularized formula should be explicit enough for concrete calculation or estimate on both sides. Thus the truncation approaches in \cite{JLR99} and \cite{IY15} do not seem to be amenable to our purpose. When $n=1$ and $F$ is totally real, one can take the archimedean component $f_{\infty}$ to be the matrix coefficient of a holomorphic discrete series so that \eqref{39} converges. This is used in \cite{RR05}, \cite{FW09} and \cite{Tsu15}. Nevertheless, general situations in the case $n=1$ (e.g., the second moment case) have not been investigated yet. In this paper, we will handle the \textit{full spectrum} in the spectral side.
\medskip 

Secondly, when $n>1,$ the automorphic weights $\phi_1'$ and $\phi_2'$ ruin the Eulerian structure of orbital integrals in the geometric side. For instance, for $\gamma\in G(F)$ with stabilizer $G_{\gamma},$ the associated orbital integral factors through 
\begin{equation}\label{50}
\int_{G_{\gamma}(F)\backslash G_{\gamma}(\mathbb{A}_F)}\phi_1'(hx)\overline{\phi_2'(hy)}dh
\end{equation}
for some $x, y\in G(\mathbb{A}_F).$ The integral \eqref{50} does not decompose into local integrals for most $\gamma$'s. In these worst cases one has to apply Cauchy-Schwartz to bound \eqref{50} by $L^2$-norm of a cusp form restricted to $G_{\gamma}(F)\backslash G_{\gamma}(\mathbb{A}_F),$ which is rather delicate to handle. These disadvantages make general orbital integrals difficult to bound individually (e.g., cf. \cite{MRY22}). Also, since Eisenstein series have degenerate terms in their Fourier expansions, by (formal) unfolding, the periods from continuous spectrum fail to be factorizable. Hence, even if ignoring the convergence issue, an elementary expansion of \eqref{39} does not give a favorite relative trace formula.
\medskip

The goals in this paper are fourfold. 
\begin{itemize}
	\item Our first aim is to develop a deformed relative trace formula
\begin{equation}\label{1000}
J_{\Spec}^{\Reg}(f,\mathbf{s};\phi_1',\phi_2')=J^{\Reg}_{\Geo}(f,\textbf{s};\phi_1',\phi_2'),\ \ \textbf{s}=(s_1,s_2),\ \ \Re(s_1),\ \Re(s_2)\geq 1,
\end{equation}
which is an identity of holomorphic functions of two complex variables.
The spectral side $J_{\Spec}^{\Reg}(f,\mathbf{s};\phi_1',\phi_2')$ is obtained by moving the degenerate terms in the Fourier expansion of automorphic forms in the residual and continuous spectrum of \eqref{39} to the geometric side. See \textsection\ref{sec2} for details. After unfolding and swapping integrals, 
\begin{align*}
J_{\Spec}^{\Reg}(f,\mathbf{s};\phi_1',\phi_2')=\int_{\widehat{G(\mathbb{A}_F)}_{\gen}}\sum_{\phi\in\mathfrak{B}_{\pi}}\Psi(s_1,\pi(f)W_{\phi},W_{\phi_1'}')\Psi(s_2,\widetilde{W}_{\phi},\widetilde{W}_{\phi_2'}')d\mu_{\pi},
\end{align*}
where $\widehat{G(\mathbb{A}_F)}_{\gen}$ is the space of generic representations,  $W_{\phi_i'}'$ (resp. $W_{\phi}$) the Whittaker function of the cusp form $\phi_i',$ (resp. $\phi$), $i=1, 2,$ and $\Psi(s,W_{\phi},W_{\phi_1'}')$ is the Rankin-Selberg period in the Whittaker form. By \cite{JPSS83}, integrands in $J_{\Spec}^{\Reg}(f,\mathbf{s})$ represent $L(1/2+s_1,\pi\times\pi_1')L(1/2+s_2,\widetilde{\pi}\times\widetilde{\pi}_2').$ Note that the contribution from the cuspidal spectrum of $J_{\Spec}^{\Reg}(f,\mathbf{s})$ is the same as that of \eqref{39}. 


\medskip

\item The second goal is to derive meromorphic continuation of both spectral and geometric sides of \eqref{1000} to $\Re(s_1)>-1/(n+1),$ $\Re(s_2)>-1/(n+1).$ We will show the absolute convergence, and cancellation of singular parts of both sides. In particular, the evaluation at $\mathbf{s}=(0,0)$ gives the desired regularization of \eqref{39}. This resolves the qualitative-analytic problem of making sense of divergent integrals \eqref{39}. 
\medskip
\item The third goal is to concatenate  orbital integrals in the geometric side in an appropriate way to make them either factorizable or be convenient to majorize by Eulerian integrals (cf. e.g., \textsection\ref{sec5}).  
\medskip
\item Lastly we provide applications of the relative trace formula  towards the second moment computation and simultaneously nonvanishing problem for Rankin-Selberg $L$-functions for $\mathrm{GL}(n+1)\times\mathrm{GL}(n)$ over number fields. 
\end{itemize}
\medskip
In all, the relative trace formula $J_{\Spec}^{\Reg}(f,\mathbf{s};\phi_1',\phi_2')=J^{\Reg}_{\Geo}(f,\textbf{s};\phi_1',\phi_2')$ as an identity of meromorphic functions is summarized by Theorem \ref{C} below. A complete description can be found in Theorem \ref{thm51} in \text\textsection \ref{sec8.3} and Theorem \ref{thm52} in \textsection\ref{8.2}.

 \subsection{The Relative Trace Formula}

An informal description of the relative trace formula is the following.

\begin{thmx}\label{C}
Let notation be as before. Let $n\geq 2.$ Let $G=\mathrm{GL}(n+1)$ and $G'=\mathrm{GL}(n).$ Let $\pi_1', \pi_2'$ be unitary cuspidal representations of $G'$ over $F.$ Let $\phi_i'\in\pi_i',$ $i=1, 2.$ Let $\theta$ be a nontrivial additive character of $F\backslash \mathbb{A}_F.$ Let $P_0'$ be the mirabolic subgroup of $G'.$ Let $\mathbf{s}=(s_1,s_2)\in\mathbb{C}^2$.  Let $f$ be a continuous function on $G(\mathbb{A}_F)$ with compact support modulo the center. Then 
\begin{equation}\label{b}
J_{\Spec}^{\Reg}(f,\mathbf{s};\phi_1',\phi_2')=J^{\Reg}_{\Geo}(f,\textbf{s};\phi_1',\phi_2'),\end{equation}
which converge absolutely in the region $\Re(s_1),\ \Re(s_2)\gg1.$ Here the spectral side 
$$
J_{\Spec}^{\Reg}(f,\mathbf{s};\phi_1',\phi_2')=\int_{\widehat{G(\mathbb{A}_F)}_{\gen}}\sum_{\phi\in\mathfrak{B}_{\pi}}\Psi(s_1,\pi(f)W_{\phi},W_{\phi_1'}')\Psi(s_2,\widetilde{W}_{\phi},\widetilde{W}_{\phi_2'}')d\mu_{\pi},
$$ 
and the geometric side 
\begin{align*}
	J^{\Reg}_{\Geo}(f,\textbf{s};\phi_1',\phi_2'):=J^{\Reg}_{\Geo,\sm}(f,\textbf{s})+^{\Reg}_{\Geo,\du}(f,\textbf{s})-J^{\Reg,\RNum{1}}_{\Geo,\bi}(f,\textbf{s})+J^{\Reg,\RNum{2}}_{\Geo,\bi}(f,\textbf{s}),
\end{align*}
with $\iota(x):=\diag(x,1)$ for all $x\in G'(\mathbb{A}_F),$ and 
\begin{align*}
&J^{\Reg}_{\Geo,\sm}(f,\textbf{s}):=\int_{P_0'(F)\backslash {G'}(\mathbb{A}_F)}\int_{{G'}(\mathbb{A}_F)}\int_{\mathbb{A}_F^n}f\left(
	\begin{pmatrix}
		y&u\\
		&1
	\end{pmatrix}
	\right)\theta(\eta xu)du\\
	&\qquad\qquad \qquad \qquad\qquad \qquad \phi_1'(x)\overline{{\phi_2'}(xy)}|\det x|^{s_1+s_2+1}|\det y|^{s_2}dydx,\\
&J^{\Reg}_{\Geo,\du}(f,\textbf{s}):=\int_{P_0'(F)\backslash {G'}(\mathbb{A}_F)}\int_{{G'}(\mathbb{A}_F)}f\left(\begin{pmatrix}
		I_{n}&\\
		\eta x&1
	\end{pmatrix}\begin{pmatrix}
		y&\\
	      &1
	\end{pmatrix}\right)\\
	&\qquad \qquad \qquad\qquad \qquad\qquad \phi'_1(x)\overline{\phi'_2(xy)}|\det x|^{s_1+s_2}|\det y|^{s_2}dydx,\\
&J^{\Reg,\RNum{1}}_{\Geo,\bi}(f,\textbf{s}):=\int_{P_0'(F)\backslash {G'}(\mathbb{A}_F)}\int_{{G'}(\mathbb{A}_F)}\int_{\mathbb{A}_F^n}\phi_1'(x)\overline{\phi_2'(xy)}|\det x|^{s_1+s_2+1}|\det y|^{s_2}\\
	&\qquad \qquad \Bigg[f\left(\begin{pmatrix}
		I_n&u\\
		&1
	\end{pmatrix}\begin{pmatrix}
		I_n&\\
		\eta x&1
	\end{pmatrix}\begin{pmatrix}
		y&\\
		&1
	\end{pmatrix}\right)+f\left(\begin{pmatrix}
		I_n&\\
		\eta x&1
	\end{pmatrix}\begin{pmatrix}
		y&u\\
		&1
	\end{pmatrix}\right)\Bigg]dudydx,
	\end{align*}
\begin{align*}
&J^{\Reg,\RNum{2}}_{\Geo,\bi}(f,\textbf{s}):=\mathop{\sum\sum}_{(\boldsymbol{\xi},t)\in F^n-\{\mathbf{0}\}}\int_{{G'}(\mathbb{A}_F)}\int_{P_0'(F)\backslash {G'}(\mathbb{A}_F)}f\left(\iota(x)^{-1}\begin{pmatrix}
		I_{n-1}&&\boldsymbol{\xi}\\
		&1&t\\
		&1&1
	\end{pmatrix}\iota(xy)\right)\\
	&\qquad \qquad \qquad\qquad \phi'_1(x)\overline{\phi'_2(xy)}|\det x|^{s_1+s_2}|\det y|^{s_2}dxdy.
\end{align*}

Moreover, \eqref{b} admits a meromorphic continuation to $\Re(s_1),\ \Re(s_2)>-1/(n+1).$ We have the following description of the above integrals.
\begin{itemize}

\item The integral $J^{\Reg}_{\Geo,\sm}(f,\textbf{s})$ is a Rankin-Selberg convolution which converges absolutely in $\Re(s_1)+\Re(s_2)>0,$ and represents  $L(1+s_1+s_2,\pi_1'\times\widetilde{\pi}_2').$

\item The integral $J^{\Reg}_{\Geo,\du}(f,\textbf{s})$ is a Rankin-Selberg convolution which converges absolutely in $\Re(s_1)+\Re(s_2)>1,$ and represents  $L(s_1+s_2,\pi_1'\times\widetilde{\pi}_2').$ 

\item The integral $J^{\Reg,\RNum{1}}_{\Geo,\bi}(f,\textbf{s})$ converges absolutely in $\Re(s_1),\ \Re(s_2)\gg1,$ and admits a meromorphic continuation to $\Re(s_1),\ \Re(s_2)>-1/(n+1).$

\item The integral $J^{\Reg,\RNum{2}}_{\Geo,\bi}(f,\textbf{s})$ converges absolutely for all $\mathbf{s}\in\mathbb{C}^2.$

\item  $J^{\Reg}_{\Geo}(f,\textbf{s};\phi_1',\phi_2')$ and $J_{\Spec}^{\Reg}(f,\mathbf{s};\phi_1',\phi_2')$ have at most simple poles on $s_1+s_2=\delta,$ $\delta\in\{-1,0,1\},$ and are holomorphic elsewhere. Moreover, 
\begin{align*}
&\underset{s_1+s_2=\delta}{\Res}J_{\Spec}^{\Reg}(f,\mathbf{s};\phi_1',\phi_2')=\underset{s_1+s_2=\delta}{\Res}J^{\Reg}_{\Geo}(f,\textbf{s};\phi_1',\phi_2'),\ \ \delta\in\{-1,0,1\}.
\end{align*}
In particular, singular parts on both sides of \eqref{b} cancel with each other.

\item The geometric side $J^{\Reg}_{\Geo}(f,\textbf{s};\phi_1',\phi_2')$ has stability in the level aspect in the sense of \cite{MR12} and \cite{FW09}. Roughly, the regular orbital integrals vanish identically when the level is large enough. See Theorem \ref{Red} in \textsection\ref{sec10}.
\end{itemize}

A precise description of \eqref{b} is given by Theorem \ref{thm51} in \text\textsection \ref{sec8.3} (when $\mathbf{s}$ is general) and Theorem \ref{thm52} in \text\textsection \ref{8.2} (when $\mathbf{s}$ is near $(0,0)$).
\end{thmx}

\begin{remark}
\begin{enumerate}
\item[(i).] The superscript `Reg' (i.e., regularization) means that we move contributions from the residue spectrum and degenerate terms of Eisenstein series in the spectral side of \eqref{39} to the geometric side  in \eqref{b} (cf. \text\textsection \ref{sec2}). Divergence is captured by the fact that $L(1+s_1+s_2,\pi_1'\times\widetilde{\pi}_2')$ has a simple pole when $s_1=s_2=0$ and $\pi_1'\simeq \pi_2'.$

\item[(ii).] In $J_{\Spec}^{\Reg}(f,\mathbf{s};\phi_1',\phi_2')$ the continuous spectrum contributes \textit{higher moment} of Rankin-Selberg $L$-functions integrated along the critical line $\Re(s)=1/2.$ 


\end{enumerate}
\end{remark}

Theorem \ref{C} resolves the qualitative-analytic problem of making sense of divergent integrals \eqref{39}. This provides a new method to investigate mean values of Rankin-Selberg $L$-functions in higher ranks. Note that the spectral side $J_{\Spec}^{\Reg}(f,\mathbf{0};\phi_1',\phi_2')$ involves sums of central $L$-values, which are frequently used in subconvexity and nonvanishing problems via  techniques from analytic theory of $L$-functions. Moreover, both of the geometric and spectral sides, as (almost) factorizable integrals, are convenient for concrete calculation or estimate once $f$ is chosen. In particular, the geometric side $J^{\Reg}_{\Geo}(f,\textbf{s};\phi_1',\phi_2')$ does not involve delicate Kloosterman sums as in the Kuznetsov formula (cf. e.g., \cite{Ye00}, \cite{Blo13}). 

We give some quantitative estimates of both sides of Theorem \ref{C} in \textsection\ref{11.2} and \textsection\ref{sec8}, obtaining new results towards Rankin-Selberg $L$-functions for $\mathrm{GL}(n+1)\times\mathrm{GL}(n)$ over \textit{number fields}: the second moment calculation of central $L$-values over the full spectrum of $\mathrm{GL}(n+1)$, and simultaneously nonvanishing of central $L$-values. See \textsection\ref{sec1.2} and \textsection\ref{sec1.3} below. 

\subsection{Second Moment of $L$-functions for $\mathrm{GL}(n+1)\times\mathrm{GL}(n)$}\label{sec1.2}
In the analytic theory of automorphic forms, it has been a very common theme of research to study the associated automorphic $L$-functions in a suitable family and to search for statistical results on average therein (e.g., cf. \cite{IS00}). In particular, the asymptotic evaluation of the \textit{second moment} of central $L$-values for $\mathrm{GL}(n+1)\times\mathrm{GL}(n)$ plays an important role (cf. \cite{BFKMMS18} for the case that $n=1$). The difficulty of this evaluation balloons as  the rank of the underlying group (or the degree of the $L$-function) increases. 

Making use of the relative trace formula Theorem \ref{C}, we establish the following asymptotic second moment calculation of central Rankin-Selberg  $L$-values for $\mathrm{GL}(n+1)\times\mathrm{GL}(n)$ over number fields, for all $n\geq 2.$ 
\begin{thmx}[Second Moment]\label{cor2.}
Let notation be as before. Let $\pi'\in\mathcal{A}_0([G'],\omega')$ be  everywhere unramified. Let $\mathfrak{N}\subsetneq\mathcal{O}_F$ be an ideal with norm $|\mathfrak{N}|$. Then 
\begin{align*}
&\int_{\widehat{G(\mathbb{A}_F)}_{\gen}(\mathfrak{N})}\sum_{\phi\in\mathfrak{B}_{\pi}}\frac{|\beta_{\phi,\mathfrak{N}}|^2}{\alpha_{\phi,\mathfrak{N}}}\cdot \frac{|\Lambda(1/2,\pi\times\pi')|^2}{\mathbf{L}(\pi)}\cdot \mathcal{H}_{f_{\infty}}(\boldsymbol{\lambda}_{\pi_{\infty}})d\mu_{\pi}\\
=&c_{F,f_{\infty}}\cdot \mathcal{H}_{f_{\infty}}^{\heartsuit}(\boldsymbol{\lambda}_{\pi_{\infty}'})\cdot\Lambda(1,\pi',\Ad)|\mathfrak{N}|^n\log|\mathfrak{N}|+O(|\mathfrak{N}|^n),
\end{align*}
where $\alpha_{\phi,\mathfrak{N}}$ and $\beta_{\phi,\mathfrak{N}}$ are local factors (cf. \textsection\ref{8.5.1}),  $\mathbf{L}(\pi)$ is defined by \eqref{137}, $\mathcal{H}_{f_{\infty}}^{\heartsuit}(\boldsymbol{\lambda}_{\pi_{\infty}})$ (resp. $\mathcal{H}_{f_{\infty}}^{\heartsuit}(\boldsymbol{\lambda}_{\pi_{\infty}'})$)  is defined by \eqref{133..} (resp. \eqref{134..}), $\boldsymbol{\lambda}_{\pi_{\infty}}$ (resp. $\boldsymbol{\lambda}_{\pi_{\infty}'}$) is the Langlands parameter of $\pi_{\infty}$ (resp. $\pi_{\infty}'$), 
\begin{align*}
	c_{F,f_{\infty}}= n(n+1)2^{r_2(n-1)}\pi^{-r_2}|\mathfrak{D}_F|^{\frac{n}{2}}\underset{s=1}{\Res}\ \zeta_F(s),
\end{align*}	
and the implied constant in $O(|\mathfrak{N}|^n)$ depends on $F,$ $\pi'$ and $f_{\infty}.$ See Theorem \ref{M} in \textsection\ref{11.2}  for a detailed description. 
\end{thmx}


\begin{cor}[Corollary \ref{cor58}]\label{cor3.}
Let notation be as before. Let $\mathcal{D}$ be a large compact region in $\mathbb{C}^{n+1}.$ Let $\pi'\in\mathcal{A}_0([G'],\omega')$ be  everywhere unramified. Then 
\begin{equation}\label{242.}
\sum_{\substack{\pi\in\mathcal{A}_0(\mathfrak{N},\omega)\\
\boldsymbol{\lambda}_{\pi_{\infty}\in \mathcal{D}}}} \frac{|L(1/2,\pi\times\pi')|^2}{L(1,\pi,\Ad)}\ll_{\mathcal{D},\pi'} |\mathfrak{N}|^n\log|\mathfrak{N}|,
\end{equation}
where the implied constant depends on $\pi'$ and $\mathcal{D}.$ Here $\mathcal{A}_0(\mathfrak{N},\omega)$ denotes the set of cuspidal representations of $G(\mathbb{A}_F)$ of level $\mathfrak{N}$ and central character $\omega$ (cf. Definition \ref{defn57} in \textsection\ref{8.5.1}).
\end{cor}
\begin{remark}
Note that \eqref{242.} gives the average Lind\"{o}f bound in the level aspect as 
	$$\sum_{\substack{\pi\in\mathcal{A}_0(\mathfrak{N},\omega),\ 
\boldsymbol{\lambda}_{\pi_{\infty}\in \mathcal{D}}}}1\asymp _{\mathcal{D}} |\mathfrak{N}|^n.
	$$  
\end{remark}

\subsection{Nonvanishing of $L$-functions for $\mathrm{GL}(n+1)\times\mathrm{GL}(n)$ }\label{sec1.3}

Nonvanishing of Rankin-Selberg $L$-functions for $\mathrm{GL}(n+1)\times\mathrm{GL}(n)$ conveys  various important arithmetic information, e.g., Landau-Siegel zeros (cf. \cite{IS00a}), Langlands functorial lifts (cf. \cite{GJR04}), the Gan-Gross-Prasad conjecture (cf. \cite{Zha14b}, \cite{Zha14a}), Whittaker periods (cf. \cite{GH16}), and the Bloch-Kato conjecture (cf. \cite{LTXZZ22}). In many of these examples (e.g., \cite{LTXZZ22}), the base fields are at least CM fields, rather than the field $\mathbb{Q}$ of rational numbers. We aim to establish nonvnaishing results for these central $L$-values over general \textit{number fields}. 

Many nonvanishing results of the aforementioned type have been established in lower ranks ($n\leq 2$) over $\mathbb{Q}$, see for instance \cite{Duk95}, \cite{IS00a}, \cite{MRY22}, \cite{ST22} for a far from exhaustive list. Moreover, the existence of simultaneous nonvanishing of central $L$-values has been proved in a few of other cases in higher ranks, e.g., \cite{Tsu21} (where the base field is $\mathbb{Q}$) and \cite{JN21} (where the Ramanujan conjecture is assumed and the base field is totally real). In these cases the underlying family of cuspidal representations is \textit{infinite}. 

Taking advantage of the relative trace formula Theorem \ref{C}, we prove, for \textit{general $n$}, for the first time an \textit{unconditionally simultaneous} nonvanishing result on central Rankin-Selberg  $L$-values for $\mathrm{GL}(n+1)\times\mathrm{GL}(n)$ over \textit{number fields}. Furthermore, the underlying family of cuspidal representations is finite.


\begin{thmx}[Simultaneous Nonvanishing]\label{cor5}
Let $F$ be a number field. Let $\mathfrak{N}_1', \mathfrak{N}_2'\subseteq \mathcal{O}_F$ be ideals such that $\mathfrak{N}_1'\mathfrak{N}_2'\subsetneq \mathcal{O}_F,$ and $\mathfrak{N}_1'+\mathfrak{N}_2'=\mathcal{O}_F.$ Let $\pi_j'\in\mathcal{A}_0([G'],\omega_j')$ with arithmetic conductor $\mathfrak{N}_j',$ $j=1, 2.$ Suppose further that $\pi_{1,v_*}'\simeq \pi_{2,v_*}'$ at a finite split place $v_*.$    
\begin{enumerate}
	\item There are infinitely many $\pi\in\mathcal{A}_0([G],\omega)$ such that 
\begin{align*}
L(1/2,\pi\times\pi_1')L(1/2,\pi\times\pi_2')\neq 0.
\end{align*}

\item Suppose further that $\pi_{1,\infty}'\simeq \pi_{2,\infty}'.$ Then  $L(1/2,\pi\times\pi_1')L(1/2,\pi\times\pi_2')\neq 0$
for some $\pi\in\mathcal{A}_0(\mathfrak{N};\mathfrak{N}_1',\mathfrak{N}_2',\sigma_{v_*},\omega,\varepsilon).$ Here $\mathcal{A}_0(\mathfrak{N};\mathfrak{N}_1',\mathfrak{N}_2',\sigma_{v_*},\omega,\varepsilon)$ is a \textit{finite} set of cuspidal representations of $G(\mathbb{A}_F)$ (cf. Definition \ref{defn63} in \textsection\ref{9.2.1}). 

\end{enumerate}
See Theorem \ref{thm53.} for a detailed description. 
\end{thmx}
\begin{remark}
\begin{enumerate}
	\item[(i).] The nonvanishing of $L(1/2,\pi\times\pi_1')L(1/2,\pi\times\pi_2')$ for some $\pi\in\mathcal{A}_0([G],\omega)$ follows from a soft argument by using Theorem \ref{C}. 
	\item[(ii).] The second part of Corollary \ref{cor5} is more delicate since 
	$$
	\#\mathcal{A}_0(\mathfrak{N};\mathfrak{N}_1',\mathfrak{N}_2',\sigma_{v_*},\omega,\varepsilon)\ll_{\mathfrak{N}_1',\mathfrak{N}_2',v_*} |\mathfrak{N}|^{n+O(\varepsilon)}
	$$ 
	is an explicit  \textit{finite} family, while the set $\mathcal{A}_0([G],\omega)$ is infinite.

\end{enumerate}	
\end{remark} 

\begin{cor}[Special Case]
Let $\pi_1'$ and $\pi_2'$ be cuspidal representations of $\mathrm{GL}(n)/\mathbb{Q}$ with levels $N_1$ and $N_2,$ respectively. Suppose $N_1N_2>1,$ $(N_1, N_2)=1,$ and $\pi_{1,p}'\simeq \pi_{2,p}'$ for some $p\nmid N_1'N_2'.$ Let $\varepsilon>0$ be fixed. Then for all large prime  $N$ with $(N,pN_1N_2)=1,$ there exists a cuspidal representation $\pi$ of $\mathrm{GL}(n+1)/\mathbb{Q}$ with level $\asymp_{p, N_1, N_2} N,$ and $\|\boldsymbol{\lambda}_{\pi_{\infty}}\|\leq N^{\varepsilon},$ such that 
$$
L(1/2,\pi\times\pi_1')L(1/2,\pi\times\pi_2')\neq 0.
$$
Here $\asymp_{p, N_1, N_2}$ means the implied constant relies on $p,$ $N_1$ and $N_2,$ and $\|\boldsymbol{\lambda}_{\pi_{\infty}}\|$ is the Euclid norm of $\boldsymbol{\lambda}_{\pi_{\infty}}$ viewed as a point in $\mathbb{C}^{n+1}.$ 
\end{cor}



\subsection{Structure of This Paper}
In \textsection\ref{sec2} we introduce the framework of the relative trace formula and describe the regularization, which is  different from \cite{JLR99}. Through \textsection\ref{sec6} to \textsection\ref{sec5} we handle the geometric side: verifying the convergence in a cone in $\mathbb{C}^2$ and establishing meromorphic continuation outside. The spectral side is treated in \textsection\ref{sec.spec}, where we obtain absolute convergence in certain region and further derive meromorphic continuation. In \textsection\ref{sec10.} we show that the meromorphic parts from the spectral side and the geometric side match perfectly.  The relative trace formula is summarized in \textsection\ref{7.3}.  In \textsection\ref{sec9..} we specify the test function and gather estimates from preceding sections on both sides to obtain applications to the second moment evaluation and the nonvanishing problem (cf. \textsection\ref{11.2} and \textsection\ref{sec8}).

\subsection{Notation}\label{notation}
Let $F$ be a number field with ring of integers $\mathcal{O}_F.$ Let $[F:\mathbb{Q}]$ be the degree. Let $N_F$ be the absolute norm. Let $\mathfrak{O}_F$ be the different of $F.$ Let $\mathbb{A}_F$ be the adele group of $F.$ Let $\Sigma_F$ be the set of places of $F.$ Denote by $\Sigma_{F,\fin}$ (resp. $\Sigma_{F,\infty}$) the set of nonarchimedean (resp. archimedean) places. For $v\in \Sigma_F,$ we denote by $F_v$ the corresponding local field and $\mathcal{O}_{F_v}$ its ring of integers. Denote by $\widehat{\mathcal{O}_F}=\prod_{v<\infty}\mathcal{O}_{F_v}.$ For a nonarchimedean place $v,$ let  $\mathfrak{p}_v$ be the maximal prime ideal in $\mathcal{O}_{F_v}.$ Given an integral ideal $\mathcal{I},$ we say $v\mid \mathcal{I}$ if $\mathcal{I}\subseteq \mathfrak{p}_v.$ Fix a uniformizer $\varpi_{v}\in\mathfrak{p}_v$ of the discrete valuation ring $\mathcal{O}_{F_v}.$ Denote by $e_v(\cdot)$ the evaluation relative to $\varpi_v$ normalized as $e_v(\varpi_v)=1.$ Denote by $\mathbb{F}_v=\mathcal{O}_{F_v}/\mathfrak{p}_v$ the residue field with cardinality $\#\mathbb{F}_v=q_v.$ We use $v\mid\infty$ to indicate an archimedean place $v$ and write $v<\infty$ if $v$ is nonarchimedean. Let $|\cdot|_v$ be the norm in $F_v.$ Put $|\cdot|_{\infty}=\prod_{v\mid\infty}|\cdot|_v$ and $|\cdot|_{\fin}=\prod_{v<\infty}|\cdot|_v.$ Let $|\cdot|_{\mathbb{A}_F}=|\cdot|_{\infty}\otimes|\cdot|_{\fin}$. We will simply write $|\cdot|$ for $|\cdot|_{\mathbb{A}_F}$ in calculation over $\mathbb{A}_F^{\times}$ or its quotient by $F^{\times}$.   

Denote by $\Tr_F$ the trace map, extended to $\mathbb{A}_F\rightarrow \mathbb{A}_{\mathbb{Q}}.$ Let $\psi_{\mathbb{Q}}$ be the additive character on $\mathbb{Q}\backslash \mathbb{A}_{\mathbb{Q}}$ such that $\psi_{\mathbb{Q}}(t_{\infty})=\exp(2\pi it_{\infty}),$ for $t_{\infty}\in \mathbb{R}\hookrightarrow\mathbb{A}_{\mathbb{Q}}.$ Let $\psi=\psi_{\mathbb{Q}}\circ \Tr_F.$ Then $\psi(t)=\prod_{v\in\Sigma_F}\psi_v(t_v)$ for $t=(t_v)_v\in\mathbb{A}_F.$ For $v\in \Sigma_F,$ let $dt_v$ be the additive Haar measure on $F_v,$ self-dual relative to $\psi_v.$ Then $dt=\prod_{v\in\Sigma_F}dt_v$ is the standard Tamagawa measure on $\mathbb{A}_F$. Let $d^{\times}t_v=\zeta_{F_v}(1)dt_v/|t_v|_v,$ where $\zeta_{F_v}(\cdot)$ is the local Dedekind zeta factor. In particular, $\Vol(\mathcal{O}_{F_v}^{\times},d^{\times}t_v)=\Vol(\mathcal{O}_{F_v},dt_v)=N_{F_v}(\mathfrak{D}_{F_v})^{-1/2}$ for all finite place $v.$ Moreover, $\Vol(F\backslash\mathbb{A}_F; dt_v)=1$ and $\Vol(F\backslash\mathbb{A}_F^{(1)},d^{\times}t)=\underset{s=1}{\Res}\ \zeta_F(s),$ where $\mathbb{A}_F^{(1)}$ is the subgroup of ideles $\mathbb{A}_F^{\times}$ with norm $1,$ and $\zeta_F(s)=\prod_{v<\infty}\zeta_{F_v}(s)$ is the finite Dedekind zeta function. More properties of $dt$ and $d^{\times}t$ can be found in \cite{Lan94}, Ch. XIV.

Let $G=\mathrm{GL}(n+1)$ and $G'=\mathrm{GL}(n).$ Denote by $Z$ (resp. $Z'$) the center of $G$ (resp. $G'$). For $z'=\alpha I_n\in Z'(\mathbb{A}_F),$ $\alpha\in \mathbb{A}_F^{\times},$ we simply write $|z_v'|_v=|\alpha_v|_v,$ $v\in \Sigma_F.$ Let $G'^0(\mathbb{A}_F)$ be the subgroup of $G'(\mathbb{A}_F)$ consisting of $g'\in G'(\mathbb{A}_F)$ with $|\det g'|=1.$ Then the subgroup $Z'(\mathbb{A}_F)G'^0(\mathbb{A}_F)$ is open and has index $1$ (or $n-1$ if $F$ is a function field) in $G'(\mathbb{A}_F).$ So we may write $G'(\mathbb{A}_F)=Z'(\mathbb{A}_F)G'^0(\mathbb{A}_F).$ Let $\overline{G}=Z\backslash G$ and $\overline{G'}=Z'\backslash G'.$ We will identify $\overline{G'}$ with $G'^0$ as the conventional notation from Rankin-Selberg theory. For $x\in G(\mathbb{A}_F)$ or $G'(\mathbb{A}_F),$ we denote by $x^{\vee}$ the transpose inverse of $x.$ Fix the embedding from $G'$ to $G:$
\begin{align*}
	\iota:\ G'\longrightarrow G,\quad \gamma\mapsto \begin{pmatrix}
		\gamma&\\
		&1
	\end{pmatrix}.
\end{align*} 

For a matrix $g=(g_{i,j})\in G(\mathbb{A}_F),$ we denote by $E_{i,j}(g)=g_{i,j},$ the $(i,j)$-th entry of $g.$ For a vector $\mathbf{v}=(v_1, \cdots, v_m)$ denote by $E_i(\mathbf{v})=v_i,$ the $i$-th component of $\mathbf{v}.$ Let $m_1, m_2\in \mathbb{N}.$ We write $M_{m_1, m_2}$ for the group of $m_1\times m_2$ matrices.

For an algebraic group $H$ over $F$, we will denote by $[H]:=H(F)\backslash H(\mathbb{A}_F).$ We equip measures on $H(\mathbb{A}_F)$ as follows: for each unipotent group $U$ of $H,$ we equip $U(\mathbb{A}_F)$ with the Haar measure such that, $U(F)$ being equipped with the counting measure and the measure of $[U]$ is $1.$ We equip the maximal compact subgroup $K$ of $H(\mathbb{A}_F)$ with the Haar measure such that $K$ has total mass $1.$ When $H$ is split, we also equip the maximal split torus of $H$ with Tamagawa measure induced from that of $\mathbb{A}_F^{\times}.$

Let $\omega$ and $\omega_1, \omega_2$ be unitary idele class characters on $\mathbb{A}_F^{\times},$ which are trivial on $\mathbb{R}_+^{\times}.$ Denote by $\mathcal{A}_0\left([G],\omega\right)$ (resp. $\mathcal{A}_0\left([G'],\omega_i\right)$) the set of cuspidal representations on $G(\mathbb{A}_F)$ (resp. $G'(\mathbb{A}_F)$) with central character $\omega$ (resp. $\omega_i$), $i=1, 2.$ Let $\omega'=\omega_1\overline{\omega}_2.$

Let $B$ (resp. $B'$) be the group of upper triangular matrices in $G$ (resp. $G'$). Let $T_B$ (resp. $T_{B'}$) be the Levi subgroup of $B$ (resp. $B'$). Let $A=Z\backslash T_B$ and $A'=Z'\backslash T_{B'}.$ Let $N$ (resp. $N'$) be the unipotent radical of $B$ (resp. $B'$). Let $R'$ be the standard parabolic subgroup of $G'$ of type $(1,n-2,1),$ and $R_0'=Z'\backslash R'.$ Let $W$ be Weyl group of $G$ with respect to $(B,T_{B}).$ Let $\Delta=\{\alpha_{1,2},\alpha_{2,3},\cdots,\alpha_{n,n+1}\}$ be the set of simple roots, and for each simple root $\alpha_{k,k+1},$ $1\leq k\leq n,$ denote by $w_k$ the corresponding reflection. Explicitly, for each $1\leq k\leq n,$
\begin{align*}
	w_k=\begin{pmatrix}
		I_{k-1} &\\
		& S_2&\\
		&&I_{n-k}
	\end{pmatrix},\ \text{where $S_2=\begin{pmatrix}
			&1\\
			1&
		\end{pmatrix}$}.
\end{align*}
For $1\leq k\leq n-1,$ denote by $w_k'$ the unique element in $G'$ such that $\iota(w_k')=w_k.$ Then $w_k$'s generate the Weyl group $W'$ of $G'$ with respect to $(B',T_{B'}).$ Let $\widetilde{w}_j'=w_j'w_{j+1}'\cdots w_{n-1}',$ $1\leq j\leq n-1.$ Denote by $N_j'=N'/(\widetilde{w}_j 'N' \widetilde{w}_j'^{-1}).$ Denote by $\widetilde{w}_n'=I_n$ and $N_n'=I_n.$ Define the generic character $\theta$ on $[N]$ by setting $\theta(u)=\prod_{j=1}^{n}\psi(u_{j,j+1})$ for $u=(u_{i,j})_{1\leq i, j\leq n+1}\in N(\mathbb{A}_F).$ Let $\theta'=\theta\mid_{[\iota(N')]}$ be the generic character on $[N'].$

Let $P$ (resp. $P'$) be the standard parabolic subgroup of $G$ (resp. $G'$) of type $(n,1)$ (resp. $(n-1,1)$). Let $P_0=Z\backslash P$ (resp. $P_0'=Z'\backslash P'$). We will denote by $Q$ a general parabolic subgroup of $G.$

Let $Z'^T(\mathbb{A}_F)=\{z=\{\diag(a,a,\cdots, a)\in Z'(\mathbb{A}_F):\ T^{-1}\leq |a|\leq T\}.$ Let $[Z'^T]=\{z=\{\diag(a,a,\cdots, a)\in Z'(F)\backslash Z'(\mathbb{A}_F):\ T^{-1}\leq |a|\leq T\}.$ Denote by $\textbf{s}=(s_1, s_2)\in\mathbb{C}^2.$

Denote by $N_P$ (resp. $N_Q$) the unipotent radical of $P$ (resp. $Q$). Let $\mathcal{S}(\mathbb{A}_F^n)$ be  the space of Bruhat-Schwartz functions on $\mathbb{A}_F^n.$ Let $\Phi\in\mathcal{S}(\mathbb{A}_F^n)$ with Fourier transform $\widehat{\Phi}$ and let  $\omega''=|\cdot|^{in\alpha}$ be a unitary character of $Z'(\mathbb{A}_F).$ Define an Eisenstein series
\begin{equation}\label{277}
E(s,x;\Phi,\omega'')=\sum_{\delta\in P_0(F)\backslash \overline{G'}(F)}\int_{Z'(\mathbb{A}_F)}\Phi(z\eta\delta x)|\det zx|^s\omega''(z)d^{\times}z
\end{equation} 
on $[{G'}].$ Then $E(s,x;\Phi,\omega'')$ converges absolutely in $\Re(s)>1$ and admits a meromorphic continuation to $\mathbb{C},$ given by 
\begin{equation}\label{278}
E(s,x;\Phi,\omega'')=E_+(s,x;\Phi,\omega'')+E_+^{\wedge}(s,x;\Phi,\omega'')+E_{\Res}(s,x;\Phi,\omega''),
\end{equation}
where 
\begin{align*}
&E_{\Res}(s,x;\Phi,\omega''):=-\frac{\Phi(0)|\det x|^s}{n(s+i\alpha)}+\frac{\widehat{\Phi}(0)|\det x|^{s-1}}{n(s-1+i\alpha)}\\
&E_+(s,x;\Phi,\omega''):=\sum_{\delta\in P_0(F)\backslash \overline{G'}(F)}\int_{|z|\geq 1}\Phi(z\eta\delta x)|\det zx|^s\omega''(z)d^{\times}z,\\
&E_+^{\wedge}(s,x;\Phi,\omega''):=\sum_{\delta\in P_0(F)\backslash \overline{G'}(F)}\int_{|z|\geq 1}\widehat{\Phi}(z\eta\delta x)|\det zx|^{1-s}\omega''^{-1}(z)d^{\times}z.
\end{align*}
Moreover, $E_+(s,x;\Phi,\omega'')$ and $E_+^{\wedge}(s,x;\Phi,\omega'')$ converges absolutely for all $s.$ Define 
\begin{equation}\label{R}
\mathcal{R}:=\big\{\mathbf{s}=(s_1,s_2)\in\mathbb{C}^2:\ \Re(s_1),\ \Re(s_2)>-1/(n+1)\big\}.
\end{equation}

Define $K'=\prod_{v}K_v',$ where each $K_v'$ is a maximal compact subgroup of $G'(F_v)$ and $K_v'=G'(\mathcal{O}_{F_v})$ if $v<\infty.$

For $j\in \{1,2\},$ let $\pi_j'=\otimes_{v\in\Sigma_F}\pi_{j,v}'\in \mathcal{A}_0([G'],\omega')$ contain  everywhere unramified vectors. Choose $\phi_j'\in \pi_j'$ to be local new forms and $\langle\phi_j',\phi_j'\rangle=1.$ 

For a function $h$ on $G(\mathbb{A}_F),$ we define $h^*$ by assigning $h^*(g)=\overline{h({g}^{-1})},$ $g\in G(\mathbb{A}_F).$ Let $F_1(s), F_2(s)$ be two meromorphic functions. Write $F_1(s)\sim F_2(s)$ if there exists an \textit{entire} function $E(s)$ such that $F_1(s)=E(s)F_2(s).$ Denote by $\alpha\asymp \beta$ for $\alpha, \beta \in\mathbb{R}$ if there are absolute constants $c$ and $C$ such that $c\beta\leq \alpha\leq C\beta.$

Throughout, we follow the $\varepsilon$-convention: that is, $\varepsilon$ will always be positive number which can be taken as small as we like, but may differ from one occurrence to another. 

\textbf{Acknowledgements}
I am very grateful to Peter Sarnak for encouragement and numerous discussions. I would like to thank Dinakar Ramakrishnan,  Djordje  Mili\'{c}evi\'{c}, Paul Nelson, and Shouwu Zhang for their precise comments and valuable suggestions. 


\section{The Relative Trace Formula and Regularization}\label{sec2}

\subsection{An Elementary Relative Trace Formula on $G$}\label{sec2.1}
Consider a smooth function $f:$ $G(\mathbb{A}_F) \rightarrow \mathbb{C}$ which is left and right $K$-finite for a compact subgroup $K$ of $G(\mathbb{A}_F)$, transforms by a unitary character $\omega^{-1}$ of $Z\left(\mathbb{A}_F\right),$ and has compact support modulo $Z\left(\mathbb{A}_F\right).$ Denote by $\mathcal{H}(G(\mathbb{A}_F),\omega^{-1})$ the set of such functions. Each $f\in \mathcal{H}(G(\mathbb{A}_F),\omega^{-1})$ defines an integral operator 
\begin{equation}\label{p}
	R(f)\phi(g)=\int_{\overline{G}(\mathbb{A}_F)}f(g')\phi(gg')dg',
\end{equation}
on the space $L^2\left([G],\omega\right)$ of functions on $[G]$ which transform under $Z(\mathbb{A}_F)$ by $\omega$ and are square integrable on $[\overline{G}].$ This operator is represented by the kernel function
\begin{equation}\label{11}
	\K^{f}(g_1,g_2)=\sum_{\gamma\in \overline{G}(F)}f(g_1^{-1}\gamma g_2),\ \ g_1, g_2\in G(\mathbb{A}_F).
\end{equation}

It is well known that $L^2\left([G],\omega\right)$ decomposes into the direct sum of the space $L_0^2\left([G],\omega\right)$ of cusp forms and spaces $L_{\Eis}^2\left([G],\omega\right)$ and $L_{\Res}^2\left([G],\omega\right)$ defined using Eisenstein series and residues of Eisenstein series respectively. Then the kernel function $\K(g_1,g_2)$ splits up as
\begin{equation}\label{ker}
	\K^f(g_1,g_2)=\K_0^f(g_1,g_2)+\K_{\Eis}^f(g_1,g_2)+\K_{\Res}^f(g_1,g_2).
\end{equation}
Denote by $\K_{\ER}(g_1,g_2)=\K_{\Eis}^f(g_1,g_2)+\K_{\Res}^f(g_1,g_2).$ Explicitly, we have the spectral expansion 
\begin{equation}\label{2.}
	\K_0^f(g_1,g_2)=\sum_{\pi\in\mathcal{A}_0([G],\omega)}\sum_{\phi\in\mathcal{B}_{\pi}}\pi(f)\phi(g_1)\overline{\phi(g_2)},
\end{equation}
where $\mathfrak{B}_{\pi}$ is a family of orthonormal basis of $\pi.$ Since the test function $f$ is $K$-finite, the sum over $\phi\in\mathcal{B}_{\pi}$ is actually finite. We often omit the superscript $f$ if the test function is clear in the context.

Let $\pi$ (resp. $\pi'$) be a unitary cuspidal representation on $G(\mathbb{A}_F)$ (resp. $G'(\mathbb{A}_F)$). Let $\phi\in \pi$ and $\phi'\in \pi'.$ Then the period 
\begin{align*}
	\mathcal{P}(\phi,\phi')=\int_{[G']}\phi(\iota(x))\phi'(x)dx.
\end{align*}
is well defined. Moreover, by Rankin-Selberg theory, it is an integral representation of the central value $L(1/2, \pi\times\pi').$ Combining this with the spectral expansion \eqref{2.} of the cuspidal kernel function $\K_0(g_1,g_2)$ we then obtain 
\begin{equation}\label{cusp}
J_0(f)=\iint\K_0(\iota(x),\iota(y))\phi'(x)\overline{\phi'(y)}dxdy=\sum_{\pi}\sum_{\phi\in\mathcal{B}_{\pi}}\mathcal{P}(\phi,\phi')\overline{\mathcal{P}(R(f)\phi,\phi')},           
\end{equation}
where the double integral is taken over $x, y\in [G'],$ and $\pi\in\mathcal{A}_0([G],\omega).$ 

Since $\K_0(g_1,g_2)$ decreases rapidly on $[G']\times[G'],$ the left hand side of \eqref{cusp} converges absolutely. Hence switching of integrals in \eqref{cusp} is granted. By Rankin-Selberg theory, \eqref{cusp} gives roughly the second moment of central $L$-values on $\mathrm{GL}(n+1)\times \mathrm{GL}(n):$
\begin{equation}\label{59} 
	J_0(f)\approx \sum_{\pi} c_f(\pi,\pi')\big|L(1/2,\pi\times\pi')\big|^2,
\end{equation}
where $c_f(\pi,\pi')$ is certain constant depending on $f$ and $\pi, \pi'.$

To compute it, the usual strategy is to take advantage of \eqref{ker} to substitute $\K_0(g_1,g_2)=\K(g_1,g_2)-\K_{\ER}(g_1,g_2)$ into \eqref{cusp}, where $\K_{\ER}(g_1,g_2)=\K_{\Eis}(g_1,g_2)+\K_{\Res}(g_1,g_2),$ obtaining \textit{formally} that 
\begin{equation}\label{2}
	J_0(f)=J_{\Geo}(f)-J_{\ER}(f),
\end{equation}
where 
\begin{align*}
	J_{\Geo}(f)=&\int_{[G']}\int_{[G']}\K(\iota(x),\iota(y))\phi'(x)\overline{\phi'(y)}dxdy,\\
	J_{\ER}(f)=&\int_{[G']}\int_{[G']}\K_{\ER}(\iota(x),\iota(y))\phi'(x)\overline{\phi'(y)}dxdy.
\end{align*}

The basic idea of relative trace formula is to compute the RHS of \eqref{2}.

\begin{enumerate}
	\item[\textbf{Geom}:] Substituting \eqref{11} into \eqref{2} then the geometric side becomes
	\begin{equation}\label{10}
		J_{\Geo}(f)=\int_{[G']}\int_{[G']}\sum_{\gamma\in \overline{G}(F)}f(\iota(x)^{-1}\gamma \iota(y))\phi'(x)\overline{\phi'(y)}dxdy.
	\end{equation}
	One can formally write $J_{\Geo}(f)$ as a linear combination of integrals by some standard geometric manipulations:
	\begin{align*}
		J_{\Geo}(f)=\sum_{\gamma}\Vol([H_{\gamma}])\int_{H_{\gamma}(\mathbb{A}_F)\backslash G'(\mathbb{A}_F)\times G'(\mathbb{A}_F)}f(\iota(x)^{-1}\gamma \iota(y))\phi'(x)\overline{\phi'(y)}dxdy,
	\end{align*}
	where $\gamma$ runs over the representatives for $Z(F)\iota(G'(F))\backslash G(F)/\iota(G'(F)),$ and $H_{\gamma}$ is the stabilizer of $\gamma.$ We call the above integral relative to $x$ and $y$ a weighted orbital integral. For the sake of simplicity, we will just call them \textit{orbital integrals} as the ordinary ones in Jacquet's relative trace formula. 
	\medskip 
	
	\item[\textbf{Spec}:] The spectral contribution $J_{\ER}(f)$ can be written as direct integrals of Eisenstein series periods. Moreover, in a simple version of the relative trace formula (e.g., cf. \cite{RR05} and \cite{FW09}), one may choose certain suitable test function $f$ to annihilate the non-cuspidal contribution $\K_{\ER}(g_1,g_2)$ so that $J_{\ER}(f)=0.$ 
\end{enumerate}

However, many orbital integrals (e.g., the one relative to $\gamma\in \iota(G'(F))$) do not converge. To overcome these analytic barriers one needs certain regularization of the relative trace formula. One of the most powerful techniques to handle the convergence problem is Arthur's truncation process (cf. e.g., \cite{Art78} and \cite{Art80}). It is expected to construct modified truncations for $J_{\Geo}(f)$ and $J_{\ER}(f)$ so that the truncated distributions can be calculated. The regularization in \cite{IY15} for a single Rankin-Selberg period should be relevant. However, the desired truncations in the geometric and spectral sides have not been established.  

Instead of truncating the kernel functions $\K(x,y)$ and $\K_{\ER}(x,y),$ we propose an alternative approach to resolve the convergence problem and compute the RHS of \eqref{2} in terms of certain $L$-values modulo some tiny tails, which vanishes in many practical cases. Our regularization consists of four major steps: 
\begin{enumerate}
	\item[(\RNum{1}).] We consider certain deformation $J_0(f,\textbf{s}),$ which is a function of two complex variables $\textbf{s}=(s_1,s_2)\in\mathbb{C}^2,$ with the property that $J_0(f)=J_0(f,\textbf{s})$ at $\textbf{s}=(0,0).$ Apply an elementary truncation on the center $Z'(\mathbb{A}_F)$ to get a truncated function $J_0^T(f,\textbf{s}),$ with $J_0^T(f,\textbf{s})\rightarrow J_0(f,\textbf{s}),$ as $T\rightarrow \infty.$ Moreover, the corresponding geometric and spectral sides $J_{\Geo}^T(f,\textbf{s})$ and $J_{\ER}^T(f,\textbf{s})$ are \textit{well defined} (i.e., convergent) and satisfy
	\begin{align*}
		J_0^T(f,\textbf{s})=J_{\Geo}^T(f,\textbf{s})-J_{\ER}^T(f,\textbf{s}).
	\end{align*} 
	See \text\textsection \ref{sec4.2} below for details. 
		
\item[(\RNum{2}).] We will introduce a refined decomposition of $J_{\Geo}^T(f,\textbf{s})$ and $J_{\ER}^T(f,\textbf{s})$ by Fourier expansion so that each term will converge absolutely for all $\textbf{s}$. This manipulation allows us to move the singular part of $J_{\ER}^T(f,\textbf{s})$ into the geometric side $J_{\Geo}^T(f,\textbf{s}).$ See \text\textsection \ref{sec2.3} for details. We will show the remaining contribution $J_{\Spec}^{\Reg,T}(f,\textbf{s})$ from the spectral side has nice growth properties. 

\item[(\RNum{3}).] In \text\textsection \ref{sec2.4} we further decompose the geometric side, according to the Bruhat decomposition, as $J^{\Reg,T}_{\Geo,\sm}(f,\textbf{s})+J^{\Reg,T}_{\Geo,\bi}(f,\textbf{s})$, obtaining the regularized relative trace formula (see \eqref{reg}):
\begin{equation}\label{19'}
	J_{\Spec}^{\Reg,T}(f,\textbf{s})=J^{\Reg,T}_{\Geo,\sm}(f,\textbf{s})+J^{\Reg,T}_{\Geo,\bi}(f,\textbf{s}).
\end{equation}

\item[(\RNum{4}).] We then show \eqref{19'} holds for $\Re(s_1)\geq 1$ and $\Re(s_2)\geq 1$ after taking $T\rightarrow \infty,$ which gives an equality of holomorphic functions therein. The last key ingredient is to establish meromorphic continuation of both sides to obtain an equality between meromorphic functions over the region 
\begin{align*}
	\mathcal{R}:=\big\{\mathbf{s}=(s_1,s_2)\in\mathbb{C}^2:\ \Re(s_1)>-1/(n+1),\ \Re(s_2)>-1/(n+1)\big\} 
\end{align*}
(cf. \eqref{R} in \textsection\ref{notation}).
\end{enumerate}
We will complete the first three steps in this section, and the rest of this paper (i.e., \textsection\ref{sec6}--\textsection\ref{sec10.}) will be devoted to the step (\RNum{4}).
	
	\subsection{Regularization (\RNum{1}): Deformation and Truncation Process}\label{sec4.2} 
	We start with a general setting. Let $f\in\mathcal{H}(G(\mathbb{A}_F),\omega).$ Let $\pi_i'\in\mathcal{A}_0([G'],\omega_i),$ and $\phi_i'\in \pi_i',$ $1\leq i\leq 2.$ Denote by $\omega'=\omega_1\overline{\omega}_2.$ Let $\textbf{s}=(s_1,s_2)\in\mathbb{C}^2.$ Define
	\begin{equation}\label{j0}
		J_0(f,\textbf{s};\phi_1',\phi_2')=\int_{[{G'}]}\int_{[{G'}]}\K_0^f(\iota(x),\iota(y))\phi_1'(x)\overline{\phi_2'(y)}|\det x|^{s_1}|\det y|^{s_2}dxdy.
	\end{equation}
	where $\K_0^f$ is the cuspidal part of the kernel function relative to the test function $f$. Since $\K_0(x,y)$ decays rapidly on $[G']\times [G']$, the above integrals converge absolutely. Moreover, when $\phi_1'=\phi_2'=\phi'$ and $s_1=s_2=0,$ then $J_0(f,\textbf{s})$ reduces to $J_0(f)$ defined in \eqref{cusp} (cf. \textsection\ref{sec2.1}). 
	
	 In order to streamline notation, henceforth will we write $J_0(f,\textbf{s})$ for $J_0(f,\textbf{s};\phi_1',\phi_2'),$ omitting $\phi_1'$ and $\phi_2'.$ The same abbreviation applies to its counterparts. 
	
	To make the relative trace formula \eqref{2} well defined, we appeal to an elementary truncation on $Z'(\mathbb{A}_F),$ the center of $G'(\mathbb{A}_F)$. For $T>0,$ let $Z'^T(\mathbb{A}_F)=\{z=aI_n\in Z'(\mathbb{A}_F):\ T^{-1}\leq |a|\leq T\}.$ Denote by $[Z'^T]=\{z=aI_n\in [Z']:\ T^{-1}\leq |a|\leq T\}.$ Define $J_0^T(f,\textbf{s})$ to be 
	\begin{align*}
		\int_{[Z'^T]}\int_{[Z'^T]}\int_{[\overline{G'}]}\int_{[\overline{G'}]}\K_0(\iota(z_1x),\iota (z_1z_2y)&\phi_1'(x)\overline{\phi_2'(y)}\omega'(z_1)\overline{\omega_2(z_2)}\\
		&|\det z_1x|^{s_1}|\det z_1z_2y|^{s_2}dxdyd^{\times}z_1d^{\times}z_2.
	\end{align*}
	Since $J_0(f,\textbf{s})$ converges absolutely, then 
	\begin{align*}
		J_0(f,\textbf{s})=\lim_{T\rightarrow \infty}J_0^T(f,\textbf{s}).
	\end{align*}
	
	Define $J_{\ER}^T(f,\textbf{s})$ to be 
	\begin{align*}
		\int_{[Z'^T]}\int_{[Z'^T]}\int_{[\overline{G'}]}\int_{[\overline{G'}]}\K_{\ER}(\iota(z_1x),\iota (z_1z_2y)&\phi_1'(x)\overline{\phi_2'(y)}\omega'(z_1)\overline{\omega_2(z_2)}\\
		&|\det z_1x|^{s_1}|\det z_1z_2y|^{s_2}dxdyd^{\times}z_1d^{\times}z_2
	\end{align*}
	the non-cuspidal counterpart of $J_0^T(f,\textbf{s}).$ Similarly, define $J_{\Geo}^T(f,\textbf{s})$ to be 
	\begin{align*}
		\int_{[Z'^T]}\int_{[Z'^T]}\int_{[\overline{G'}]}\int_{[\overline{G'}]}\K(\iota(z_1x),\iota(z_1z_2y))&\phi_1'(x)\overline{\phi_2'(y)}\omega'(z_1)\overline{\omega_2(z_2)}\\
		&|\det z_1x|^{s_1}|\det z_1z_2y|^{s_2}dxdyd^{\times}z_1d^{\times}z_2.
	\end{align*}
	
	Note that $\K$ and $\K_{\ER}$ increase slowly on $\overline{G'}(\mathbb{A}_F)\times\overline{G'}(\mathbb{A}_F),$ and $\phi'$ decays rapidly on $[\overline{G'}].$ Then the above integrals are well defined and converge absolutely for fixed $T$ and $\mathbf{s}$ of fixed real parts.  Moreover, we have by \eqref{ker} that  
	\begin{equation}\label{8}
		J_{0}^T(f,\textbf{s})=J_{\Geo}^T(f,\textbf{s})-J_{\ER}^T(f,\textbf{s}).
	\end{equation}
	
	However, neither $J_{\Geo}^T(f,\textbf{s})$ nor $J_{\ER}^T(f,\textbf{s})$ would converge when $T\rightarrow \infty.$ We will show the singular part of $J_{\Geo}^T(f,\textbf{s})$ and $J_{\ER}^T(f,\textbf{s})$ cancel with each other; more precisely, when $\textbf{s}\in \mathbb{C}^2$ with $\Re(s_1)\gg 0$ and $\Re(s_2)\gg 0,$ both $J_{\Geo}^T(f,\textbf{s})$ and $J_{\ER}^T(f,\textbf{s})$ converge as  $T\rightarrow \infty.$ Our argument relies on a decomposition of kernel functions via certain Fourier expansion in the following subsection.

	\subsection{Regularization (\RNum{2}): Manipulation of Kernel Functions}\label{sec2.3}
	We will follow the observation in \cite{Yan19a} and \cite{Yan19} to rewrite the spectral expansion \eqref{ker} into certain variant form (see Lemma \ref{decom} in \text\textsection \ref{sec2.3}) by an iteration of Poisson summations. Our regularization starts with this spectral expansion and incorporates the idea of meromorphic continuation into the relative trace formula. 
	
	\subsubsection{Mirabolic Fourier Expansion}
	Assume $n\geq2.$ Let $\psi$ be a nontrivial additive character on $F\backslash \mathbb{A}_F.$ Then for any $\alpha=(\alpha_1,\cdots,\alpha_{n})\in F^{n},$ define a character $\psi_{\alpha}:$ $N(\mathbb{A}_F)\rightarrow \mathbb{C}$ by
	\begin{align*}
		\psi_{\alpha}(u)=\prod_{i=1}^{n}\psi\left(\alpha_iu_{i,i+1}\right),\quad \forall\ u=(u_{i,j})_{(n+1)\times (n+1)}\in N(\mathbb{A}_F).
	\end{align*}  
	Write $\psi_k=\psi_{(0,\cdots,0,1,\cdots,1)}$ (where the first $n+1-k$ components are $0$'s and the remaining $k$ components are $1$) and $\theta=\psi_{(1,\cdots,1)},$ the standard generic character used to define Whittaker functions. 
	
	For $1\leq k\leq n,$ let $B_{n+1-k}$ be the standard Borel subgroup (i.e. the subgroup consisting of nonsingular upper triangular matrices) of $\mathrm{GL}(n+1-k);$ let $N_{n+1-k}$ be the unipotent radical of $B_{n+1-k}.$ For any $i,$ $j\in\mathbb{N},$ let $M_{i, j}$ be
	the additive group scheme of $i\times j$-matrices. Define the unipotent radicals 
	$$
	N_{(k,1,\cdots,1)}=\bigg\{\begin{pmatrix}
		I_{k} & B\\
		& D\\
	\end{pmatrix}:\ B\in M_{k,n+1-k},\ D\in N_{n+1-k}\bigg\},\ \text{$1\leq k\leq n$ }.
	$$
	For $1\leq k\leq n,$ set the generalized mirabolic subgroups 
	\begin{align*}
		R_{k}&=\bigg\{\left(
		\begin{array}{cc}
			A&C\\
			0&B
		\end{array}
		\right):\ A\in \mathrm{GL}(k),\ C\in M_{k,n+1-k},\ B\in N_{n+1-k}
		\bigg\}.
	\end{align*}
	For $2\leq k\leq n,$ define subgroups $R_{k}^0$ of $R_k$ by
	\begin{align*}
	\Bigg\{\left(
		\begin{array}{ccc}
			A&B'&C\\
			0&a&D\\
			0&0&B
		\end{array}
		\right):\ A\in \mathrm{GL}(k-1),\  \begin{pmatrix}
		B'&C\\
		a&D
		\end{pmatrix}\in M_{k,n+2-k},\ B\in N_{n+1-k}
		\Bigg\}.
	\end{align*} 
	Also we define $R_0=R_1^0:=N_{(1,1,\cdots,1)}$ to be the unipotent radical of the standard Borel subgroup of $G=\mathrm{GL}(n+1).$ Then we have the mirabolic Fourier expansion:
	
	\begin{lemma}\label{Fourier}
Let $h$ be a continuous function on $P_0(F)\backslash G(\mathbb{A}_F).$ Then we have 
		\begin{equation}\label{fourier}
			h(x)=\sum_{k=1}^{n+1}\sum_{\delta_k\in R_{n+1-k}(F)\backslash R_{n}(F)}\int_{[N_{(n+1-k,1,\cdots,1)}]}h(u\delta_kx)\psi_{k-1}(u)du
		\end{equation}
		if the RHS converges absolutely and locally uniformly.
	\end{lemma}

	\subsubsection{Reformulation of the Relative Trace Formula}
	For $0\leq i, j\leq n,$ define 
	\begin{align*}
		\mathcal{F}_i\mathcal{F}_j\K(g_1,g_2)=&\sum_{\alpha_i\in R_{n+1-i}(F)\backslash R_{n}(F)}\sum_{\beta_j\in R_{n+1-j}(F)\backslash R_{n}(F)}\int_{[N_{(n+1-i,1,\cdots,1)}]}\\
		\qquad &\int_{[N_{(n+1-j,1,\cdots,1)}]}\K(u\alpha_ig_1,v\beta_jg_2)\psi_{i-1}(u)\psi_{j-1}(v)dudv.
	\end{align*}
	
	Denote by $\mathcal{F}_0\mathcal{F}_0\K(g_1,g_2)=\K(g_1,g_2),$ and
	\begin{align*}
		\mathcal{F}_i\mathcal{F}_0\K(g_1,g_2)=&\sum_{\alpha_i\in R_{n+1-i}(F)\backslash R_{n}(F)}\int_{[N_{(n+1-i,1,\cdots,1)}]}\K(u\alpha_ig_1,g_2)\psi_{i-1}(u)du,\\
		\mathcal{F}_0\mathcal{F}_j\K(g_1,g_2)=&\sum_{\alpha_j\in R_{n+1-j}(F)\backslash R_{n}(F)}\int_{[N_{(n+1-j,1,\cdots,1)}]}\K(g_1,v\beta_jg_2)\psi_{j-1}(v)dv.
	\end{align*}
	
	Similarly we can define $\mathcal{F}_i\mathcal{F}_j\K_{\ER}(g_1,g_2)$ by replacing $\K(g_1,g_2)$ with $\K_{\ER}(g_1,g_2)$ in the above definition. 
	
	\begin{lemma}\label{decom}
		Let notation be as before. Let $g_1, g_2\in G(\mathbb{A}_F).$ Then 
		\begin{align*}
			\K_0(g_1,g_2)=&\mathcal{F}_0\mathcal{F}_0\K(g_1,g_2)-\sum_{i=1}^{n}\mathcal{F}_i\mathcal{F}_0\K(g_1,g_2)-\sum_{j=1}^{n}\mathcal{F}_{0}\mathcal{F}_j\K(g_1,g_2)\\
			&\quad +\sum_{i=1}^{n}\sum_{j=1}^{n}\mathcal{F}_{i}\mathcal{F}_j\K(g_1,g_2)-\mathcal{F}_{n+1}\mathcal{F}_{n+1}\K_{\ER}(g_1,g_2).
		\end{align*}
	\end{lemma}
	\begin{proof}
		Let $\pi\in \mathcal{A}_0\left([G],\omega\right),$ and $\phi\in\pi,$ denote by  
		\begin{align*}
			W_{\phi}(x)=\int_{[N]}\phi(ux)\overline{\theta}(u)du,
		\end{align*}
		the Whittaker function associated to $\phi.$ Take the Fourier-Whittaker expansion:
		\begin{equation}\label{3'}
			\phi(x)=\sum_{\gamma\in N(F)\backslash P_0(F)}W_{\phi}(\gamma x)=\sum_{\gamma\in N(F)\backslash P_0(F)}\int_{[N]}\phi(u\gamma x)\overline{\theta}(u)du.
		\end{equation}
		
		It the follows from \eqref{3'} and the spectral expansion of $\K_0$ that 
		\begin{align*}
			\K_0(g_1,g_2)=\sum_{\gamma_1\in N(F)\backslash P_0(F)}\sum_{\gamma_2\in N(F)\backslash P_0(F)}\int_{[N]}\int_{[N]}\K_0(u\gamma_1 g_1,v\gamma_2g_2)\overline{\theta}(u)\theta(v)dudv.
		\end{align*}
		
		In conjunction with \eqref{ker} we can write $\K_0(g_1,g_2)$ as 
		\begin{equation}\label{5}
			\K_0(g_1,g_2)=\mathcal{F}_{n+1}\mathcal{F}_{n+1}\K(g_1,g_2)-\mathcal{F}_{n+1}\mathcal{F}_{n+1}\K_{\ER}(g_1,g_2).
		\end{equation}
		
		By Lemma \ref{Fourier} we have the expressions:
		\begin{align*}
			\mathcal{F}_0\mathcal{F}_j\K(g_1,g_2)=&\sum_{i=1}^{n+1}\mathcal{F}_i\mathcal{F}_j\K(g_1,g_2),\ \ 0\leq j\leq n+1,\\
			\mathcal{F}_i\mathcal{F}_0\K(g_1,g_2)=&\sum_{j=1}^{n+1}\mathcal{F}_i\mathcal{F}_j\K(g_1,g_2),\ \ 0\leq i\leq n+1,
		\end{align*}
		from which $\mathcal{F}_{n+1}\mathcal{F}_{n+1}\K(g_1,g_2)$ is equal to 
		\begin{align*}
			\mathcal{F}_0\mathcal{F}_0\K(g_1,g_2)-\sum_{i=1}^{n}\mathcal{F}_i\mathcal{F}_0\K(g_1,g_2)-\sum_{j=1}^{n}\mathcal{F}_{0}\mathcal{F}_j\K(g_1,g_2)+\sum_{i=1}^{n}\sum_{j=1}^{n}\mathcal{F}_{i}\mathcal{F}_j\K(g_1,g_2).
		\end{align*}
		Therefore, the lemma follows from \eqref{5} and the above equality.
	\end{proof}
\begin{remark}
When $n=1,$ Lemma \ref{decom} is a consequence of applying Poisson summation twice. 
\end{remark}
	Since $\mathcal{F}_0\mathcal{F}_0\K(g_1,g_2)=\K(g_1,g_2),$ then by Lemma \ref{decom}, $\K(g_1,g_2)-\K_0(g_1,g_2)-\mathcal{F}_{n+1}\mathcal{F}_{n+1}\K_{\ER}(g_1,g_2)$ is equal to 
	\begin{equation}\label{12}
		\sum_{i=1}^{n}\mathcal{F}_i\mathcal{F}_0\K(g_1,g_2)+\sum_{j=1}^{n}\mathcal{F}_{0}\mathcal{F}_j\K(g_1,g_2)-\sum_{i=1}^{n}\sum_{j=1}^{n}\mathcal{F}_{i}\mathcal{F}_j\K(g_1,g_2).
	\end{equation}
	Note that kernel functions in \eqref{12} are geometric, i.e., one has $\K$ rather than $\K_{\ER}$ in \eqref{12}. We then write the non-cuspidal kernel $\K_{\ER}(g_1,g_2)=\K(g_1,g_2)-\K_0(g_1,g_2)$ into the sum of the non-constant Fourier coefficient $\mathcal{F}_{n+1}\mathcal{F}_{n+1}\K_{\ER}(g_1,g_2)$ and the geometric kernel functions \eqref{12}. This feature is one of the main differences between our treatment of the spectral side and Arthur's truncation approach.

	\begin{defn}[$\mathcal{F}_{i,j}J_{\Geo}^T(f,\mathbf{s})$]\label{6}
		Let $0\leq i, j\leq n.$ Define $\mathcal{F}_{i,j}J_{\Geo}^T(f,\textbf{s})$ to be
		\begin{align*}
			\int_{[Z'^T]}\int_{[Z'^T]}\int_{[\overline{G'}]}\int_{[\overline{G'}]}&\mathcal{F}_{i}\mathcal{F}_{j}\K(\iota(z_1x),\iota(z_1z_2y))\phi_1'(x)\overline{\phi_2'(y)}\omega'(z_1)\overline{\omega_2(z_2)}\\
			&|\det z_1x|^{s_1}|\det z_1z_2y|^{s_2}dxdyd^{\times}z_1d^{\times}z_2.
		\end{align*}
	\end{defn}
	
	\begin{lemma}\label{lem6}
		Let notation be as before. Let $0\leq i, j\leq n.$ Then $\mathcal{F}_{i,j}J_{\Geo}^T(f,\textbf{s})\equiv 0$ unless $i, j\in \{0, 1\}.$
	\end{lemma}
	\begin{proof}
		Suffice it to show $\mathcal{F}_{i,j}J_{\Geo}^T(f,\mathbf{s})\equiv 0$ for $2\leq i\leq n$ and $0\leq j\leq n.$ By definition of $\mathcal{F}_{i,j}J_{\Geo}^T(f,\mathbf{s}),$ we only need to show 
		\begin{equation}\label{9}
			\int_{[\overline{G'}]}\int_{[\overline{G'}]}\mathcal{F}_{i}\mathcal{F}_{j}\K(\iota(z_1x),\iota(z_1z_2y))\phi_1'(x)\overline{\phi_2'(y)}|\det x|^{s_1}|\det y|^{s_2}dxdy=0.
		\end{equation}
		for $2\leq i\leq n$ and $0\leq j\leq n.$ Let 
		\begin{align*}
			R_{k}'&=\bigg\{\left(
			\begin{array}{cc}
				A&C\\
				0&B
			\end{array}
			\right):\ A\in GL_k,\ C\in M_{k\times(n-k)},\ B\in N_{n-k}
			\bigg\}.
		\end{align*}
		So $\iota(R_k')=R_k\cap \iota(G').$ Then we have 
		\begin{equation}\label{14}
			R_{n+1-i}(F)\backslash R_n(F)=\iota(R_{n+1-i}'(F)\backslash G'(F)).
		\end{equation}
		
		Substitute the definition of $\mathcal{F}_{i}\mathcal{F}_j\K(x,y)$ into the Definition \eqref{6}. Note that $i\geq 2,$ then by \eqref{14} the left hand side of \eqref{9} becomes
		\begin{align*}
			&\int_{[\overline{G'}]}\int_{[\overline{G'}]}\sum_{z\in Z'(F)}\sum_{\alpha_i}\phi'_1(x)\overline{\phi_2'(y)} h(\alpha_ix,z_1,z_2,z,y)|\det x|^{s_1}|\det y|^{s_2}dxdy\\
			=&\int_{[\overline{G'}]}\int_{R'_{n+1-i}(F)\backslash \overline{G'}(\mathbb{A}_F)}\sum_{z\in Z'(F)}\phi'_1(x)\overline{\phi_2'(y)}h(x,z_1,z_2,z,y)|\det x|^{s_1}|\det y|^{s_2}dxdy,
		\end{align*}
		where $\alpha_i$ ranges through $R_{n+1-i}'(F)\backslash \overline{G'}(F),$ and 
		\begin{align*}
			h(x,z_1,z_2,z,y):=\int_{[N_{(n+1-i,1,\cdots,1)}]}\mathcal{F}_0\mathcal{F}_j\K(u\iota(z_1zx),\iota(z_2y))\psi_{i-1}(u)du.
		\end{align*}
		
		Let $N_{n+1-i, i-1}'$ be the unipotent radical of the standard parabolic of $G'$ of type $(n+1-i, i-1).$ Then by a change of variable the function $h(x)=h(x,z_1,z_2,z,y)$ is left invariant under $N_{n+1-i, i-1}'(\mathbb{A}_F),$ i.e., $h(x,z_1,z_2,z,y)=h(vx,z_1,z_2,z,y)$ for $v\in N_{n+1-i, i-1}'(\mathbb{A}_F).$ Integrating over $[N_{n+1-i, i-1}']$ we then see the left hand side of \eqref{9} is equal to 
		
		\begin{equation}\label{15}
			\int_{[\overline{G'}]}\int\sum_{z\in Z'(F)}\mathcal{P}_i(\phi_1'(x))\overline{\phi_2'(y)}h(x,z_1,z_2,z,y)|\det x|^{s_1}|\det y|^{s_2}dxdy,
		\end{equation}
		where $x$ runs through the region $R'_{n+1-i}(F)\backslash \overline{G'}(\mathbb{A}_F)$ and 
		\begin{align*}
			\mathcal{P}_i(\phi_1'(x))=\int_{[N_{n+1-i, i-1}']}\phi_1'(vx)dv
		\end{align*}
		is the constant term of the Fourier expansion of $\phi'$ relative to $[N_{n+1-i, i-1}'].$ Since $\phi_1'$ is \textit{cuspidal}, $\mathcal{P}_i(\phi_1'(x))=0$ for $i\geq 2.$ Hence the desired equality follows from cuspidality of $\phi_1'$ and \eqref{15}.
	\end{proof}

	\begin{defn}[$J_{\Eis}^{\Reg,T}(f,\textbf{s})$]
		Let notation be as before. Let $J^{\Reg,T}_{\Eis}(f,\textbf{s})$ be 
		\begin{align*}
			\int_{[Z'^T]}\int_{[Z'^T]}\int_{[\overline{G'}]}\int_{[\overline{G'}]}&\mathcal{F}_{n+1}\mathcal{F}_{n+1}\K_{\ER}(\iota(z_1x),\iota(z_1z_2y))\phi_1'(x)\overline{\phi_2'(y)}\omega'(z_1)\overline{\omega_2(z_2)}\\
		&|\det z_1x|^{s_1}|\det z_1z_2y|^{s_2}dxdyd^{\times}z_1d^{\times}z_2,
		\end{align*}
	\end{defn}
Note that $J_{\Eis}^{\Reg,T}(f,\textbf{s})$ converges absolutely for fixed $T$ and $\textbf{s}\in\mathbb{C}^2$ since $\phi_1'$ and $\phi_2'$ decay rapidly on $[\overline{G'}].$ The subscript `Eis' indicates that only Eisenstein spectrum in the spectral decomposition of $\K_{\ER}$ will contribute to the integral. See \text\textsection \ref{sec.spec} for a detailed discussion. We then make the following definition.

	\begin{defn}[$J_{\Spec}^{\Reg,T}(f,\textbf{s})$]
		Let notation be as before. Define 
		\begin{align*}
			J_{\Spec}^{\Reg,T}(f,\textbf{s})=J_0^T(f,\textbf{s})+J_{\Eis}^{\Reg,T}(f,\textbf{s}),\ \ \textbf{s}\in\mathbb{C}^2,\ T>0.
		\end{align*}
	\end{defn}
	We put the superscript `$\Reg$' on $J_{\Spec}$ to indicate that $J_{\Spec}^{\Reg,T}(f,\textbf{s})$ is certain regularization of $J_0^T(f)+J_{\ER}^T(f).$ Note that since we write some part of $\K_{\ER}$ (i.e., \eqref{12}) into the geometric side, $J_{\Spec}^{\Reg,T}(f,\textbf{0})\neq J_0^T(f)+J_{\ER}^T(f),$ where $\textbf{0}=(0,0).$ Nevertheless, we always have $J_0^T(f,\textbf{0})=J_0^T(f),$ which is  the contribution from the cuspidal spectrum. 
	
	One of the major advantages of working with $J_{\Spec}^{\Reg,T}(f,\textbf{s})$ is that (cf. \text\textsection \ref{sec.spec}) 
	$$
	J_{\Spec}^{\Reg}(f,\textbf{s}):=\lim_{T\rightarrow \infty}J_{\Spec}^{\Reg,T}(f,\textbf{s})
	$$ converges absolute in $\Re(s_1)>1/2$ and $\Re(s_2)>1/2,$ and thus defines a holomorphic function of two variables $\textbf{s}$ therein. After some further investigation we will obtain a meromorphic continuation of $J_{\Spec}^{\Reg}(f,\textbf{s})$ to $\textbf{s}\in\mathcal{R}:=\{(s_1,s_2):\  \Re(s_1)>-1/(n+1),\ \Re(s_2)>-1/(n+1)\}$ and further to $\mathbb{C}^2$ by functional equation. This evades the obstruction that $J_0^T(f)+J_{\ER}^T(f)$ does not converge in general as $T\rightarrow \infty.$

	\subsection{Regularization (\RNum{3}): The Truncated Relative Trace Formula}\label{sec2.4}
	With the preparations in \text\textsection \ref{sec4.2} and \text\textsection \ref{sec2.3}, we deduce by Lemmas \ref{decom} and \ref{lem6} that 
	\begin{equation}\label{13}
		J_{\Spec}^{\Reg,T}(f,\textbf{s})=J_{\Geo}^T(f,\textbf{s})-\mathcal{F}_{1,0}J_{\Geo}^T(f,\textbf{s})-\mathcal{F}_{0,1}J_{\Geo}^T(f,\textbf{s})+\mathcal{F}_{1,1}J_{\Geo}^T(f,\textbf{s}).
	\end{equation}
	Recall that $\mathcal{F}_{i,j}J_{\Geo}^T(f,\textbf{s})$ is defined by Definition \ref{6} in \text\textsection \ref{sec2.3}. 
	
	Since $\mathcal{F}_{i,j}J_{\Geo}^T(f,\textbf{s})$ is defined from the whole kernel function $\K(x,y)$, the RHS of \eqref{13} could be regarded as a linear combination of four relative trace formulas. Typically one can unfold each $\mathcal{F}_{i,j}J_{\Geo}^T(f,\textbf{s})$ into a sum of orbital integrals according to the decomposition of $\iota(G'(F))\backslash \overline{G}(F)/\iota(G'(F)).$ However, this conventional approach is not compatible with our truncation on the center $Z'$ of $G'$ since some of the orbital integrals (e.g., the one relative to $\gamma=\Id$), which we call \textit{singular}, does not converge as $T\rightarrow \infty.$ We then propose another strategy to handle the RHS of \eqref{13} as follows. 
	
	Consider the Bruhat decomposition $\overline{G}(F)=P_0(F)\sqcup P_0(F)w_nP_0(F),$ where $w_n$ is the Weyl element corresponding to the mirabolic subgroup $P_0$ of $G.$ According to \eqref{11} we can write $\K(g_1,g_2)=\K_{\sm}(g_1,g_2)+\K_{\bi}(g_1, g_2),$ where 
	\begin{align*}
		\K_{\sm}(g_1,g_2)=\sum_{\gamma\in P_0(F)}f(g_1^{-1}\gamma g_2),\quad \K_{\bi}(g_1,g_2)=\sum_{\gamma\in P_0(F)w_nP_0(F)}f(g_1^{-1}\gamma g_2).
	\end{align*}
	
	\subsubsection{Orbital Integrals on the Small Cell}
	For $0\leq i, j\leq n+1,$ define 
	\begin{align*}
		\mathcal{F}_i\mathcal{F}_j\K_{\sm}(g_1,g_2)=&\sum_{\alpha_i\in R_{n+1-i}(F)\backslash R_{n}(F)}\sum_{\beta_j\in R_{n+1-j}(F)\backslash R_{n}(F)}\int_{[N_{(n+1-i,1,\cdots,1)}]}\\
		\qquad &\int_{[N_{(n+1-j,1,\cdots,1)}]}\K_{\sm}(u\alpha_ig_1,v\beta_jg_2)\psi_{i-1}(u)\psi_{j-1}(v)dudv.
	\end{align*}
	
	Similar to Definition \ref{6} we define, for $0\leq i, j\leq n+1,$ 
	\begin{align*}
		\mathcal{F}_{i,j}^{\sm}J_{\Geo}^T(f,\textbf{s})&:=\int_{[Z'^T]}\int_{[Z'^T]}\int_{[\overline{G'}]}\int_{[\overline{G'}]}\mathcal{F}_{i}\mathcal{F}_{j}\K_{\sm}(\iota(z_1x),\iota(z_1z_2y))\\
		&\phi_1'(x)\overline{\phi_2'(y)}\omega'(z_1)\overline{\omega_2(z_2)}|\det z_1|^{s_1}|\det z_1z_2|^{s_2}dxdyd^{\times}z_1d^{\times}z_2.
	\end{align*}
	
	We will call $\mathcal{F}_{i,j}^{\sm}J_{\Geo}^T(f,\textbf{s})$ orbitals integrals on the \textit{small} cell. 
	\begin{defn}[$J^{\Reg,T}_{\Geo,\sm}(f,\textbf{s})$]
		Let notation be as before. Define the corresponding regularized geometric side $J^{\Reg,T}_{\Geo,\sm}(f,\textbf{s})$ to be 
		\begin{align*}
			\mathcal{F}_{0,0}^{\sm}J_{\Geo}^T(f,\textbf{s})-\mathcal{F}_{1,0}^{\sm}J_{\Geo}^T(f,\textbf{s})-\mathcal{F}_{0,1}^{\sm}J_{\Geo}^T(f,\textbf{s})+\mathcal{F}_{1,1}^{\sm}J_{\Geo}^T(f,\textbf{s}).
		\end{align*}
	\end{defn}
	
	While not all constituents in $J^{\Reg,T}_{\Geo,\sm}(f,\textbf{s})$ converge as $T\rightarrow \infty,$ e.g, the integral $\mathcal{F}_{0,0}^{\sm}J_{\Geo}^T(f,\textbf{s})\approx T$ when $T$ is large, we will appeal to Fourier expansion (which in this case is in fact Poisson summation) in \text\textsection \ref{sec6} to show $J^{\Reg,T}_{\Geo,\sm}(f,\textbf{s})$ converges when $T\rightarrow \infty,$ under the assumption that $\Re(s_1)+\Re(s_2)>0.$ Moreover, it admits a meromorphic continuation to $\textbf{s}\in\mathbb{C}^2.$ 
	
	\subsubsection{Orbital Integrals on the Big Cell}\label{2.4.2}
	For $0\leq i, j\leq n+1,$ define 
	\begin{align*}
		\mathcal{F}_i\mathcal{F}_j\K_{\bi}(g_1,g_2)=&\sum_{\alpha_i\in R_{n+1-i}(F)\backslash R_{n}(F)}\sum_{\beta_j\in R_{n+1-j}(F)\backslash R_{n}(F)}\int_{[N_{(n+1-i,1,\cdots,1)}]}\\
		\qquad &\int_{[N_{(n+1-j,1,\cdots,1)}]}\K_{\bi}(u\alpha_ig_1,v\beta_jg_2)\psi_{i-1}(u)\psi_{j-1}(v)dudv.
	\end{align*}
	
	Similar to Definition \ref{6} we define, for $0\leq i, j\leq n+1,$ 
	\begin{align*}
		\mathcal{F}_{i,j}^{\bi}J_{\Geo}^T(f,\textbf{s})&:=\int_{[Z'^T]}\int_{[Z'^T]}\int_{[\overline{G'}]}\int_{[\overline{G'}]}\mathcal{F}_{i}\mathcal{F}_{j}\K_{\bi}(\iota(z_1x),\iota(z_1z_2y))\\
		&\phi_1'(x)\overline{\phi_2'(y)}\omega'(z_1)\overline{\omega_2(z_2)}|\det z_1x|^{s_1}|\det z_1z_2y|^{s_2}dxdyd^{\times}z_1d^{\times}z_2.
	\end{align*}
	
	We will call $\mathcal{F}_{i,j}^{\bi}J_{\Geo}^T(f,\textbf{s}),$ $0\leq i, j\leq n+1,$ orbitals integrals on the \textit{big} cell. We will show they
	converge as $T\rightarrow \infty,$ supposing $\Re(s_1)+\Re(s_2)>1.$ 

	Let notation be as before. Then we have the \textit{regularized relative trace formula}: 
	\begin{equation}\label{reg}
		J_{\Spec}^{\Reg,T}(f,\textbf{s})=J^{\Reg,T}_{\Geo,\sm}(f,\textbf{s})+J^{\Reg,T}_{\Geo,\bi}(f,\textbf{s}),
	\end{equation}
	where $\mathbf{s}=(s_1,s_2)\in\mathbb{C}^2,$ and $J^{\Reg,T}_{\Geo,\bi}(f,\textbf{s})$ is defined by 
	\begin{equation}\label{big}
		\mathcal{F}_{0,0}J^{\bi,T}_{\Geo}(f,\textbf{s})-\mathcal{F}_{1,0}J^{\bi,T}_{\Geo}(f,\textbf{s})-\mathcal{F}_{0,1}J^{\bi,T}_{\Geo}(f,\textbf{s})+\mathcal{F}_{1,1}J^{\bi,T}_{\Geo}(f,\textbf{s}).
	\end{equation}

	\subsection{Regularization (\RNum{4}): Remove the Truncation}\label{sec2.5}
According to the analysis in \textsection\ref{sec6}--\textsection\ref{sec10.} we will see that $J^{\Reg,T}_{\Geo,\sm}(f,\textbf{s})$ converges as $T\rightarrow \infty,$ in the range $\Re(s_1)+\Re(s_2)>0;$ $J^{\Reg,T}_{\Geo,\bi}(f,\textbf{s})$ converges in the range $\Re(s_1)+\Re(s_2)>1$ when $T\rightarrow \infty;$ and $J_{\Spec}^{\Reg,T}(f,\textbf{s})$ converges absolutely when $\Re(s_1)\gg 0$ and $\Re(s_2)\gg 0.$ Hence, one can take $T\rightarrow \infty$ in \eqref{reg} in the region $\Re(s_1)\gg 0$ and $\Re(s_2)\gg 0,$ obtaining
\begin{equation}\label{41}
J_{\Spec}^{\Reg}(f,\textbf{s})=J^{\Reg}_{\Geo,\sm}(f,\textbf{s})+J^{\Reg}_{\Geo,\bi}(f,\textbf{s}),
\end{equation} 
as holomorphic functions of $\mathbf{s}$ in the region $\Re(s_1)\gg 0$ and $\Re(s_2)\gg 0.$ 

In order to remove the truncation in \eqref{41}, the meromorphic continuation to $$\textbf{s}\in\mathcal{R}:=\{(s_1,s_2):\  \Re(s_1)>-1/(n+1),\ \Re(s_2)>-1/(n+1)\}$$ remains to be an involved task and will occupy \textsection\ref{sec6}--\textsection\ref{sec10.}.
	
\subsection{Test Functions and Automorphic Weights}\label{2.6}
Recall that $\pi_i'\in\mathcal{A}_0([G'],\omega_i),$ $1\leq i\leq 2.$ Let $S\subset\Sigma_{F}$ be a finite set containing the archimedean places such that $\pi_1'$ and $\pi_2'$ are unramified outside $v\in S.$ Let $\phi_i'\in \pi_i',$ $1\leq i\leq 2,$ be cusp forms which are right-$G'(\mathcal{O}_{F_v})$-invariant for all $v\notin S.$ Let $\mathfrak{N}\subsetneq\mathcal{O}_F$ be an ideal such that $\{v\in\Sigma_{F,\fin}:\ v\mid\mathfrak{N}\}\cap S=\emptyset.$ Write $|\mathfrak{N}|:=N_F(\mathfrak{N}),$ the absolute norm of $\mathfrak{N}.$ 

For $v\mid\mathfrak{N},$ let $K_{v}=K_0(\mathfrak{p}_{v}^{e_v(\mathfrak{N})})$ the Hecke congruence of level $\mathfrak{p}_{v}.$ When $v<\infty$ and $v\nmid\mathfrak{N},$ set $K_v=G(\mathcal{O}_{F_v}).$ For $v\mid \infty,$ set $K_v$ to be the maximal compact subgroup of $G(F_v).$ Let $K_0(\mathfrak{N})=\prod_vK_v.$ Define $K'=\prod_{v}K_v',$ where each $K_v'$ is a maximal compact subgroup of $G'(F_v)$ and $K_v'=G'(\mathcal{O}_{F_v})$ if $v<\infty.$

Let $\omega$ be a unitary Hecke character which is unramified outside $v\mid \mathfrak{N}.$ For $v\mid\mathfrak{N},$ define a function on $G(F_v),$ supported on $Z(F_v)\backslash K_v,$ by  
\begin{equation}\label{f_v}
f_v(z_vk_v)=\Vol(\overline{K_v})^{-1}\omega_v(z)^{-1}\omega_v(E_{n+1,n+1}(k_v))^{-1},\ z_v\in Z(F_v),\ k_v\in K_v,
\end{equation}
where $\overline{K_v}$ is the image of $K_v=K_0(\mathfrak{p}_{v}^{e_v(\mathfrak{N})})$ in $\overline{G}(F_v),$ and $E_{n+1,n+1}(k_v)$ is the $(n+1,n+1)$-th entry of $k_v\in K_v.$

Denote by $\mathcal{F}_S(\mathfrak{N},\omega^{-1})$ the space of linear combination of the functions $f=\otimes_vf_v,$ where  $f_v$ is defined by \eqref{f_v} if $v\mid\mathfrak{N};$ and for for $v\nmid\mathfrak{N},$ 
$$
f_v(g_v)=\int_{Z(F_v)}\tilde{f}_v(z_vg_v)\omega_v(z_v)d^{\times}z_v
$$  
with $\tilde{f}$ being defined as follows
\begin{itemize}
\item $v\in S,$ let $\tilde{f}_v$ be a continuous function on $G(F_v)$ with a compact support. 
\item $v\notin S$ and $v\nmid \mathfrak{N},$ set $\tilde{f}_v=\Vol(K_v)^{-1}\textbf{1}_{K_v}.$  
\end{itemize}

From now on we will take $f\in \mathcal{F}_S(\mathfrak{N},\omega^{-1}).$ The restriction $\mathfrak{N}\neq\mathcal{O}_F$ is not essential, but it simplifies the calculation of the orbital integral $J^{\Reg,T}_{\Geo,\bi}(f,\textbf{s}).$ See \text\textsection \ref{5.1} for details. In practice it won't spoil the generality of the relative trace formula.

	\section{Orbital Integrals Relative to the Small Cell}\label{sec6}
	In this section, we will show that $J^{\Reg,T}_{\Geo,\sm}(f,\textbf{s}),$ which is defined by 
	\begin{align*}
		\mathcal{F}_{0,0}^{\sm}J_{\Geo}^T(f,\textbf{s})-\mathcal{F}_{1,0}^{\sm}J_{\Geo}^T(f,\textbf{s})-\mathcal{F}_{0,1}^{\sm}J_{\Geo}^T(f,\textbf{s})+\mathcal{F}_{1,1}^{\sm}J_{\Geo}^T(f,\textbf{s}),
	\end{align*}
	converges when $T\rightarrow \infty,$ in the region $\Re(s_1)+\Re(s_1)>0.$
	
	Denote by $\eta=(0,\cdots,0,1)\in F^n.$ For $\textbf{x}=(x_1,\cdots, x_n)\in\mathbb{A}_F^n$ we set
	\begin{align*}
		\check{f}_P(\textbf{x};z_2,y)=\int_{N_P(\mathbb{A}_F)}f(u\iota(z_2y))\prod_{i=1}^n\overline{\psi}(x_iE_{i,n+1}(u))du,
	\end{align*}
where $E_{i,j}(u)$ is the $(i,j)$-th entry of the matrix $u,$ and $\psi$ is the additive character on $F\backslash\mathbb{A}_F$ defining the generic character $\theta.$ 	
	
	Then $\check{f}_P(\textbf{x};z_2,y)$ is a Schwartz-Bruhat function on $\mathbb{A}_F^n$ for all $z_2\in\mathbb{A}_F^{\times}$ and $y\in\overline{G'}(\mathbb{A}_F).$ Using $\check{f}_P(\textbf{x};z_2,y)$ as a smooth section we define the associated Eisenstein series $E(x,s;\check{f}_P(\cdot,z_2,y))$ to be 
	\begin{align*}
		E(x,s;\check{f}_P(\cdot;z_2,y))=	\sum_{\delta\in P_0'(F)\backslash \overline{G'}(F)}\int_{Z'(\mathbb{A}_F)}\check{f}_P(\eta z_1x;z_2,y)|\det z_1x|^{s}d^{\times}z_1,
	\end{align*}
	which converges absolutely when $\Re(s)>1.$

	\begin{prop}\label{prop14}
		Let notation be as before. Let $\textbf{s}=(s_1,s_2)$ be such that $\Re(s_1+s_2)>0.$ Then the limit 
		\begin{align*}
			J^{\Reg}_{\Geo,\sm}(f,\textbf{s}):=\lim_{T\rightarrow \infty}J^{\Reg,T}_{\Geo,\sm}(f,\textbf{s})
		\end{align*}
		exists and is equal to
		\begin{equation}\label{34}
		\int_{Z'(\mathbb{A}_F)}\int_{\overline{G'}(\mathbb{A}_F)}\mathcal{P}(\textbf{s};\phi_1',\phi_2',\check{f}_P(\cdot;z_2,y))\omega'(z_1)\overline{\omega_2(z_2)}|\det z_2y|^{s_2}dyd^{\times}z_2.
		\end{equation}
		where the period $\mathcal{P}(\textbf{s};\phi_1',\phi_2',\check{f}_P(\cdot;z_2,y))$ is defined by 
		\begin{align*}\label{37}
			\int_{[\overline{G'}]}\phi_1'(x)\overline{R(y)\phi_2'(x)}E(x,s_1+s_2+1;\check{f}_P(\cdot;z_2,y))dx
		\end{align*}
		is a Rankin-Selberg period representing the $L$-function $L(s_1+s_2+1,\pi_1'\times\widetilde{\pi}_2').$ In articular, $J^{\Reg}_{\Geo,\sm}(f,\textbf{s})$ extends to a meromorphic function such that 
		\begin{align*}
			J^{\Reg}_{\Geo,\sm}(f,\textbf{s})\sim \Lambda(s_1+s_2+1,\pi_1'\times\widetilde{\pi}_2').
		\end{align*}
	\end{prop}
	\begin{proof}
		By definition we have for  $g_1, g_2\in G(\mathbb{A}_F)$ that 
		\begin{align*}
			\mathcal{F}_0\mathcal{F}_1\K_{\sm}&(g_1,g_2)=\int_{[N_P]}\sum_{\gamma\in P_0(F)}f(g_1^{-1}\gamma vg_2)dv,\\
			\mathcal{F}_1\mathcal{F}_1\K_{\sm}&(g_1,g_2)=\int_{[N_P]}\int_{[N_P]}\sum_{\gamma\in P_0(F)}f(g_1^{-1}u^{-1}\gamma vg_2)dudv.
		\end{align*}
		
		Since $N_P$ is a normal subgroup of $P_0,$ then 
		$$
		\mathcal{F}_1\mathcal{F}_1\K_{\sm}(g_1,g_2)=\mathcal{F}_0\mathcal{F}_1\K_{\sm}(g_1,g_2),
		$$ 
		noting that the volume of $[N_P]=N_P(F)\backslash N_P(\mathbb{A}_F)$ is equal to $1.$ Therefore, 
		\begin{align*}
			\mathcal{F}_{0,1}^{\sm}J_{\Geo}^T(f,\textbf{s})=\mathcal{F}_{1,1}^{\sm}J_{\Geo}^T(f,\textbf{s}).
		\end{align*}
		
		So $J^{\Reg,T}_{\Geo,\sm}(f,\textbf{s})=\mathcal{F}_{0,0}^{\sm}J_{\Geo}^T(f,\textbf{s})-\mathcal{F}_{1,0}^{\sm}J_{\Geo}^T(f,\textbf{s}).$ Since $\mathcal{F}_0\mathcal{F}_0\K_{\sm}(g_1,g_2)$ is left $P_0(F)$-invariant as a function of $g_1\in G(\mathbb{A}_F).$ Then we can appeal to Lemma \ref{Fourier} to deduce that 
		\begin{equation}\label{16}
			\mathcal{F}_0\mathcal{F}_0\K_{\sm}(g_1,g_2)-\mathcal{F}_1\mathcal{F}_0\K_{\sm}(g_1,g_2)=\sum_{i=2}^{n+1}\mathcal{F}_i\mathcal{F}_0\K_{\sm}(g_1,g_2).
		\end{equation}
		
		A similar argument as in the proof of Lemma \ref{lem6} implies that $\mathcal{F}_{i,0}^{\sm}J_{\Geo}^T(f,\textbf{s})=0$ for $2\leq i\leq n.$ Therefore, it follows from \eqref{16} that 
		\begin{align*}
			J^{\Reg,T}_{\Geo,\sm}(f,\textbf{s})=\mathcal{F}_{n+1,0}^{\sm}J_{\Geo}^T(f,\textbf{s}),
		\end{align*}
		which is defined by 
		\begin{align*}
			\int_{[Z'^T]}\int_{[Z'^T]}&\int_{[\overline{G'}]}\int_{[\overline{G'}]}\sum_{\alpha\in N(F)\backslash P_0(F)}\int_{[N]}\K_{\sm}(u\alpha \iota(z_1x),\iota(z_1z_2y))\theta(u)du\\
			&\quad \phi'(x)\overline{\phi'(y)}\omega'(z_1)\overline{\omega_2(z_2)}|\det z_1x|^{s_1}|\det z_1z_2y|^{s_2}dxdyd^{\times}z_1d^{\times}z_2.
		\end{align*}
		
		Unfolding $\K_{\sm}$ and changing variables then $J^{\Reg,T}_{\Geo,\sm}(f,\textbf{s})$ becomes 
		\begin{align*}
	&\int_{Z'^T(\mathbb{A}_F)}\int_{Z'^T(\mathbb{A}_F)}\int_{\overline{G'}(\mathbb{A}_F)}\int_{N'(F)\backslash \overline{G'}(\mathbb{A}_F)}\sum_{\delta\in N_P(F)}\int_{[N]}f(\iota(z_1x)^{-1}u^{-1}\delta^{-1}\iota(z_1z_2y))\\
	&\qquad \theta(u)du\phi_1'(x)\overline{\phi_2'(y)}\omega'(z_1)\overline{\omega_2(z_2)}|\det z_1|^{s_1}|\det z_1z_2|^{s_2}dxdyd^{\times}z_2d^{\times}z_1,
		\end{align*}
		where we make use of the fact that $N(F)\backslash P_0(F)\simeq N'(F)\backslash G'(F).$
		
		Note that $N(\mathbb{A}_F)=N_P(\mathbb{A}_F)\times \iota(N'(\mathbb{A}_F)).$ Then $J^{\Reg,T}_{\Geo,\sm}(f,\textbf{s})$ is equal to 
		\begin{align*}
			&\int_{Z'^T(\mathbb{A}_F)}\int_{Z'^T(\mathbb{A}_F)}\int_{\overline{G'}(\mathbb{A}_F)}\int_{N'(F)\backslash \overline{G'}(\mathbb{A}_F)}\int_{N_P(\mathbb{A}_F)}\int_{[N']}f(\iota(z_1x)^{-1}v^{-1}u^{-1}\iota(z_1z_2y))\\
			&\quad \theta(v)dv\psi(E_{n,n+1}(u))du\phi_1'(x)\overline{\phi_2'(z_2y)}|\det z_1x|^{s_1}|\det z_1z_2y|^{s_2}dxdyd^{\times}z_2d^{\times}z_1.
		\end{align*}
		
		Changing variables, then the above integral becomes
		\begin{align*}
			&\int_{Z'^T(\mathbb{A}_F)}\int_{Z'^T(\mathbb{A}_F)}\int_{\overline{G'}(\mathbb{A}_F)}\int_{N'(\mathbb{A}_F)\backslash \overline{G'}(\mathbb{A}_F)}\int_{N_P(\mathbb{A}_F)}\int_{[N']}f(\iota(z_1x)^{-1}v^{-1}u^{-1}\iota(z_1z_2y))\\
			&\theta(v)\psi(E_{n,n+1}(u))duW_{\phi_1'}(x)\overline{\phi_2'(y)}|\det z_1x|^{s_1}|\det z_1z_2y|^{s_2}dxdy\omega'(z_1)\overline{\omega_2(z_2)}d^{\times}z_2d^{\times}z_1,
		\end{align*}
		where
		\begin{align*}
			W_{\phi'}(x):=\int_{[N']}\phi'(nx)\overline{\theta}(n)dn
		\end{align*}
		is the Whittaker function of $\phi'$ relative to the generic character ${\theta}.$ One more change of variables yields that the integral $J^{\Reg,T}_{\Geo,\sm}(f,\textbf{s})$ is equal to 
		\begin{align*}
			\int_{Z'^T(\mathbb{A}_F)}&\int_{Z'^T(\mathbb{A}_F)}\int_{\overline{G'}(\mathbb{A}_F)}\int_{N'(\mathbb{A}_F)\backslash \overline{G'}(\mathbb{A}_F)}\check{f}_P(\eta z_1x;z_2,y) W_{\phi_1'}(x)\overline{W_{\phi_2'}(xy)}\\
			&\qquad \qquad \qquad\omega'(z_1)\overline{\omega_2(z_2)}|\det z_1x|^{s_1+s_2+1}|\det z_2y|^{s_2}dxdyd^{\times}z_2d^{\times}z_1.
		\end{align*}

		Moreover, since $\Re(s_1)+\Re(s_2)>0,$ and $\check{f}_P$ is a Schwartz-Bruhat function on $G(\mathbb{A}_F),$ then by Rankin-Selberg theory, $\lim_{T\rightarrow\infty}J^{\Reg,T}_{\Geo,\sm}(f,\textbf{s})$ exists and is equal to the absolutely convergent integral
		\begin{equation}\label{25}
	\iiint\check{f}_P(\eta x;z_2,y)W_{\phi_1'}(x)\overline{W_{\phi_2'}(xy)}|\det x|^{s}\omega'(z_1)\overline{\omega_2(z_2)}|\det z_2y|^{s_2}dxdyd^{\times}z_2,
		\end{equation}
		where $s=s_1+s_2+1,$ $x$ ranges over $N'(\mathbb{A}_F)\backslash G'(\mathbb{A}_F),$ $y$ (resp. $z_2$) runs through $G'(\mathbb{A}_F)$ (resp. $Z'(\mathbb{A}_F)$). Thus, $J^{\Reg}_{\Geo,\sm}(f,\textbf{s})$ is given by \eqref{25} when $\Re(s_1)+\Re(s_2)>0,$ and converges absolutely therein.
		
		Since the Eisenstein series $E(x,1+s_1+s_2;\check{f}_P(\cdot;z_2,y))$ converges absolutely in the region $\Re(s_1)+\Re(s_2)>0,$ one can unfold the Eisenstein series to see that 
		$$ 
		\int_{Z'(\mathbb{A}_F)}\int_{\overline{G'}(\mathbb{A}_F)}\mathcal{P}(s;\phi',\check{f}_P(\cdot;z_2,y))dyd^{\times}z_2
		$$
		is equal to 
		\begin{align*}
			\int_{Z'(\mathbb{A}_F)}\int_{Z'(\mathbb{A}_F)}\int_{\overline{G'}(\mathbb{A}_F)}\int_{P_0'(F)\backslash \overline{G'}(\mathbb{A}_F)}&\check{f}_P(\eta z_1x;z_2,y){\phi_1'}(x)\overline{{\phi_2'}(xy)}\\
			&|\det z_1x|^{s_1+s_2+1}d^{\times}z_1dxdyd^{\times}z_2.
		\end{align*}
		Executing the Fourier expansions of $\phi_1'$ and $\phi_2'$ the above integral becomes \eqref{25}, proving \eqref{34}. 
		
		By Rankin-Selberg theory, the period $\mathcal{P}(\textbf{s};\phi_1',\phi_2',\check{f}_P(\cdot;z_2,y)$ is an integral representation for $\Lambda(s_1+s_2+1,\pi_1'\times\widetilde{\pi}_2').$ It admits a meromorphic continuation to $\mathbb{C}$ with the property that $\mathcal{P}(\textbf{s};\phi_1',\phi_2',\check{f}_P(\cdot;z_2,y)/\Lambda(s_1+s_2+1,\pi_1'\times\widetilde{\pi}_2')$ is entire. Moreover, $f$ is a compactly supported function on $\overline{G}(\mathbb{A}_F).$ Therefore, by \eqref{34} the function $J^{\Reg}_{\Geo,\sm}(f,\textbf{s})$ also admits a meromorphic continuation to $\mathbb{C}$ such that $J^{\Reg}_{\Geo,\sm}(f,\textbf{s})/\Lambda(s_1+s_2+1,\pi_1'\times\widetilde{\pi}_2')$ is entire. 
	\end{proof}

	Let $\textbf{v}_j'=\otimes_v\textbf{v}_{j,v}'\in \pi_j'=\otimes_{v}\pi_{j,v}'$ be the vector corresponding to the automorphic form $\phi_j'.$ Let $W_{\phi_j',v}$ be a local Whittaker functional associated to the vector $\textbf{v}_{j,v}'$ and the generic character induced from $\psi_v$ so that 
	$$
	W_{\phi_2'}(x)=\prod_{v\in\Sigma_F}W_{\phi_2',v}(x_v).
	$$
	
	Let $x_v\in G'(F_v).$ Write $\eta x_v=(x_{v,1}, x_{v,2}, \cdots, x_{v,n})$ for the $n$-th row of $x_v.$ Let $s_2\in \mathbb{C}.$ Define $m_v(x_v,s_2)$ to be
	\begin{align*}
		\int_{G'(F_v)}\int_{N_P(F_v)}f_v(u_v\iota(y_v))|\det(y_v)|_v^{s_2}
		\overline{W_{\phi_2',v}(x_vy_v)}\prod_{i=1}^n\overline{\psi}_v(x_{v,i}E_{i,n+1}(u_v))du_vdy_v.
	\end{align*}
	
	\begin{lemma}\label{lem10}
		Let notation be as above. Let $v\notin S.$ Then 
		\begin{equation}\label{110}
			m_v(x_v,s_2)=\frac{\overline{W_{\phi_2',v}(x_v)}}{\Vol(K_v)N_{F_v}(\mathfrak{O}_{F_v})^{n/2}}\cdot \textbf{1}_{(\mathfrak{O}_{F_v}^{-1})^n}(\eta x_v).
		\end{equation}
	\end{lemma}
	\begin{proof}
		By definition, $f_v(u_v\iota(y_v))$ is nonvanishing if and only if
		\begin{equation}\label{111}
			u_v\iota(y_v)\in Z(F_v)K_v,\ \text{i.e.},\ \lambda_vu_v\iota(y_v)\in K_v
		\end{equation} 
		for some $\lambda_v\in F_v^{\times}.$ So $\lambda_v\in \mathcal{O}_{F_v}^{\times},$ $y_v\in K_v'$ and $u_v\in N_P(\mathcal{O}_{F_v}).$ Hence,
		$$
		f_v(u_v\iota(y_v))=\frac{1}{\Vol(K_v)}\textbf{1}_{K_v'}(y_v)\cdot \textbf{1}_{N_P(\mathcal{O}_{F_v})}(u_v).
		$$
		One thus can compute the integral $m_v(x_v,s_2)$ as 
		\begin{align*}
			\int_{K_v'}\int_{N_P(\mathcal{O}_{F_v})}f_v(u_v\iota(y_v))|\det(y_v)|_v^{s_2}
			\overline{W_{\phi_2',v}(x_vy_v)}\prod_{i=1}^n\overline{\psi}(x_{v,i}E_{i,n+1}(u_v))du_vdy_v.
		\end{align*}
		Since $\overline{W_{\phi_2',v}(x_vy_v)}=\overline{W_{\phi_2',v}(x_v)}$ for $y_v\in K_v',$ then
		\begin{align*}
			m_v(x_v,s_2)=\frac{\overline{W_{\phi_2',v}(x_v)}}{\Vol(K_v)}\cdot \widehat{\textbf{1}_{\mathcal{O}_{F_v}^n}}(\eta x_v).
		\end{align*}
		Then \eqref{110} follows from a computation of the Fourier transform $\widehat{\textbf{1}_{\mathcal{O}_{F_v}^n}}.$
	\end{proof}
	
	Since \eqref{34} converges absolutely when $\Re(s_1)+\Re(s_2)>0,$ we can unfold the Eisenstein series $E(x,s_1+s_2+1;\check{f}_P(\cdot;z_2,y))$ and write 
	$$
	J^{\Reg}_{\Geo,\sm}(f,\textbf{s})=\prod_{v\in\Sigma_F}J^{\Reg}_{\Geo,\sm,v}(f_v,\textbf{s}),
	$$ 
	where
	\begin{align*}
		J^{\Reg}_{\Geo,\sm,v}(f_v,\textbf{s}):=\int_{N'(F_v)\backslash {G'}(F_v)}W_{\phi_1',v}(x_v)m_v(x_v,s_2)|\det x_v|_v^{s_1+s_2+1}dx_v.
	\end{align*}
	
	\begin{prop}\label{prop11'}
		Let notation be as before. Then 
		\begin{align*}
		\frac{J^{\Reg}_{\Geo,\sm}(f,\textbf{s})}{W_{\phi_1'}(I_n)\overline{W_{\phi_2'}(I_n)}}=\frac{N_F(\mathfrak{D}_F)^{n(s_1+s_2+\frac{1}{2})}\Lambda(s_1+s_2+1,\pi'_1\times\widetilde{\pi}_2')}{\Vol(K_0(\mathfrak{N}))}\prod_{v\in S}\mathcal{F}_{v}^{\natural}(f_v,\textbf{s};\phi_1',\phi_2'),
		\end{align*} 
		where 
		$$
		\mathcal{F}_{v}^{\natural}(f_v,\textbf{s};\phi_1',\phi_2')=\frac{J^{\Reg}_{\Geo,\sm,v}(f_v,\textbf{s})}{L_v(s_1+s_2+1,\pi'_{1,v}\times\widetilde{\pi}_{2,v}')W_{\phi_1',v}(I_n)\overline{W_{\phi_2',v}(I_n)}}
		$$
		is an entire function of $\textbf{s}.$
	\end{prop}
	\begin{proof}
		Let $v\notin S.$ By Lemma \ref{lem10}, $\mathcal{F}_{0,1}J^{\bi}_{\Geo,v}(f_v,\textbf{s})$ is equal to 
		\begin{align*}
			\frac{N_{F_v}(\mathfrak{O}_{F_v})^{-\frac{n}{2}}}{\Vol(K_v)}\int_{N'(F_v)\backslash {G'}(F_v)}W_{\phi_1',v}(x_v)\overline{W_{\phi_2',v}(x_v)}\cdot \textbf{1}_{(\mathfrak{O}_{F_v}^{-1})^n}(\eta x_v)|\det x_v|_v^{s_1+s_2+1}dx_v.
		\end{align*}
		
		By definition $\phi_1'$ and $\phi_2'$ are right $K_v'$-invariant, i.e., spherical at $v.$ Then we obtain by Rankin-Selberg theory, for $v\nmid \mathfrak{D}_F,$ that 
		\begin{align*}
			J^{\Reg,v}_{\Geo,\sm}(f_v,\textbf{s})=\frac{W_{\phi_1',v}(I_n)\overline{W_{\phi_2',v}(I_n)}}{\Vol(K_v)}\cdot L_v(s_1+s_2+1,\pi_{1,v}'\times\widetilde{\pi}_{2,v}').
		\end{align*}
		
		Now we assume $v\mid \mathfrak{D}_F.$ For $\Re(s)>1,$ define
		\begin{align*}
			I_v(s):=\int_{N'(F_v)\backslash {G'}(F_v)}W_{\phi_1',v}(x_v)\overline{W_{\phi_2',v}(x_v)}\cdot \textbf{1}_{(\mathfrak{O}_{F_v}^{-1})^n}(\eta x_v)|\det x_v|_v^{s}dx_v.
		\end{align*}
		
		Appeal to the Iwasawa decomposition $G'(F_v)=N'(F_v)T_{B'}(F_v)K_v'$ to write $x_v=n_v'\varpi_v^{\mathbf{m}}k_v'$ and $dx_v=dn_v'\delta_{B'}^{-1}(\varpi_v^{\mathbf{m}})dk_v',$ with $\mathbf{m}=(m_1,\cdots, m_n)\in \mathbb{Z}^n,$ where $\varpi_v^{\mathbf{m}}:=\diag(\varpi_v^{m_1},\cdots, \varpi_v^{m_n}).$ Then $I_v(s)$ is equal to
		\begin{align*}
			\int_{G'(\mathcal{O}_{F_v})}\sum_{\mathbf{m}\in\mathbb{Z}^n}W_{\phi_1',v}(\varpi_v^{\mathbf{m}})\overline{W_{\phi_2',v}(\varpi_v^{\mathbf{m}})}\cdot \textbf{1}_{(\mathfrak{O}_{F_v}^{-1})^n}(\eta \varpi_v^{\mathbf{m}}k_v')|\det (\varpi_v^{\mathbf{m}})|_v^{s}\delta_{B'}^{-1}(\varpi_v^{\mathbf{m}})dk_v'.
		\end{align*}
		
		Let $\mathbf{m}'=(m_1,\cdots, m_{n-1}, 0).$ Then we can rewrite the above expression as 
		\begin{equation}\label{200} 
			I_v(s)=\sum_{\mathbf{m}'\in\mathbb{Z}^{n-1}}W_{\phi_1',v}(\varpi_v^{\mathbf{m}'})\overline{W_{\phi_2',v}(\varpi_v^{\mathbf{m}'})}|\det (\varpi_v^{\mathbf{m}'})|_v^{s}\delta_{B'}^{-1}(\varpi_v^{\mathbf{m}'})\cdot l_v(s),
		\end{equation}
		where 
		\begin{align*}
			l_v(s)=\sum_{m\in \mathbb{Z}}\int_{K_v'}\textbf{1}_{(\mathfrak{O}_{F_v}^{-1})^n}(\eta \varpi_v^{m}k_v')q_v^{-nms}dk_v'
		\end{align*}
		
		Note that $\min\{e_v(E_{n,1}(k_v')),\cdots, e_v(E_{n,n}(k_v'))\}=0.$ Then 
		\begin{align*}
			l_v(s)=\sum_{m\geq -e_v(\mathfrak{D}_{F_v})}\int_{K_v'}\textbf{1}_{(\mathfrak{O}_{F_v}^{-1})^n}(\eta \varpi_v^{m}k_v')q_v^{-nms}dk_v',
		\end{align*}
		and $\textbf{1}_{(\mathfrak{O}_{F_v}^{-1})^n}(\eta \varpi_v^{m}k_v')\equiv 1$ for $m\geq -e_v(\mathfrak{D}_{F_v})$ and $k_v'\in K_v'.$ Therefore, 
		\begin{equation}\label{201}
			l_v(s)=\Vol(K_v')\sum_{m\geq -e_v(\mathfrak{D}_{F_v})}q_v^{-nms}=\frac{N_{F_v}(\mathfrak{D}_{F_v})^{ns}}{1-q_v^{-ns}}=N_{F_v}(\mathfrak{D}_{F_v})^{ns}\zeta_{F_v}(ns),
		\end{equation}
		where $\zeta_{F_v}$ is the local Dedekind zeta function attached to $F_v.$
		
		Appealing to Macdonald formula (cf. \cite{Mac79}) we then have 
		\begin{equation}\label{202}
			\sum_{\mathbf{m}'\in\mathbb{Z}^{n-1}}W_{\phi_1',v}(\varpi_v^{\mathbf{m}'})\overline{W_{\phi_2',v}(\varpi_v^{\mathbf{m}'})}|\det (\varpi_v^{\mathbf{m}'})|_v^{s}\delta_{B'}^{-1}(\varpi_v^{\mathbf{m}'})=\frac{L(s,\pi_{1,v}\times\widetilde{\pi}_{2,v})}{\zeta_{F_v}(ns)}.
		\end{equation}
		
		Then Proposition \ref{prop11'} follows from \eqref{200}, \eqref{201}, and \eqref{202}.
	\end{proof}

	\section{Orbital Integrals Relative to the Big Cell: \RNum{1}}\label{sec4}
	In this section we compute the orbital integrals $\mathcal{F}_{1,1}J^{\bi,T}_{\Geo}(f,\textbf{s}),$ $\mathcal{F}_{0,1}J^{\bi,T}_{\Geo}(f,\textbf{s})$ and $\mathcal{F}_{1,0}J_{\Geo}^{\bi,T}(f,\textbf{s})$ relative to the big cell $P_0(F)w_nP_0(F).$ Recall that their definitions are given in \textsection\ref{2.4.2}.

	\subsection{The Orbital Integral $\mathcal{F}_{1,1}J^{\bi,T}_{\Geo}(f,\textbf{s})$}
	We will make use of the cuspidality of $\phi'$ to show $\mathcal{F}_{1,1}J^{\bi,T}_{\Geo}(f,\textbf{s})$ vanishes identically for all $T>0$ and $\textbf{s}\in\mathbb{C}^2.$ 
	
	\begin{prop}
		Let notation be as before. Then for $\textbf{s}\in\mathbb{C}^2$ and $T>0,$
		\begin{equation}\label{17}
			\mathcal{F}_{1,1}J^{\bi,T}_{\Geo}(f,\textbf{s})=0.
		\end{equation}
	\end{prop}
	\begin{proof}
		Denote by $P'_0$ the mirabolic subgroup of $G'.$ Let $N_{P'}$ be the unipotent radical of $P_0'.$ By definition, 	
		\begin{align*}
			\mathcal{F}_1\mathcal{F}_1\K_{\bi}(g_1,g_2)=&\int_{[N_P]}\int_{[N_P]}\sum_{\gamma\in P_0(F)w_nP_0(F)}f(g_1^{-1}u^{-1}\gamma vg_2)dudv.
		\end{align*}
		
		By Bruhat decomposition there exists a set $\Phi(F)$ consisting of finitely many disjoint subgroups of $P_0(F)$ such that 
		\begin{align*}
			P_0(F)w_{n}P_0(F)=P_0(F)\times \{w_n\}\times \Phi(F)
		\end{align*}
		as a bijection of sets. So we can express $\mathcal{F}_{1,1}J^{\bi,T}_{\Geo}(f,\mathbf{s})$ as 
		\begin{align*}
			\int_{[Z'^T]}&\int_{Z'^T(\mathbb{A}_F)}\int_{[\overline{G'}]}\int_{\overline{G'}(\mathbb{A}_F)}\int_{N_P(\mathbb{A}_F)}\int_{[N_P]}\sum_{\beta\in \Phi(F)}f(\iota(z_1x)^{-1}u^{-1}w_n\beta v\iota(z_1z_2y))dudv\\
			&\qquad\phi_1'(x)\overline{\phi_2'(y)}\omega'(z_1)\overline{\omega_2(z_2)}|\det z_1x|^{s_1}|\det z_1z_2y|^{s_2}dxdyd^{\times}z_1d^{\times}z_2.
		\end{align*}
		
		Let $u'\in N_{P'}(\mathbb{A}_F).$ Observe that  $w_n\iota(u')w_n\in N_P(\mathbb{A}_F).$ So
		\begin{align*}
			&\int_{N_P(\mathbb{A}_F)}\int_{[N_P]}\sum_{\beta\in \Phi(F)}f(\iota(z_1x)^{-1}\iota(u')u^{-1}w_n\beta v\iota(z_1z_2y))dudv\\
=&\int_{N_P(\mathbb{A}_F)}\int_{[N_P]}\sum_{\beta\in \Phi(F)}f(\iota(z_1x)^{-1}u^{-1}w_n\beta v\iota(z_1z_2y))dudv.
		\end{align*}
		
		Substituting this into the definition of $\mathcal{F}_{1,1}J^{\bi,T}_{\Geo}(f,\textbf{s})$ and executing a change of variables we obtain $\mathcal{F}_{1,1}J^{\bi,T}_{\Geo}(f,\textbf{s})$  is equal to 
		\begin{align*}
			\int_{[Z'^T]}&\int_{Z'^T(\mathbb{A}_F)}\int_{[\overline{G'}]}\int_{\overline{G'}(\mathbb{A}_F)}\int_{N_P(\mathbb{A}_F)}\int_{[N_P]}\sum_{\beta\in \Phi(F)}f(\iota(z_1x)^{-1}u^{-1}w_n\beta v\iota(z_1z_2y))dudv\\
			&\qquad\phi_1'(u'x)\overline{\phi_2'(y)}\omega'(z_1)\overline{\omega_2(z_2)}|\det z_1x|^{s_1}|\det z_1z_2y|^{s_2}dxdyd^{\times}z_1d^{\times}z_2
		\end{align*}
		for an arbitrary $u'\in N_{P'}(\mathbb{A}_F).$ Integrating the above integral over $u'\in [N_{P'}]$ we conclude from the cuspidality of $\phi_1'$ that $\mathcal{F}_{1,1}J^{\bi,T}_{\Geo}(f,\textbf{s})\equiv 0.$
	\end{proof}
	
	\subsection{The Orbital Integral: $\mathcal{F}_{0,1}J^{\bi,T}_{\Geo}(f,\textbf{s})$}\label{sec5.2}
	By definition, 
	\begin{align*}
		\mathcal{F}_{0,1}J^{\bi,T}_{\Geo}(f,\textbf{s})=&\int_{[Z'^T]}\int_{[Z'^T]}\int_{[\overline{G'}]}\int_{[\overline{G'}]}\int_{[N_P]}\sum_{\gamma}f(\iota(z_1x)^{-1}\gamma u\iota(z_1z_2y))\\
		&\phi_1'(x)\overline{\phi_2'(y)}\omega'(z_1)\overline{\omega_2(z_2)}|\det z_1x|^{s_1}|\det z_1z_2y|^{s_2}dudxdyd^{\times}z_1d^{\times}z_2,
	\end{align*}
	where $\gamma$ ranges through $Z(F)\backslash P(F)w_nP(F).$ 
	
	Note that $\mathcal{F}_{0,1}J^{\bi,T}_{\Geo}(f,\textbf{s})$ converges absolutely for all $\textbf{s}$. We then do the double coset decomposition for $Z(F)\backslash P(F)w_nP(F)$ to handle $\mathcal{F}_{0,1}J^{\bi,T}_{\Geo}(f,\textbf{s}).$ 
	
	\begin{lemma}\label{lem13}
		Let notation be as before. Let $\gamma=\begin{pmatrix}
		I_{n-1}&\\
		&1&\\
		&1&1
	\end{pmatrix}.$ Then the set $\{w_n, \gamma\}$ is a complete set of representatives for $Z(F)\iota(G'(F))\backslash P(F)w_nP(F)/P(F).$
	\end{lemma}
	\begin{proof}
		Let $g\in P(F)w_nP(F).$ Suppose $g\notin \iota(G'(F))w_nP(F).$ Then there exists $x\in G'(F)$ and $p\in P(F)$ such that 
		\begin{align*}
			\iota(x^{-1})g=\begin{pmatrix}
				I_{n-2}&&\\
				&-1&1\\
				&&1
			\end{pmatrix}w_n\begin{pmatrix}
				I_{n-2}&&\\
				&1&1\\
				&&1
			\end{pmatrix}p=\gamma p,
		\end{align*}
	  Hence, $g=\iota(x)\gamma p\in \iota(G'(F))\gamma P(F).$ Now we need to check $w_n$ and $\gamma$ are not equivalent. Suppose for the sake of contrary that $\gamma\in \iota(G'(F))w_nP(F).$ Note that 
		\begin{equation}\label{103}
			\gamma=\begin{pmatrix}
				I_{n-1}&\\
				&-1&1\\
				&&1
			\end{pmatrix}w_n\begin{pmatrix}
				I_{n-1}&\\
				&1&1\\
				&&1
			\end{pmatrix}.
		\end{equation}
		We then get a contradiction by comparing the canonical Bruhat form of $\gamma$ and the elements in $\iota(G'(F))w_nP(F),$ which cannot be of the form \eqref{103}.
	\end{proof}
	
	According to Lemma \ref{lem13} we define $\mathcal{F}_{0,1}^{\bi,1}J_{\Geo}^T(f,\textbf{s})$ by
	\begin{align*}
		&\int_{[Z'^T]}\int_{[Z'^T]}\int_{[\overline{G'}]}\int_{[\overline{G'}]}\int_{[N_P]}\sum_{\gamma\in \iota(G'(F))w_nP(F)/Z(F)}f(\iota(z_1x)^{-1}\gamma u\iota(z_1z_2y))du\\
		&\qquad \phi_1'(x)\overline{\phi_2'(y)}\omega'(z_1)\overline{\omega_2(z_2)}|\det z_1x|^{s_1}|\det z_1z_2y|^{s_2}dxdyd^{\times}z_1d^{\times}z_2.
	\end{align*}
	
	\begin{lemma}\label{lem15}
		Let notation be as before. Then $\mathcal{F}_{0,1}^{\bi,1}J_{\Geo}^T(f,\textbf{s})\equiv 0.$
	\end{lemma}
	\begin{proof}
		A straightforward calculation shows that the stabilizer of $w_n$ is 
		\begin{align*}
			H_{w_n}(F)=\Bigg\{\begin{pmatrix}
				A&U\\
				&1\\
				&&1
			\end{pmatrix}\times \begin{pmatrix}
				A&&U\\
				&1\\
				&&1
			\end{pmatrix}\in \iota(G'(F))\times P_0(F)
			\Bigg\}.
		\end{align*}
		
		Therefore, $\mathcal{F}_{0,1}^{\bi,1}J_{\Geo}^T(f,\textbf{s})$ is equal to 
		\begin{align*}
		\int_{[Z'^T]}\int_{[Z'^T]}&\int_{H_{w_n}(F)\backslash (\overline{G'}(F)\times\overline{G'}(F))}\int_{[N_P]}f(\iota(z_1x)^{-1}w_n u\iota(z_1z_2y))du\\
			&\quad \phi_1'(x)\overline{\phi_2'(y)}\omega'(z_1)\overline{\omega_2(z_2)}|\det z_1x|^{s_1}|\det z_1z_2y|^{s_2}dxdyd^{\times}z_1d^{\times}z_2.
		\end{align*}
		
		Let $u'\in N_{P'}(\mathbb{A}_F).$ Then $w_n\iota(u')w_n\in N_P(\mathbb{A}_F).$ Hence $\mathcal{F}_{0,1}^{\bi,1}J_{\Geo}^T(f,\textbf{s})$ becomes
		\begin{align*}
		\int_{[Z'^T]}\int_{[Z'^T]}&\int_{H_{w_n}(F)\backslash (\overline{G'}(F)\times\overline{G'}(F))}\int_{[N_P]}f(\iota(z_1u'x)^{-1}w_n u\iota(z_1z_2y))du\\
			&\qquad \phi_1'(x)\overline{\phi_2'(y)}\omega'(z_1)\overline{\omega_2(z_2)}|\det z_1x|^{s_1}|\det z_1z_2y|^{s_2}dxdyd^{\times}z_1d^{\times}z_2.
		\end{align*}
		
		Change variables $u'^{-1}x\mapsto x$ we then rewrite $\mathcal{F}_{0,1}^{\bi,1}J_{\Geo}^T(f,\textbf{s})$ as 
		\begin{align*}
			\int_{[Z'^T]}\int_{[Z'^T]}&\int_{H_{w_n}(F)\backslash (\overline{G'}(F)\times\overline{G'}(F))}\int_{[N_P]}f(\iota(z_1x)^{-1}w_n u\iota(z_1z_2y))du\\
			& \phi_1'(u'x)\overline{\phi_2'(y)}\omega'(z_1)\overline{\omega_2(z_2)}|\det z_1x|^{s_1}|\det z_1z_2y|^{s_2}dxdyd^{\times}z_1d^{\times}z_2.
		\end{align*}
		for all $u\in N_{P'}(\mathbb{A}_F).$ Taking the integral relative to $u'$ through $[N_{P'}]$ we then obtain from the cuspidality of $\phi_1'$ that $\mathcal{F}_{0,1}^{\bi,1}J_{\Geo}^T(f,\textbf{s})=0.$ 
	\end{proof}
	
	Hence, by Lemma \ref{lem13} and \ref{lem15} we obtain 
	\begin{align*}
		\mathcal{F}_{0,1}J^{\bi,T}_{\Geo}(f,\textbf{s})=&\int_{[Z'^T]}\int_{[Z'^T]}\int_{[\overline{G'}]}\int_{[\overline{G'}]}\int_{[N_P]}\sum_{\delta}f(\iota(z_1x)^{-1}\delta u\iota(z_1z_2y))\phi_1'(x)\\
	&\overline{\phi_2'(y)}|\det z_1x|^{s_1}|\det z_1z_2y|^{s_2}dudxdyd^{\times}z_1\omega'(z_1)\overline{\omega_2(z_2)}d^{\times}z_2,
	\end{align*}
where $\delta$ ranges over $\iota(G'(F))\gamma P(F)/Z(F).$
	
	To handily investigate the analytic behavior of $\mathcal{F}_{0,1}J^{\bi,T}_{\Geo}(f,\textbf{s})$ we introduce some notation as follows. For $\textbf{x}\in M_{1,n}(\mathbb{A}_F),$ set $u(\textbf{x})=\begin{pmatrix}
			I_n&\\
			\textbf{x}&1
		\end{pmatrix}.$ Denote by 
	\begin{equation}\label{102}
		f_P^{\dag}(\textbf{x};z_2,y)=\int_{N_P(\mathbb{A}_F)}f(u(\textbf{x})n\iota(z_2y))dn.
	\end{equation}
	Then $f_P^{\dag}$ is a Schwartz-Bruhat function on $\mathbb{A}_F^n\times Z'(\mathbb{A}_F)\times \overline{G'}(\mathbb{A}_F).$ Define 
	\begin{equation}\label{274}
		E_{\gamma}(x,s;f_P^{\dag}(\cdot;z_2,y))=\sum_{\delta\in P_0(F)\backslash \overline{G'}(F)}\int_{Z'(\mathbb{A}_F)}f_P^{\dag}(\eta z_1\delta x;z_2,y)|\det z_1x|^{s}d^{\times}z_1.
	\end{equation}
Then $E_{\gamma}(x,s;f_P^{\dag}(\cdot;z_2,y))$ converges absolutely when $\Re(s)>1.$ Moreover, for fixed $z_2\in Z'(\mathbb{A}_F)$ and $y\in \overline{G'}(\mathbb{A}_F),$ the Eisenstein series $E_{\gamma}(x,s;f_P^{\dag}(\cdot;z_2,y))$ has a meromorphic continuation to the whole $s$-plane (with at most simple poles at $s\in\{0,1\}$) and admits the function equation connecting the evaluations at $s$ and $1-s.$

\begin{prop}\label{prop22}
Let notation be as before. Let $n\geq 2.$ Let $\textbf{s}=(s_1,s_2)\in\mathbb{C}^2$ be such that $\Re(s_1+s_2)>0.$ Then 
		\begin{align*}
			\mathcal{F}_{0,1}J^{\bi}_{\Geo}(f,\textbf{s}):=\lim_{T\rightarrow\infty} \mathcal{F}_{0,1}J^{\bi,T}_{\Geo}(f,\textbf{s})
		\end{align*}
		exists, and is equal to the absolutely convergent integral
		\begin{equation}\label{56}
			\int_{Z'(\mathbb{A}_F)}\int_{\overline{G'}(\mathbb{A}_F)}\mathcal{P}(\textbf{s};\phi_2', f_P^{\dag}(\cdot;z_2,y))\omega'(z_1)\overline{\omega_2(z_2)}|\det z_2y|^{s_2}
			dyd^{\times}z_2,
		\end{equation} 
		where $\mathcal{P}(\textbf{s};\phi_2', f_P^{\dag}(\cdot;z_2,y))$ is defined to be 
		\begin{equation}\label{240}
		\int_{[ \overline{G'}]}	\phi_1'(x)\overline{\phi_2'(xy)}E_{\gamma}(x,s_1+s_2+1;f_P^{\dag}(\cdot;z_2,y))dx,
		\end{equation}
which is a Rankin-Selberg convolution representing $\Lambda(1+s_1+s_2,\pi_1'\times\widetilde{\pi}_2').$ Moreover, 
$$
\mathcal{F}_{0,1}J^{\bi}_{\Geo}(f,\textbf{s})\ll  |\mathfrak{N}|^{-n\Re(s_1+s_2)}, \ \ \Re(s_1)+\Re(s_2)>0,
$$
where the implied constant relies on $\phi_1',$ $\phi_2'$ and $f_S.$
	\end{prop}
\begin{proof}
Recall $P_0'$ is the mirabolic subgroup of $G'.$ Let 
$$
\Delta \iota(P_0')=\Bigg\{\left(\begin{pmatrix}
	p'&\\
	&1
\end{pmatrix}, \begin{pmatrix}
	p'&\\
	&1
\end{pmatrix}
\right):\ \ p'\in P_0'\Bigg\}.
$$ 
 
Then the stabilizer of $\gamma$ is $H_{\gamma}=\Delta \iota(P_0').$ Thus, $\mathcal{F}_{0,1}J^{\bi,T}_{\Geo}(f,\textbf{s})$ is equal to 
\begin{align*}
	\int_{Z'^T(\mathbb{A}_F)}&\int_{Z'^T(\mathbb{A}_F)}\int_{\Delta P_0'(F)\backslash (\overline{G'}(\mathbb{A}_F)\times\overline{G'}(\mathbb{A}_F))}\int_{N_P(\mathbb{A}_F)}f(\iota(z_1x)^{-1}\gamma u\iota(z_1z_2y))\\
	&\qquad \phi_1'(x)\overline{\phi_2'(y)}\omega'(z_1)\overline{\omega_2(z_2)}|\det z_1x|^{s_1}|\det z_1z_2y|^{s_2}dudxdyd^{\times}z_1d^{\times}z_2. 
\end{align*}

By a change of variables we have
\begin{align*}
	\mathcal{F}_{0,1}^{\bi}&J_{\Geo}^T(f,\textbf{s})=\int_{\overline{G'}(\mathbb{A}_F)}\int_{Z'^T(\mathbb{A}_F)}\int_{N_P(\mathbb{A}_F)}\int_{Z'^T(\mathbb{A}_F)}\int_{P_0'(F)\backslash \overline{G'}(\mathbb{A}_F)}f(u(\eta z_1x) n\iota(z_2y))\\
	&\phi_1'(x)\overline{\phi_2'(xy)}\omega'(z_1)\overline{\omega_2(z_2)}|\det z_1x|^{s_1+s_2+1}|\det z_2y|^{s_2}dxd^{\times}z_1dnd^{\times}z_2dy,
\end{align*}
where we replace the variable $u\in N_P(\mathbb{A}_F)$ with $n$ to distinguish it from $u(\eta z_1x).$ 
 
Hence, $\mathcal{F}_{0,1}J^{\bi,T}_{\Geo}(f,\textbf{s})$ is majorized by 
\begin{align*}
\mathcal{J}:=&\int_{{G'}(\mathbb{A}_F)}\int_{N_P(\mathbb{A}_F)}\int_{Z'(\mathbb{A}_F)}\int_{P_0'(F)\backslash \overline{G'}(\mathbb{A}_F)}|f(u(\eta z_1x) n\iota(y))|\\
	&|\phi_1'(x){\phi_2'(xy)}||\det z_1x|^{\Re(s_1+s_2)+1}|\det y|^{\Re(s_2)}dxd^{\times}z_1dndy.
\end{align*}

We proceed to show that $\mathcal{J}$ converges. To start with, we characterize the range of $z_2,$ $y$ and $n$ that contributes to $f(u(\eta z_1x) n\iota(z_2y)).$ 
By Iwasawa decomposition we can reparametrize $z_1x$ by $z_1pk,$ $p\in[P_0']$ and $k\in K'.$ Fix $x$  (i.e., switch the integrals), we can write $n=\begin{pmatrix}
	I_n&k^{-1}u\\
	&1
\end{pmatrix},$ and $y=k^{-1}z_2an'k',$ where $a=\diag(a_1,\cdots, a_{n-1},1)\in A'(\mathbb{A}_F),$ $k'\in K',$ and $n'=\begin{pmatrix}
	u''&u'\\
	&1
\end{pmatrix}\in N'(\mathbb{A}_F).$ Then $\mathcal{J}$ is equal to 
\begin{align*}
&\int_{K'}\int_{K'}\int_{Z'(\mathbb{A}_F)}\int_{N'(\mathbb{A}_F)}\int_{A'(\mathbb{A}_F)}\int_{\mathbb{A}_F^n}\int_{Z'(\mathbb{A}_F)}\int_{[P_0']}\Big|f\left(\begin{pmatrix}
	I_n&\\
	\eta z_1&1
\end{pmatrix}\begin{pmatrix}
z_2an'k'&u\\
	&1
\end{pmatrix}\right)\Big|\\
&|\phi_1'(pk){\phi_2'(pan'k')}||\det z_1p|^{\Re(s_1+s_2)+1}dpd^{\times}z_1du|\det z_2a|^{\Re(s_2)}d^{\times}adn'd^{\times}z_2dkdk'.
\end{align*}

Writing $z_{1,\fin}=z'/z,$ where $z, z'\in \widehat{\mathcal{O}_{F}}-\{0\}$ and $z$ and $z'$ coprime. By local analysis in Lemmas \ref{lem18}, \ref{lem18.} and \ref{lem18'} below, we derive that 
\begin{equation}\label{118}
\mathcal{J}\ll\max_{y^*\in C}\int\iint\Bigg|\phi_1'(p){\phi_2'\left(p\begin{pmatrix}
z^{-1}I_{n-1}\\
&1
\end{pmatrix}y^*\right)}\Bigg|\frac{|\det p|^{A}d^{\times}z|z'|_{\fin}^{ns'}d^{\times}z'dp}{|z|_{\fin}^{A}\Vol(K_0(\mathfrak{N}))},
\end{equation}
where $p$ ranges through $[P_0'],$ $z$ ranges over $\widehat{\mathcal{O}_{F}}-\{0\}$ and $z$ is coprimes with $\mathfrak{N},$ $z'\in \mathfrak{N}\widehat{\mathcal{O}_{F}}-\{0\},$  $s'=1+\Re(s_1+s_2),$   $C$ is a compact set in $G'(\mathbb{A}_F)$ determined by $\supp\tilde{f},$ $A>0$ is a constant depending on $\mathbf{s}$ and $\supp\tilde{f}.$

Due to the rapid decay of cusp forms $\phi_1'$ and $\phi_2'$ (cf. e.g., \cite{MS12}), the RHS of \eqref{118} converges. In particular, $\mathcal{J}\ll |\mathfrak{N}|^{-ns'}/\Vol(K_0(\mathfrak{N}))\ll |\mathfrak{N}|^{-n\Re(s_1+s_2)}.$

Therefore, when $\Re(s_1)+\Re(s_2)>0,$ 
\begin{align*}
	\mathcal{F}_{0,1}J^{\bi}_{\Geo}(f,\textbf{s}):=\lim_{T\rightarrow\infty} \mathcal{F}_{0,1}J^{\bi,T}_{\Geo}(f,\textbf{s})
\end{align*}
exists and is equal to \eqref{56}. The rest of Proposition \ref{prop22} can be deduced from the Rankin-Selberg theory.
\end{proof}
\begin{remark}
One may write $\mathcal{F}_{0,1}J^{\bi}_{\Geo,\gamma(0)}(f,\textbf{s})$ as
\begin{align*}
	\int_{[ \overline{G'}]}\phi_1'(x)\overline{\phi'(x)}\iint_{{G'}(\mathbb{A}_F)}\sum_{\phi'\in\mathfrak{B}_{\pi_2'}}\overline{\langle\pi_2'(y)\phi_2',\phi'\rangle}E_{\gamma}(x,s';f_P^{\dag}(\cdot,s_2;z_2,y))dyd^{\times}z_2dx,
\end{align*}
from which it seems likely that $\mathcal{F}_{0,1}J^{\bi}_{\Geo,\gamma(0)}(f,\textbf{s})$ is the  product of a Rankin-Selberg period and a Godement-Jacquet integral (cf. \cite{GJ06}). However, local calculation reveals that $f_P^{\dag}(\textbf{x},s_2;z_2,y),$ as a function of $z_2y\in G'(\mathbb{A}_F)\subset M_{n\times n}(\mathbb{A}_F),$ cannot be extended to a Schwartz-Bruhat function on $M_{n\times n}(\mathbb{A}_F).$ In nonarchimedean places, there are delicate intertwining between $z_1,$ $y$ and $n.$ See the following lemmas for details.
\end{remark}

\begin{lemma}\label{lem18}
Let notation be as before. Let $v\mid\infty.$ Let $\mathbf{c}_v=(c_1,\cdots,c_n)\in F_v^n$ and $\mathfrak{u}_v=\transp{(u_1,\cdots,u_n)}\in F_v^n.$ Let $n_v=\begin{pmatrix}
		I_n&\mathfrak{u}_v\\
		&1
	\end{pmatrix}.$ Suppose that $f_v(u(\mathbf{c}_v) n_v\iota(y_v))\neq 0.$ Then 
$c_l\ll 1,$ $u_l\ll 1,$ $1\leq l\leq n,$ with the implied constants depending on $\supp\tilde{f}_v;$ and $y_v$ lies inside a compact set determined by $\supp\tilde{f}_v.$
\end{lemma}
\begin{proof}
By definition,  $f_v(u(\mathbf{c}_v) n_v\iota(y_v))\neq 0$ amounts to 
\begin{equation}\label{105.}
	\lambda\begin{pmatrix}
		I_n&\\
	\mathbf{c}_v&1
	\end{pmatrix}\begin{pmatrix}
	I_n&\mathfrak{u}_v\\
		&1
	\end{pmatrix}\begin{pmatrix}
	y_v&\\
	&1
\end{pmatrix}\in \supp \tilde{f}_v
\end{equation} 
for some $\lambda\in F_v^{\times}.$ Recall $\supp \tilde{f}_v$ is a compact subset of $G(F_v).$ As a consequence of \eqref{105.} and its inverse, we have, for $1\leq l\leq n,$ that 
\begin{align*}
\begin{cases}
|\lambda c_l|_v\ll 1,\ \ |\lambda|_v^{-1}\ll 1,\ \ |\lambda^{-1}c_l|_v\ll 1\\
|\lambda+\lambda c_lu_l|_v\ll 1,\ |\lambda u_l|_v\ll 1.
\end{cases}
\end{align*}

Hence, $|c_l|_v\ll 1,$ $|u_l|_v\ll 1,$ where the implied constants relies on $\supp\tilde{f}_v.$ So $\begin{pmatrix}
	I_n&\\
	\mathbf{c}_v&1
\end{pmatrix}\begin{pmatrix}
	I_n&\mathfrak{u}_v\\
	&1
\end{pmatrix}$ lies in a compact set, determined by $\supp\tilde{f}_v.$ Consequently, $\lambda\diag(y_v,1)$ lies in a compact set determined by $\supp\tilde{f}_v.$ In conjunction with that $|\lambda|_v^{-1}\ll 1$ we then deduce that $y_v$ lies in a compact set determined by $\supp\tilde{f}_v.$
\end{proof}

\begin{lemma}\label{lem18.}
Let notation be as before. Let $v\notin S.$ Let $z_v=\varpi_v^{l_0}I_n\in Z'(F_v)$ and $n_v=\begin{pmatrix}
	I_n&\mathfrak{u}_v\\
	&1
\end{pmatrix}\in N_P(F_v),$ where $\mathfrak{u}_v=\begin{pmatrix}
	u_{v,1}\\
	u_{v,2}
\end{pmatrix}\in M_{n,1}(F_v)$ with $u_{v,1}\in M_{n-1,1}(F_v).$ Suppose that $f_v(u(\eta z_v) n_v\iota(y_v))\neq 0.$ 
\begin{itemize}
\item If $v\mid\mathfrak{N},$ then $l_0\geq 1,$ $y_v\in K_v',$ and $u_v\in N_P(\mathcal{O}_{F_v});$ 
\item If $v\nmid\mathfrak{N},$ we have the following two possibilities:
\begin{enumerate}
\item[(a)] $l_0\geq 0,$ $y_v\in K_v',$ and $u_v\in N_P(\mathcal{O}_{F_v});$
\item [(b)] $l_0\leq -1,$ $y_v=\diag(\varpi_v^{-l_0}I_{n-1},\varpi_v^{-2l_0}),$ $u_{v,1}\in \varpi_v^{-l_0}\mathcal{O}_{F_v}^{n-1},$ and $u_{v,2}\in -\varpi_v^{-l_0}+\varpi_v^{-2l_0}\mathcal{O}_{F_v}.$
\end{enumerate}
\end{itemize}

\end{lemma}
\begin{proof}
By Iwasawa decomposition, $y_v=z_{2,v}a_vn_v'k_v,$ with $n'_v=\begin{pmatrix}
	u''&u'\\
	&1
\end{pmatrix}\in N'(F_v).$ Then $f_v(u(\eta z_v) n_v\iota(y_v))\neq 0$ implies that 
\begin{equation}\label{105}
		\lambda_v\begin{pmatrix}
			I_n&\\
			\eta z_{1,v}&1
		\end{pmatrix}\begin{pmatrix}
			z_{2,v}a_vn_v'&\mathfrak{u}_v\\
			&1
		\end{pmatrix}\in K_v,
	\end{equation} 
	for some $\lambda_v\in F_v^{\times},$ where $K_v$ is defined in \text\textsection \ref{2.6}. Write $z_{2,v}a_v=\diag(t_1,\cdots, t_{n-1},t_n),$ and $t_v'=\diag(t_1,\cdots, t_{n-1}).$ Then $e_v(t_1)=e_v(t_2)=\cdots=e_v(t_{n-1})=-e_v(\lambda_v)$ and $u_v''\in \mathrm{GL}(n-1,\mathcal{O}_{F_v}).$ Thus \eqref{105} becomes
	\begin{align*}
		\begin{pmatrix}
			I_{n-1}&\\
			&1\\
			& z_{1,v}&1
		\end{pmatrix}\begin{pmatrix}
			I_{n-1}&\lambda_vt_v'u_v'&\lambda_vu_{v,1}\\
			&\lambda_vt_{n}&\lambda_vu_{v,2}\\
			&&\lambda_v
		\end{pmatrix}\in K_v,
	\end{align*}
 So $\lambda_vt_v'u_v', \lambda_vu_{v,1}\in\mathcal{O}_{F_v}^{n-1},$ and
	\begin{equation}\label{191}
		\begin{pmatrix}
			1\\
			z_{1,v}&1
		\end{pmatrix}\begin{pmatrix}
			\lambda_vt_{n}&\lambda_vu_{v,2}\\
			&\lambda_v
		\end{pmatrix}\in \begin{pmatrix}
			I_{n-1}&\\
			&\mathrm{GL}(2,\mathcal{O}_{F_v})
		\end{pmatrix}\cap K_v,
	\end{equation}
	forcing that $e_v(t_n)+2e_v(\lambda_v)=0$ by considering the determinant. So \eqref{191} becomes
	\begin{equation}\label{192}
		\begin{pmatrix}
			1\\
			z_{1,v}&1
		\end{pmatrix}\begin{pmatrix}
			\lambda_v^{-1}&\lambda_vu_{v,2}\\
			&\lambda_v
		\end{pmatrix}\in K_{2,v},
	\end{equation}
	where $K_{2,v}$ is $\mathrm{GL}(2,\mathcal{O}_{F_v})$ when $v\nmid\mathfrak{N}$ and $K_{2,v}$ is the Hecke congruence subgroup of level $\mathfrak{q}_v^{e_v(\mathfrak{N})}$ if $v\mid\mathfrak{N}.$ 
	\begin{itemize}
		\item Suppose that $e_v(\lambda_v)=0.$ Then by \eqref{192} we have $z_{1,v}\in 
\varpi_v^{e_v(\mathfrak{N})}\mathcal{O}_{F_v}$ and  $u_{v,2}\in \mathcal{O}_{F_v}.$ Hence, $t_vn_v'\in G'(\mathcal{O}_{F_v}),$ $u_v\in \mathcal{O}_{F_v}^n.$ This gives the case (a).
		
		\item Suppose that $e_v(\lambda_v)\leq -1.$ Write $\lambda_v=\varpi_v^{-l},$ $l\geq 1.$ Then \eqref{192} becomes
		\begin{equation}\label{193}
			\begin{pmatrix}
				\varpi_v^l&u_{v,2}\varpi_v^{-l}\\
				\varpi_v^{l_0+l}&\varpi_v^{l_0-l}u_{v,2}+\varpi_v^{-l}
			\end{pmatrix}\in K_{2,v},
		\end{equation}
and we must have $n\nmid\mathfrak{N}$. By \eqref{193}, $l_0=-l,$ and $u_{v,2}\in -\varpi_v^{l}+\varpi_v^{2l}\mathcal{O}_{F_v}.$ Then $a_v=\diag(\varpi_v^{-l},\cdots, \varpi_v^{-l},1),$ $z_{2,v}=\varpi_v^{2l}I_n,$ and $n_v'\in N'(\mathcal{O}_{F_v});$ and 
		$$
		\begin{pmatrix}
			I_{n-1}&&u_{v,1}\\
			&1&u_{v,2}\\
			&&1
		\end{pmatrix}\in \begin{pmatrix}
			I_{n-1}&&\varpi_v^l\mathcal{O}_{F_v}^{n-1}\\
			&1&-\varpi_v^l+\varpi_v^{2l}\mathcal{O}_{F_v}\\
			&&1
		\end{pmatrix} \subseteq N_P(\mathcal{O}_{F_v}).
		$$
	This is the case (b).
	\end{itemize}
Therefore, Lemma \ref{lem18.} follows.
\end{proof}

\begin{lemma}\label{lem18'}
	Let notation be as before. Let $v\in S\cap\Sigma_{F,\fin}.$ Let $n_v=\begin{pmatrix}
		I_n&\mathfrak{u}_v\\
		&1
	\end{pmatrix}.$ Let $l_0\leq -1.$
Suppose that $f_v(u(\eta \varpi_v^{l_0}) n_v\iota(y_v))\neq 0$. Then 
\begin{equation}\label{119}
	\begin{cases}
	u_v\in\varpi_v^{-l_0}C_1\\
	y_v\in \diag(\varpi_v^{-l_0}I_{n-1},\varpi_v^{-2l_0}) C_2,
	\end{cases}
\end{equation}
where $C_1\subset M_{n,1}(F_v)$ and $C_2\subset G'(F_v)$ are compact sets determined by $\supp\tilde{f}_v.$ 
\end{lemma}
\begin{proof}
Let $d_v^{\pm}=\min_{1\leq i,j\leq n+1}\{e_v(E_{i,j}(x_v)):\ x_v\in\supp\tilde{f}_v\ \text{or}\ x_v^{-1}\in\supp\tilde{f}_v\}.$
By definition,  $f_v(\eta \varpi_v^{l_0}) n_v\iota(y_v))\neq 0$ amounts to 
\begin{equation}\label{112}
	\lambda\begin{pmatrix}
		I_n&\\
		\eta\varpi_v^{l_0}&1
	\end{pmatrix}\begin{pmatrix}
	y_v&\mathfrak{u}_v\\
		&1
	\end{pmatrix}\in \supp \tilde{f}_v
\end{equation} 
for some $\lambda\in F_v^{\times}.$ So $e_v(\lambda)\leq -d_v^-$ by considering the $(n+1,n+1)$-th entry of the inverse of \eqref{112}. Also, $-e_v(\lambda)+l_0\geq d_v^-.$ 

By Iwasawa decomposition, $y_v=t_vn_v'k_v',$ with $n'_v=\begin{pmatrix}
	u_v''&u_v'\\
	&1
\end{pmatrix}\in N'(F_v),$ and $t_v=\diag(t_1,\cdots, t_{n-1},t_n),$ and $t_v'=\diag(t_1,\cdots, t_{n-1}).$ Write $\mathfrak{u}_v=\begin{pmatrix}
	u_{v,1}\\
	u_{v,2}
\end{pmatrix}\in M_{n,1}(F_v)$ with $u_{v,1}\in M_{n-1,1}(F_v).$ Then \eqref{112} becomes
\begin{align*}
	\begin{pmatrix}
		I_{n-1}&\\
		&1\\
		& \varpi_v^{l_0}&1
	\end{pmatrix}\begin{pmatrix}
\lambda_vt_v'u_v''&\lambda_vt_v'u_v'&\lambda_vu_{v,1}\\
		&\lambda_vt_{n}&\lambda_vu_{v,2}\\
		&&\lambda_v
	\end{pmatrix}\in \supp \tilde{f}_v.
\end{align*}
So $\lambda_vt_v'u_v''$ and $\lambda_vu_{v,1}$ support in a compact set. In particular, $d_v^+\leq e_v(\lambda_v)+e_v(t_{l})\leq -d_v^-$, $1\leq l\leq n-1,$ leading to that $u_v'$ also lies in a compact set. So
\begin{align*}
	\begin{pmatrix}
I_{n-1}\\
&1\\
&\varpi_v^{l_0}&1
	\end{pmatrix}\begin{pmatrix}
I_{n-1}\\
&\lambda_vt_{n}&\lambda_vu_{v,2}\\
&&\lambda_v
	\end{pmatrix}\in \begin{pmatrix}
I_{n-1}\\
&\mathrm{GL}(2,F_v)
\end{pmatrix}\cap \supp \tilde{f}_v.
\end{align*}

Analyzing the above constraints we derive that 
\begin{align*}
\begin{cases}
d_v^++d_v^--2l_0\leq e_v(t_n)\leq -d_v^+-d_v^--2l_0\\
d_v^++l_0\leq e_v(\lambda)\leq -d_v^-+l_0,\ e_v(\lambda_v)+e_v(u_{v,2})\geq d_v^+\\
d_v^+\leq e_v(\lambda_v)+e_v(t_{l})\leq -d_v^-,\ 1\leq l\leq n-1.
\end{cases}
\end{align*}

So \eqref{119} follows.
\end{proof}

\subsection{Meromorphic continuation}
Now we show meromorphic continuation of $\mathcal{F}_{0,1}J^{\bi}_{\Geo,\gamma(0)}(f,\textbf{s}).$ By Poisson summation one can write $E_{\gamma}(x,s;f_P^{\dag}(\cdot,s_2;z_2,y))$ as 
\begin{align*}
E_{\gamma,+}(x,s;f_P^{\dag}(\cdot;z_2,y))+E_{\gamma,+}^{\wedge}(x,s;f_P^{\dag}(\cdot;z_2,y))+E_{\gamma,\Res}(x,s;f_P^{\dag}(\cdot;z_2,y)),
\end{align*} 
where 
\begin{align*}
E_{\gamma,+}(x,s;f_P^{\dag}(\cdot;z_2,y)):=\sum_{\delta}\int_{|z_1|\geq 1}f_P^{\dag}(\eta z_1\delta x;z_2,y)|\det z_1x|^{s}d^{\times}z_1,\\
E_{\gamma,+}^{\wedge}(x,s;f_P^{\dag}(\cdot;z_2,y)):=\sum_{\delta}\int_{|z_1|\geq 1}\hat{f}_P^{\dag}(\eta z_1\delta \transp{x}^{-1};z_2,y)\overline{\omega}'\omega'(z_1)|\det z_1x|^{1-s}d^{\times}z_1,\\
E_{\gamma,\Res}(x,s;f_P^{\dag}(\cdot;z_2,y)):=-\frac{f_P^{\dag}(\mathbf{0};z_2,y)|\det x|^s}{n(s)}+\frac{\hat{f}_P^{\dag}(\mathbf{0};z_2,y)|\det x|^{s-1}}{n(s-1)}.
\end{align*}
Here $\delta$ ranges over $P_0(F)\backslash \overline{G'}(F),$ $|z_1|\geq 1$ means $z_1\in Z'(\mathbb{A}_F)$ with $|z_1|\geq 1,$ and 
\begin{align*}
\hat{f}_P^{\dag}(\mathbf{x};z_2,y):=\int_{\mathbb{A}_F^n}\int_{N_P(\mathbb{A}_F)}f(u(\textbf{c})n\iota(z_2y))\theta(\mathbf{x}\cdot\transp{\mathbf{c}})dnd\mathbf{c}.
\end{align*}

Note that $E_{\gamma,+}(x,s;f_P^{\dag}(\cdot;z_2,y))$ and $E_{\gamma,+}^{\wedge}(x,s;f_P^{\dag}(\cdot;z_2,y))$ converge absolutely for all $(s_1,s_2)\in\mathbb{C}^2.$ Define  (at least formally) $\mathcal{F}_{0,1}J^{\bi}_{\Geo,+}(f,\textbf{s})$ to be
\begin{align*}
\int_{Z'(\mathbb{A}_F)}\int_{\overline{G'}(\mathbb{A}_F)}\int_{[ \overline{G'}]}	\phi_1'(x)\overline{\phi_2'(xy)}E_{\gamma,+}(x,s';f_P^{\dag}(\cdot;z_2,y))\overline{\omega}'(z_2)|\det z_2y |^{s_2}
	dxdyd^{\times}z_2,
\end{align*} 
where we identify $\overline{G'}$ with $G'^0.$

Likewise, define $\mathcal{F}_{0,1}J^{\bi,\wedge}_{\Geo,+}(f,\textbf{s})$ (resp. $\mathcal{F}_{0,1}J^{\bi}_{\Geo,\Res}(f,\textbf{s})$) in a similar way by inserting $E_{\gamma,+}^{\wedge}(x,s';f_P^{\dag}(\cdot;z_2,y))$ (resp. $E_{\gamma,\Res}(x,s';f_P^{\dag}(\cdot;z_2,y))$) instead of the Eisenstein series  $E_{\gamma,+}(x,s';f_P^{\dag}(\cdot;z_2,y))$. Then  (at least formally) 
\begin{equation}\label{260}
\mathcal{F}_{0,1}J^{\bi}_{\Geo}(f,\textbf{s})=\mathcal{F}_{0,1}J^{\bi}_{\Geo,+}(f,\textbf{s})+\mathcal{F}_{0,1}J^{\bi,\wedge}_{\Geo,+}(f,\textbf{s})+\mathcal{F}_{0,1}J^{\bi}_{\Geo,\Res}(f,\textbf{s}).
\end{equation}

By Proposition \ref{prop22}, both sides of \eqref{260} are well defined in $\Re(s_1)+\Re(s_2)>0.$ 
\begin{lemma}\label{lem19}
Let notation be as before. The integral $\mathcal{F}_{0,1}J^{\bi}_{\Geo,+}(f,\textbf{s})$ converges absolutely in $\mathbb{C}^2.$ In particular, $\mathcal{F}_{0,1}J^{\bi}_{\Geo,+}(f,\textbf{s})$ is entire, and satisfies the estimate
\begin{align*}
	\mathcal{F}_{0,1}J^{\bi}_{\Geo,+}(f,\textbf{s})\ll |\mathfrak{N}|^{-n\Re(s_1+s_2)},
\end{align*}
where the implied constant relies on $\textbf{s},$ $\phi_1',$ $\phi_2'$ and $f_S.$
\end{lemma}
\begin{proof}
The proof is similar to that of Proposition \ref{prop22}. 
Write $z_{1,\fin}=z'/z$ with $z, z'\in \widehat{\mathcal{O}_{F}}-\{0\}$ and $z$ and $z'$ coprime. From Lemma \ref{lem18} we have that $|z_{1,\infty}|_{\infty}\ll 1.$ So $|\det pz_1|\geq 1$ forces that $|z'|_{\fin}\gg |z|_{\fin}|\det p|^{-\frac{1}{n}}.$ It follows from \eqref{118} that $\mathcal{F}_{0,1}J^{\bi}_{\Geo,+}(f,\textbf{s})$ is 
\begin{align*}
\ll\max_{y^*\in C}\int_{[P_0']}\iint\Bigg|\phi_1'(p){\phi_2'\left(p\begin{pmatrix}
z^{-1}I_{n-1}\\
&1
\end{pmatrix}y^*\right)}\Bigg||\det p|^{A}\frac{d^{\times}z}{|z|_{\fin}^{A}}|z'|_{\fin}^{ns'}d^{\times}z'dp,
\end{align*}
where $z\in\widehat{\mathcal{O}_{F}}-\{0\}$ and $z$ is coprimes with $\mathfrak{N},$ $z'\in \mathfrak{N}\widehat{\mathcal{O}_{F}}-\{0\},$ satisfying that $|z'|_{\fin}\gg |z|_{\fin}|\det p|^{-\frac{1}{n}}.$ Hence,
$\mathcal{F}_{0,1}J^{\bi}_{\Geo,+}(f,\textbf{s})$ is 
\begin{align*}
\ll\frac{|\mathfrak{N}|^{-n\Re(s')}}{\Vol(K_0(\mathfrak{N}))}\max_{y^*\in C}\int_{[P_0']}\int_{\widehat{\mathcal{O}_{F}}-\{0\}}\Bigg|\phi_1'(p){\phi_2'\left(p\begin{pmatrix}
z^{-1}I_{n-1}\\
&1
\end{pmatrix}y^*\right)}\Bigg|\frac{|\det p|^{A+|s'|+\frac{1}{n}}d^{\times}zdp}{|z|_{\fin}^{A+n|s'|+1}},
\end{align*}
which converges due to the rapid decay of cusp forms.
\end{proof}
\begin{remark}
\begin{itemize}
\item Note that in Proposition \ref{prop22}, the convergence domain comes from that of the Eisenstein series. In Lemma \ref{lem19} we get entireness from the holomorphy of $E_{\gamma,+}(x,s;f_P^{\dag}(\cdot;z_2,y)).$ 

\item Instead of using \cite{MS12}, one can derive a more explicit majorization by unfolding the cusp forms into Whittaker functions (cf. \cite{Yan22}, \textsection 8). 
\end{itemize}
\end{remark}

We also need the following generalization of Lemma \ref{lem18.}.
\begin{lemma}\label{lem20}
	Let notation be as before. Let $v\notin S.$ Let $\mathbf{c}_v=(c_1,\cdots,c_n)\in F_v^n$ and $\mathfrak{u}_v=\transp{(u_1,\cdots,u_n)}\in F_v^n.$ Let $n_v=\begin{pmatrix}
		I_n&\mathfrak{u}_v\\
		&1
	\end{pmatrix}.$ 
	\begin{enumerate}
		\item[(i).] For $v\mid\mathfrak{N},$ one has $\min_{1\leq i\leq n}\{e_v(c_i)\}\geq e_v(\mathfrak{N}).$ 		
		\item[(ii).] Suppose $\min_{1\leq i\leq n}\{e_v(c_i)\}\geq 0.$ Then $f_v(u(\mathbf{c}_v) n_v\iota(y_v))\neq 0$ unless $y_v\in K_v',$ $\mathbf{c}_v\in\mathcal{O}_{F_v}^n$ and $\mathfrak{u}_v\in \mathcal{O}_{F_v}^n.$
       \item[(iii).] Suppose $e_v(c_{i_0})=\min_{1\leq i\leq n}\{e_v(c_i)\}=-l\leq -1.$ Then $f_v(u(\mathbf{c}_v) n_v\iota(y_v))\neq 0$ unless $u_i\in \varpi_v^l\mathcal{O}_{F_v},$ $1\leq i<n,$ and $u_n\in -c_{i_0}^{-1}+\varpi_v^{2l}\mathcal{O}_{F_v},$ and 
		$$
		y_v=w_{n-1}\cdots w_{i_0}\begin{pmatrix}
			\varpi_v^lI_{n-1}&\\
			&\varpi_v^{2l}
		\end{pmatrix}\begin{pmatrix}
			I_{n-1}&\\
			\mathbf{c}_v'&1
		\end{pmatrix}k_v',
		$$
		where $w_j$ is the Weyl element realized in $K'$ corresponding to the $j$-th simple root, $\mathbf{c}_v'\in  \varpi_v^{-l}M_{\mathfrak{u}_v\mathbf{c}_v}+\mathcal{O}_{F_v}^{n-1},$ with $M_{\mathfrak{u}_v\mathbf{c}_v}\in\mathcal{O}_{F_v}^{n-1}$ being the vector consisting of the $(n,i)$-th entry of $w_{i_0}\cdots w_{n-1}\mathfrak{u}_v\mathbf{c}_v,$ $1\leq i\leq n-1.$
	\end{enumerate}

\end{lemma}
\begin{proof}
	Suppose that $f_v(u(\mathbf{c}_v) n_v\iota(y_v))\neq 0,$ which amounts to	\begin{equation}\label{262}
	   \lambda_v\begin{pmatrix}
			I_n&\\
			\mathbf{c}_v&1
		\end{pmatrix}\begin{pmatrix}
			y_v&\mathfrak{u}_v\\
			&1
		\end{pmatrix}\in K_v,\ \text{i.e.}, \ \lambda_v\begin{pmatrix}
			y_v&\mathfrak{u}_v\\
			\mathbf{c}_vy_v&1+\mathbf{c}_v\mathfrak{u}_v
		\end{pmatrix}\in K_v
	\end{equation}
	for some $\lambda_v\in {\varpi_v^{\mathbb{Z}}}.$ Taking inverse of \eqref{262} gives 
	\begin{equation}\label{263}
		\lambda_v^{-1}\begin{pmatrix}
			y_v^{-1}+y_v^{-1}\mathfrak{u}_v\mathbf{c}_v&-y_v^{-1}\mathfrak{u}_v\\
			-\mathbf{c}_v&1
		\end{pmatrix}\in K_v.
	\end{equation}
	From \eqref{263} one deduces that $e_v(\lambda_v)\leq 0.$ We may write $\lambda_v=\varpi_v^{-l'},$ $l'\geq 0.$ 
	\begin{enumerate}
	\item Suppose $v\mid\mathfrak{N}.$ By Cramer's rule there exists $k_v\in K_v'$ such that 
	\begin{align*}
		k_v\begin{pmatrix}
			I_n&\\
			\mathbf{c}_v&1
		\end{pmatrix}k_v^{-1}=\begin{pmatrix}
			I_n&\\
			\eta\varpi_v^{\min_{1\leq i\leq n}\{e_v(c_i)\}}&1
		\end{pmatrix}.
	\end{align*} 
	Substituting this into \eqref{262} we then derive from Lemma \ref{lem18.} that (i) holds.
		\item Suppose $l'=0.$ Then by \eqref{263}, $\mathbf{c}_v\in\mathcal{O}_{F_v}^n,$ $y_v^{-1}\mathfrak{u}_v\in\mathcal{O}_{F_v}^n,$ and thereby $y_v^{-1}\in M_{n\times n}(\mathcal{O}_{F_v}).$ By \eqref{262} we have $y_v\in M_{n, n}(\mathcal{O}_{F_v}).$ Hence, $y_v\in K_v',$ $\mathbf{c}_v\in\mathcal{O}_{F_v}^n,$ and $\mathfrak{u}_v\in\mathcal{O}_{F_v}^n.$  
		
		\item Suppose $l'\geq 1.$ Since $K_v$ is a group, then analyzing the $(n+1)$-th row of \eqref{263} leads to $l'=-\max_{1\leq i\leq n}\{e_v(c_i)\}=-l.$ Let $i_0$ be the smallest index such that $e_v(c_{i_0})=-l.$ Parametrize $y_v$ by the Iwasawa form $w_{i_0}\cdots w_{n-1}y_v=t_vn_v'k_v',$ where $w_j$ is the Weyl element realized in $K'$ corresponding to the $j$-th simple root, $t_v=(\varpi_v^{l_1},\cdots, \varpi_v^{l_n}),$ $n_v'$ is a lower triangle uniponent radical of a Borel, and $k_v'\in K_v'.$ Considering the determinant and diagonal entries of \eqref{262} gives 
		\begin{equation}\label{264}
			\begin{cases}
				l_1+l_2+\cdots+ l_n=(n+1)l\\
				l_1\geq l,\ l_1\geq l,\ \cdots,\ l_{n}\geq l.
			\end{cases}
		\end{equation} 
		Investigating the $(n+1,n)$-th entry of \eqref{262} gives that $\lambda_vc_{i_0}\varpi_v^{l_n}\in\mathcal{O}_{F_v},$ namely, $l_n\geq 2l.$ In conjunction with \eqref{264} we then derive that $l_1=\cdots=l_{n-1}=l$ and $l_{n}=2l.$ Since the upper left $(n-1)\times(n-1)$-block lies in $M_{n-1, n-1}(\mathcal{O}_{F_v}),$ then the upper left $(n-1)\times(n-1)$-block of $n_v'$ lies in $M_{n-1,n-1}(\mathcal{O}_{F_v}).$ Hence, we may write $w_{i_0}\cdots w_{n-1}y_v=t_vn_v''k_v'',$ with $n_v''=\begin{pmatrix}
			I_{n-1}&\\
			\mathbf{c}_v'&1
		\end{pmatrix}$ with $\mathbf{c}_v'\in F_v^{n-1},$ $k_v''\in K_v'.$ Moreover we may assume $k_v''=1$ as $f_v$ is right $\iota(K_v')$-invariant. 
		
		Now we consider the upper left $n\times n$-block of \eqref{263}, i.e., 
		\begin{equation}\label{265}
			\begin{pmatrix}
				I_{n-1}&\\
				-\mathbf{c}_v'&1
			\end{pmatrix}\begin{pmatrix}
				I_{n-1}&\\
				&\varpi_v^{-l}
			\end{pmatrix}w_{i_0}\cdots w_{n-1}\cdot (I_n+\mathfrak{u}_v\mathbf{c}_v)\in M_{n, n}(\mathcal{O}_{F_v}).
		\end{equation}
		By the $(n+1)$-th column of \eqref{262} we have $e_v(u_i)\geq l,$ $1\leq i\leq n.$ Hence, $\mathfrak{u}_v\mathbf{c}_v\in M_{n\times n}(\mathcal{O}_{F_v}).$ Therefore, we deduce from \eqref{265} that
		\begin{align*}
			\mathbf{c}_v'-\varpi_v^{-l}M_{\mathfrak{u}_v\mathbf{c}_v}\in \mathcal{O}_{F_v}^{n-1},\ \ e_v(1+u_nc_{i_0})\geq l.
		\end{align*} 
	\end{enumerate}
	
	Hence Lemma \ref{lem20} follows. 
\end{proof}

We now proceed to investigate the analytic behavior of  $\mathcal{F}_{0,1}J^{\bi,\wedge}_{\Geo,+}(f,\textbf{s})$. 
\begin{prop}\label{prop19}
The function $\mathcal{F}_{0,1}J^{\bi,\wedge}_{\Geo,+}(f,\textbf{s})$ converges absolutely in the region 
\begin{equation}\label{R.}
\mathcal{R}:=\big\{\mathbf{s}=(s_1,s_2)\in\mathbb{C}^2:\ \Re(s_1),\ \Re(s_2)>-1/(n+1)\big\},
\end{equation}
and extends to a holomorphic function therein. Moreover, 
\begin{equation}\label{63.}
\mathcal{F}_{0,1}J^{\bi,\wedge}_{\Geo,+}(f,\textbf{s})=O(|\mathfrak{N}|^{2+n\Re(s_1+s_2)}),\ \ \textbf{s}\in\mathcal{R},
\end{equation}
where the implied constant relies on $\varepsilon,$ $\textbf{s},$ $\phi_1',$ $\phi_2'$ and $f_S.$
\end{prop}
\begin{proof}
Expanding the Eisenstein series $E_{\gamma,+}^{\wedge}(x,s';f_P^{\dag}(\cdot;I_n,y)),$ changing variables $x\mapsto \transp{x}^{-1},$ writing $x=pk,$ $p\in [P'_0]$ and $k\in K',$ $\mathcal{F}_{0,1}J^{\bi,\wedge}_{\Geo,+}(f,\textbf{s})$ is majorized by 
\begin{align*}
\mathcal{J}:=&\int_{K'}\int_{G'(\mathbb{A}_{F})}\int_{Z'(\mathbb{A}_F)}\int_{N_P(\mathbb{A}_F)}\Big|\int_{\mathbb{A}_F^n}f(u(\mathbf{c}) n\iota(y))\theta(\eta z_1k\transp{\mathbf{c}})d\mathbf{c}\Big|\\
&\int_{\substack{p\in [P_0']\\ |\det pz_1|\geq 1}}\big|\overline{\phi_1'(pk)}\phi_2'(pk\transp{y}^{-1})\big|dpdn|\det z_1|^{-\Re(s_1+s_2)}d^{\times}z_1|\det y|^{\Re(s_2)}dydk.
\end{align*}

Change variables $\mathbf{c}\mapsto k^{-1}\mathbf{c},$ $n\mapsto \iota(\transp{k})n\iota(\transp{k}^{-1}),$ $y\mapsto \transp{k}y$ to obtain 
\begin{align*}
\mathcal{J}=&\int_{K'}\int_{G'(\mathbb{A}_{F})}\int_{Z'(\mathbb{A}_F)}\int_{N_P(\mathbb{A}_F)}\Big|\int_{\mathbb{A}_F^n}f(u(\mathbf{c}) n\iota(y))\theta(\eta z_1\transp{\mathbf{c}})d\mathbf{c}\Big|dn\\
	&\int_{[P'_0]}\big|\overline{\phi_1'(pk)}\phi_2'(p\transp{y}^{-1})\big|\textbf{1}_{|\det pz_1|\geq 1}dp|\det z_1|^{-\Re(s_1+s_2)}d^{\times}z_1|\det y|^{\Re(s_2)}dydk.
\end{align*}
By orthogonality, $z_{1,\fin}$ supports in $Z'(\mathfrak{N}^{-1}\mathcal{I}),$ where $\mathcal{I}$ is a fractional ideal depending on $\supp f_S.$ Denote by $\mathcal{Z}=Z'(F_{\infty})Z'(\mathfrak{N}^{-1}\mathcal{I}).$ 

Let $\|\phi_2'\|_{\infty}$ be the sup-norm of $\phi_2'.$ Then  
\begin{align*}
\int_{|\det pz_1|\geq 1}\big|\overline{\phi_1'(pk)}\phi_2'(p\transp{y}^{-1})\big|dp\leq \|\phi_2'\|_{\infty}\int_{[P_0']}\big|\phi_1'(pk)\big|\textbf{1}_{|\det pz_1|\geq 1}dp.
\end{align*}
Therefore, 
\begin{align*}
\mathcal{J}\ll \int_{\mathcal{Z}}|\det z_1|^{-\Re(s_1+s_2)}\Big[\prod_{v\in\Sigma_F}\mathcal{J}_v(z_{1,v})\Big]\int_{K'}\int_{\substack{[P_0']}}\big|\phi_1'(pk)\big|\textbf{1}_{|\det pz_1|\geq 1}dpdkd^{\times}z_1,
\end{align*}
where $\mathcal{J}_v(z_{1,v})$ is defined by 
\begin{align*}
\int\int\Big|\int_{F_v^n}f_v(u(\mathbf{c}_v) n_v\iota(y_v))\int_{\mathcal{O}_{F_v}^{\times}}\theta_v(\eta \beta_vz_{1,v}\transp{\mathbf{c}}_v)d^{\times}\beta_vd\mathbf{c}_v\Big|dn_v|\det y_v|_v^{\Re(s_2)}dy_v
\end{align*}
if $v<\infty,$ where $y_v$ (resp. $n_v$) ranges over $G'(F_v)$ (resp. $N_P(F_v)$); and is defined as
\begin{align*}
\int_{G'(F_v)}\int_{N_P(F_v)}\Big|\int_{F_v^n}f_v(u(\mathbf{c}_v) n_v\iota(y_v))\theta_v(\eta z_{1,v}\transp{\mathbf{c}}_v)d\mathbf{c}_v\Big|dn_v|\det y_v|_v^{\Re(s_2)}dy_v
\end{align*}
if $v\mid\infty$. The integral over $\beta_v$ comes from a change of variable $z_{1}\mapsto z_1\prod_{v<\infty}\beta_v$ and taking the integral over $\beta_v$'s.

Let $v\in\Sigma_{F}.$ Let $l(\mathbf{c}_v)=\min_{1\leq j\leq n}\{e_v(c_j)\},$ where $\mathbf{c}_v=(c_1,\cdots,c_n).$

\begin{itemize}
\item Suppose $v\notin S$ and $v\nmid\mathfrak{N}.$ Then $\mathcal{J}_v(z_{1,v})=\mathcal{J}_v^+(z_{1,v})+\mathcal{J}_v^-(z_{1,v}),$ where $\mathcal{J}_v^+(z_{1,v})$ is defined by 
\begin{align*}
\int_{G'(F_v)}\int_{N_P(F_v)}\Big|\int_{F_v^n}f_v(u(\mathbf{c}_v) n_v\iota(y_v))\theta_v(\eta z_{1,v}\transp{\mathbf{c}}_v)\textbf{1}_{l(\mathbf{c}_v)\geq 0}d\mathbf{c}_v\Big|dn_v|\det y_v|_v^{\Re(s_2)}dy_v,
\end{align*}
and $\mathcal{J}_v^-(z_{1,v})$ is defied similarly, with $\textbf{1}_{l(\mathbf{c}_v)\geq 0}$ replaced by $\textbf{1}_{l(\mathbf{c}_v)< 0}.$

By Lemma \ref{lem20} we have $\mathcal{J}_v^+(z_{1,v})=1,$ and $\mathcal{J}_v^-(z_{1,v})$ is 
\begin{equation}\label{22}
\ll \sum_{l\leq -1}|z_{1,v}|_v^{-1}q_v^{-(n-1)l}q_v^{(n+1)l}q_v^{(n+1)l\Re(s_2)}\ll |z_{1,v}|_v^{-1}q_v^{-2-(n+1)\Re(s_2)},
\end{equation}
where the implied constant is absolute. Here $q_v^{(n-1)l}$ is the contribution from $c_1, \cdots, c_{n-1},$ while the contribution from $c_n$ is $|z_{1,v}|_v^{-1}$ due to the Ramanujan sum $\int_{\mathcal{O}_{F_v}^{\times}}\theta_v(\eta \beta_vz_{1,v}\transp{\mathbf{c}}_v)d^{\times}\beta_v$. The factor $q_v^{-(n+1)l}$ (resp. $q_v^{-(n+1)l\Re(s_2)}$) is the contribution from $n_v$ (resp. $y_v$) according to the second part of Lemma \ref{lem20}.

\item Suppose $v\mid\mathfrak{N}.$ By the first part of Lemma \ref{lem20} we have $l(\mathbf{c}_v)\geq e_v(\mathfrak{N}).$ So 
	$$
	\mathcal{J}_v(z_{1,v})=\mathcal{J}_v^+(z_{1,v})\ll q_v^{-ne_v(\mathfrak{N})}\|f_v\|_{\infty}\textbf{1}_{e_v(z_{1,v})\geq -e_v(\mathfrak{N})}.
	$$ 
\item Suppose $v\mid\infty.$ By Lemma \ref{lem18} the integral $\mathcal{J}_v(z_{1,v})$ converges, as the supports of $\mathbf{c}_v,$ $n_v$ and $y_v$ are compact. Here we notice that the archimedean Fourier transform is a Schwartz function $h_v$ of $z_{1,v}.$

\item Suppose $v\in S\cap\Sigma_{F,\fin}.$ Write simply that  $l=l(\mathbf{c}_v)=\min_{1\leq j\leq n}\{e_v(c_j)\}.$ 

If $l\geq 0,$ then $f_{v}(u(\mathbf{c}_{v}) n_{v}\iota(y_{v}))\neq 0$ unless $u_v\in C_1$ and
$y_v\in C_2,$ where  $C_1\subset F_v^n$ and $C_2\subset G'(F_v)$ are compact sets determined by $\supp\tilde{f}_v.$ 

If $l\leq -1,$ then there exists $k_0\in K_v'$ (depending on $\mathbf{c}_v$) such that $\iota(k_0^{-1})u(\mathbf{c}_v)\iota(k_0)=u(\eta\varpi_v^{l}).$ Appealing to Lemma \ref{lem18'} we then derive that $f_{v}(u(\mathbf{c}_{v}) n_{v}\iota(y_{v}))\neq 0$ unless
$u_v\in \varpi_v^{-l}C_1'$ and
$y_v\in \varpi_v^{-l}K_v'\diag(I_{n-1},\varpi_v^{-l}) C_2',$ where  $C_1'\subset F_v^n$ and $C_2'\subset G'(F_v)$ are compact sets determined by $\supp\tilde{f}_v.$ 

Parallel to \eqref{22} we derive that $\mathcal{J}_v^-(z_{1,v})\ll |z_{1,v}|_v^{-1}(1+q_v^{-2-(n+1)\Re(s_2)}),$ where the implied constant depends on $\supp\tilde{f}_v.$ Note that there are only finitely many $v$'s in $S.$ So $\prod_{v\in S\cap\Sigma_{F,\fin}}\mathcal{J}_v^-(z_{1,v})\ll \prod_{v\in S\cap\Sigma_{F,\fin}}|z_{1,v}|_v^{-1}.$ 
\end{itemize}

Write $|z_{1,\fin}|_{\fin}=m^{-1},$ $m\in |\mathfrak{N}|^{-1}M_S^{-1}\mathbb{Z}_{\geq 1},$ where $M_S\in\mathbb{Q}$ is determined by $\supp f_S.$ Collecting the above discussions we then deduce that $\mathcal{J}$ is 
\begin{align*}
	\ll &\frac{|\mathfrak{N}|^{-n}}{\Vol(K_0(\mathfrak{N}))}\int_{Z'(F_{\infty})}|\det z_{1,\infty}|_{\infty}^{-\Re(s_1+s_2)}\prod_{v\mid\infty}\mathcal{J}_v(z_{1,v})\int_{K'}\int_{\substack{[P_0']}}\big|\phi_1'(pk)\big|\\
	&\prod_{\substack{v\in\Sigma_{F,\fin}}}(1+O(q_v^{-2-(n+1)\Re(s_2)}))\sum_{\substack{\mathfrak{m}:\ N_F(\mathfrak{m})=m\\m\in |\mathfrak{N}|^{-1}M_S^{-1}\mathbb{Z}_{\geq 1}\\ m^n\leq |\det z_{1,\infty}|_{\infty}|\det p|}}m^{1+n\Re(s_1+s_2)}dpdkd^{\times}z_1,
\end{align*}
where $\mathfrak{m}$'s are fractional ideals in $F$, and $N_F(\mathfrak{m})$ is the absolute norm.

Consequently, when $-2-(n+1)\Re(s_2)<-1,$ i.e., $\Re(s_2)>-1/(n+1),$  
\begin{align*}
	\mathcal{J}\ll &\frac{|\mathfrak{N}|^{-n}}{\Vol(K_0(\mathfrak{N}))}\cdot (|\mathfrak{N}|M_S)^{2+n\Re(s_1+s_2)}\int_{Z'(F_{\infty})}|\det z_{1,\infty}|_{\infty}^{2/n}\prod_{v\mid\infty}\mathcal{J}_v(z_{1,v})d^{\times}z_1\\
	&\int_{K'}\int_{\substack{[P_0']}}\big|\phi_1'(pk)\big||\det p|^{\frac{2}{n}+\Re(s_1+s_2)}\textbf{1}_{|\det p|\geq (|\mathfrak{N}|M_S)^{-n}}dpdk.
\end{align*}

Since $\phi_1'$ decays rapidly over $Z'(\mathbb{A}_F)P_0'(F)\backslash G'(\mathbb{A}_F),$ then 
$$
\int_{K'}\int_{\substack{[P_0']}}\big|\phi_1'(pk)\big||\det p|^{\frac{2}{n}+\Re(s_1+s_2)}dpdk<\infty.
$$

Moreover, by the support of $f_{\infty}$ we have 
$$
\int_{Z'(F_{\infty})}|\det z_{1,\infty}|_{\infty}^{2/n}\prod_{v\mid\infty}\mathcal{J}_v(z_{1,v})d^{\times}z_1<\infty.
$$
Hence \eqref{63.} follows.
\end{proof}


\begin{prop}\label{prop21.}
The function $\mathcal{F}_{0,1}J^{\bi}_{\Geo,\Res}(f,\textbf{s})$ converges absolutely outside $s_1+s_2\in\{0,-1\},$ 
and extends to a meromorphic function in $\mathcal{R}$ with possible simple poles at $s_1+s_2\in \{0, -1\}.$ 
\end{prop}
\begin{proof}
We omit the proof here as it is quite similar to (but easier than) that of Proposition \ref{prop19}. 
\end{proof}

Combining Propositions \ref{prop22}, \ref{prop19} and \ref{prop21.} we deduce the following.
\begin{thm}\label{thm22}
Let notation be as before. Then $\mathcal{F}_{0,1}J^{\bi}_{\Geo}(f,\textbf{s})$ converges absolutely in $\Re(s_1)+\Re(s_2)>0,$ admits a meromorphic continuation to the region $\mathcal{R}$ (cf. \eqref{R.}) with at most simple poles at $s_1+s_2\in \{0,-1\}.$

 Moreover, for $\mathbf{s}\in\mathcal{R},$ the holomorphic parts of $\mathcal{F}_{0,1}J^{\bi}_{\Geo}(f,\textbf{s})$ satisfy 
\begin{align*}
\mathcal{F}_{0,1}J^{\bi}_{\Geo,+}(f,\textbf{s})\ll |\mathfrak{N}|^{-n\Re(s_1+s_2)},\ \ \mathcal{F}_{0,1}J^{\bi,\wedge}_{\Geo,+}(f,\textbf{s})=O(|\mathfrak{N}|^{2+n\Re(s_1+s_2)}),
\end{align*}
 where the implied constants depend on $\pi_1',$ $\pi_2',$ and $f_S$. 
\end{thm}

\subsection{The Orbital Integral: $\mathcal{F}_{1,0}J^{\bi,T}_{\Geo}(f,\textbf{s})$}
	Similar to Lemma \ref{lem13}, the set $\{w_n, \gamma\}$ is a complete set of representatives for $Z(F)P(F)\backslash P(F)w_nP(F)/\iota(G'(F)).$ In parallel with Lemma \ref{lem15} we can reduce $\mathcal{F}_{1,0}J^{\bi,T}_{\Geo}(f,\textbf{s})$ to 
	\begin{align*}
		\int_{[Z'^T]}\int_{[Z'^T]}&\int_{[\overline{G'}]}\int_{[\overline{G'}]}\int_{[N_P]}\sum_{\gamma\in Z(F)\backslash P(F)\gamma\iota(G'(F))}f(\iota(z_1x)^{-1}u\gamma \iota(z_2y))\\
		&\qquad \phi_1'(x)\overline{\phi_2'(y)}\omega'(z_1)\overline{\omega_2(z_2)}|\det z_1x|^{s_1}|\det z_2y|^{s_2}dudxdyd^{\times}z_1d^{\times}z_2.
	\end{align*}

Recall that for $\textbf{x}\in M_{n,1}(\mathbb{A}_F),$ we set $u(\textbf{x})=\begin{pmatrix}
	I_n&\\
	\textbf{x}&1
\end{pmatrix}$. Denote by 
$$
^{\dag}f_P(\textbf{x};z_2,y)=\int_{N_P(\mathbb{A}_F)}f(nu(\textbf{x})\iota(z_2y))dn.
$$ 
Then $^{\dag}f_P$ is a Schwartz-Bruhat function on $\mathbb{A}_F^n\times Z'(\mathbb{A}_F)\times \overline{G'}(\mathbb{A}_F).$ Define 
\begin{align*}
	E_{\gamma}(x,s;^{\dag}f_P(\cdot;z_2,y))=\sum_{\delta}\int_{Z'(\mathbb{A}_F)}{^{\dag}f_P}(\eta z_1\delta x;z_2,y)|\det z_1x|^{s}d^{\times}z_1,
\end{align*}
which converges absolutely when $\Re(s)>1.$ Here $\delta$ ranges over $P_0(F)\backslash \overline{G'}(F).$ Moreover, for fixed $z_2\in Z'(\mathbb{A}_F)$ and $y\in \overline{G'}(\mathbb{A}_F),$ $E_{\gamma}(x,s;{^{\dag}f_P}(\cdot;z_2,y))$ has a meromorphic continuation to the whole $s$-plane and admits the function equation connecting the evaluations at $s$ and $1-s.$ Similar to Proposition \ref{prop22}, we obtain the following:
\begin{prop}\label{prop23}
Let notation be as before. Let $\textbf{s}=(s_1,s_2)$ be such that $\Re(s_1)+\Re(s_2)>0.$ Then 
\begin{align*}
	\mathcal{F}_{1,0}J^{\bi}_{\Geo}(f,\textbf{s}):=\lim_{T\rightarrow\infty} \mathcal{F}_{1,0}J^{\bi,T}_{\Geo}(f,\textbf{s})
\end{align*}
exists, and is represented by  the absolutely convergent integral 
\begin{align*}
\int_{[\overline{G'}]}\int_{Z'(\mathbb{A}_F)}\mathcal{P}(\textbf{s};\phi_2', ^{\dag}f_P(\cdot;z_2,y))\omega'(z_1)\overline{\omega_2(z_2)}|\det(z_2y)|^{s_2}d^{\times}z_2dy,
\end{align*} 
where $\mathcal{P}(\textbf{s};\phi_2', f_P^{\dag}(\cdot;z_2,y))$ is defined to be 
\begin{align*}
\int_{[\overline{G'}]}\phi_1'(x)\overline{\phi_2'(xy)}E_{\gamma}(x,s_1+s_2+1;^{\dag}f_P(\cdot;z_2,y))dyd^{\times}z_2,
\end{align*}
which is a Rankin-Selberg convolution representing $\Lambda(1+s_1+s_2,\pi_1'\times\widetilde{\pi}_2').$ Moreover, $$
\mathcal{F}_{1,0}J^{\bi}_{\Geo}(f,\textbf{s})\ll |\mathfrak{N}|^{-n(s_1+s_2)},\ \ \Re(s_1)+\Re(s_2)>0,
$$ 
where the implied constant depends on $\pi_1',$ $\pi_2',$ and $f_S$. 
\end{prop}

Parallel to Propositions \ref{prop19} and \ref{prop21.} we obtain 
\begin{thm}\label{thm24}
Let notation be as before. Then $\mathcal{F}_{1,0}J^{\bi}_{\Geo}(f,\textbf{s})$ converges absolutely in $\{(s_1,s_2)\in\mathbb{C}^2,\ \Re(s_1)+\Re(s_2)>0\},$ admits a meromorphic continuation to the region $\mathcal{R},$ with at most simple poles at $s_1+s_2\in \{0,-1\}.$ 

Moreover, for $\mathbf{s}\in\mathcal{R},$ the holomorphic parts of $\mathcal{F}_{1,0}J^{\bi}_{\Geo}(f,\textbf{s})$ satisfy 
\begin{align*}
\mathcal{F}_{1,0}J^{\bi}_{\Geo,+}(f,\textbf{s})\ll |\mathfrak{N}|^{-n\Re(s_1+s_2)},\ \ \mathcal{F}_{1,0}J^{\bi,\wedge}_{\Geo,+}(f,\textbf{s})=O(|\mathfrak{N}|^{2+n\Re(s_1+s_2)}),
\end{align*}
 where the implied constants depend on $\pi_1',$ $\pi_2',$ and $f_S$. 
\end{thm}
\begin{proof}
Write $\mathcal{F}_{1,0}J^{\bi}_{\Geo}(f,\textbf{s};\phi_1',\phi_2')$ for $\mathcal{F}_{1,0}J^{\bi}_{\Geo}(f,\textbf{s})$ to highlight the dependence of the cusp forms $\phi_1'$ and $\phi_2'$ and their orders. Note that after a change of variables,
\begin{equation}\label{23}
\mathcal{F}_{1,0}J^{\bi}_{\Geo}(f,\textbf{s};\phi_1',\phi_2')=\mathcal{F}_{0,1}J^{\bi}_{\Geo}(f^{\vee},\textbf{s}^{\vee};\phi_2',\phi_1'),
\end{equation}
where for $\mathbf{s}=(s_1,s_2),$ $\mathbf{s}^{\vee}:=(s_2,s_1);$ and for $g\in G(\mathbb{A}_F),$ $f^{\vee}(g):=f(g^{-1}).$ Here $\mathcal{F}_{0,1}J^{\bi}_{\Geo}(f^{\vee},\textbf{s}^{\vee};\phi_2',\phi_1')$ means the integral with $f$ (resp. $\mathbf{s}$) replaced by $f^{\vee}$ (resp. $\mathbf{s}^{\vee}$), and the cusp forms $\phi_1'$ and $\phi_2'$ are switched.
	
Then Theorem \ref{thm24} follows from \eqref{23} and the proof of Theorem \ref{thm22}.	
\end{proof}

\begin{remark}
The equality \eqref{23} gives an explicit relation between $\mathcal{F}_{0,1}J^{\bi}_{\Geo}(f,\textbf{s})$ and $\mathcal{F}_{1,0}J^{\bi}_{\Geo}(f,\textbf{s})$.
\end{remark}

\section{Orbital Integrals Relative to the Big Cell: \RNum{2}}\label{sec5}
In this section we deal with the distribution $\mathcal{F}_{0,0}J_{\Geo}^{\bi,T}(f,\mathbf{s}),$ which  is defined as  
	\begin{align*}
		\int_{[Z'^T]}\int_{[Z'^T]}\int_{[\overline{G'}]}\int_{[\overline{G'}]}&\sum_{\gamma\in P_0(F)w_nP_0(F)}f(\iota(z_1x)^{-1}\gamma\iota(z_1z_2y))\phi'_1(x)\overline{\phi'_2(y)}\\
&\omega'(z_1)\overline{\omega_2(z_2)}|\det z_1x|^{s_1}|\det z_1z_2y|^{s_2}dxdyd^{\times}z_1d^{\times}z_2.
	\end{align*}
	
	The analytic behavior of the integral $\mathcal{F}_{0,0}J^{\bi,T}_{\Geo}(f,\mathbf{s})$ is more delicate than those of  $\mathcal{F}_{0,1}J^{\bi,T}_{\Geo}(f,\mathbf{s})$ and $\mathcal{F}_{1,0}J^{\bi,T}_{\Geo}(f,\mathbf{s}).$ We will write $\mathcal{F}_{0,0}J^{\bi,T}_{\Geo}(f,\mathbf{s})$ as a linear combination of orbital integrals via a double coset decomposition. This is a typical treatment of the geometric side. Note that the orbital integrals are not purely geometric as $\mathcal{F}_{0,0}J^{\bi,T}_{\Geo}(f,\mathbf{s})$ involves some automorphic information (e.g., $\phi'_1$). Some new ingredients are required to handle them. Let
	\begin{align*}
		&\gamma_1=\begin{pmatrix}
			I_{n-1}&&\\
			&&1\\
			&1&
		\end{pmatrix},\qquad 
		\gamma_2=\begin{pmatrix}
			I_{n-1}&\\
			&&1\\
			&1&1
		\end{pmatrix},\qquad \gamma_3=\begin{pmatrix}
			I_{n-1}&\\
			&1&1\\
			&1&
		\end{pmatrix},\\
	&\gamma_4=\begin{pmatrix}
		I_{n-1}&&\tau\\
		&&1\\
		&1&
	\end{pmatrix},\qquad 
	\gamma_5=\begin{pmatrix}
		I_{n-1}&&\tau\\
		&&1\\
		&1&1
	\end{pmatrix},\qquad \gamma_6=\begin{pmatrix}
		I_{n-1}&&\tau\\
		&1&1\\
		&1&
	\end{pmatrix},\\
&\gamma_{\tau}=\begin{pmatrix}
		I_{n-1}&&\tau\\
		&1&\\
		&1&1
	\end{pmatrix},\qquad \gamma(t)=\begin{pmatrix}
			I_{n-1}&\\
			&1&t\\
			&1&1
		\end{pmatrix},\quad t\in F-\{1\},
	\end{align*}
where $\tau=\transp{(1,0,\cdots,0)}\in F^{n-1}.$ Let $\Phi_1=\{\gamma_1, \gamma_2, \gamma_3, \gamma_4, \gamma_5, \gamma_6\}.$ Denote by $\Phi_2=\big\{\gamma_{\tau}, \gamma(t):\ t\in F-\{1\}\big\}.$
	
	\begin{lemma}\label{lem16}
		Let notation be as before. Then $\Phi:=\Phi_1\sqcup \Phi_2$ form a complete set of representatives for $\iota(G')(F)\backslash P_0(F)w_nP_0(F)/\iota(G')(F).$
	\end{lemma}
	\begin{proof}
		Let $g\in P_0(F)w_nP_0(F).$ By Bruhat composition we can write $g=p_1w_np_2$ with $p_1, p_2\in P_0(F).$ Let $p_1=m_1n_1,$ $p_2=m_2n_2$ be the Levi decomposition. We may assume assume
		\begin{align*}
		n_1=\begin{pmatrix}
		I_{n-1}&\\
		&1&\alpha_1\\
		&&1
		\end{pmatrix},\ \ n_2=\begin{pmatrix}
		I_{n-1}&&\beta\tau\\
		&1&\alpha_2\\
		&&1
	\end{pmatrix},\ \ \alpha_1, \alpha_2, \beta\in F.
		\end{align*}
		\begin{itemize}
			\item Suppose $\alpha_1=\alpha_2=0.$ Then $g\in \iota(G')(F)\gamma_1 \iota(G')(F)\bigcup \iota(G')(F)\gamma_4 \iota(G')(F).$
			
			\item Suppose $\alpha_1=0$ but $\alpha_2\neq 0.$ Then 
			$$
			g\in \iota(G')(F)\gamma_2 \iota(G')(F)\bigcup \iota(G')(F)\gamma_5 \iota(G')(F).
			$$
			
			\item Suppose $\alpha_1\neq 0$ but $\alpha_2= 0.$ Then 
			$$
			g\in \iota(G')(F)\gamma_3 \iota(G')(F)\bigcup \iota(G')(F)\gamma_6 \iota(G')(F).
			$$
			
			\item Suppose $\alpha_1\alpha_2\neq 0,$ and $\beta=0.$ Then $g\in \iota(G')(F)\gamma(t) \iota(G')(F)$ for some $t\in F-\{1\}.$
			
			\item Suppose $\alpha_1\alpha_2\neq 0,$ and $\beta\neq 0.$ 
			\begin{itemize}
			\item If $\alpha_1\alpha_2=-1,$ then $g\in \iota(G')(F)\gamma_{\tau}\iota(G')(F).$
			\item If $\alpha_1\alpha_2\neq -1,$ then $g\in \iota(G')(F)\gamma(t)\iota(G')(F)$ for some $t\in F-\{0,1\}.$ Here we make use of the fact that 
			\begin{align*}
			\begin{pmatrix}
			I_{n-1}&c\tau\\
			&1\\
			&&1
			\end{pmatrix}n_1w_n\begin{pmatrix}
			I_{n-1}&&\\
			&1&\alpha_2\\
			&&1
		\end{pmatrix}=n_1w_n\begin{pmatrix}
		I_{n-1}&&(1+\alpha_1\alpha_2)c\tau\\
		&1&\alpha_2\\
		&&1
	\end{pmatrix}.
			\end{align*}
			\end{itemize}
		\end{itemize}
		Therefore, we have 
		\begin{equation}\label{90}
			P_0(F)w_nP_0(F)=\bigcup_{\gamma\in \Phi}\iota(G')(F)\gamma \iota(G')(F).
		\end{equation}
		By uniqueness of Bruhat canonical form the union in \eqref{90} is disjoint.
	\end{proof}
	
	For $\delta\in \Phi,$ denote by $\mathcal{O}_{\delta}(F)$ the orbits of $\delta$ in $P_0(F)w_nP_0(F).$ Let 
	\begin{align*}
		H_{\delta}'=\bigg\{(x,y)\in G'\times G':\ \iota(x)\delta=\delta\iota(y)\bigg\}
	\end{align*} 
	be the stabilizer of $\delta.$ Then $\mathcal{O}_{\delta}(F)=\iota(H_{\delta}'\backslash (G'(F)\times G'(F))).$
	
	According to Lemma \ref{lem16} we can decompose $\mathcal{F}_{0,0}J_{\Geo}^{\bi,T}(f,\textbf{s})$ as 
	\begin{equation}\label{18}
		\mathcal{F}_{0,0}J^{\bi,T}_{\Geo}(f,\textbf{s})=\sum_{\delta \in \Phi}J_{\Geo,\delta}^{\bi,T}(f,\textbf{s}),
	\end{equation}
	where 
	\begin{align*}
		J_{\Geo,\delta}^{\bi,T}(f,\textbf{s}):=&\int_{[Z'^T]}\int_{[Z'^T]}\int_{[\overline{G'}]}\int_{[\overline{G'}]}\sum_{\gamma\in\mathcal{O}_{\delta}(F)}f(\iota(z_1x)^{-1}\gamma\iota(z_1z_2y))\\
		&\qquad \phi_1'(x)\overline{\phi_2'(y)}\omega'(z_1)\overline{\omega_2(z_2)}|\det z_1x|^{s_1}|\det z_1z_2y|^{s_2}dxdyd^{\times}z_1d^{\times}z_2.
	\end{align*}
	
	Substituting $\mathcal{O}_{\delta}(F)=\iota(H_{\delta}'\backslash (G'(F)\times G'(F)))$ into the above definition to rewrite $J_{\Geo,\delta}^{\bi,T}(f,\textbf{s})$ as 
	\begin{align*}
		\int_{[Z'^T]}\int_{[Z'^T]}\int_{[\overline{G'}]}&\int_{[\overline{G'}]}\sum_{(\gamma_1,\gamma_2)\in H_{\delta}'(F)\backslash (G'(F)\times G'(F))}f(\iota(z_1^{-1}x^{-1}\gamma_1^{-1})\delta \iota(\gamma_2z_1z_2y))\\
		&\phi_1'(x)\overline{\phi_2'(y)}\omega'(z_1)\overline{\omega_2(z_2)}|\det z_1x|^{s_1}|\det z_1z_2y|^{s_2}dxdyd^{\times}z_1d^{\times}z_2.
	\end{align*}
	Note that the kernel function $\K(g_1,g_2)$ is convergent absolutely and is slowly increasing with respective to $g_1$ and $g_2.$ Also, $\phi'$ is rapidly decaying on $[\overline{G'}],$ then
	$\mathcal{F}_{0,0}J_{\Geo}^{\bi,T}(f,\textbf{s})$ converges absolutely. Therefore, \eqref{18} is well defined and it converges absolutely as well. We need to show it is allowed to take limit $T\rightarrow \infty$ into the orbital integral decomposition \eqref{18}. In fact, we will show such a limit version of \eqref{18} is admissible for all $\textbf{s}\in\mathbb{C}^2$ and thus defines an entire function of $\textbf{s}.$
	
	\subsection{Vanishing Orbital Integrals}\label{5.1}
In this subsection, we will show $J_{\Geo,\delta}^{\bi,T}(f,\textbf{s})\equiv 0$ for $\gamma\in \Phi_1.$ This is  a consequence of the support of $f_{v},$ $v\mid\mathfrak{N}.$
	\begin{lemma}\label{lem17}
		Let notation be as before. Let $\gamma\in \Phi_1.$ Then 
		\begin{align*}
			J_{\Geo,\delta}^{\bi,T}(f,\textbf{s})\equiv 0.
		\end{align*}
	\end{lemma}
	\begin{proof}
		Let $v\mid\mathfrak{N}.$ Consider $f_v(\iota(z_{1,v}x_v)^{-1}\gamma\iota(z_{2,v}y_v))$ for $\gamma\in \Phi_1.$ 
		
		\begin{enumerate}
			\item[(\RNum{1}).] Suppose $\gamma=\gamma_1.$ Then 
			\begin{align*}
				\iota(z_{1,v}x_v)^{-1}\gamma\iota((z_{1,v}z_{2,v}y_v)=\begin{pmatrix}
					*&*\\
					*&0
				\end{pmatrix}\in G(F_v),
			\end{align*}
			namely, the last line of $\iota(z_{1,v}x_v)^{-1}\gamma\iota(z_{1,v}z_{2,v}y_v)$ is of the form $(*,\cdots, *,0).$ Hence $\iota(z_{1,v}x_v)^{-1}\gamma\iota(z_{1,v}z_{2,v}y_v)\notin Z(F_v)K_v,$ whose $(n+1, n+1)$-th entry is nonvanishing. So  $f_v(\iota(z_{1,v}x_v)^{-1}\gamma_1\iota(z_{1,v}z_{2,v}y_v))=0.$
			
			\item[(\RNum{2}).] Suppose $\gamma=\gamma_2.$ Consider the Iwasawa decomposition
			\begin{align*}
				z_{1,v}x_v=b_vk_v,\ \ z_{2,v}y_v=z_{1,v}^{-1}b_v'k_v',
			\end{align*}
			where $b_v, b_v'\in B'(F_v)$ and $k_v, k_v'\in K_v'=G'(\mathcal{O}_{F_v}).$ Note that $\iota(K_v')\subset K_v.$ So 
			$$
			f_v(\iota(z_{1,v}x_v)^{-1}\gamma_1\iota(z_{1,v}z_{2,v}y_v))=f_v(\iota(b_v)^{-1}\gamma_2\iota(b_v')).
			$$
			However, a straightforward calculation shows that 
			$$
			\iota(b_v)^{-1}\gamma_2\iota(b_v')=\begin{pmatrix}
				*&*\\
				0&*\\
				*&1
			\end{pmatrix}\in G(F_v),
			$$
			namely, the last second line of $\iota(b_v)^{-1}\gamma_2\iota(b_v')$ is $(0,\cdots, 0, *),$ where the first $n$ entries are zeros. So it cannot lie in $Z(F_v)K_v,$ the Hecke congruence subgroup. Hence $f_v(\iota(z_{1,v}x_v)^{-1}\gamma_2\iota(z_{1,v}z_{2,v}y_v))=0.$
			
			\item[(\RNum{3}).] Suppose $\gamma=\gamma_3.$ Then taking the inverse and by a similar argument to the Case \RNum{2} we obtain $f_v(\iota(z_{1,v}x_v)^{-1}\gamma_3\iota(z_{2,v}y_v))=0.$
		\end{enumerate}
		
		In all, we have $f_v(\iota(z_{1,v}x_v)^{-1}\gamma\iota(z_{1,v}z_{2,v}y_v))=0$ for all $\gamma\in \Phi,$ $z_{1,v}x_v, z_{2,v}y_v\in G'(F_v),$ proving Lemma \ref{lem17}.
	\end{proof}
	
	By Lemma \ref{lem16} and Lemma \ref{lem17} we have 
	\begin{equation}\label{91}
		\mathcal{F}_{0,0}J_{\Geo}^{\bi,T}(f,\textbf{s})=J_{\Geo,\gamma(0)}^{\bi,T}(f,\textbf{s})+J_{\Geo,\gamma_{\tau}}^{\bi,T}(f,\textbf{s})+\sum_{t\in F-\{0,1\}}J_{\Geo,\gamma(t)}^{\bi,T}(f,\textbf{s}).
	\end{equation}

Following the convention in \cite{RR05} we shall call 
\begin{equation}\label{29}
J_{\Geo,\reg}^{\bi,T}(f,\textbf{s}):=\sum_{\substack{t\in F-\{0,1\}}}J_{\Geo,\gamma(t)}^{\bi,T}(f,\textbf{s}).
\end{equation}
\textit{regular} orbital integrals. Denote by $J_{\Geo,\du}^{\bi,T}(f,\textbf{s})=J_{\Geo,\gamma(0)}^{\bi,T}(f,\textbf{s}).$ We will see $J_{\Geo,\du}^{\bi,T}(f,\textbf{s})$ is the `dual' of $J_{\Geo,\sm}^{\bi,T}(f,\textbf{s})$ via the functional equation. 

The orbital integral $J_{\Geo,\du}^{\bi,T}(f,\textbf{s})$ will be dominant as one of the main terms if $(s_1,s_2)$ approaches $(0,0),$ while the sum of $J_{\Geo,\gamma_{\tau}}^{\bi,T}(f,\textbf{s})$ and regular orbital integrals $J_{\Geo,\gamma(t)}^{\bi,T}(f,\textbf{s}),$ as $t$ runs through $F-\{0,1\},$ will be shown to serve as tiny error terms. Moreover, if the test function $f$ has compact support modulo the center, then $\mathcal{F}_{0,0}J_{\Geo,\gamma(t)}^{\bi,T}(f,\textbf{s})\equiv 0$ for all $t\neq 0,1$ as long as $|\mathfrak{N}|$ is large enough. Hence, we will get stability in the sense of \cite{MR12} and \cite{FW09}. Here we recall $v\mid\mathfrak{N}$ is the place where we take the local test function to be the characteristic function of a Hecke congruence subgroup (cf. \textsection\ref{2.6}).

\subsection{Orbits Alteration} 
The stabilizers of $\gamma(0),$ $\gamma(t)$ and $\gamma_{\tau}$ are different, which makes the analytic behaviors of corresponding orbital integrals quite distinguishing. We will handle them by diverse methods. For $\gamma(0)$ and $\gamma(t)$ we simply use the usual manipulation of orbital integrals. However, the stabilizers of $\gamma_{\tau}$ is less favorable. We thus provide another treatment of the orbits by rewriting it into disjoint union of \textit{conjugacy classes}, despite the original orbits are relative.

\subsubsection{The Orbital Integral $J_{\Geo,\du}^{\bi,T}(f,\textbf{s})$}
Note that the stabilizer of $\gamma(0)$ in $G'\times G'$ is $\Delta P_0'(F),$ the diagonal embedding of the mirabolic subgroup $P_0'$ of $G'.$ So  
\begin{align*}
J_{\Geo,\du}^{\bi,T}(f,\textbf{s})=&\int_{Z'^T(\mathbb{A}_F)}\int_{Z'^T(\mathbb{A}_F)}\int_{\overline{G'}(\mathbb{A}_F)}\int_{P_0'(F)\backslash\overline{G'}(\mathbb{A}_F)}f(\iota(z_1x)^{-1}\gamma(0)\iota(z_1z_2xy))\\
	&\quad \phi'_1(x)\overline{\phi'_2(xy)}\omega'(z_1)\overline{\omega_2(z_2)}|\det z_1x|^{s_1+s_2}|\det z_2y|^{s_2}dxdyd^{\times}z_1d^{\times}z_2.
\end{align*}

Analytic behaviors of $J_{\Geo,\du}^{\bi,T}(f,\textbf{s})$ will be given in \textsection\ref{sec9.2}.

\subsubsection{The Orbital Integral $J_{\Geo,\gamma_{\tau}}^{\bi,T}(f,\textbf{s})$}\label{5.2.2}
A straightforward calculation shows that the stabilizer of $\gamma_{\tau}$ is $\Delta S_0(F),$ where $S_0$ is the matrix group of the form 
\begin{align*}
\begin{pmatrix}
1&\mathbf{b}&c_1\\
&A&\mathbf{c}'\\
&&1
\end{pmatrix},\ \ A\in\mathrm{GL}(n-2),\ \mathbf{b}\in M_{1,n-2},\ c_1\in M_{1,1}, \mathbf{c}'\in M_{n-2,1}.
\end{align*} 

Changing variables we obtain 
\begin{align*}
J_{\Geo,\gamma_{\tau}}^{\bi,T}(f,\textbf{s})=&\int_{Z'^T(\mathbb{A}_F)}\int_{Z'^T(\mathbb{A}_F)}\int_{\overline{G'}(\mathbb{A}_F)}\int_{S_0(F)\backslash\overline{G'}(\mathbb{A}_F)}f(\iota(z_1x)^{-1}\gamma_{\tau}\iota(z_1z_2xy))\\
	&\quad \phi'_1(x)\overline{\phi'_2(xy)}\omega'(z_1)\overline{\omega_2(z_2)}|\det z_1x|^{s_1+s_2}|\det z_2y|^{s_2}dxdyd^{\times}z_1d^{\times}z_2.
\end{align*}
However, the usual treatment of orbital integrals does not work well, since it is quite inconvenient to handle the integral 
\begin{align*}
\int_{[S_0]}\phi'_1(x)\overline{\phi'_2(xy)}|\det x|^{s_1+s_2}|\det y|^{s_2}dxdy.
\end{align*}

We thereby propose a different treatment. Let $\Phi_0(F)$ be a complete set of representative of $S_0(F)\backslash P_0'(F).$ Let $N_j$ be the unipotent radical of the standard parabolic subgroup of $\mathrm{GL}(j+1)$ of type $(j,1).$ So matrices in $N_j$ are upper triangle. By Bruhat decomposition we can take   
\begin{align*}
\Phi_0(F)=\bigsqcup_{j=1}^{n-2} \begin{pmatrix}
	\mathrm{GL}(1,F)\\
	&I_{n-1}
\end{pmatrix}w_1w_2\cdots w_j\begin{pmatrix}
N_j(F)\\
&I_{n-j-1}
\end{pmatrix}
\end{align*} 
if $n\geq 3,$ and $\Phi_0(F)=\begin{pmatrix}
	\mathrm{GL}(1,F)\\
	&1
\end{pmatrix}$ if $n=2.$ Define 
\begin{align*}
\mathcal{S}(F):=\Big\{\gamma_1^{-1}\gamma_{\tau}\gamma_1:\ \gamma_1\in \Phi_0(F)\Big\}.
\end{align*}
	
Let $\Phi_0'(F)$ be a complete set of representative of $P_0'(F)\backslash G'(F).$ Then $\gamma\in \mathcal{O}_{\gamma_{\tau}}(F)$ can be written uniquely as $\gamma_1^{-1}\gamma_2\gamma_1$, where $\gamma_1\in \Phi_0'(F),$ and $\gamma_2\in \mathcal{S}(F).$ Then $J_{\Geo,\gamma_{\tau}}^{\bi,T}(f,\textbf{s})$ becomes 
\begin{align*}
&\int_{Z'^T(\mathbb{A}_F)}\int_{Z'^T(\mathbb{A}_F)}\int_{\overline{G'}(\mathbb{A}_F)}\int_{P_0'(F)\backslash\overline{G'}(\mathbb{A}_F)}\sum_{\gamma\in \mathcal{S}(F)}f(\iota(z_1x)^{-1}\gamma\iota(z_1z_2xy))\\
&\quad \phi'_1(x)\overline{\phi'_2(xy)}\omega'(z_1)\overline{\omega_2(z_2)}|\det z_1x|^{s_1+s_2}|\det z_2y|^{s_2}dxdyd^{\times}z_1d^{\times}z_2.
\end{align*}

By definition, $\mathcal{S}(F)$ consists of matrices of the form 
\begin{align*}
\begin{pmatrix}
I_{n-1}&&\boldsymbol{\xi}\\
&1\\
&1&1
\end{pmatrix},\ \ \boldsymbol{\xi}=\transp{(\xi_1,\cdots,\xi_{n-1})}\in M_{n-1,1}(F)-\textbf{0}.
\end{align*}
	
\subsubsection{The Orbital Integral $J_{\Geo,\gamma(t)}^{\bi,T}(f,\textbf{s}):$ $t\neq 0,1$}\label{5.2.3}
Let $t\in F-\{0,1\}.$ The stabilizer of $\gamma(t)$ is $\Delta H,$ the diagonal embedding of $H,$ where $H=\diag(\mathrm{GL}(n-1),1)$ is the Levi component of $P_0'.$ Changing variables we obtain 
\begin{align*}
	J_{\Geo,\gamma(t)}^{\bi,T}(f,\textbf{s})=&\int_{Z'^T(\mathbb{A}_F)}\int_{Z'^T(\mathbb{A}_F)}\int_{\overline{G'}(\mathbb{A}_F)}\int_{H(F)\backslash\overline{G'}(\mathbb{A}_F)}f(\iota(z_1x)^{-1}\gamma(t)\iota(z_1z_2xy))\\
	&\quad \phi'_1(x)\overline{\phi'_2(xy)}\omega'(z_1)\overline{\omega_2(z_2)}|\det z_1x|^{s_1+s_2}|\det z_2y|^{s_2}dxdyd^{\times}z_1d^{\times}z_2.
\end{align*}

Notice that $P_0'(F)=H(F)N_{P'}(F),$ where $N_{P'}(F)$ is the unipotent radical of $P'(F).$ Define  
\begin{align*}
\mathcal{S}^*(F):=\Big\{\gamma_1^{-1}\gamma(t)\gamma_1:\ \gamma_1\in N_{P'}(F)\Big\}.
\end{align*}
Similar to \textsection\ref{5.2.2} we can rewrite $J_{\Geo,\gamma(t)}^{\bi,T}(f,\textbf{s})$ as
\begin{align*}
&\int_{Z'^T(\mathbb{A}_F)}\int_{Z'^T(\mathbb{A}_F)}\int_{\overline{G'}(\mathbb{A}_F)}\int_{P_0'(F)\backslash\overline{G'}(\mathbb{A}_F)}\sum_{\gamma\in\mathcal{S}^*(F)}f(\iota(z_1x)^{-1}\gamma(t)\iota(z_1z_2xy))\\
	&\quad \phi'_1(x)\overline{\phi'_2(xy)}\omega'(z_1)\overline{\omega_2(z_2)}|\det z_1x|^{s_1+s_2}|\det z_2y|^{s_2}dxdyd^{\times}z_1d^{\times}z_2.
\end{align*}

By definition, $\mathcal{S}^*(F)$ consists of matrices of the form 
\begin{align*}
	\begin{pmatrix}
		I_{n-1}&&\boldsymbol{\xi}\\
		&1&t\\
		&1&1
	\end{pmatrix},\ \ \boldsymbol{\xi}=\transp{(\xi_1,\cdots,\xi_{n-1})}\in M_{n-1,1}(F).
\end{align*}

	\subsection{The Orbital Integral  $J_{\Geo,\du}^{\bi}(f,\textbf{s})$}\label{sec9.2}

	Let $\textbf{x}=(x_1,\cdots, x_n)\in \mathbb{A}_F^n.$ Recall that $u(\textbf{x})=\begin{pmatrix}
		I_n&\\
		\textbf{x}&1
	\end{pmatrix}.$ Denote by $f^{\dag}(\textbf{x},z_2,y)=f(u(\textbf{x})\iota(z_2y)).$ Then $f^{\dag}$ is a Schwartz-Bruhat function on $\mathbb{A}_F^n\times Z'(\mathbb{A}_F)\times \overline{G'}(\mathbb{A}_F).$ 
	
	Let $\Re(s)>1.$ Define $E_{\du}(x,s;f^{\dag}(\cdot,z_2,y))$ to be 
	\begin{equation}\label{eisdu}
\sum_{\delta\in P_0(F)\backslash \overline{G'}(F)}\int_{Z'(\mathbb{A}_F)}f^{\dag}(\eta z_1\delta x,z_2,y)|\det z_1x|^{s}d^{\times}z_1.
	\end{equation}
	Then for fixed $z_2\in Z'(\mathbb{A}_F)$ and $y\in \overline{G'}(\mathbb{A}_F),$ the Eisenstein series $E_{\du}(x,s;z_2,y;f^{\dag})$ has a meromorphic continuation to the whole $s$-plane and admits the function equation connecting the evaluations at $s$ and $1-s.$
	
	\begin{prop}\label{prop18'}
		Let notation be as before. Let $\textbf{s}=(s_1,s_2)$ be such that $\Re(s_1)+\Re(s_2)>1.$ Then 
		\begin{align*}
			J_{\Geo,\du}^{\bi}(f,\textbf{s}):=\lim_{T\rightarrow\infty} J_{\Geo,\du}^{\bi,T}(f,\textbf{s})
		\end{align*}
		converges absolutely and is equal to 
		\begin{equation}\label{45}
		\int_{Z'(\mathbb{A}_F)}\int_{\overline{G'}(\mathbb{A}_F)}\mathcal{P}(s_1+s_2;\phi_1'.\phi_2',f^{\dag}(\cdot,y))\overline{\omega}'(z_2)|\det z_2y|^{s_2}dyd^{\times}z_2,
		\end{equation} 
		where 
		\begin{align*}
			\mathcal{P}(s;\phi_1',\phi_2',f^{\dag}(\cdot,y))=\int_{[\overline{G'}]}\phi'(x)\overline{R(y)\phi'(x)}E_{\du}(x,s;f^{\dag}(\cdot,z_2,y))dx
		\end{align*}
		is a Rankin-Selberg period representing the $L$-function $\Lambda(s_1+s_2,\pi_1'\times\widetilde{\pi}_2').$ In particular, $J_{\Geo,\du}^{\bi}(f,\textbf{s})/\Lambda(s_1+s_2,\pi_1'\times\widetilde{\pi}_2')$ is entire. 
	\end{prop}
	\begin{proof}
		A straightforward calculation shows that the stabilizer of $\gamma(0)$ is $H_{\gamma(0)}'=\Delta P_0',$ the diagonal embedding of the mirabolic subgroup of $G'$ into its product $P_0'\times P_0'.$ Therefore, the orbital integral $J_{\Geo,\du}^{\bi,T}(f,\textbf{s})$ becomes 
		\begin{align*}
			\int_{[Z'^T]}\int_{[Z'^T]}\int_{[\overline{G'}]}\int_{[\overline{G'}]}&\sum_{(\gamma_1,\gamma_2)\in \Delta P_0'(F)\backslash (G'(F)\times G'(F))}f(\iota(z_1^{-1}x^{-1}\gamma_1^{-1})\gamma(0) \iota(\gamma_2z_1z_2y))\\
			& \phi_1'(x)\overline{\phi_2'(y)}\omega'(z_1)\overline{\omega_2(z_2)}|\det z_1x|^{s_1}|\det z_1z_2y|^{s_2}dxdyd^{\times}z_1d^{\times}z_2.
		\end{align*}
		
		After a change of variables we then write $J_{\Geo,\du}^{\bi,T}(f,\textbf{s})$ as 
		\begin{align*}
			\int_{Z'^T(\mathbb{A}_F)}\int_{\overline{G'}(\mathbb{A}_F)}&\int_{Z'^T(\mathbb{A}_F)}f(\iota(z_1^{-1}x^{-1})\gamma(0) \iota(z_1z_2xy))\omega'(z_1)\overline{\omega_2(z_2)}\\
			& \int_{P_0'(F)\backslash\overline{G'}(\mathbb{A}_F)}\phi_1'(x)\overline{\phi_2'(xy)}|\det z_1x|^s|\det z_1z_2xy|^{s_2}dxd^{\times}z_1dyd^{\times}z_2.
		\end{align*}

		Let $\eta=(0,\cdots,0,1)\in F^n.$ Then $J_{\Geo,\gamma(0)}^{\bi,T}(f,\textbf{s})$ is equal to 
		\begin{equation}\label{44}
			\int_{Z'^T(\mathbb{A}_F)}\int_{\overline{G'}(\mathbb{A}_F)}\int_{[\overline{G'}]}\phi_1'(x)\overline{\phi_2'(xy)}E_{\du}^T(x,s_1+s_2;f^{\dag}(\cdot,z_2,y))dxdyd^{\times}z_2,
		\end{equation}
		where $E_{\du}^T(x,s;f^{\dag}(\cdot,z_2,y))$ is a truncated Eisenstein series defined by 
		\begin{align*}
		\sum_{\delta\in P_0(F)\backslash \overline{G'}(F)}\int_{Z'^T(\mathbb{A}_F)}f^{\dag}(\eta z_1\delta x,z_2,y)|\det z_1x|^sd^{\times}z_1.
		\end{align*}
		Note that $E_{\du}(x,s;f^{\dag}(\cdot,z_2,y))$ converges absolutely when $\Re(s)>1.$ Then 
		$$
		\lim_{T\rightarrow \infty}E_{\du}^T(x,s;z_2,y)=E_{\du}(x,s;z_2,y;f^{\dag}),\ \ \Re(s_1+s_2)>0.
		$$
		
		Note also that $f(u(\mathbf{x})\iota(y))\neq 0$ unless $y$ lies in a compact set of $G'(\mathbb{A}_F),$ uniformly for $\mathbf{x}\in\mathbb{A}_F^n.$ Therefore, $J_{\Geo,\gamma(0)}^{\bi}(f,\textbf{s})$ is equal to 
		\begin{align*}
			\int_{Z'(\mathbb{A}_F)}\int_{\overline{G'}(\mathbb{A}_F)}\int_{[\overline{G'}]}\phi_1'(x)\overline{\phi_2'(xy)}E_{\du}(x,s_1+s_2;z_2,y;f^{\dag})dxdyd^{\times}z_2,
		\end{align*}
	which converges absolutely. Then \eqref{45} follows.
	\end{proof}
\begin{remark}
Under certain choice of $f$ (e.g., $f=f^*$) one obtains that $J_{\Geo,\gamma(0)}^{\bi}(f,\textbf{s})$ and $J^{\Reg}_{\Geo,\sm}(f,\textbf{s})$ have the same special value (or residue if they are not holomorphic) at $s_1+s_2=0.$ A special case of this phenomenon is the equality between the two unipotent orbital integrals in \cite{RR05}, which explains the factor $2$ in the leading term of Duke's theorem (cf. \cite{Duk95}). 
\end{remark}

	Denote by
	\begin{align*}\label{156}
		l_v(x_v,s_2)=\int_{G'(F_v)}f_v(u(\eta x_v)\iota(y_v))|\det(y_v)|_v^{s_2}\overline{W_{\phi_2',v}(x_vy_v)}dy_v.
	\end{align*}		
	
	Since \eqref{45} converges absolutely when $\Re(s_1)+\Re(s_2)>1,$ we can unfold the Eisenstein series and write 
	\begin{equation}\label{154}
		J_{\Geo,\du}^{\bi}(f,\textbf{s})=\prod_{v\in\Sigma_F}J^{\bi}_{\Geo,\du,v}(f_v,\textbf{s}),
	\end{equation}
	where
	\begin{align*}
		J^{\bi}_{\Geo,\du,v}(f_v,\textbf{s}):=\int_{N'(F_v)\backslash {G'}(F_v)}W_{\phi_1',v}(x_v)l_v(x_v,\eta x_v,s_2)|\det x_v|_v^{s_1+s_2}dx_v.
	\end{align*}
	
	Hence \eqref{45} is factorizable. We then compute the local integrals $J^{\bi}_{\Geo,\gamma(0),v}(f_v,\textbf{s})$ in the RHS of \eqref{154} as local $L$-factors $L_v(s_1+s_2,\pi_{1,v}'\times\widetilde{\pi}_{2,v}')$ if $v<\infty.$ 
	
	\begin{prop}\label{prop21}
		Let notation be as before. Then 
		\begin{equation}\label{155}
			\frac{J_{\Geo,\du}^{\bi}(f,\textbf{s})}{W_{\phi_1'}(I_n)\overline{W_{\phi_2'}(I_n)}}=\frac{\Lambda(s_1+s_2,\pi'_1\times\widetilde{\pi}_2')}{\Vol(K_0(\mathfrak{N}))|\mathfrak{N}|^{n(s_1+s_2)}}\prod_{v\mid \infty}\mathcal{F}_v^{\dagger}(f_v,\textbf{s};\phi_1',\phi_2'),
		\end{equation} 
		where 
		\begin{equation}\label{150} 
			\mathcal{F}_v^{\dagger}(f_v,\textbf{s};\phi_1',\phi_2')=\frac{J^{\bi}_{\Geo,\du,v}(f_v,\textbf{s})}{L_v(s_1+s_2,\pi'_{1,v}\times\widetilde{\pi}_{2,v}')W_{\phi_1',v}(I_n)\overline{W_{\phi_2',v}(I_n)}}
		\end{equation}
		is an entire function of $\textbf{s}.$
	\end{prop}
	\begin{proof}
		Suppose $v<\infty$. By definition, $f_v(u(\eta x_v)\iota(y_v))$ is nonvanishing if and only if $u(\eta x_v)\iota(y_v)\in Z(F_v)K_v,\ \text{i.e.},\ \lambda_vu(\eta x_v)\iota(y_v)\in K_v$ for some $\lambda_v\in F_v^{\times}.$ Taking the inverse we then obtain 
		\begin{align*}
			\begin{pmatrix}
				\lambda_v^{-1}y_v^{-1}&\\
				&\lambda_v^{-1}
			\end{pmatrix}\begin{pmatrix}
				I_{n}&\\
				-\eta x_v&1
			\end{pmatrix}\in K_v.
		\end{align*}
		Therefore, $\lambda_v\in\mathcal{O}_{F_v}^{\times},$ $y_v\in K_v'$ and $\eta x_v\in \mathfrak{N}\mathcal{O}_{F_v}^n,$ deducing
		\begin{equation}\label{106}
			l_v(x_v,s_2)=\frac{\overline{W_{\phi_2',v}(x_v)}}{\Vol(K_v)}\cdot \textbf{1}_{\mathfrak{N}\mathcal{O}_{F_v}^n}(\eta x_v).
		\end{equation} 
		
		Therefore, $J_{\Geo,\du,v}^{\bi}(f,\textbf{s})$ is equal to 
		\begin{align*}
			\frac{1}{\Vol(K_v)}\int_{N'(F_v)\backslash {G'}(F_v)}W_{\phi_1',v}(x_v)\overline{W_{\phi_2',v}(x_v)}\cdot \textbf{1}_{\mathfrak{N}\mathcal{O}_{F_v}^n}(\eta x_v)|\det x_v|_v^{s_1+s_2}dx_v.
		\end{align*}
		
		By definition $\phi_1'$ and $\phi_2'$ are right $K_v'$-invariant, i.e., spherical at $v.$ Then we obtain by Rankin-Selberg theory (cf. e.g., \cite{JPSS83}) that 
		\begin{align*}
			J^{\bi}_{\Geo,\du,v}(f_v,\textbf{s})=\frac{W_{\phi_1',v}(I_n)\overline{W_{\phi_2',v}(I_n)}}{\Vol(K_v)|\mathfrak{N}|^{n(s_1+s_2)}}\cdot L_v(s_1+s_2,\pi_{1,v}'\times\widetilde{\pi}_{2,v}'),
		\end{align*}
		and \eqref{150} is entire. Thus \eqref{155} holds when $\Re(s_1)+\Re(s_2)>1.$ Since there are only finitely many $v$'s satisfying $v\mid \infty,$ and each $\mathcal{F}_v^{\dagger}(f_v,\textbf{s};\phi_1',\phi_2')$ is entire, then \eqref{155} can be extended as an equality of meromorphic functions over the whole complex plane. Then Proposition  \ref{prop21} follows.
	\end{proof}
 
\subsection{The Orbital Integral $J^{\Reg,\RNum{2}}_{\Geo,\bi}(f,\textbf{s})$}\label{sec10}
In this section we investigate 
$$
J^{\Reg,\RNum{2}}_{\Geo,\bi}(f,\textbf{s}):=\lim_{T\rightarrow\infty}\Big[J_{\Geo,\gamma_{\tau}}^{\bi,T}(f,\textbf{s})+J_{\Geo,\reg}^{\bi,T}(f,\textbf{s})\Big].
$$
where $J_{\Geo,\gamma_{\tau}}^{\bi,T}(f,\textbf{s})$ and $J_{\Geo,\reg}^{\bi,T}(f,\textbf{s})$ are defined by \eqref{91} and \eqref{29}, respectively. 

We will show $J^{\Reg,\RNum{2}}_{\Geo,\bi}(f,\textbf{s})$ converges absolutely for all $\mathbf{s}.$ Moreover,   the regular orbital integral $\lim_{T\rightarrow\infty}J_{\Geo,\reg}^{\bi,T}(f,\textbf{s})$ is stable in the sense of \cite{MR12} and \cite{FW09}.
	\begin{thm}\label{Red}
		Let notation be as before. For $n\geq 1,$ we have the following.  
		\begin{enumerate}
			\item The integral	$J^{\Reg,\RNum{2}}_{\Geo,\bi}(f,\textbf{s})$ is equal to
\begin{align*}
\sum_{\substack{(\boldsymbol{\xi},t)\in F^n\\
			(\boldsymbol{\xi},t)\neq \textbf{0}}}&\int_{P_0'(F)\backslash {G'}(\mathbb{A}_F)}\int_{{G'}(\mathbb{A}_F)}f\left(\iota(x)^{-1}\begin{pmatrix}
		I_{n-1}&&\boldsymbol{\xi}\\
		&1&t\\
		&1&1
	\end{pmatrix}\iota(xy)\right)\\
	&\qquad \qquad \qquad \qquad \phi_1'(x)\overline{\phi_2'(xy)}|\det x|^{s_1+s_2}|\det y|^{s_2}dydx,
\end{align*}
which converges absolutely. 
\item Suppose $|\mathfrak{N}|$ is large enough, and $f_S$ is fixed. Then $J^{\Reg,\RNum{2}}_{\Geo,\bi}(f,\textbf{s})$ is equal to
\begin{align*}
\sum_{\substack{\boldsymbol{\xi}\in F^{n-1}\\
			\boldsymbol{\xi}\neq \textbf{0}}}&\int_{P_0'(F)\backslash {G'}(\mathbb{A}_F)}\int_{{G'}(\mathbb{A}_F)}f\left(\iota(x)^{-1}\begin{pmatrix}
		I_{n-1}&&\boldsymbol{\xi}\\
		&1&\\
		&1&1
	\end{pmatrix}\iota(xy)\right)\\
	&\qquad \qquad \qquad \qquad \phi_1'(x)\overline{\phi_2'(xy)}|\det x|^{s_1+s_2}|\det y|^{s_2}dydx,
\end{align*}
and $J^{\Reg,\RNum{2}}_{\Geo,\bi}(f,\textbf{s})\ll |\mathfrak{N}|,$
where the implied constant relies on $\textbf{s},$ $f_S,$ $\phi_1'$ and $\phi_2'.$ In particular, $J^{\Reg,\RNum{2}}_{\Geo,\bi}(f,\textbf{s})\equiv 0$ if $n=1.$ 
\end{enumerate} 
	\end{thm}
\subsubsection{Global Reduction by Cauchy-Schwartz Inequality}
By Cauchy-Schwartz,
\begin{align*}
\big|J^{\Reg,\RNum{2}}_{\Geo,\bi}(f,\textbf{s})\big|\leq \sqrt{\mathcal{J}_{\Geo,\bi}^{\Reg,\RNum{2}}(f,\phi_1',s_1)\tilde{\mathcal{J}}_{\Geo,\bi}^{\Reg,\RNum{2}}(f,\phi_2',s_2)},
\end{align*}
where $\mathcal{J}_{\Geo,\bi}^{\Reg,\RNum{2}}(f,\phi_1',s)$ is defined by 
\begin{align*}
\sum_{\substack{(\boldsymbol{\xi},t)\in F^n\\
			(\boldsymbol{\xi},t)\neq \textbf{0}}}\iint\Bigg|f\left(\iota(x)^{-1}\begin{pmatrix}
		I_{n-1}&&\boldsymbol{\xi}\\
		&1&t\\
		&1&1
	\end{pmatrix}\iota(xy)\right)\Bigg|\big|\phi_1'(x)\big|^2|\det x|^{2\Re(s)}dxdy,
\end{align*}
and $\tilde{\mathcal{J}}_{\Geo,\bi}^{\Reg,\RNum{2}}(f,\phi_2',s)$ is defined by 
\begin{align*}
\sum_{\substack{(\boldsymbol{\xi},t)\in F^n\\
			(\boldsymbol{\xi},t)\neq \textbf{0}}}&\iint\Bigg|f\left(\iota(x)^{-1}\begin{pmatrix}
		I_{n-1}&&\boldsymbol{\xi}\\
		&1&t\\
		&1&1
	\end{pmatrix}\iota(xy)\right)\Bigg|\big|\phi_2'(xy)\big|^2|\det xy|^{2\Re(s)}dxdy,
\end{align*}
with $x$ (resp. $y$) ranges over $P_0'(F)\backslash {G'}(\mathbb{A}_F)$ (resp. $G'(\mathbb{A}_F)$).

Define $f^{-1}$ by $f^{-1}(g)=f(g^{-1}),$ $g\in G(\mathbb{A}_F).$ Changing variable $x\mapsto xy^{-1}$ and $y\mapsto y^{-1}$ and swapping integrals, we then derive that 
$$
\tilde{\mathcal{J}}_{\Geo,\bi}^{\Reg,\RNum{2}}(f,\phi_2',s)=\mathcal{J}_{\Geo,\bi}^{\Reg,\RNum{2}}(f^{-1},\phi_2',-s).
$$ 

Therefore, it suffices to show that $\mathcal{J}_{\Geo,\bi}^{\Reg,\RNum{2}}(f,\phi_1',s)$ converges absolutely.

\subsubsection{Counting Rational Points}
By Siegel's results on $[P_0']$       every element may be represented in the form $b=ak^*\in A^*(F_{\infty})\times\Omega^*,$ where $$A^*(F_{\infty}):=\big\{\diag(a_1,\cdots,a_{n-1},1)\in T_{B'}(F_{\infty}):\ |a_{1,v}|_v\geq\cdots\geq |a_{n-1,v}|_v,\ v\mid\infty\big\},$$ and $\Omega^*$ is a large enough fixed compact set in $P_0'(\mathbb{A}_F)$. Let $\mathcal{V}$ be the set of rational primes such that $k_v^*\in G'(\mathcal{O}_{F_v})$ whenever $v\notin \mathcal{V}.$ So $\mathcal{V}$ is fixed, relying only on $\Omega^*$.

Now we analyze $f\left(\iota(x)^{-1}\begin{pmatrix}
	I_{n-1}&&(1-t)\boldsymbol{\xi}\\
	&1&t\\
	&1&1
\end{pmatrix}\iota(xy)\right)$ to determine the support of $x$ and $y.$ By Iwasawa decomposition we may write $x=zbk,$ $z\in \mathbb{A}_F^{\times},$ $b=ak^*\in [P_0']\subset A^*(F_{\infty})\times\Omega^*,$ and $k\in K'.$ So
$$
f\left(\iota(x)^{-1}\begin{pmatrix}
	I_{n-1}&&(1-t)\boldsymbol{\xi}\\
	&1&t\\
	&1&1
\end{pmatrix}\iota(xy)\right)
$$
factorizes as the product $\mathfrak{F}_{\infty}\mathfrak{F}_{\fin},$ where 
\begin{equation}\label{55}
	\mathfrak{F}_{\infty}:=f_{\infty}\left(\iota(z_{\infty}ak_{\infty}^*k_{\infty})^{-1}\begin{pmatrix}
		I_{n-1}&&(1-t)\boldsymbol{\xi}\\
		&1&t\\
		&1&1
	\end{pmatrix}\iota(z_{\infty}ak_{\infty}^*k_{\infty}y_{\infty})\right)
\end{equation}
and $\mathfrak{F}_{\fin}:=\prod_{v< \infty}\mathfrak{F}_v,$ with 
\begin{equation}\label{58}
	\mathfrak{F}_v:=f_v\left(\iota(z_vk_v^*k_v)^{-1}\begin{pmatrix}
		I_{n-1}&&(1-t)\boldsymbol{\xi}\\
		&1&t\\
		&1&1
	\end{pmatrix}\iota(z_vk_v^*k_vy_v)\right),\ \ v<\infty. 
\end{equation}

\begin{lemma}\label{lem9.2}
Let notation be as before. Define $\mathfrak{X}(f;C)$ by 
\begin{align*}
\big\{(\xi_1,\cdots,\xi_{n-1},t)\in F^n:\ \xi_j\in |\mathfrak{N}|C^{-1}\mathcal{O}_F,\ \frac{t}{t-1}\in |\mathfrak{N}|C^{-1}\mathcal{O}_F,\ 1\leq j<n\big\}.
\end{align*}

Then there exists an integer $C\in \mathbb{Z}$ depending on $\otimes_{v\in S\cap\Sigma_{F,\fin}}\supp f_v$ such that $\prod_{v<\infty}\mathfrak{F}_v=0$ 
unless $(\boldsymbol{\xi},t)\in\mathfrak{X}(f;C).$ Moreover,
\begin{align*}
	\int_{Z'(F_v)}\int_{G'(F_v)}\Bigg|f_v\left(\iota(z_vk_v^*k_v)^{-1}\begin{pmatrix}
		I_{n-1}&&(1-t)\boldsymbol{\xi}\\
		&1&t\\
		&1&1
	\end{pmatrix}\iota(z_vk_v^*k_vy_v)\right)\Bigg|dy_vd^{\times}z_v
\end{align*}
is $\ll q_v^{ne_v(\mathfrak{N})}\big(\big|\min\{e_v(t)-e_v(t-1), e_v(\xi_j):\ 1\leq j\leq n-1\}\big|+\big|e_v(t-1)\big|+1\big),$ where the implied constant relies on $S.$
\end{lemma}
\begin{proof}
Write $\boldsymbol{\xi}=\transp{(\xi_1,\cdots,\xi_{n-1})}\in F^{n-1}.$ Let $v<\infty,$ write $z_v=\varpi_v^rI_n.$


\begin{enumerate}
	\item[(A)] Suppose $r\geq 0.$ By Cramer's rule there exists $k_v^{(1)}\in G(\mathcal{O}_{F_v})$ such that 
	\begin{align*}
\iota(z_v)^{-1}\begin{pmatrix}
			I_{n-1}&&(1-t)\boldsymbol{\xi}\\
			&1&t\\
			&1&1
		\end{pmatrix}\iota(z_v)=(1-t)k_{v}^{(1)}\begin{pmatrix}
			\frac{1}{1-t}I_{n-1}&&\varpi_v^{-r}\boldsymbol{\xi}\\
			&\frac{1}{1-t}&\frac{\varpi_v^{-r}t}{1-t}\\
			&&1
		\end{pmatrix}.
	\end{align*}
	
	\item[(B)] Suppose $r< 0.$ By Cramer's rule there exists $k_{v}^{(2)}\in G(\mathcal{O}_{F_v})$ such that 
	\begin{align*}
\iota(z_v)^{-1}\begin{pmatrix}
			I_{n-1}&&(1-t)\boldsymbol{\xi}\\
			&1&t\\
			&1&1
		\end{pmatrix}\iota(z_v)=\varpi_v^{-r}(1-t)k_{v}^{(2)}\begin{pmatrix}
			\frac{\varpi_v^r}{1-t}I_{n-1}&&\boldsymbol{\xi}\\
			&\frac{\varpi_v^{2r}}{1-t}&\frac{\varpi_v^{r}}{1-t}\\
			&&1
		\end{pmatrix}.
	\end{align*}
	
\end{enumerate}

Let $t\in F-\{1\}.$ Consider the various cases as follows.
\begin{enumerate}
	\item Let $v\notin S\cup\mathcal{V}$ and $v\nmid\mathfrak{N}.$ 
	\begin{itemize}
		\item Suppose $e_v(t-1)\geq 0.$ Then from the case (B) and the support of $f_v$ one must have $r\geq 0.$ So 
		\begin{align*}\mathfrak{F}_v=f_v\left(
	\begin{pmatrix}
			\frac{1}{1-t}I_{n-1}&&\varpi_v^{-r}\boldsymbol{\xi}\\
			&\frac{1}{1-t}&\frac{\varpi_v^{-r}t}{1-t}\\
			&&1
		\end{pmatrix}\iota(k_v^*k_vy_v)
	\right)\neq 0
		\end{align*}
	if and only if there exists some $\lambda_v\in F_v^{\times}$ such that 
	\begin{equation}\label{9.3}
\lambda_v\begin{pmatrix}
			\frac{1}{1-t}I_{n-1}&&\varpi_v^{-r}\boldsymbol{\xi}\\
			&\frac{1}{1-t}&\frac{\varpi_v^{-r}t}{1-t}\\
			&&1
		\end{pmatrix}\iota(k_v^*k_vy_v)\in G(\mathcal{O}_{F_v}),
	\end{equation}
	which forces that $\lambda_v\in \mathcal{O}_{F_v}^{\times},$ $y_v\in (t-1)K_v',$ and 
		\begin{align*}
			\begin{cases}
				e_v(t)-e_v(t-1)\geq r\geq 0\\
				e_v(\xi_j)\geq r\geq 0,\ 1\leq j\leq n-1.
			\end{cases}
		\end{align*}
					
		\item Suppose $e_v(t-1)\leq -1.$ Then $e_v(t)=e_v(t-1).$ From the case (B) we obtain $e_v(t-1)\leq r\leq -1;$ in the case (A) we have $0\leq r\leq e_v(t)-e_v(t-1),$ implying $r=0.$ So 
		\begin{equation}\label{9.4..}
			\mathfrak{F}_v=\begin{cases}
			f_v\left(
	\begin{pmatrix}
			\frac{1}{1-t}I_{n-1}&&\boldsymbol{\xi}\\
			&\frac{1}{1-t}&\frac{t}{1-t}\\
			&&1
		\end{pmatrix}\iota(k_v^*k_vy_v)
	\right),&\text{if $r=0$}\\
	f_v\left(
	\begin{pmatrix}
			\frac{\varpi_v^r}{1-t}I_{n-1}&&\boldsymbol{\xi}\\
			&\frac{\varpi_v^{2r}}{1-t}&\frac{p^{r}}{1-t}\\
			&&1
		\end{pmatrix}\iota(k_v^*k_vy_v)
	\right),& \text{if $r<0$}.
			\end{cases}
				\end{equation}
		
		Hence $\mathfrak{F}_v\neq 0$ unless $\boldsymbol{\xi}\in \mathcal{O}_{F_v}^{n-1}$ and $e_v(t)-e_v(t-1)\geq 0$ if $r=0,$ and $\mathfrak{F}_v\neq 0$ unless $\boldsymbol{\xi}\in \mathcal{O}_{F_v}^{n-1}$ and $r-e_v(t-1)\geq 0$ if $r<0.$ 
		
		In the case that $r<0,$ $r-e_v(t-1)\geq 0$ implies that $e_v(t-1)\leq r<0.$ So $e_v(t)-e_v(t-1)=0.$ 
	\end{itemize}
	
	Therefore, in the above cases we have $\mathfrak{F}_v\neq 0$ unless
	\begin{align*}
			\begin{cases}
				e_v(t)-e_v(t-1)\geq 0\\
				e_v(\xi_j)\geq  0,\ 1\leq j\leq n-1
			\end{cases}
		\end{align*} 
		and $y_v\in (t-1)K_v'$ or $(t-1)\begin{pmatrix}
			\varpi_v^{-r}\\
			&\varpi_v^{-2r}
		\end{pmatrix}K_v'.$
	
	\item Let $v\in S$ or $v\in\mathcal{V}.$ Then by definition, $\supp f_v$ is $Z(F_v)D_v,$ where $D_v$ is a compact set of $G(F_v).$ Hence, similar to \eqref{9.3}, there exists some $\lambda_v\in F_v^{\times}$ such that 
	\begin{equation}\label{9.4}
	\lambda_v\begin{pmatrix}
			\frac{1}{1-t}I_{n-1}&&\varpi_v^{-r}\boldsymbol{\xi}\\
			&\frac{1}{1-t}&\frac{\varpi_v^{-r}t}{1-t}\\
			&&1
		\end{pmatrix}\iota(k_v^*k_vy_v)\in \iota(k_v^*k_v)D_v
	\end{equation}
	if $r\geq 0;$ and if $r<0,$ we have 
	\begin{equation}\label{9.5}
	\lambda_v\begin{pmatrix}
			\frac{\varpi_v^r}{1-t}I_{n-1}&&\boldsymbol{\xi}\\
		&\frac{\varpi_v^{2r}}{1-t}&\frac{p^{r}}{1-t}\\
			&&1
		\end{pmatrix}\iota(k_v^*k_vy_v)\in \iota(k_v^*k_v)D_v.
	\end{equation}
	
	Note that both \eqref{9.4} and \eqref{9.5} implies that 
	\begin{align*}
			\begin{cases}
				e_v(t)-e_v(t-1)\geq -C\\
				e_v(\xi_j)\geq  -C,\ 1\leq j\leq n-1.
			\end{cases}
		\end{align*} 
		for some constant $C\geq 0$ which relies at most on $S,$ and $y_v$ supports in a translation of a compact subset of $K_v'$ (depending at most on $t$ and $r$).
	
		\item Finally we consider the case that $v\mid\mathfrak{N}.$ Then $v\notin\mathcal{V}.$ By Iwasawa decomposition, we can write $k_v^*k_vy_v=\varpi_v^{r'}a_v'u_v'k_v'$ for some $r'\in\mathbb{Z},$ $a_v'\in \overline{T_{B'}}(F_v),$ $u_v'\in N'(F_v)$ and $k_v'\in K_v'.$ Since $f_v=\Vol(K_v)^{-1}\textbf{1}_{Z(F_v)K_v},$ then  
		\begin{align*}
			\mathfrak{F}_v=f_v\left(\begin{pmatrix}
		*&*&\varpi_v^{-r}(1-t)\boldsymbol{\xi}\\
		&\varpi_v^{r'}&\varpi_v^{-r}t\\
		&\varpi_v^{r+r'}&1
	\end{pmatrix}\right)
		\end{align*}  
		and thus $\mathfrak{F}_v\neq 0$ unless $\lambda_v\begin{pmatrix}
		*&*&\varpi_v^{-r}(1-t)\boldsymbol{\xi}\\
		&\varpi_v^{r'}&\varpi_v^{-r}t\\
		&\varpi_v^{r+r'}&1
	\end{pmatrix}\in K_v$ for some $\lambda_v\in F_v,$ which leads to that 
		\begin{equation}\label{cons}
			\begin{cases}
				e_v(\lambda_v)+r'=0,\ e_v(\lambda_v)\geq 0,\\
				e_v(\lambda_v)+r+r'\geq e_v(\mathfrak{N})\geq 1,\\
				e_v(\lambda_v)-r+e_v(t)\geq 0,\\
				e_v(\lambda_v)-r+e_v(1-t)+e_v(\xi_j)\geq 0,\ 1\leq j\leq n-1,\\
				2e_v(\lambda_v)+r'+e_v(1-t)=0.
			\end{cases}
		\end{equation}
		Notice that the last constraint in \eqref{cons} comes from the determinant of the lower right $2\times 2$-corner. So $r\geq e_v(\mathfrak{N}),$ $e_v(1-t)=-e_v(\lambda_v)\leq 0,$ and $e_v(\xi_j)\geq r\geq e_v(M),$ $1\leq j\leq n-1.$
		
		Moreover, it follows from the assumption $\mathfrak{F}_v\neq 0$ that $y_v$ supports in a translation of  $K_v'$ (depending at most on $t$ and $r$).
		\end{enumerate}

Putting the above discussions together, we then see that $\prod_{v<\infty}\mathfrak{F}_v=0$ 
unless $(\boldsymbol{\xi},t)\in\mathfrak{X}(f;C).$ Moreover, for $v<\infty,$ the possible values of $r$ are $\leq \big|\min\{e_v(t)-e_v(t-1), e_v(\xi_j):\ 1\leq j\leq n-1\}\big|+\big|e_v(t-1)\big|+1,$ and is equal to $1$ if $e_v(\xi_1)=\cdots=e_v(\xi_{n-1})=e_v(t)=e_v(t-1)=0;$ the support of $y_v$ is a translation of a compact subset $C_v'$ of $K_v'$ (depending at most on $t$ and $r$), and $C_v'=K_v'$ if $e_v(\xi_1)=\cdots=e_v(\xi_{n-1})=e_v(t)=e_v(t-1)=0.$
Hence Lemma \ref{lem9.2} follows from the fact that $\|f_v\|_v\ll q_v^{ne_v(\mathfrak{N})},$ where the implied constant relies on $S.$
\end{proof}

\begin{lemma}\label{cor9.4}
Let notation be as before. 
\begin{enumerate}
	\item There exists $x, y$ such that $\mathfrak{F}_{\infty}=0$ unless $(\boldsymbol{\xi},t)$ satisfies that
\begin{equation}\label{992}
	|\xi_j|_v\ll |z_va_{j,v}|_v,\ 1\leq j\leq n-1,\ |t|_v\ll |z_v|_v\ll |t-1|_v,\ v\mid\infty,
\end{equation} 
where where $\mathfrak{F}_{\infty}$ is defined by \eqref{55}, and $y_{\infty}$ ranges over a compact set $\mathcal{Y}_{\infty}(t)=\prod_{v\mid\infty}\mathcal{Y}_v(t)$  determined by $k_{\infty}^*k_{\infty}$, $f_{\infty}$ and $t$:
\begin{align*}
	\mathcal{Y}_v&(t):=(k_v^*k_v)^{-1}\cdot \big\{
	z_{v}'\diag(a_{1}',\cdots, a_{n-1}',1)u_{v}k_{v}':\ |t-1|_v^{-1}\ll |z_v'|_v\ll |t-1|_v,\\
	& 1\ll |a_{j,v}'|_v\ll |t-1|_v,\ 1\leq j< n,\ \text{$u_v$ lies in a compact set determined by $f_{\infty}$}
	\big\}.
\end{align*}
Here all the implied constants depend on $\supp f_{\infty},$

\item When $|\mathfrak{N}|$ is large enough, $\mathfrak{F}_{\infty}\mathfrak{F}_{\fin}=0$ unless $t=0.$
\item For all $\mathfrak{N}\subseteq \mathcal{O}_F,$ the number of $t\in F$ such that $\mathfrak{F}_{\infty}\mathfrak{F}_{\fin}\neq 0$ is $O(1)$, where the implied constant $O(1)$ depends on $\mathfrak{N}$ and $\supp f_{S}.$ 
\end{enumerate}
\end{lemma}
\begin{proof}
By definition, $\mathfrak{F}_{\infty}\neq 0$ implies that 
\begin{equation}\label{990}
	\lambda_{\infty}\iota(z_{\infty}ak_{\infty}^*k_{\infty})^{-1}\begin{pmatrix}
		I_{n-1}&&(1-t)\boldsymbol{\xi}\\
		&1&t\\
		&1&1
	\end{pmatrix}\iota(z_{\infty}ak_{\infty}^*k_{\infty}y_{\infty})\in \mathcal{D}_{\infty}
\end{equation}
for some $\lambda_{\infty}\in F_{\infty}^{\times},$ where $\mathcal{D}_{\infty}$ is a compact set of $G(F_{\infty})$ determined by $\supp f_{\infty}.$ 

Write $k_{\infty}^*k_{\infty}y_{\infty}=z_{\infty}'a'u_{\infty}k_{\infty}'$ in the Iwasawa coordinates: $z_{\infty}'\in F_{\infty}^{\times},$ $a'=\diag(a_{1}',\cdots, a_{n-1}',1)\in \overline{T_{B'}}(F_{\infty}),$ $u_{\infty}\in N'(F_{\infty})$ and $k_{\infty}'\in K_{\infty}'$. Then \eqref{990} becomes 
\begin{equation}\label{991}
	\lambda_{\infty}\iota(z_{\infty}a)^{-1}\begin{pmatrix}
		I_{n-1}&&(1-t)\boldsymbol{\xi}\\
		&1&t\\
		&1&1
	\end{pmatrix}\iota(z_{\infty}z_{\infty}'aa'u_{\infty})\in \mathcal{D}_{\infty}',
\end{equation}
where $\mathcal{D}_{\infty}':=\iota(k_{\infty}^*k_{\infty})\mathcal{D}_{\infty}\iota(k_{\infty}'^{-1}).$

Denote by $\mathcal{M}_{\infty}=\prod_{v\mid\infty}\mathcal{M}_v$ the matrix on the left hand side of \eqref{991}. Then \eqref{991} becomes $\mathcal{M}_{\infty}\in \mathcal{D}_{\infty}'=\prod_{v\mid\infty}\mathcal{D}_v',$ which is a compact set. Let $v\mid\infty.$

Consider the upper-left $(n-1)\times n$-corner of $\mathcal{M}_v$ we derive that $u_v$ lies in a compact set determined by $\supp f_{\infty},$ and $|a_{j,v}'|_v\asymp |\lambda_v|_v^{-1},$ $1\leq j\leq n-1.$

Multiplying the $(n,n+1)$-th entry of $\mathcal{M}_v$ with the $(n+1,n)$-th entry of $\mathcal{M}_v^{-1}$ we then obtain that $\big|\frac{t}{t-1}\big|_v\ll 1.$ In conjunction with the $(i,n+1)$-th entry, $1\leq j\leq n-1,$ we have 
$$
|\lambda_va_{j,v}^{-1}z_v^{-1}(t-1)\xi_v|_v\ll 1,\ 1\leq j\leq n-1.
$$

Considering the $(n+1,n+1)$-th entry of $\mathcal{M}_v^{-1}$ we obtain $|\lambda_v|_v|t-1|_v\gg 1.$ So $|a_{j,v}^{-1}z_v^{-1}\xi_v|_v\ll 1,$ $1\leq j\leq n-1.$

Consider the $(n,n+1)$-th entry of $\mathcal{M}_v$ and  $(n+1,n+1)$-th entry of $\mathcal{M}_v$ we derive that $|\lambda_v|_v\ll 1$ and $|z_v|_v\gg |t|_v.$ Consider the $(n+1,n)$-th entry of $\mathcal{M}_v^{-1}$ and  $(n+1,n+1)$-th entry of $\mathcal{M}_v$ we derive that $|z_v|_v\ll |t-1|_v.$ 

Hence \eqref{992} holds and $1\ll |\lambda_v|_v^{-1}\ll |t-1|_v.$ So 
$$
1\ll |a_{j,v}'|_v\ll |t-1|_v,\ \ 1\leq j\leq n-1.
$$

Multiplying the $(n,n)$-th entry of $\mathcal{M}_v$ with $(n+1,n+1)$-th entry of $\mathcal{M}_v^{-1}$ we obtain that $|z_v'|_v\ll |t-1|_v.$ Multiplying the $(n+1,n+1)$-th entry of $\mathcal{M}_v$ with $(n,n)$-th entry of $\mathcal{M}_v^{-1}$ we obtain that $|z_v'|_v\gg |t-1|_v^{-1}.$ So $|t-1|_v^{-1}\ll |z_v'|_v\ll |t-1|_v.$ Putting the discussions together we conclude that 
$y_{\infty}\in \mathcal{Y}_{\infty}(t).$ 

Note that $|t|_v\ll |z_v|_v\ll |t-1|_v$ implies that $\big|\frac{t}{t-1}\big|_v\ll 1.$ Suppose that $|\mathfrak{N}|$ is large enough and $\mathfrak{F}_{\infty}\mathfrak{F}_{\fin}\neq 0.$ Then by Lemma \ref{lem9.2} we have $\big|\frac{t}{t-1}\big|_v\geq C^{-1}|\mathfrak{N}|,$ which is large. This contradicts with $\big|\frac{t}{t-1}\big|_v\ll 1.$ Hence, when $|\mathfrak{N}|$ is large enough, $\mathfrak{F}_{\infty}\mathfrak{F}_{\fin}=0$ unless $t=0.$ 

When $\mathfrak{N}$ is arbitrary, we notice that there are $O(1)$ integers $m\in\mathcal{O}_{F}$ satisfying that $\frac{t}{t-1}=C^{-1}|\mathfrak{N}|m$ and $\big|\frac{t}{t-1}\big|_{\infty}\ll 1.$
\end{proof}

\subsubsection{Proof of Theorem \ref{Red}}
Let $\mathfrak{X}^{\dagger}(a,z_{\infty})$ be the set of rational points $\textbf{0}\neq (\boldsymbol{\xi},t)\in\mathfrak{X}(f;C)$ (cf. Lemma \ref{lem9.2}) which satisfy the constraint  \eqref{992}.

Changing variables $\boldsymbol{\xi}\mapsto (1-t)\boldsymbol{\xi},$ 
$\mathcal{J}_{\Geo,\bi}^{\Reg,\RNum{2}}(f,\phi_1',s)$ becomes
\begin{align*}
\sum_{\substack{(\boldsymbol{\xi},t)\in F^n\\
			(\boldsymbol{\xi},t)\neq \textbf{0}}}\iint\Bigg|f\left(\iota(x)^{-1}\begin{pmatrix}
		I_{n-1}&&(1-t)\boldsymbol{\xi}\\
		&1&t\\
		&1&1
	\end{pmatrix}\iota(xy)\right)\Bigg|\big|\phi_1'(x)\big|^2|\det x|^{2\Re(s)}dxdy,
\end{align*}
where $x$ (resp. $y$) ranges over $P_0'(F)\backslash {G'}(\mathbb{A}_F)$ (resp. $G'(\mathbb{A}_F)$). Here we make use of the fact that $f(\cdots)=0$ if $t=1$ to exclude the case that $t=1$ in the changing of variable $\boldsymbol{\xi}\mapsto (1-t)\boldsymbol{\xi}.$ Swapping the integrals,  $\mathcal{J}_{\Geo,\bi}^{\Reg,\RNum{2}}(f,\phi_1',s)$ is 
 \begin{align*}
\ll 
	\int_{A^*(F_{\infty})}I(\phi_1',a)\Bigg[\max_{\substack{k^*\in\Omega^*\\ k\in K'}}\int_{G'(\mathbb{A}_F)}\int_{Z'(\mathbb{A}_F)}\sum_{\substack{(\boldsymbol{\xi},t)}}\big|\mathfrak{F}_{\infty}\cdot \mathfrak{F}_{\fin}\big|d^{\times}zdy\Bigg]\frac{|\det a|_{\infty}^{2\Re(s)}d^{\times}a}{\delta_{B'}(a)},
\end{align*}
where $\mathfrak{F}_{\infty}$ and $\mathfrak{F}_{\fin}$ are defined by \eqref{55} and \eqref{58}, $(\boldsymbol{\xi},t)\in\mathfrak{X}^{\dagger}(a,z_{\infty})$, and 
 $$
 I(\phi_1',a):=\int_{K'}\int_{\Omega^*}\big|\phi_1'(ak^*k)\big|^2dk^*dk.
 $$
 
 Let $E(t):=1+|t-1|_{\infty}^{\varepsilon}+|t-1|_{\infty}^{-\varepsilon}+|t/(t-1)|_{\infty}^{\varepsilon}.$ By Lemma \ref{lem9.2} and the supnorm $\|f\|_{\infty}\ll 1,$ the integral $\mathcal{J}_{\Geo,\bi}^{\Reg,\RNum{2}}(f,\phi_1',s)$ is 
\begin{align*}
\ll |\mathfrak{N}|^n\int_{A^*(F_{\infty})}I(\phi_1',a)\Bigg[\max_{\substack{k_{\infty}^*\in\Omega_{\infty}^*\\ k_{\infty}\in K_{\infty}'}}\int\int\sum_{\substack{(\boldsymbol{\xi},t)}}E(t)\big|\mathfrak{F}_{\infty}\big|d^{\times}z_{\infty}dy_{\infty}\Bigg]\frac{|\det a|_{\infty}^{2\Re(s)}d^{\times}a}{\delta_{B'}(a)},
\end{align*}
where $y_{\infty}$ (resp. $z_{\infty}$) ranges over $G'(F_{\infty})$ (resp. $F_{\infty}^{\times}$), and $(\boldsymbol{\xi},t)\in\mathfrak{X}^{\dagger}(a,z_{\infty}).$

By Lemma \ref{lem9.2}, $(\boldsymbol{\xi},t)\in\mathfrak{X}^{\dagger}(a,z_{\infty})$ could be parametrized by 
\begin{align*}
	\frac{t}{t-1}=|\mathfrak{N}|C^{-1}{m},\ \ \xi_j=|\mathfrak{N}|C^{-1}{m}_j,\ \ {m},\ {m}_j\in\mathcal{O}_F,\ 1\leq j\leq n-1.
\end{align*}

In conjunction with \eqref{992} in Lemma \ref{cor9.4}, one has $|m|_{\infty}\ll C^{[F:\mathbb{Q}]}|\mathfrak{N}|^{-[F:\mathbb{Q}]},$ and 
$$
|m_j|_{\infty}\ll C^{[F:\mathbb{Q}]}|\mathfrak{N}|^{-[F:\mathbb{Q}]}|z_{\infty}a_{j}|_{\infty}\ll C^{[F:\mathbb{Q}]}|\mathfrak{N}|^{-[F:\mathbb{Q}]}|t-1|_{\infty}|a_{j}|_{\infty},\ 1\leq j<n.
$$

When $t=0,$ $|t-1|_{\infty}=1.$ Suppose $t\neq 0.$ Then $m\neq 0.$ So $|C^{-1}|\mathfrak{N}|{m}-1|_{\infty}\gg C^{-[F:\mathbb{Q}]},$ implying that $|t-1|_{\infty}\ll C^{[F:\mathbb{Q}]}\ll 1.$  In all, we have $|t-1|_{\infty}\ll 1.$ Notice that the set $\mathcal{Y}_v(t)$ (cf. Lemma \ref{cor9.4}) is empty unless $|t-1|_v\gg 1$ for all $v\mid\infty.$ So $|t-1|_v\asymp 1$ for all $v\mid\infty.$ As a consequence, $|z_v|_v\ll 1,$ $v\mid\infty.$ 

Also, if $t\neq 0,$ then $|m|_{\infty}\geq 1.$ It follows from the parametrization $\frac{t}{t-1}=|\mathfrak{N}|C^{-1}{m}$ that $|t|_{\infty}\gg (1+C^{[F:\mathbb{Q}]}|\mathfrak{N}|^{-[F:\mathbb{Q}]})^{-1}.$ 
Hence, 
\begin{equation}\label{994}
\mathfrak{X}^{\dagger}(a,z_{\infty})\textbf{1}_{t\neq 0}\subseteq \mathfrak{X}^{\heartsuit}(a)\textbf{1}_{\substack{|z_{\infty}|_{\infty}\gg (1+C^{[F:\mathbb{Q}]}|\mathfrak{N}|^{-[F:\mathbb{Q}]})^{-1},\ |z_v|_v\ll 1,\ v\mid\infty}}
\end{equation} 
where the set $\mathfrak{X}^{\heartsuit}(a)$ is defined by 
\begin{align*}
	\big\{(\boldsymbol{\xi},t)\in\mathfrak{X}(f;C):\ |t/(t-1)|_{\infty}\ll 1,\ |t-1|_v\asymp 1,\ v\mid\infty;\ |\xi_j|_{\infty}\ll |a_{j}|_{\infty},\ j<n\big\}.
\end{align*}
In particular, $(\boldsymbol{\xi},t)$ ranges over only finitely many values for fixed $a.$ 
  
By Lemma \ref{cor9.4} and \eqref{994}, we have 
\begin{align*}
	\max_{\substack{k_{\infty}^*\in\Omega_{\infty}^*\\ k_{\infty}\in K_{\infty}'}}\int_{F_{\infty}^{\times}}\int_{G'(F_{\infty})}\sum_{\substack{(\boldsymbol{\xi},t)\in \mathfrak{X}^{\dagger}(a,z_{\infty})\\ t\neq 0}}E(t)\big|\mathfrak{F}_{\infty}\big|dy_{\infty}d^{\times}z_{\infty}\ll \#\mathfrak{X}^{\heartsuit}(a)\ll 1,
\end{align*}
where the implied constant depends on $f_S$ and $\mathfrak{N}.$ 

On the other hand, when $t=0,$ by Lemma \ref{cor9.4} the support of $y_{\infty}$ is compact, and $|z|_v\ll 1,$ $v\mid \infty.$ Since $(\boldsymbol{\xi},t)\neq\textbf{0},$ we must have $\boldsymbol{\xi}\neq \textbf{0}.$ Hence, $\xi_1\neq 0,$ implying that  $|m_1|_{\infty}\gg 1.$ So it follows from \eqref{992} that $|\mathfrak{N}|^{[F:\mathbb{Q}]}C^{-[F:\mathbb{Q}]}\ll |z_{\infty}|_{\infty}|a_1|_{\infty}.$ Hence,
\begin{align*}
	\max_{\substack{k_{\infty}^*\in\Omega_{\infty}^*\\ k_{\infty}\in K_{\infty}'}}\int_{F_{\infty}^{\times}}\int_{G'(F_{\infty})}\sum_{\substack{(\boldsymbol{\xi},t)\in \mathfrak{X}^{\dagger}(a,z_{\infty})\\ t=0}}E(t)\big|\mathfrak{F}_{\infty}\big|dy_{\infty}d^{\times}z_{\infty}
\end{align*}
is majorized by 
\begin{align*}
	\int_{\substack{|z_{\infty}|_{\infty}\gg \frac{|\mathfrak{N}|^{[F:\mathbb{Q}]}|a_1|^{-1}_{\infty}}{C^{[F:\mathbb{Q}]}}\\ |z_v|_v\ll1,\ v\mid\infty}}\prod_{j=1}^{n-1}\Big[1+\frac{C^{[F:\mathbb{Q}]}|z_{\infty}|_{\infty}|a_j|_{\infty}}{|\mathfrak{N}|^{[F:\mathbb{Q}]}}\Big]^{1+\varepsilon}d^{\times}z_{\infty}\ll \frac{|a_1|_{\infty}^{(n-1)(1+\varepsilon)}}{|\mathfrak{N}|^{(n-1)[F:\mathbb{Q}]}},
\end{align*}
where the implied constant depends on $n,$ $\varepsilon,$ and $f_S.$

Therefore, it follows from the rapid decay of $\phi_1'$ that $\mathcal{J}_{\Geo,\bi}^{\Reg,\RNum{2}}(f,\phi_1',s)$ converges is $\Re(s)$ lies a compact set of $\mathbb{R}.$

In particular, when $|\mathfrak{N}|$ is large, the constraint $|m|_{\infty}\ll C^{[F:\mathbb{Q}]}|\mathfrak{N}|^{-[F:\mathbb{Q}]}$ forces that $m=0,$ i.e., $t=0.$ So  $J_{\Geo,\reg}^{\bi,T}(f,\textbf{s})\equiv 0,$ and thus
\begin{align*}
	\mathcal{J}_{\Geo,\bi}^{\Reg,\RNum{2}}(f,\phi_1',s)\ll |\mathfrak{N}|\int_{A^*(F_{\infty})}I(\phi_1',a)|a_1|_{\infty}^{(n-1)(1+\varepsilon)}\frac{|\det a|_{\infty}^{2\Re(s)}d^{\times}a}{\delta_{B'}(a)}\ll |\mathfrak{N}|.
\end{align*} 

So Theorem \ref{Red} holds.

\section{Spectral Side: Convergence and Meromorphic Continuation}\label{sec.spec} 
In this section we focus on the convergence of the spectral side $J_{\Spec}^{\Reg,T}(f,\mathbf{s})$ as the truncation parameter $T$ tends to infinity, where $\mathbf{s}=(s_1,s_2)\in\mathbb{C}^2$. Recall  $J_{\Spec}^{\Reg,T}(f,\mathbf{s}):=J_{\Eis}^{\Reg,T}(f,\mathbf{s})+J_{0}^{T}(f,\mathbf{s})$ (cf. \text\textsection \ref{sec4.2}).  Since $\K_0$ decays rapidly on $[\overline{G}]\times [\overline{G}],$ the function $J_0^T(f,\mathbf{s})$ converges absolutely everywhere, as $T\rightarrow\infty,$  to 
$$
J_0(f,\mathbf{s})=\sum_{\substack{\pi\in\mathcal{A}_0([G],\omega)\\ \phi\in \mathcal{B}_{\pi}}}\int_{[G']}\pi(f)\phi(\iota(x))\phi_1'(x)|\det x|^{s_1}dx\int_{[G']}\overline{\phi(\iota(y))\phi_2'(y)}|\det y|^{s_2}dy,
$$ 
where $\mathcal{B}_{\pi}$ is an orthonormal basis of $\pi.$ Note that $\mathcal{B}_{\pi}$ is a finite set since $f$ is $K$-finite.

It then remains to prove the convergence of $J_{\Eis}^{\Reg,T}(f,\mathbf{s})$ as $T\rightarrow\infty,$ which will be done by a purely analytic argument in the \textsection \ref{10.1} under the condition that $\Re(s_1)$ and $\Re(s_2)$  are large enough. 
 
\subsection{Convergence in a Right Half Plane}\label{10.1}
Recall that  $J^{\Reg,T}_{\Eis}(f,\textbf{s})$ is defined by
\begin{align*}
	\int_{[Z'^T]}\int_{[Z'^T]}\int_{[\overline{G'}]}\int_{[\overline{G'}]}&\mathcal{F}_{n+1}\mathcal{F}_{n+1}\K_{\ER}(\iota(z_1x),\iota(z_1z_2y))\phi_1'(x)\overline{\phi_2'(y)}\omega'(z_1)\overline{\omega_2(z_2)}\\
	&|\det z_1x|^{s_1}|\det z_1z_2y|^{s_2}dxdyd^{\times}z_1d^{\times}z_2,
\end{align*}

Since $\phi_1'$ and $\phi_2'$ decay rapidly on $[\overline{G'}],$ then for fixed $T>0,$ $J_{\Eis}^{\Reg,T}(f,\textbf{s})$ converges absolutely and thus defines a holomorphic function of $\mathbf{s}\in\mathbb{C}^2.$ We then wish to take $T\rightarrow \infty$ to remove the truncation. However, the convergence when $T$ tends to infinity is rather subtle. We start with defining $\lim_{T\rightarrow \infty}J_{\Eis}^{\Reg,T}(f,\textbf{s}).$ 

By spectral decomposition, $\K_{\ER}=\K-\K_0,$ implying that $\mathcal{F}_{n+1}\mathcal{F}_{n+1}\K_{\ER}=\mathcal{F}_{n+1}\mathcal{F}_{n+1}\K-\mathcal{F}_{n+1}\mathcal{F}_{n+1}\K_0.$ By Fourier expansion we have $\mathcal{F}_{n+1}\mathcal{F}_{n+1}\K_0=\K_0.$ Hence, we can write $J_{\Eis}^{\Reg,T}(f,\textbf{s})=J_{\Kl}^T(f,\mathbf{s})-J_0^T(f,\mathbf{s}),$ where 
\begin{align*}
	J_{\Kl}^T(f,\mathbf{s}):=\int_{[Z'^T]}\int_{[Z'^T]}&\int_{X}\int_{X}J_{\Kuz}(\iota(z_1x),\iota(z_1z_2y))\phi_1'(x)\overline{\phi_2'(y)}|\det z_1x|^{s_1}\\
	&|\det z_1z_2y|^{s_2}\omega'(z_1)\overline{\omega_2(z_2)}dxdyd^{\times}z_1d^{\times}z_2,
\end{align*}
with $X=N'(\mathbb{A}_F)\backslash \overline{G'}(\mathbb{A}_F),$ and for $g_1, g_2\in G(\mathbb{A}_F),$
\begin{align*}
	J_{\Kuz}(g_1,g_2):=\int_{[N]}\int_{[N]}\K(n_1g_1,n_2g_2)\theta(n_1)\overline{\theta}(n_2)dn_1dn_2.
\end{align*}

Note that $J_{\Kuz}(g_1, g_2)$ is a relative trace formula on $G(\mathbb{A}_F)$ of Kuznetsov type. We will show that $J_{\Kuz}(\iota(x),\iota(y))$ is majorized by a gauge. 

Let $x, y\in A(\mathbb{A}_F)$ be given by $x=\diag (x_1\cdots x_{n}, \cdots, x_1x_2,  x_1, 1),$ and $y=\diag (y_1\cdots y_{n}, \cdots, y_1y_2,  y_1, 1)\in A(\mathbb{A}_F).$ We say a function $\mathcal{G}$ on $A(\mathbb{A}_F)\times A(\mathbb{A}_F)$ is a gauge if it is a positive function of the form 
\begin{align*}
	\mathcal{G}(x,y)=\xi(x_1,\cdots, x_{n}, y_1, \cdots, y_n)\cdot\left( |x_1x_2\cdots x_{n}|^{-M}+|y_1y_2\cdots y_{n}|^{-M}\right)
\end{align*}
where $M\geq 0$ and $\xi$ is a Schwartz-Bruhat function on $(\mathbb{A}_F^{\times})^{2n}.$

\begin{prop}\label{Kl}
	Let notation be as above. Let $T_{B'}$ be the Levi of $B'.$ Then as a function of $(x, y)\in T_{B'}(\mathbb{A}_F)\times T_{B'}(\mathbb{A}_F),$ $J_{\Kuz}(\iota(x),\iota(y))$ is a majorized by a finite sum of gauges on $A(\mathbb{A}_F)\times A(\mathbb{A}_F).$ 
\end{prop}
\begin{proof}
	By definition of the kernel function $\K(x,y)$ we have   
	\begin{align*}
		J_{\Kuz}(\iota(x),\iota(y))=\int_{[N]}\int_{[N]}\sum_{\gamma\in \overline{G}(F)}f(\iota(x)^{-1}n_1^{-1}\gamma n_2\iota(y))\theta(n_1)\overline{\theta}(n_2)dn_1dn_2,
	\end{align*}
	which converges absolutely since $\K(x,y)$ is continuous and $[N]$ is compact.
	
	Then we consider the double coset $N(F)\backslash \overline{G}(F)/N(F),$ whose element is of the form $wa,$ where $w$ is a Weyl element and $a\in A(F).$ Let 
	\begin{align*}
		H_{wa}:=\big\{(n_1, n_2)\in N\times N: \ n_1^{-1}wan_2a^{-1}w^{-1}\in Z\big\}
	\end{align*} 
	be the stabilizer relative to the representative $wa.$ Then 
	\begin{align*}
		J_{\Kuz}(\iota(x),\iota(y))=\sum_{wa\in\Phi}\int f(\iota(x)^{-1}n_1^{-1}wa n_2\iota(y))\theta(n_1)\overline{\theta}(n_2)dn_1dn_2,
	\end{align*}
	where $\Phi$ is a set of complete representatives for $N(F)\backslash \overline{G}(F)/N(F),$ and $(n_1,n_2)\in H_{wa}(F)\backslash N(\mathbb{A}_F)\times N(\mathbb{A}_F).$ So $J_{\Kuz}(\iota(x),\iota(y))$ becomes
	\begin{align*}
		\sum_{wa\in\Phi}C_{wa}\int_{H_{wa}(\mathbb{A}_F)\backslash N(\mathbb{A}_F)\times N(\mathbb{A}_F)}f(\iota(x)^{-1}n_1^{-1}wa n_2\iota(y))\theta(n_1)\overline{\theta}(n_2)dn_1dn_2,
	\end{align*}
	where 
	$$
	C_{wa}=\int_{[H_{wa}]}\theta(n_1')\overline{\theta}(n_2')dn_1'dn_2'.
	$$
	
	We say $wa\in\Phi$ is \textit{relevant} if $C_{wa}\neq 0,$ i.e., $\theta(n_1')\overline{\theta}(n_2')$ is trivial on $H_{wa}(\mathbb{A}_F).$ Denote by $\Phi^{*}$ the set of relevant elements in $\Phi.$ By the classification of relevant orbits in \cite{Gol87} (see also Prop. 1 in \cite{JR92}) one can take the following realization: $\Phi^*$ consists of $wa,$ where $w$ is the long Weyl element inside a stantard parabolic subgroup $P\subseteq G$ of type $(k_1, \cdots, k_r),$ and $a\in Z(F)\backslash\diag (T_{k_1}(F),\cdots, T_{k_r}(F))$ (modulo some further relations), with $T_{k_j}$ being the maximal split torus of $\mathrm{GL}(k_j).$ For instance, when $P=B$ the Borel, then $w=I_n$, $a=I_n$ and $H_{wa}=N.$ Therefore,
	\begin{align*}
		J_{\Kuz}(\iota(x),\iota(y))=\sum_{wa\in\Phi^*}\vol([H_{wa}])J_{\Kuz}(\iota(x),\iota(y);wa),
	\end{align*}
	where $J_{\Kuz}(\iota(x),\iota(y);wa)$ is defined by 
	\begin{align*}
		\int_{H_{wa}(\mathbb{A}_F)\backslash N(\mathbb{A}_F)\times N(\mathbb{A}_F)}f(\iota(x)^{-1}n_1^{-1}wa n_2\iota(y))\theta(n_1)\overline{\theta}(n_2)dn_1dn_2.
	\end{align*}
	
	By definition of $\Phi^*$ each $w$ corresponds to a unique (i.e., the minimal one) parabolic subgroup $P$ containing $w.$ Suppose $w\neq I_n.$ Then by Levi decomposition it suffices to consider the extreme case that $P=G$ and $w$ is the long Weyl element. 
	
	Recall that the test function $f$ is bi-$K$-finite. Hence there is some compact subgroup $K_0\subset G(\mathbb{A}_{F,\fin})$ such that $f$ is bi-$K_0$-invariant. Let $K_0=\prod_{v<\infty}K_{0,v}.$ Note that $J_{\Kuz}(\iota(x),\iota(y);wa)=\prod_{v\leq \infty}J_{\Kuz,v}(\iota(x_v),\iota(y_v);wa),$ where the local factor $J_{\Kuz,v}(\iota(x_v),\iota(y_v);wa)$ is defined by 
	$$\int_{H_{wa}(F_v)\backslash N(F_v)\times N(F_v)}f_v(\iota(x_v)^{-1}n_1^{-1}wa n_2\iota(y_v))\theta_v(n_1)\overline{\theta}_v(n_2)dn_1dn_2.
	$$
	Then for each finite place $v,$ $J_{\Kuz,v}(\iota(x_v),\iota(y_v);wa)$ is bi-$K_{0,v}$-invariant. So there exists a compact subgroup $N_{0,v}\subseteq K_{0,v}\cap N(F_v),$ depending only on $f_v,$ such that 
	$$
	J_{\Kuz,v}(\iota(x_v),\iota(y_v);wa)=J_{\Kuz,v}(\iota(x_vu_v),\iota(y_vu_v');wa),\ \text{for all $u_v, u_v'\in N_{0,v}$}.
	$$
	On the other hand, $$J_{\Kuz,v}(\iota(x_vu_v),\iota(y_vu_v');wa)=\theta(x_vu_vx_v^{-1})\overline{\theta}(y_vu_v'y_v^{-1})J_{\Kuz,v}(\iota(x_v),\iota(y_v);wa).$$
	There exists a constant $C_v$ depending only on $N_{0,v}$ and $\theta$ such that $\theta(x_vu_vx_v^{-1})=1$ and $\overline{\theta}(y_vu_v'y_v^{-1})=1$ if and only if $|\alpha_i(x_v)|_v\leq C_v$ and $|\alpha_i(y_v)|_v\leq C_v,$ where $\alpha_i$'s are the simple roots of $G(F)$ relative to $B.$ Note that for all but finitely many $v<\infty,$ $K_{0,v}=G(\mathcal{O}_{F,v}).$ Thus we can take the corresponding $C_v=1.$ Hence for any $x_v, y_v\in A(F_v),$ $J_{\Kuz,v}(\iota(x_v),\iota(y_v);wa)\neq 0$ implies that $|\alpha_i(x_v)|_v\leq C_v,$ and $|\alpha_i(y_v)|_v\leq C_v,$ $1\leq i\leq n-1,$ and $C_v=1$ for all but finitely many finite places $v.$ Denote the compact set by
	$$
	A_{f,\fin}=\Big\{a=(a_v)\in A(\mathbb{A}_{F,\fin}):\ |\alpha_i(a_v)|_v\leq C_v,\ 1\leq i\leq n-1\Big\}.
	$$
	Then $\supp J_{\Kuz,v}(\iota(x_v),\iota(y_v);wa)\subseteq (A(\mathbb{A}_{F,\infty})A_{f,\fin})\times (A(\mathbb{A}_{F,\infty})A_{f,\fin}).$
	
	For any $g=\otimes_v(g_{i,j,v})\in G(\mathbb{A}_F).$ We define $\|g_v\|_v=\max_{i,j}|g_{i,j,v}|_v$ if $v<\infty;$ and 
	$$
	\|g_v\|_v=\Big[\sum_{i,j}|g_{i,j,v}|_v^2\Big]^{1/2},\ \text{if}\ v\mid \infty.
	$$
	Then $\|g_v\|_v=1$ for almost all $v.$ The height function $\|g\|=\prod_v\|g_v\|_v$ is therefore well defined as a finite product. Also, by the compactness of $\supp f_v,$ and $\supp J_{\Kuz}(f,x;wa)\subseteq A(\mathbb{A}_{F,\infty})A_{f,\fin},$ we have $\|w^{-1}x_vwx_va\|_v\leq C_v'$ for some constant $C_v'$ depending only on $f_v,$ $v<\infty,$ and $C_v'=1$ for almost all $v$'s.
	
	Now we investigate the archimedean integral $J_{\Kuz,v}(\iota(x_v),\iota(y_v);wa),$ i.e., $v\mid\infty.$ Note that $f_v$ is a compactly supported on $\overline{G}(F_v).$ Then $J_{\Kuz,v}(\iota(x_v),\iota(y_v);wa)=0$ unless $n_{1,v}^{-1}g_vwn_{2,v} \in \supp f_v,$ where $g_v=\iota(x_v)^{-1}wa\iota(y_v)w^{-1}.$ Hence we have $\|n_{1,v}^{-1}g_v wn_{2,v}w^{-1}\|_{v}\leq C_v$ for some constant $C_v$ depending only on $f.$ A straightforward computation (or by Lemma 5.1 of \cite{Jac09}) shows that $\|n_{1,v}\|_v+\|n_{2,v}\|_v+\|g_v\|_v\leq C_v'$ for some constant $C_v'$ depending only on $f.$ So $f_v(n_{1,v}g_vwn_{2,v})$ has compact support relative to $n_{1,v}$ and $n_{2,v}.$ Therefore, $J_{\Kuz,v}(\iota(x_v),\iota(y_v);wa)=0$ unless $n_{1,v},$ $n_{2,v}$ run through a compact set of $N(F_v)$ and $\|g\|_v$ is bounded.
	
	For $\alpha=\diag(\alpha_1\cdots\alpha_n,\cdots,\alpha_1,1)\in A(\mathbb{A}_F),$ denote by
	$$
	\theta_{\alpha}(u)=\prod_{i=1}^{n-1}\psi_{F/\mathbb{Q}}\left(\alpha_iu_{i,i+1}\right),\quad \forall\ u=(u_{i,j})_{n\times n}\in N(\mathbb{A}_F).
	$$
	Then $J_{\Kuz,v}(\iota(x_v),\iota(y_v);wa)$ is equal to 
	\begin{align*}
		\delta_{w}(\iota(x))\delta_{w}(\iota(y))\int_{H_{wa}(\mathbb{A}_F)\backslash N(\mathbb{A}_F)\times N(\mathbb{A}_F)}f(n_1^{-1}\iota(x)^{-1}wa\iota(y)n_2)\theta_{x}(n_1)\overline{\theta}_y(n_2)dn_1dn_2,
	\end{align*}
	where $\delta_w$ is the modular character of the parabolic subgroup associated to $w.$ 
	
	Since $n_1$ and $n_2$ lie in a compact set determined by $\supp f,$ then for a fixed $a\in A(F),$ $v\mid \infty,$ the $v$-th component of 
	$$
	\int_{H_{wa}(\mathbb{A}_F)\backslash N(\mathbb{A}_F)\times N(\mathbb{A}_F)}f(n_1^{-1}\iota(x)^{-1}wa\iota(y)n_2)\theta_{x}(n_1)\overline{\theta}_y(n_2)dn_1dn_2
	$$
	is a Schwartz function of $(x,y)$ as the Fourier transform of a compactly supported smooth function on $G(F_{\infty})$. Hence for each $a$, the function $J_{\Kuz,v}(\iota(x_v),\iota(y_v);wa)$ is a Schwartz-Bruhat function on $A(\mathbb{A}_F)\times A(\mathbb{A}_F).$ Moreover, $J_{\Kuz,v}(\iota(x_v),\iota(y_v);wa)=0$ unless $x^{-1}wayw^{-1}\in A^*:=\{t\in A(\mathbb{A}_F): \ \|t\|\leq \prod_{v}C_v'\}.$ 
	
	By properties of the height $\|\cdot\|$ (e.g., cf. \cite{Art05} p. 70) one has 
	\begin{align*}
		\#\left(w^{-1}x\cdot A^*\cdot wy^{-1}\cap A(F)\right)\leq C\cdot (|\det x|^{M}+|\det x|^{-M}) (|\det y|^{M}+|\det y|^{-M}),
	\end{align*}
	for some constants $C$ and $M$ depending on $\supp f.$ Therefore, 
	\begin{align*}
		\sum_{a\in A(F)}|J_{\Kuz}(f,x;wa)|=\sum_{\substack{a\in A(F)\\ a\in w^{-1}x\cdot A^*\cdot wy^{-1}}}|J_{\Kuz}(f,x;wa)|
	\end{align*}
	is a finite sum of Schwartz-Bruhat functions. So $\sum_{a\in A(F)}|J_{\Kuz}(f,x;wa)|$ is majorized by $\xi(x,y)\cdot (|x_1\cdots x_n|^{-M}+|y_1\cdots y_n|^{-M})$ for some $M\geq 0$ and Schwartz-Bruhat function $\xi$ on $A(\mathbb{A}_F)\times A(\mathbb{A}_F).$  
	
	The remaining case is that $w=I_n,$ i.e., $P=B.$ In this case 
	\begin{align*}
		\sum_{a\in A(F)}J_{\Kuz}(f,x;wa)=\delta_{w}(y)\int_{N(\mathbb{A}_F)}f(x^{-1}yn)\overline{\theta}_y(n)dn
	\end{align*}
	is also a Schwartz-Bruhat function of $(x,y),$ observing that $x^{-1}y$ supports in a compact set of $\mathbb{A}_F^{\times}.$ So it is majorized by a gauge. Then Proposition \ref{Kl} follows.
\end{proof}

\begin{cor}\label{cor39}
	Let $\mathbf{s}=(s_1,s_2)\in\mathbb{C}^2$ be such that $\Re(s_1)\gg 0$ and $\Re(s_2)\gg 0.$ Then $J_{\Kl}^T(f,\mathbf{s})$ converges absolutely and uniformly, as $T\rightarrow\infty,$ to the function $J_{\Kl}(f,\mathbf{s})$ defined by 
	\begin{align*}
		\int_{N'(\mathbb{A}_F)\backslash G'(\mathbb{A}_F)}\int_{N'(\mathbb{A}_F)\backslash G'(\mathbb{A}_F)}J_{\Kuz}(x,y)\phi_1'(x)\overline{\phi_2'(y)}|\det x|^{s_1}|\det y|^{s_2}dxdy.
	\end{align*}
\end{cor}
\begin{proof}
	By Proposition \ref{Kl} the function $J_{\Kuz}(x,y)$ is majorized by a gauge. So the integrand $J_{\Kuz}(x,y)\phi_1'(x)\overline{\phi_2'(y)}|\det x|^{s_1}|\det y|^{s_2}\delta_{B'}^{-1}(x)\delta_{B'}^{-1}(y)$ decay rapidly on $T_{B'}(\mathbb{A}_F)\times T_{B'}(\mathbb{A}_F)$ when $\Re(s_1)$ and $\Re(s_2)$ are large. Therefore, $J_{\Kl}^T(f,\mathbf{s})$ converges absolutely and uniformly when $T\rightarrow \infty$ by the dominant control theorem, and the limit function is $J_{\Kl}(f,\mathbf{s}).$ 
\end{proof}
\begin{remark}
	We do not essentially make use of cuspidality of $\phi_1'$ and $\phi_2'$ here, except that they are slowly increasing. So Corollary \ref{cor39} still holds when $\phi_1'$ and $\phi_2'$ are Eisenstein series.  
\end{remark}

By Corollary \ref{cor39}, when $\Re(s_1)$ and $\Re(s_2)$ are large, the function $J_{\Eis}^{\Reg,T}(f,\mathbf{s})=J_{\Kl}^T(f,\mathbf{s})-J_0^T(f,\textbf{s})$ converges when $T\rightarrow\infty.$ So we can define
\begin{equation}\label{Eis}
	J_{\Eis}^{\Reg}(f,\mathbf{s}):=\lim_{T\rightarrow\infty}J_{\Eis}^{\Reg,T}(f,\mathbf{s})=J_{\Kl}(f,\mathbf{s})-J_0(f,\mathbf{s}),
\end{equation}
which converges absolutely when $\Re(s_1)\gg 0$ and $\Re(s_2)\gg 0,$ depending on the test function $f.$  

However, $J_{\Eis}^{\Reg}(f,\mathbf{s})$ may not converge everywhere, e.g., it may diverge on the plane $s_1+s_2=1$ when $\pi_1'\simeq\pi_2'.$ Our strategy to overcome this is to find meromorphic continuation of $J_{\Eis}^{\Reg}(f,\mathbf{s})$ to $\mathbb{C}^2,$ which is similar to \cite{Yan19}. To achieve it, we will apply spectral decomposition of the kernel function $\K_{\ER}$ and (after a switch of integrals) write $J_{\Eis}^{\Reg}(f,\mathbf{s})$ as an infinitely sum of Rankin-Selberg periods. The justification of switching integrals is given in \textsection \ref{7.1} (cf. Proposition \ref{exp}).

\subsection{Spectral Decomposition of the Kernel Function}\label{7.1}
In this subsection, we review briefly the spectral theory of automorphic representation of reductive groups, and then apply the results to the non-cuspidal kernel function $\K_{\ER}=\K_{\Res}+\K_{\Eis}$. 
	
Let $X(G)_F$ be the space set of $F$-rational characters of $G.$ Denote by  $\mathfrak{a}_G=\Hom_{\mathbb{Z}}(X(G)_{F},\mathbb{R}).$ Let $\mathfrak{a}_G^*=X(G)_{F}\otimes \mathbb{R}.$ Let $Q$ be a standard parabolic subgroup of $G$ of type $(n_1,\cdots,n_l).$ Let $M_Q$ (resp. $N_Q$) be the Levi subgroup (resp. unipotent radical) of $Q.$ Let $X(M_Q)_F$ be the space set of $F$-rational characters of $M_Q.$ Denote by $\mathfrak{a}_Q=\mathfrak{a}_{M_Q}=\Hom_{\mathbb{Z}}(X(M_Q)_{F},\mathbb{R})$ and $\mathfrak{a}_Q^*=\mathfrak{a}_{M_Q}^*=X(M_Q)_{F}\otimes \mathbb{R}.$ Write $\mathfrak{a}_0$ (resp. $\mathfrak{a}_0^*$) for $\mathfrak{a}_{B}$ (resp. $\mathfrak{a}_{B}^*$). The abelian group $X(M_Q)_F$ has a canonical basis of rational characters $\chi_i: m\rightarrow \det m_i,$ for $,m=\begin{pmatrix}
m_1\\
&\ddots\\
&&m_l	
\end{pmatrix}\in M_Q.$ We then choose $\big\{n_i\chi_i^{-1}:\ 1\leq i\leq l\big\}$ to be a basis for $\mathfrak{a}_Q^*,$ and choose a dual basis on $\mathfrak{a}_Q$ to identified both $\mathfrak{a}_Q$ and $\mathfrak{a}_{Q}^*$ with $\mathbb{R}^l.$ Furthermore, define 
$$
H_{M_Q}(m)=\left(\frac{\log |\det m_1|}{n_1},\cdots, \frac{\log |\det m_l|}{n_l}\right)\in\mathfrak{a}_Q,
$$
where $m=\diag(m_1,\cdots, m_l)\in M_Q(\mathbb{A}_F).$ Denote by $\langle\cdot,\cdot\rangle$ the pair for $\mathfrak{a}_Q$ and $\mathfrak{a}_Q^*$ under the above chosen identification. We also write $H_Q$ for $H_{M_Q}$ to simplify notation.
	
	Then by spectral theory (e.g., cf. p. 256 and p. 263 of \cite{Art79}), the decomposition of the Hilbert space $L^2\left(Z_{G}(\mathbb{A}_F)N_Q(\mathbb{A}_F)M_Q(F)\backslash G(\mathbb{A}_F)\right)$ into right $G(\mathbb{A}_F)$-invariant subspaces is determined by the spectral data $\chi=\{(M,\sigma)\},$ where the pair $(M,\sigma)$ consists of a Levi subgroup $M$ of $G$ and a cuspidal representation $\sigma\in\mathcal{A}_0\left(Z(\mathbb{A}_F)\backslash M^1(\mathbb{A}_F)\right),$ where $M^1(\mathbb{A}_F):=\big\{m\in M(\mathbb{A}_F):\ H_M(m)=0\big\},$ and the class $(M,\sigma)$ derives from the equivalence relation $(M,\sigma)\sim(M',\sigma')$ if and only if $M$ is conjugate to $M'$ by a Weyl group element $w,$ and $\sigma'=\sigma^w$ on $Z(\mathbb{A}_F)\backslash M^1(\mathbb{A}_F).$ Let $\mathfrak{X}$ be the set of equivalence classes $\chi=\{(M,\sigma)\}$ of these pairs, we thus have
	\begin{equation}\label{49}
		L^2\left(Q\right):=L^2\left(Z_{G}(\mathbb{A}_F)N_Q(\mathbb{A}_F)M_Q(F)\backslash G(\mathbb{A}_F)\right)=\bigoplus_{\chi\in\mathfrak{X}}L^2\left(Q\right)_{\chi},
	\end{equation}
	where $L^2\left(Q\right)_{\chi}$ consists of functions $\phi\in L^2\left(Z_{G}(\mathbb{A}_F)N_Q(\mathbb{A}_F)M_Q(F)\backslash G(\mathbb{A}_F)\right)$ such that for each standard parabolic subgroup $Q_1$ of $G,$ with $Q_1\subset Q,$ and almost all $x\in G(\mathbb{A}_F),$ the projection of the function 
	$$
	m\mapsto x.\phi_{Q_1}(m)=\int_{N_{Q_1}(F)\backslash N_{Q_1}(\mathbb{A}_F)}\phi(nmx)dn
	$$
	onto the space of cusp forms in $L^2\left(Z_{G}(\mathbb{A}_F)M_{Q_1}(F)\backslash M_{Q_1}^1(\mathbb{A}_F)\right)$ transforms under $M_{Q_1}^1(\mathbb{A}_F)$ as a sum of representations $\sigma,$ in which $(M_{Q_1},\sigma)\in\chi.$ If there is no such pair in $\chi,$ $x.\phi_{{Q_1}}$ will be orthogonal to $\mathcal{A}_0\left(Z_{G}(\mathbb{A}_F)M_{Q_1}(F)\backslash M_{Q_1}^1(\mathbb{A}_F)\right).$ Denote by $\mathcal{H}_Q$ the space of such $\phi$'s. Let $\mathcal{H}_{Q,\chi}$ be the subspace of $\mathcal{H}_Q$ such that for any $(M,\sigma)\notin \chi,$ with $M=M_{Q_1}$ and $Q_1\subset Q,$ we have
	$$
	\int_{M(F)\backslash M(\mathbb{A}_F)^1}\int_{N_{Q_1}(F)\backslash N_{Q_1}(\mathbb{A}_F)}\psi_0(m)\phi(nmx)dn=0,
	$$
	for any $\psi_0\in L^2_{0}\left(M(F)\backslash M(\mathbb{A}_F)^1\right)_{\sigma},$ and almost all $x.$ This leads us to Langlands' result to decompose $\mathcal{H}_Q$ as an orthogonal direst sum $\mathcal{H}_Q=\bigoplus_{\chi\in\mathfrak{X}}\mathcal{H}_{Q,\chi}.$ Let $\mathcal{B}_Q$ be an orthonormal basis of $\mathcal{H}_Q,$ then we can choose $\mathcal{B}_Q=\bigcup_{\chi\in\mathfrak{X}}\mathcal{B}_{Q,\chi},$ where $\mathcal{B}_{Q,\chi}$ is an orthonormal basis of the Hilbert space $\mathcal{H}_{Q,\chi}.$ We may assume that vectors in each $\mathcal{B}_{Q,\chi}$ are $K$-finite and are pure tensors.
	
 Given a spectral datum $\chi=\{(M_Q,\sigma)\},$ and $\lambda\in\mathfrak{a}_Q^*\otimes\mathbb{C},$ we have the induced representation $\mathcal{I}_{Q}(\cdot, \lambda):=\Ind_{Q(\mathbb{A}_F)}^{G(\mathbb{A}_F)}\sigma\cdot e^{\langle \lambda, H_{M}(\cdot)\rangle}$ with the representation space being $\mathcal{H}_Q.$ Explicitly, for $\phi\in \mathcal{H}_Q,$ one has 
 $$
 (\mathcal{I}_Q(y,\lambda)\phi)(x)=\phi(xy)e^{\langle \lambda+\rho_Q,H_{M_Q}(xy)\rangle}e^{-\langle \lambda+\rho_Q,H_{M_Q}(x)\rangle},\ \ x, y\in G(\mathbb{A}_F),
 $$
 where $\rho_Q$ is the half sum of the positive roots of $(Q, A_Q).$ The associated Eisenstein series is defined by 
\begin{align*}
E(x,\phi,\lambda)=\sum_{\delta\in Q(F)\backslash G(F)}\phi(\delta x)e^{\langle \lambda+\rho_Q,H_{M_Q}(\delta x)\rangle}.
\end{align*}
 
Let $\chi=\{(M,\sigma)\}\in\mathfrak{X}.$ Define 
\begin{align*}
\K_{\chi}(x,y)=\sum_{Q}\frac{1}{n_Q}\left(\frac{1}{2\pi i}\right)^{\dim A_Q/A_G}\int_{i\mathfrak{a}_Q^*/i\mathfrak{a}_G^*}\sum_{\phi\in\mathfrak{B}_{Q,\chi}}E(x,\mathcal{I}(\lambda, f)\phi,\lambda)\overline{E(y,\phi,\lambda)}d\lambda,
\end{align*}
where $Q$ runs over standard parabolic subgroups that contain $M,$ and $n_Q$ is the number of chambers in $\mathfrak{a}_Q,$ and 
\begin{equation}\label{172}
E(x,\mathcal{I}(\lambda, f)\phi,\lambda)=\int_{G(\mathbb{A}_F)}\tilde{f}(y)E(xy,\phi,\lambda)dy.
\end{equation}
Then by \cite{Art79} the kernel function $\K_{\ER}$ admits the decomposition $\K_{\ER}(x,y)=\sum_{\chi}\K_{\chi}(x,y),$ where $\chi$ runs over proper spectral data.

Plugging the spectral decomposition of $\K_{\ER}$ into the above expression we show
\begin{prop}\label{exp}
Let notation be as before. Let $n\geq 1.$ Let $\Re(s_1), \Re(s_2)\gg 0.$ Then we have the absolutely convergent spectral expansion 
\begin{equation}\label{J}
J_{\Eis}^{\Reg}(f,\textbf{s})=\sum_{\chi}\sum_{Q}\frac{1}{n_Q}\left(\frac{1}{2\pi i}\right)^{\dim A_Q/A_G}\int_{i\mathfrak{a}_Q^*/i\mathfrak{a}_G^*}\sum_{\phi\in\mathfrak{B}_{Q,\chi}}\mathcal{H}_{\phi}(\lambda,s_1,s_2)d\lambda,
\end{equation}
where $\mathcal{H}_{\phi}(\lambda,s_1,s_2):=\Psi_{\lambda}(s_1,\mathcal{I}(\lambda,f)\phi,\phi_1')\overline{\Psi_{\lambda}(\overline{s_2},\phi,\phi_2')}$ with 
\begin{align*}
\Psi_{\lambda}(s_1,\mathcal{I}(\lambda,f)\phi,\phi_1')&=\int_{N'(\mathbb{A}_F)\backslash G'(\mathbb{A}_F)}\int_{[N]}E(ux,\mathcal{I}(\lambda, f)\phi,\lambda)\overline{\theta}(u)duW_1'(x)|\det x|^{s_1}dx\\
\Psi_{\lambda}(s_2,\phi,\phi_2')&=\int_{N'(\mathbb{A}_F)\backslash G'(\mathbb{A}_F)}\int_{[N]}E(ux,\phi,\lambda)\overline{\theta}(u)duW_2'(x)|\det x|^{s_2}dx.
\end{align*}
\end{prop}
\begin{proof}
We write $J_{\Eis}^{\Reg}(f,s_1, s_2;\phi_1',\phi_2')$ for $J_{\Eis}^{\Reg}(f,\mathbf{s})$ in this proof. First we assume $f$ is of the form $h*h^*,$ where $h^*(g):=h(\transp{g}^{-1}),$ $g\in G(\mathbb{A}_F).$ Then by \eqref{Eis} the function $J_{\Eis}^{\Reg}(f,s_1,s_2;\phi_1',\phi_2')$ is equal to 
\begin{align*}
\int_{{N'(\mathbb{A}_F)\backslash G'(\mathbb{A}_F)}}\int_{{N'(\mathbb{A}_F)\backslash{G'}(\mathbb{A}_F)}}{W}_{\infty}(\iota(x),\iota(y))W_1'(x)\overline{W_2'(y)}\det x|^{s_1}|\det y|^{s_2}dxdy,
\end{align*}
where $W_{1}'$ (resp. $W_2'$) is the Whittaker function associated to $\phi_1'$ (resp. $\phi_2'$), and 
\begin{align*}
	{W}_{\infty}(g_1,g_2):=\int_{[N]}\int_{[N]}\K_{\ER}(ug_1,vg_2)\theta(u)\overline{\theta(v)}dudv,\ \ g_1, g_2\in G(\mathbb{A}_F).
\end{align*}
Appealing to the spectral decomposition $\K_{\ER}(x,y)=\sum_{\chi}\K_{\chi}(x,y)$ then the function $J_{\Eis}^{\Reg}(f,s_1,s_2;\phi_1',\phi_2')$ is equal to 
\begin{equation}\label{176}
\iint\sum_{\chi}\sum_{Q}\frac{1}{n_Q}\left(\frac{1}{2\pi i}\right)^{\dim A_Q/A_G}\int\sum_{\phi}\mathcal{W}_{1,\chi}(x,\lambda;\phi)\overline{\mathcal{W}_{2,\chi}(y,\lambda;\phi)}d\lambda dxdy,
\end{equation}
where $x, y\in N'(\mathbb{A}_F)\backslash G'(\mathbb{A}_F),$ $\lambda\in i\mathfrak{a}_Q^*/i\mathfrak{a}_G^*,$  $\phi$ runs over $\mathfrak{B}_{Q,\chi},$ and for $1\leq j\leq 2,$
\begin{align*}
\mathcal{W}_{j,\chi}(x,\lambda;\phi):=W_j'(x)|\det x|^{s_j}\int_{[N]}E(u\iota(x),\mathcal{I}(\lambda, h)\phi,\lambda)\overline{\theta}(u)du.
\end{align*}

$J_{\Eis}^{\Reg}(f,s_1,s_2;\phi_1',\phi_2')$ converges absolutely for fixed $\phi_1', \phi_2'$ in $\Re(s_1)\gg 0$ and $\Re(s_2)\gg 0.$ In particular, $J_{\Eis}^{\Reg}(f,s_j,\overline{s_j};\phi_j',\phi_j')<\infty$ when $\Re(s_j)\gg 0,$ $j=1, 2.$ 

Define (at least formally) $\tilde{J}_{\Eis}^{\Reg}(f,s_j,\overline{s_j};\phi_j',\phi_j')$ to be 
\begin{equation}\label{177}
\iint\sum_{\chi}\sum_{Q}\frac{1}{n_Q}\left(\frac{1}{2\pi i}\right)^{\dim A_Q/A_G}\int\sum_{\phi}\big|\mathcal{W}_{j,\chi}(x,\lambda;\phi)\overline{\mathcal{W}_{j,\chi}(y,\lambda;\phi)}\big|d\lambda dxdy,
\end{equation}
which is precisely $J_{\Eis}^{\Reg}(f,s_1,s_2;\phi_1',\phi_2')$ with $\mathcal{W}_{1,\chi}(x,\lambda;\phi)\overline{\mathcal{W}_{2,\chi}(y,\lambda;\phi)}$ replaced by $|\mathcal{W}_{j,\chi}(x,\lambda;\phi)\overline{\mathcal{W}_{j,\chi}}(y,\lambda;\phi)|,$ $1\leq j\leq 2.$

By Minkowski's integral inequality, $\tilde{J}_{\Eis}^{\Reg}(f,s_j,\overline{s_j};\phi_j',\phi_j')$ is bounded by 
\begin{align*}
&\Bigg[\int_{N'(\mathbb{A}_F)\backslash G'(\mathbb{A}_F)}\Bigg[\sum_{\chi}\sum_{Q}\frac{1}{n_Q}\left(\frac{1}{2\pi i}\right)^{\dim A_Q/A_G}\int_{i\mathfrak{a}_Q^*/i\mathfrak{a}_G^*}\sum_{\phi\in\mathfrak{B}_{Q,\chi}}\\
&\quad \times\Big|\int_{[N]}E(u\iota(x),\mathcal{I}(\lambda, h)\phi,\lambda)\overline{\theta}(u)du\Big|^2d\lambda\Bigg]^{\frac{1}{2}}|W'(x)||\det x|^{\Re(s_j)}dx\Bigg]^2\\
=&\Bigg[\int_{N'(\mathbb{A}_F)\backslash G'(\mathbb{A}_F)}\Big[W_{\infty}(\iota(x),\iota(x))\Big]^{\frac{1}{2}}|W'(x)||\det x|^{\Re(s_j)}dx\Bigg]^2,
\end{align*}
which, by Proposition \ref{Kl}, converges when $\Re(s_j)\gg 0.$ So \eqref{177} converges. In conjunction with Cauchy inequality, \eqref{176} converges absolutely. We can switch the integrals in \eqref{176} to write $J_{\Eis}^{\Reg}(f,s_1,\overline{s_2};\phi_1',\phi_2')$ as
\begin{equation}\label{175}
\sum_{\chi}\sum_{Q}\frac{1}{n_Q}\left(\frac{1}{2\pi i}\right)^{\dim A_Q/A_G}\int\sum_{\phi}\Psi_{\lambda}(s_1,\mathcal{I}(\lambda,f)\phi,\phi_1')\overline{\Psi_{\lambda}(s_2,\phi,\phi_2')}d\lambda,
\end{equation}
where $\lambda$ ranges over $i\mathfrak{a}_Q^*/i\mathfrak{a}_G^*,$  $\phi$ runs over $\mathfrak{B}_{Q,\chi}.$ 

For general test function $f$ we can apply Dixmier-Malliavin's decomposition to write it as a finite sum of convolution products: $f=\sum_{j}h_{1,j}*h_{2,j}.$ Then \eqref{J} follows from \eqref{175} and Cauchy-Schwartz inequality.
\end{proof}

We then investigate the analytic behavior of each period $\mathcal{H}_{\phi}(\lambda,s_1,s_2)$ as a meromorphic function in \textsection \ref{9.3} below, where we show 
\begin{align*}
\sum_{\chi}\sum_{Q}\frac{1}{n_Q}\left(\frac{1}{2\pi i}\right)^{\dim A_Q/A_G}\int_{i\mathfrak{a}_Q^*/i\mathfrak{a}_G^*}\sum_{\phi\in\mathfrak{B}_{Q,\chi}}\mathcal{H}_{\phi}(\lambda,s_1,s_2)d\lambda
\end{align*}
converges absolutely for \textit{all} $\mathbf{s}\in\mathbb{C}^2,$ where $\chi$ ranges through $\mathfrak{X}$ minus \textit{two} cuspidal data. This would give us meromorphic continuation of $J_{\Eis}^{\Reg}(f,\mathbf{s})$ as shown below.

\subsection{Meromorphic Continuation of $J_{\Eis}^{\Reg}(f,\mathbf{s})$ to $\mathbb{C}^2$}\label{9.3}
Let $n\geq 1.$ Let $\mathbf{s}=(s_1,s_2)$ be such that $\Re(s_1)\gg 0$ and $\Re(s_2)\gg 0.$ Let $Q$ be a parabolic subgroup of $G$ of type $(n_1,\cdots,n_l),$ $l\geq 2.$ Let $\chi=\{(M_Q,\sigma)\}$ be a cuspidal datum, where $\sigma=\sigma_1\otimes\cdots\otimes\sigma_l$ is a cuspidal representation of the Levi subgroup $M_Q.$ Denote by 
\begin{align*}
\mathcal{J}_{\chi}(s_1,s_2,\phi_1',\phi_2')=\frac{1}{n_Q}\sum_{\phi\in\mathfrak{B}_{Q,\chi}}\left(\frac{1}{2\pi i}\right)^{\dim A_Q/A_G}\int_{i\mathfrak{a}_{Q}^*/i\mathfrak{a}_G^*}\mathcal{H}_{\phi}(\lambda,s_1,s_2)d\lambda.
\end{align*}

Write $\lambda=(\lambda_1,\cdots,\lambda_l)$ with $\lambda_1+\cdots+\lambda_l=0.$ According to the discussion in \cite{JPSS83} and \cite{Jac09} we have 
\begin{align*}
	\mathcal{H}_{\phi}(\lambda,s_1,s_2)\sim&\frac{\prod_{j=1}^l\Lambda(s_1+1/2+\lambda_j/n_j,\pi_1'\times\sigma_j)\Lambda(s_2+1/2-\lambda_j/n_j,\widetilde{\pi}_2'\times\widetilde{\sigma}_j)}{\prod_{\substack{i\leq l\\  j\leq l\\ i\neq j}}\Lambda(1+\lambda_i/n_i-\lambda_j/n_j,\sigma_i\times\widetilde{\sigma}_j)\Lambda(1-\lambda_i/n_i+\lambda_j/n_j,\widetilde{\sigma}_i\times\sigma_j)}.
\end{align*}
Note that the RHS, as a rational function of complete $L$-functions, is meromorphic. Hence, we obtain a meromorphic continuation of $\mathcal{H}_{\phi}(\lambda,s_1,s_2)$ as a function of $\lambda\in i\mathfrak{a}_P^*/i\mathfrak{a}_G^*,$ $(s_1,s_2)\in\mathbb{C}^2.$ Moreover, $\mathcal{H}_{\phi}(\lambda,s_1,s_2)$ is entire if neither $\sigma_j\simeq\widetilde{\pi}_1'$ nor 
$\sigma_j\simeq\pi_2',$ $1\leq j\leq l.$ In particular, $\mathcal{H}_{\phi}(\lambda,s_1,s_2)$ is entire if $l\geq 3.$ 

\begin{defn}
We say $\mathcal{J}_{\chi}(s_1,s_2,\phi_1',\phi_2')$ is \textit{regular} if $l\geq 3,$ or $l=2$ and $\sigma_j\not\simeq\widetilde{\pi}_1',$ $\sigma_j\not\simeq\pi_2',$ $1\leq j\leq 2;$ \textit{semi-singular} if $l=2,$ $\pi_1'\not\simeq\pi_2',$ and $\sigma_j\simeq\widetilde{\pi}_1'$ or $\sigma_j\simeq\pi_2',$ $1\leq j\leq 2;$ \textit{singular} if  $l=2,$ $\pi_1'\simeq\pi_2',$ and either $\sigma_1\simeq\widetilde{\pi}_1'$ or $\sigma_2\simeq\widetilde{\pi}_1'.$ 
\end{defn}

Let $\mathfrak{X}^{\reg}$ consist of $\chi\in\mathfrak{X}$ such that $\mathcal{J}_{\chi}(s_1,s_2,\phi_1',\phi_2')$ is regular. Let $\mathfrak{X}^{\semi}$ (resp. $\mathfrak{X}^{\sing}$) consist of $\chi\in\mathfrak{X}$ such that $\mathcal{J}_{\chi}(s_1,s_2,\phi_1',\phi_2')$ is semi-singular (resp. singular). Then $\#\mathfrak{X}^{\semi}=4$ and $\#\mathfrak{X}^{\sing}=2.$

Let $*$ be a script in $\{\reg, \semi,\sing\}.$ We define 
\begin{equation}\label{182}
J_{\Eis,*}^{\Reg}(f,\mathbf{s}):=\sum_{\chi\in\mathfrak{X}^*}\sum_{Q}\frac{1}{n_Q}\left(\frac{1}{2\pi i}\right)^{\dim A_Q/A_G}\int\sum_{\phi\in\mathfrak{B}_{Q,\chi}}\mathcal{H}_{\phi}(\lambda,s_1,s_2)d\lambda,
\end{equation}
where $\lambda$ ranges through $i\mathfrak{a}_Q^*/i\mathfrak{a}_G^*\simeq i\mathbb{R}.$

\begin{thm}\label{thm40}
Let notation be as before. Then 
$$
J_{\Eis}^{\Reg}(f,\mathbf{s})=J_{\Eis,\reg}^{\Reg}(f,\mathbf{s})+J_{\Eis,\semi}^{\Reg}(f,\mathbf{s})+J_{\Eis,\sing}^{\Reg}(f,\mathbf{s}),
$$ 
where $J_{\Eis,\reg}^{\Reg}(f,\mathbf{s})$ converges absolutely for all $\mathbf{s}\in\mathbb{C}^2;$ $J_{\Eis,\semi}^{\Reg}(f,\mathbf{s})$ admits a holomorphic continuation to $\mathbb{C}^2;$ and $J_{\Eis,\sing}^{\Reg}(f,\mathbf{s})$ admits a meromorphic continuation to $\mathbb{C}^2,$ with possible (simple) poles on $s_1+s_2=1$ and $s_1+s_2=0.$ In particular, $J_{\Eis}^{\Reg}(f,\mathbf{s})$ admits a meromorphic continuation to $\mathbb{C}^2,$ with possible (simple) poles on $s_1+s_2=1$ and $s_1+s_2=0.$
\end{thm}
\begin{proof}
Theorem \ref{thm40} follows from Lemmas \ref{lem41}, \ref{lem43} and \ref{lem44} below.
\end{proof}

\subsubsection{Regular Periods}
Let $\chi\in\mathfrak{X}^{\reg}.$ Then $\mathcal{H}_{\phi}(\lambda,s_1,s_2)$ continues to an entire function of $\lambda,$ $s_1$ and $s_2.$ Recall (cf. \eqref{182})
\begin{align*}
J_{\Eis,\reg}^{\Reg}(f,\mathbf{s}):=\sum_{\chi\in\mathfrak{X}^{\reg}}\sum_{Q}\frac{1}{n_Q}\left(\frac{1}{2\pi i}\right)^{\dim A_Q/A_G}\int_{i\mathfrak{a}_Q^*/i\mathfrak{a}_G^*}\sum_{\phi\in\mathfrak{B}_{Q,\chi}}\mathcal{H}_{\phi}(\lambda,s_1,s_2)d\lambda.
\end{align*}

\begin{lemma}\label{lem41}
Let notation be as before. Then $J_{\Eis,\reg}^{\Reg}(f,\mathbf{s})$ converges absolutely for all $\mathbf{s}\in\mathbb{C}^2.$ In particular, $J_{\Eis,\reg}^{\Reg}(f,\mathbf{s})$ is entire.
\end{lemma}
\begin{proof}
Let $\chi\in\mathfrak{X}^{\reg}.$ By definition $\mathcal{H}_{\phi}(\lambda,s_1,s_2)$ is entire, and moreover is bounded in the vertical strip (cf. \cite{Jac09}). By Proposition \ref{exp} there exists $c=c_{f,\phi_1',\phi_2'}>0$ be such that $J_{\Eis,\reg}^{\Reg}(f,\mathbf{s})$ converges absolutely when $\Re(s_1)\geq c$ and $\Re(s_2)\geq c.$

Suppose $f=h*h^*$ and $s_1=\overline{s_2}=s\in\mathbb{C}$ be arbitrary. Let $J_{\Eis,\reg,j}^{\Reg}(f,\mathbf{s})$ be the counterpart of $J_{\Eis,\reg}^{\Reg}(f,\mathbf{s})$ with $\phi_1'$ and $\phi_2'$ replaced with $\phi_j'$ and $\phi_j',$ $1\leq j\leq 2.$  Note that when $\chi=\{(M,\sigma)\}\in\mathfrak{X}^{\reg},$ the dual $\widetilde{\chi}=\{(M,\widetilde{\sigma})\}\in\mathfrak{X}^{\reg}.$ Then by functional equation we have $J_{\Eis,\reg,j}^{\Reg}(f,\mathbf{s})=J_{\Eis,\reg,j}^{\Reg}(f,-\mathbf{s}),$ $1\leq j\leq 2.$  

Let $c_0=|s|+c.$ Then by maximal principle there exists $s_{\phi}$ with $\Re(s_{\phi})=c_0$ or $\Re(s_{\phi})=-c_0$ such that 
\begin{equation}\label{178}
|\mathcal{H}_{\phi}(\lambda,s_1,s_2)|\leq |\mathcal{H}_{\phi}(\lambda,s_{\phi},\overline{s_{\phi}})|
\end{equation}
for all $\phi\in\mathfrak{B}_{Q,\chi}$ and all $\chi\in\mathfrak{X}^{\reg}.$ Observe that 
\begin{align*}
\sum_{\phi\in\mathfrak{B}_{Q,\chi}}\big|\mathcal{H}_{\phi}(\lambda,s_1,s_2)\big|=\sum_{\phi\in\mathfrak{B}_{Q,\chi}}\big|\Psi_{\lambda}(s,\mathcal{I}(\lambda,h)\phi,\phi_j')\big|^2.
\end{align*}
In conjunction with \eqref{178} and the functional equation we conclude that 
\begin{equation}\label{179}
\sum_{\chi\in\mathfrak{X}^{\reg}}\sum_{Q}\frac{1}{n_Q}\left(\frac{1}{2\pi i}\right)^{\dim A_Q/A_G}\int_{i\mathfrak{a}_Q^*/i\mathfrak{a}_G^*}\sum_{\phi\in\mathfrak{B}_{Q,\chi}}\big|\mathcal{H}_{\phi}(\lambda,s_1,s_2)\big|d\lambda
\end{equation}
is majorized by $\tilde{J}_{\Eis}^{\Reg}(f,c_0,c_0;\phi_j',\phi_j')$ defined in \eqref{177}. So \eqref{179} converges. This proves Lemma \ref{lem41} in the situation that $f=h*h^*$ and $s_1=\overline{s_2}.$ The general case follows from the Dixmier-Malliavin's decomposition and Cauchy inequality.
\end{proof}

\subsubsection{Semi-singular Periods}\label{9.4}
The contribution from $\chi\in\mathfrak{X}^{\semi}$ to $J_{\Eis}^{\Reg}(f,\mathbf{s})$ is 
\begin{align*}
J_{\Eis,\semi}^{\Reg}(f,\mathbf{s}):=\frac{1}{2}\sum_{\chi\in\mathfrak{X}^{\semi}}\sum_{Q}\frac{1}{2\pi i}\int_{i\mathfrak{a}_Q^*/i\mathfrak{a}_G^*}\sum_{\phi\in\mathfrak{B}_{Q,\chi}}\mathcal{H}_{\phi}(\lambda,s_1,s_2)d\lambda.
\end{align*}

Let $\mathcal{D}=\mathcal{D}(\pi_1',\pi_2')$ be a standard (open) zero-free region of the $L$-function  $\Lambda(s,\pi_1'\otimes\omega\omega'^{-1})\Lambda(s,\widetilde{\pi}_1'\otimes\omega\omega'^{-1})\Lambda(s,\pi_2'\otimes\omega\omega'^{-1})\Lambda(s,\widetilde{\pi}_2'\otimes\omega\omega'^{-1})$ (cf. e.g., \cite{Bru06}). The boundary $\partial\mathcal{D}$ lies strictly on the left of $\Re(s)=1.$ Let $\mathcal{C}=\big\{c\in\mathbb{C}:\ 1-(1+n^{-1})c\in\mathcal{D},\ 1+(1+n^{-1})c\in\mathcal{D}\big\}.$ Then $\mathcal{C}$ is a ``neighborhood'' of the line $\Re(s)=0.$ The boundary $\partial\mathcal{C}$ has two branches. Let $\partial\mathcal{C}^{+}$ be the branch in $\Re(s)>0.$

\begin{lemma}\label{lem43}
Let notation be as before. Then $J_{\Eis,\semi}^{\Reg}(f,\mathbf{s})$ admits a holomorphic continuation to $\mathbb{C}^2.$ 
\end{lemma}
\begin{proof}
By \cite{JPSS83} and \cite{Jac09} we have
\begin{align*}
\mathcal{H}_{\phi}(\lambda,s_1,s_2)\sim&\frac{L(s_1+1/2+\lambda,\pi_1'\times\sigma_1)L(s_1+1/2-\lambda/n,\pi_1'\times\sigma_2)}{L(1+(1+n^{-1})\lambda,\sigma_1\times\widetilde{\sigma}_2)}\\
	&\quad \times \frac{L(s_2+1/2-\lambda,\pi_2'\times\widetilde{\sigma}_1)L(s_2+1/2+\lambda/n,\widetilde{\pi}_2'\times\widetilde{\sigma}_2)}{L(1-(1+n^{-1})\lambda,\sigma_2\times\widetilde{\sigma}_1)}.
\end{align*}

Note that $\#\mathfrak{X}^{\semi}=4.$ So we can investigate one by one. Without loss of generality we may assume $\chi=\{(M_Q,\sigma)\},$ where $Q$ is of type $(1,n)$ and $\sigma=\sigma_1\otimes\sigma_2$ with $\sigma_2\simeq \widetilde{\pi}_1'$ and $\sigma_1=\omega\omega'^{-1}.$ Note that $\pi_1'\not\simeq \pi_2'.$ Then 
$$
\frac{L(s_2+1/2-\lambda,\pi_2'\times\widetilde{\sigma}_1)L(s_2+1/2+\lambda/n,\widetilde{\pi}_2'\times\widetilde{\sigma}_2)}{L(1-(1+n^{-1})\lambda,\sigma_2\times\widetilde{\sigma}_1)}
$$
is holomorphic for all $\lambda\in\mathcal{C}.$ Denote by 
\begin{equation}\label{180}
J_{\chi}^{\Reg}(f,\mathbf{s}):=\frac{1}{2\pi i}\int_{i\mathfrak{a}_Q^*/i\mathfrak{a}_G^*}\sum_{\phi\in\mathfrak{B}_{Q,\chi}}\mathcal{H}_{\phi}(\lambda,s_1,s_2)d\lambda,
\end{equation}
which is holomorphic when $s_1>1/2.$ Let $s_1>1/2$ and $s_1\in 1/2+\mathcal{C}.$ Then 
\begin{align*}
J_{\chi}^{\Reg}(f,\mathbf{s})=\frac{1}{2\pi i}\int_{\partial\mathcal{C}^+}\sum_{\phi\in\mathfrak{B}_{Q,\chi}}\mathcal{H}_{\phi}(\lambda,s_1,s_2)d\lambda-\sum_{\phi\in\mathfrak{B}_{Q,\chi}}\underset{\lambda=n(s_1-1/2)}{\Res}\mathcal{H}_{\phi}(\lambda,s_1,s_2),
\end{align*}
which is a holomorphic function for $s_1\in 1/2+\mathcal{C}$ and $s_2\in\mathbb{C}.$ When $\Re(s_1)\leq 1/2,$
\begin{align*}
J_{\chi}^{\Reg}(f,\mathbf{s})=\frac{1}{2\pi i}\int_{i\mathfrak{a}_P^*/i\mathfrak{a}_G^*}\sum_{\phi\in\mathfrak{B}_{Q,\chi}}\mathcal{H}_{\phi}(\lambda,s_1,s_2)d\lambda-\sum_{\phi\in\mathfrak{B}_{Q,\chi}}\underset{\lambda=n(s_1-1/2)}{\Res}\mathcal{H}_{\phi}(\lambda,s_1,s_2),
\end{align*}
which is holomorphic therein. Therefore, we obtain an analytic continuation of $J_{\chi}^{\Reg}(f,\mathbf{s})$ to $(s_1,s_2)\in\mathbb{C}^2.$ 

Similar arguments applies to remaining $3$ cases. In all, we obtain an analytic continuation of $J_{\Eis,\semi}^{\Reg}(f,\mathbf{s})$ to $(s_1,s_2)\in\mathbb{C}^2.$ 
\end{proof}

\subsubsection{Singular Periods}\label{9.5}
The contribution of $\chi\in\mathfrak{X}^{\sing}$ to $J_{\Eis}^{\Reg}(f,\mathbf{s})$ (cf. \eqref{182}) is 
\begin{equation}\label{182'}
J_{\Eis,\sing}^{\Reg}(f,\mathbf{s}):=\frac{1}{2}\sum_{\chi\in\mathfrak{X}^{\sing}}\sum_{Q}\frac{1}{2\pi i}\int_{i\mathfrak{a}_Q^*/i\mathfrak{a}_G^*}\sum_{\phi\in\mathfrak{B}_{Q,\chi}}\mathcal{H}_{\phi}(\lambda,s_1,s_2)d\lambda.
\end{equation}

Define the auxiliary functions 
\begin{equation}\label{183}
\mathcal{G}_{\chi}(s_1,s_2,\phi_1',\phi_2'):=\sum_{\phi\in\mathfrak{B}_{Q,\chi}}\underset{\lambda=n(s_1-1/2)}{\Res}\Psi_{\lambda}(s_1,\mathcal{I}(\lambda,f)\phi,\phi_1')\overline{\Psi_{\lambda}(\overline{s_2},\phi,\phi_2')}
\end{equation}
and similarly
\begin{equation}\label{198}
\widetilde{\mathcal{G}}_{\chi}(s_1,s_2,\phi_1',\phi_2'):=\sum_{\phi\in\mathfrak{B}_{Q,\chi}}\underset{\lambda=n(1/2-s_2)}{\Res}\Psi_{\lambda}(s_1,\mathcal{I}(\lambda,f)\phi,\phi_1')\overline{\Psi_{\lambda}(\overline{s_2},\phi,\phi_2')}.
\end{equation}

\begin{lemma}\label{lem44}
Let notation be as before. Then $J_{\Eis,\sing}^{\Reg}(f,\mathbf{s})$ admits a meromorphic continuation to $\mathbb{C}^2,$ with possible (simple) poles on $s_1+s_2=1$ and $s_1+s_2=0.$
\end{lemma}
\begin{proof}
Since $\#\mathfrak{X}^{\sing}=2,$ suffice it to consider each summand of $J_{\Eis,\sing}^{\Reg}(f,\mathbf{s}).$ We may assume $\chi=\{(M_Q,\sigma)\},$ where $Q$ is of type $(1,n)$ and $\sigma=\sigma_1\otimes\sigma_2$ with $\sigma_2\simeq \widetilde{\pi}_1'\simeq \widetilde{\pi}_2',$ $\sigma=\omega\omega'^{-1}.$ Then by \cite{JPSS83} and \cite{Jac09} we have
\begin{align*}
\mathcal{H}_{\phi}(\lambda,s_1,s_2)\sim&\frac{L(s_1+1/2+\lambda,\pi'\times\sigma_1)L(s_1+1/2-\lambda/n,\pi'\times\widetilde{\pi}')}{L(1+(1+n^{-1})\lambda,\pi'\times\sigma_1)}\\
	&\quad \times \frac{L(s_2+1/2-\lambda,\widetilde{\pi}'\times\widetilde{\sigma}_1)L(s_2+1/2+\lambda/n,\pi'\times\widetilde{\pi}')}{L(1-(1+n^{-1})\lambda,\widetilde{\pi}'\times\widetilde{\sigma}_1)}.
\end{align*}

Let $J_{\chi}^{\Reg}(f,\mathbf{s})$ be defined in \eqref{180}. Then $J_{\chi}^{\Reg}(f,\mathbf{s})$ is converges absolutely when $\Re(s_1)>1/2$ and $\Re(s_2)>1/2.$ Let $s_1>1/2,$ $s_1\in 1/2+\mathcal{C},$ and $\Re(s_2)>1/2.$ Then by Cauchy integral formula $J_{\Eis,\sing}^{\Reg}(f,\mathbf{s})$ is equal to 
\begin{equation}\label{197}
\frac{1}{2}\sum_{\chi\in\mathfrak{X}^{\sin}}\Bigg[\sum_{Q}\frac{1}{2\pi i}\int_{\partial\mathcal{C}^+}\sum_{\phi\in\mathfrak{B}_{Q,\chi}}\mathcal{H}_{\phi}(\lambda,s_1,s_2)d\lambda-\mathcal{G}_{\chi}(s_1,s_2,\phi_1',\phi_2')\Bigg],
\end{equation}
where $\mathcal{G}_{\chi}(s_1,s_2,\phi_1',\phi_2')$ is defined by \eqref{183}. Note that 
\begin{align*}
\mathcal{G}_{\chi}(s_1,s_2,\phi_1',\phi_2')\sim \frac{L(s_2+\frac{n+1}{2}-ns_1,\widetilde{\pi}'\times\widetilde{\sigma}_1)L(s_1+s_2,\pi'\times\widetilde{\pi}')}{L(1+\frac{n+1}{2}-(n+1)s_1,\widetilde{\pi}'\times\widetilde{\sigma}_1)}.
\end{align*}
Then $\mathcal{G}_{\chi}(s_1,s_2,\phi_1',\phi_2')$ continues to a meromorphic function in the region where  $s_1\in 1/2+\mathcal{C}$ or $\Re(s_1)\leq 1/2,$ and $s_2\in\mathbb{C},$ with at most simple poles on $s_1+s_2=0$ and $s_1+s_2=1.$ We then need to continue the function 
\begin{equation}\label{181}
\frac{1}{2\pi i}\int_{\partial\mathcal{C}^+}\sum_{\phi\in\mathfrak{B}_{Q,\chi}}\mathcal{H}_{\phi}(\lambda,s_1,s_2)d\lambda,\ \ s_1>1/2,\ s_1\in 1/2+\mathcal{C},\ \Re(s_2)>1/2.
\end{equation}
By the analytic behavior of $\mathcal{H}_{\phi}(\lambda,s_1,s_2),$ \eqref{181} extends to a holomorphic function in the region $s_1\in\mathcal{D}^-=\big\{s\in\mathbb{C}:\ \Re(s)< 1/2\ \text{or}\ s\in 1/2+\mathcal{C}\big\},$ and $s_2\in\mathcal{D}^+=\big\{s\in\mathbb{C}:\ \Re(s)> 1/2\ \text{or}\ s\in 1/2-\mathcal{C}\big\}.$ Let $s_2\in\mathcal{D}^+$ with $\Re(s_2)\leq 1/2.$ Then 
\begin{align*}
\frac{1}{2\pi i}\int_{\partial\mathcal{C}^+}\sum_{\phi\in\mathfrak{B}_{Q,\chi}}\mathcal{H}_{\phi}(\lambda,s_1,s_2)d\lambda=\frac{1}{2\pi i}\int_{i\mathfrak{a}_P^*/i\mathfrak{a}_G^*}\sum_{\phi\in\mathfrak{B}_{Q,\chi}}\mathcal{H}_{\phi}(\lambda,s_1,s_2)d\lambda+\widetilde{\mathcal{G}}_{\chi},
\end{align*}
where $\widetilde{\mathcal{G}}_{\chi}=\widetilde{\mathcal{G}}_{\chi}(s_1,s_2,\phi_1',\phi_2')$ is defined by \eqref{198}.

Note that $\frac{1}{2\pi i}\int_{i\mathfrak{a}_P^*/i\mathfrak{a}_G^*}\sum_{\phi\in\mathfrak{B}_{Q,\chi}}\mathcal{H}_{\phi}(\lambda,s_1,s_2)d\lambda$ is holomorphic when $s_1\in\mathcal{D}^{-}$ and $\Re(s_2)<1/2.$ Moreover, we have 
\begin{align*}
\widetilde{\mathcal{G}}_{\chi}(s_1,s_2,\phi_1',\phi_2')\sim \frac{L(s_1+\frac{n+1}{2}-ns_2,{\pi}'\times{\sigma}_1)L(s_1+s_2,\pi'\times\widetilde{\pi}')}{L(1+\frac{n+1}{2}-(n+1)s_2,{\pi}'\times{\sigma}_1)},
\end{align*}
implying that $\widetilde{\mathcal{G}}_{\chi}(s_1,s_2,\phi_1',\phi_2')$ continues to a meromorphic function in the region where $s_2\in\mathcal{D}^-$ and $s_1\in\mathbb{C},$ with at most simple poles on $s_1+s_2=0$ and $s_1+s_2=1.$

Therefore, we obtain a meromorphic continuation of $J_{\Eis,\sing}^{\Reg}(f,\mathbf{s})$ to $\mathbb{C}^2,$ with at most simple poles on $s_1+s_2=0$ and $s_1+s_2=1.$
 \end{proof}

\section{Singularity Matching}\label{sec10.}
Let $n\geq 1.$ Let $\mathbf{s}=(s_1,s_2)\in\mathcal{R}.$ In this section we let $Q$ be the parabolic subgroup of $G$ of type $(1,n),$ with Levi subgroup $M_Q.$ Suppose $\pi_1'\simeq \pi_2'\simeq\pi'.$ Let $\phi_j'\in\pi_j',$ $1\leq j\leq2.$ Let $\chi=\{(M_Q,\sigma)\}$ be a cuspidal datum such that $\sigma=\sigma_1\otimes\sigma_2$ with $\sigma_2\simeq\widetilde{\pi}'$ and $\sigma_1=\omega\omega'^{-1}.$ Denote by 
\begin{align*}
\mathcal{J}_{\chi}(f,\mathbf{s},\phi_1',\phi_2')=\sum_{\phi\in\mathfrak{B}_{Q,\chi}}\frac{1}{2\pi i}\int_{i\mathfrak{a}_{Q}/i\mathfrak{a}_G}\mathcal{H}_{\phi}(\lambda,s_1,s_2)d\lambda,
\end{align*}
where $\mathcal{H}_{\phi}(\lambda,s_1,s_2):=\Psi_{\lambda}(s_1,\mathcal{I}(\lambda,f)\phi,\phi_1')\overline{\Psi_{\lambda}(\overline{s_2},\phi,\phi_2')}.$ Note that $\mathcal{J}_{\chi}(f,\mathbf{s},\phi_1',\phi_2')$ is precisely the integral defined in \eqref{180}. We use the notation $\mathcal{J}_{\chi}(f,\mathbf{s},\phi_1',\phi_2')$ here to highlight the dependence of $\phi_1'$ and $\phi_2'.$

We have meromorphic continuation of $\mathcal{J}_{\chi}(f,\mathbf{s},\phi_1',\phi_2')$ to $\mathbb{C}^2$ with possible simple poles on $s_1+s_2=0$ and $s_1+s_2=1$ (cf. \textsection \ref{9.5}). In this section, we compute the \textit{meromorphic part} of $\mathcal{J}_{\chi}(f,\mathbf{s},\phi_1',\phi_2')$ explicitly in terms of $f,$ $\phi_1'$ and $\phi_2'.$ The main result is that the singularities from the spectral side (i.e., $\mathcal{G}_{\chi}(s_1,s_2,\phi_1',\phi_2')$ (cf. \eqref{183}) and $\widetilde{\mathcal{G}}_{\chi}(s_1,s_2,\phi_1',\phi_2')$ (cf. \eqref{198}) ) and the geometric side cancel with each other. In fact, we will show $\mathcal{G}_{\chi}(s_1,s_2,\phi_1',\phi_2')$ (resp. $\widetilde{\mathcal{G}}_{\chi}(s_1,s_2,\phi_1',\phi_2')$) is closely related to  $\mathcal{F}_{0,1}J^{\bi}_{\Geo}(\transp{f},\textbf{s})$ (resp. $\mathcal{F}_{1,0}J^{\bi}_{\Geo}(\transp{f},\textbf{s})$) under the functional equation, where $\transp{f}$ is defined by $\transp{f}(g)=f(\transp{g}).$
 
Define the Eisenstein series $E_P^{\dagger}(x,s;\widehat{f}_P,y)$ to be 
\begin{equation}\label{276}
\sum_{\delta \in P'(F)\backslash G'(F)}\int_{Z'(\mathbb{A}_F)}\int_{N_P(\mathbb{A}_F)}\int_{\mathbb{A}_F^n}f\left(u(\mathbf{x})n\iota(y)\right)\theta(\eta \delta zx\mathfrak{u})d\mathbf{x}dn|\det zx|^{s}d^{\times}z,
\end{equation}
where we write $n=\begin{pmatrix}
	I_n&\mathfrak{u}\\
	&1
\end{pmatrix}\in N_P(\mathbb{A}_F).$

Then $E_P^{\dagger}(x,s;\widehat{f}_P,y)$ converges absolutely when $\Re(s)>1,$ admits a functional equation and meromorphic continuation to $s\in\mathbb{C}.$ Likewise, the Eisenstein series $\widetilde{E}_P^{\dagger}(x,s_1+s_2;\widehat{f}_P,y)$ defined by 
\begin{align*}
	\sum_{\delta \in P'(F)\backslash G'(F)}\int_{Z'(\mathbb{A}_F)}\int_{N_P(\mathbb{A}_F)}\int_{\mathbb{A}_F^n}f\left(nu(\mathbf{x})\iota(y)\right)\theta(\eta \delta zx\mathfrak{u})d\mathbf{x}dn|\det zx|^{s}d^{\times}z
\end{align*}
admits the same analytic properties. Recall (cf. \eqref{183})
$$
\mathcal{G}_{\chi}(s_1,s_2,\phi_1',\phi_2'):=\sum_{\phi\in\mathfrak{B}_{Q,\chi}}\underset{\lambda=n(s_1-1/2)}{\Res}\Psi_{\lambda}(s_1,\mathcal{I}(\lambda,f)\phi,\phi_1')\overline{\Psi_{\lambda}(\overline{s_2},\phi,\phi_2')}.
$$
By \cite{JPSS83} and \cite{Jac09} we have
\begin{align*}
\mathcal{H}_{\phi}(\lambda,s_1,s_2)\sim&\frac{L(s_1+1/2+\lambda,\pi_1'\times\sigma_1)L(s_1+1/2-\lambda/n,\pi_1'\times\sigma_2)}{L(1+(1+n^{-1})\lambda,\sigma_1\times\widetilde{\sigma}_2)}\\
&\quad \times \frac{L(s_2+1/2-\lambda,\pi_2'\times\widetilde{\sigma}_1)L(s_2+1/2+\lambda/n,\widetilde{\pi}_2'\times\widetilde{\sigma}_2)}{L(1-(1+n^{-1})\lambda,\sigma_2\times\widetilde{\sigma}_1)}.
\end{align*}
As a consequence we have
\begin{align*}
\mathcal{G}_{\chi}(s_1,s_2,\phi_1',\phi_2')\sim \frac{L(s_2-ns_1+\frac{n+1}{2},\widetilde{\pi}_2'\otimes\overline{\omega}'\overline{\omega})L(s_1+s_2,\pi_1'\times\widetilde{\pi}_2')}{L(1-(n+1)s_2+\frac{n+1}{2},\widetilde{\pi}_2'\otimes\overline{\omega}'\overline{\omega})},
\end{align*}
which has a simple at $s_1=1-s_2.$ 

The main results in this section is the following.
\begin{thm}\label{thm49}
Let notation be as before. Then $\mathcal{G}_{\chi}(s_1,s_2,\phi_1',\phi_2')$ is equal to 
\begin{equation}\label{194}
-\int_{G'(\mathbb{A}_F)}\int_{[\overline{G'}]}\phi_1'(x)\overline{\phi_2'(xy)}E_P^{\dagger}(x,s_1+s_2;\widehat{f}_P,y)dx\overline{\omega'}(y)|\det y|^{s_2}dy.
\end{equation}
Similarly, $\widetilde{\mathcal{G}}_{\chi}(s_1,s_2,\phi_1',\phi_2')$ is equal to 
\begin{align*}
\int_{G'(\mathbb{A}_F)}\int_{[\overline{G'}]}\phi_1'(x)\overline{\phi_2'(xy)}\widetilde{E}_P^{\dagger}(x,s_1+s_2;\widehat{f}_P,y)dx\overline{\omega'}(y)|\det y|^{s_2}dy.
\end{align*}
Moreover, $\mathcal{G}_{\chi}(s_1,s_2,\phi_1',\phi_2')$ (resp. $\widetilde{\mathcal{G}}_{\chi}(s_1,s_2,\phi_1',\phi_2')$) admits a meromorphic continuation to $$
\mathcal{R}:=\big\{\mathbf{s}=(s_1,s_2)\in\mathbb{C}^2:\ \Re(s_1),\ \Re(s_2)>-1/(n+1)\big\}.
$$
\end{thm}

\subsection{Singular Eisenstein Periods}\label{10.2}
Let notation be as above. Recall 
\begin{align*}
	H_Q: M_Q(\mathbb{A}_F)\rightarrow \mathbb{R}^2,\  m=\begin{pmatrix}
		m_1&\\
		&m_2
	\end{pmatrix}\mapsto H_Q(m)=\left(\log |\det m_1|,\frac{\log |\det m_2|}{n}\right).
\end{align*}
Let $\rho_Q$ be the half sum of positive roots of $(Q,A_Q).$ Then $e^{\rho_QH_Q(m)}=\delta_Q^{1/2}(m)=|\det m_1|^{n/2}|\det m_2|^{-1/2}$ for $m=\begin{pmatrix}
	m_1&\\
	&m_2
\end{pmatrix}\in M_Q(\mathbb{A}_F).$

Let $\pi'$ and $\sigma_2$ be cuspidal automorphic representations of $G'(\mathbb{A}_F)$ of central character $\omega'$ and $\omega$ respectively. Let $\sigma_1=\omega'\omega^{-1}.$

Denote by $\mathcal{A}_0^{Q}$ the space of functions $\phi:$ $G(\mathbb{A}_F)\longrightarrow\mathbb{C}$ such that $\phi(amug)=\delta_Q^{\frac{1}{2}}\phi(g)$ for $m\in M_Q(\mathbb{A}_F),$ $u\in N_Q(\mathbb{A}_F),$ $a\in \diag(1,A_{G'}),$ and for all $k\in K$ the function $m\mapsto \phi(mk)$ is a cuspidal automorphic form on $G'(\mathbb{A}_F).$ Then $\mathcal{A}_0^{Q}$ decomposes as a direct sum of $\mathcal{A}_{0,\sigma}^{Q}$ over cuspidal representations $\sigma$ of $M_Q(\mathbb{A}_F).$ We can identify $\mathcal{A}_0^{Q}$ with the representation $\Ind _{Q(\mathbb{A}_F)}^{G(\mathbb{A}_F)}\mathcal{A}_{0}([M_Q])$ by $\phi\mapsto \delta_Q^{-1/2}\pi(g)\phi$. Let $\lambda\in\mathbb{C}.$ Denote by $\boldsymbol{\lambda}=(\lambda,-\lambda).$ Write $\boldsymbol{\lambda}H_Q(m)$ for the dot product of vectors $\boldsymbol{\lambda}$ and $H_Q.$ Set $\phi_{\boldsymbol{\lambda}}(g)=\phi(mk)e^{\boldsymbol{\lambda}H_Q(m)},$ where $g=umk$ is the Iwasawa decomposition with $u\in N_Q(\mathbb{A}_F),$ $m=\begin{pmatrix}
	m_1&\\
	&m_2
\end{pmatrix}\in M_Q(\mathbb{A}_F)$ and $k\in K.$ With the action $\mathcal{I}(g,s)\phi=\left(\phi(\cdot g)_{\boldsymbol{\lambda}+\rho_Q}\right)_{-\boldsymbol{\lambda}-\rho_Q}$ the space $\mathcal{A}_0^Q$ becomes  $\Ind _{Q(\mathbb{A}_F)}^{G(\mathbb{A}_F)}\mathcal{A}_{0}([M_P])\otimes \alpha^{\lambda},$ where $\alpha(\diag(m_1,m_2))=|\det m_2|^{-1/n}|\det m_1|.$ Define the Eisenstein series:
\begin{align*}
	E(g,\phi,\lambda)=\sum_{\delta\in Q(F)\backslash G(F)}\phi_{\boldsymbol{\lambda}+\rho_Q}(\delta g),
\end{align*}
which converges absolutely when $\Re(\lambda)> \rho_Q$ and admits meromorphic continuation to the whole complex plane and a functional equation. Moreover, it is holomorphic on the vertical line $\Re(\lambda)=0.$

Let $\phi\in \mathcal{A}_{0,\sigma}^Q$ with $\sigma=\sigma_1\otimes\sigma_2$ being a cuspidal representation of the Levi subgroup of $Q(\mathbb{A}_F).$ Let 
\begin{align*}
	W_{\lambda}(g)=\int_{[N]}E(ug,\phi,\lambda)\overline{\theta}(u)du,\ \ \Re(\lambda)\gg 0.
\end{align*}  

Let $v\in\Sigma_{F}.$ Denote by $\sigma_{v,\lambda}=\Ind_{Q(F_v)}^{G(F_v)}\sigma_v\otimes\alpha_v^{\lambda}.$ Unfolding the Eisenstein series one can show (e.g., cf. \cite{Sha10}, p. 122) that $W_{\lambda}(g)$ is factorizable:
\begin{align*}
	W_{\lambda}(g)=\prod_{v\in \Sigma_F}W_{\lambda,v}(g_v),
\end{align*}
where each factor $W_{\lambda,v}(g_v)$ can be given by a Jacquet integral. 

Let $W'=\prod_{v\in \Sigma_F}W_v'$ be a Whittaker function associated to the cuspidal representation $\pi'.$ Define the Rankin-Selberg convolution by 
\begin{equation}\label{r-s}
	\Psi_{\lambda}(s,\phi,\phi'):=\int_{N'(\mathbb{A}_F)\backslash G'(\mathbb{A}_F)}W_{\lambda}\left(\begin{pmatrix}
		x&\\
		&1
	\end{pmatrix}\right)W'(x)|\det x|^sdx,\ \ \Re(s)\gg 0.
\end{equation}

Integrals of form $\Psi_{\lambda}(s,\phi,\phi')$ will play an important role in the spectral side of our regularized relative trace formula. They are actually the possible non-holomorphic part of $J_{\Eis}^{\Reg}(f,\textbf{s}),$ the contribution from the continuous spectrum. We will study in this section the analytic behavior of $\Psi_{\lambda}(s,\phi,\phi')$ and compute their possible poles and residues. Together with further manipulation on the convergence of the whole spectral side, we will derive the meromorphic continuation and thus the regularization of the spectral side. See \text\textsection \ref{sec.spec} below.

By the local theory developed in \cite{JPSS83} and \cite{Jac09} the integral 
\begin{align*} 
	\Psi_{\lambda,v}(s,\phi_v,\phi'_v):=\int_{N'(F_v)\backslash G'(F_v)}W_{\lambda,v}\left(\begin{pmatrix}
		x_v&\\
		&1
	\end{pmatrix}\right)W_v'(x_v)|\det x_v|_v^sdx_v	
\end{align*}
represents the Rankin-Selberg $L$-function $L_v(s+1/2,\sigma_v\times\pi_v'\otimes e^{\boldsymbol{\lambda}H_Q})$ up to a holomorphic factor, which is $L_v(1-2\lambda,\sigma_1\times\sigma_2)^{-1}$ at all but finitely many places. Hence \eqref{r-s} is well defined when $\Re(s)$ and $\Re(\lambda)$ large. The convergence of $\Psi_{\lambda}(s,\phi,\phi')$ when $\Re(s)\gg 0$ can also be verified by making use of a modified truncation operator, see \cite{IY15} for details. 

Note that $L(s+1/2,\sigma\times\pi'\otimes e^{\boldsymbol{\lambda}H_Q})=L(s+1/2+\lambda,\sigma_1\times\pi')L(s+1/2-\lambda/n,\sigma_2\times\pi').$ Hence, $\Psi_{\lambda}(s,\phi,\phi')$ admits an meromorphic continuation to  unless $\sigma_2\equiv \widetilde{\pi}',$ in which case $\Psi_{\lambda}(s,\phi,\phi')$ has a simple pole at $s=\pm 1/2-\lambda$ and is holomorphic elsewhere. 

We shall compute $\Psi_{\lambda}(s,\phi,\phi')$ more explicitly and compute the possible residue at $s=1/2$. In the case of $\mathrm{GL}(2),$ this is done in \cite{JL70}. However, it seems likely that the higher rank case has not been investigated yet. 

In the process of regularization of $J_{\Eis}^{\Reg}(f,\textbf{s}),$ we have to regard $\Psi_{\lambda}(s,\phi,\phi')$ as a function of both $s$ and $\lambda.$ Hence, it is more convenient for us to apply the Godement section (cf. \cite{God95}) in $\Ind_{P(\mathbb{A}_F)}^{G(\mathbb{A}_F)}\sigma\otimes e^{\boldsymbol{\lambda}H_Q}$ in the calculation rather than the standard section, which is independent of $\lambda$ when restricted to the compact subgroup $K.$ 

Note that  $\Ind_{P(\mathbb{A}_F)}^{G(\mathbb{A}_F)}\sigma\otimes e^{\boldsymbol{\lambda}H_Q}$ is the restricted tensor product over all places $v\in\Sigma_F$ of analogously defined spaces $\Ind_{Q(F_v)}^{G(F_v)}\sigma_v\otimes e^{\boldsymbol{\lambda}H_{Q,v}}$. We may write a pure tensor section $\phi\in  \Ind_{P(\mathbb{A}_F)}^{G(\mathbb{A}_F)}\sigma\otimes e^{\boldsymbol{\lambda}H_Q}$ as $\phi=\otimes_v\phi_v,$ with $\phi_v\in \Ind_{Q(F_v)}^{G(F_v)}\sigma_v\otimes e^{\boldsymbol{\lambda}H_{Q,v}}$ and $\phi_v\equiv 1$ for all but finitely many $v\in\Sigma_F.$ 

Let $v$ be a finite place such that $\sigma_v$ is unramified. Let $\Phi_v$ be the characteristic function of $M_{n\times (n+1)}(\mathcal{O}_{F_v}).$ Then the function $h_v(g_v,{\lambda}),$ defined by
\begin{align*}
\sigma_{1,\lambda,v}(\det g_v)|\det g_v|_v^{\frac{n}{2}}\int_{G'(F_v)}\Phi_v[(0,g_v')g_v]\phi_{2,-\lambda,v}(g_v'^{-1})\sigma_{1,\lambda,v}(\det g'_v)|\det g'_v|_v^{\frac{n+1}{2}}dg_v'^{-1},
\end{align*}
is a section in $\Ind_{Q(F_v)}^{G(F_v)}\sigma_v\otimes e^{\boldsymbol{\lambda}H_{Q,v}},$ where $\phi_{2,-\lambda,v}$ is a vector in $\sigma_{2,-\lambda,v}.$ An explicit calculation shows that 
\begin{align*}
h_v(I_{n+1},{\lambda})=L_v(1+2\lambda,\sigma_{1,v}\times\widetilde{\sigma}_{2,v}).
\end{align*}
Hence, the function $\phi_v^*(g_v,{\lambda}):=h_v(g_v,{\lambda})/L_v(1+2\lambda,\sigma_{1,v}\times\widetilde{\sigma}_{2,v})$ is a standard section, i.e., $\phi_v^*(k_v,{\lambda})\equiv 1,$ $\forall$ $k_v\in K_v.$

Let $\mathcal{S}(n\times(n+1),\mathbb{A}_F)$ be the space of Bruhat-Schwartz functions on $M_{n\times (n+1)}(\mathbb{A}_F).$ Then $\phi^*_{\Phi,\phi_2}(g,{\lambda}),$ defined by 
\begin{equation}\label{163}
\frac{\sigma_{1,\lambda}(\det g)|\det g|^{\frac{n}{2}}}{L(1+2\lambda,\sigma_{1}\times\widetilde{\sigma}_{2})}\int_{G'(\mathbb{A}_F)}\Phi[(0,g')g]\phi_{2,-\lambda}(g'^{-1})\sigma_{1,\lambda}(\det g')|\det g'|^{\frac{n+1}{2}}dg',
\end{equation}
is a global section in $\Ind_{P(\mathbb{A}_F)}^{G(\mathbb{A}_F)}\sigma\otimes e^{\boldsymbol{\lambda}H_Q},$ where $\Phi\in\mathcal{S}(n\times(n+1),\mathbb{A}_F)$ and $\phi_{2,-\lambda}$ is a vector in $\sigma_{2,-\lambda}.$ Combining the local theory developed in \cite{JPSS83} (the proposition in \textsection6, p.428) and \cite{Jac09} (Proposition 7.1, p. 98) we conclude that every global section of $\Ind_{P(\mathbb{A}_F)}^{G(\mathbb{A}_F)}\sigma\otimes e^{\boldsymbol{\lambda}H_Q}$ is a finite sum of $F({\lambda})\phi^*_{\Phi,\phi_2}(g,{\lambda}),$ where $\Phi\in\mathcal{S}(n\times(n+1),\mathbb{A}_F),$ $\phi_2\in \sigma_{2},$ and  $F({\lambda})=\prod_v F_v(\lambda)$ is an elementary function of $\lambda$ with $F_v(\lambda)\equiv 1$ for all but finitely many places. 

Write $u=\begin{pmatrix}
	I_n&\mathfrak{u}\\
	&1
\end{pmatrix}\in N_P(\mathbb{A}_F).$ Let $\chi=\{(M_Q,\sigma)\}.$ Let $w=\begin{pmatrix}
	&1\\
	I_{n}&
\end{pmatrix}.$ Define the vector valued Eisenstein series associated to $\phi\in\mathfrak{B}_{Q,\chi}$ by 
\begin{align*}
e_v(x,\phi,s):=\sum_{\delta \in P'(F)\backslash G'(F)}\int_{Z'(\mathbb{A}_F)}\int_{N_P(\mathbb{A}_F)}\phi(wu)\overline{\theta}(\eta \delta zx\mathfrak{u})d\mathfrak{u}|\det zx|^{s}d^{\times}z.
\end{align*}
By Poisson summation, $e_v(x,\phi,s)$ converges absolutely when $\Re(s)>1,$ and admits a functional equation and meromorphic continuation to $\lambda\in\mathbb{C}$ with possible (simple) poles at $s\in \{0,1\}.$

\begin{prop}\label{prop46}
Let notation be as before. Let $\chi=\{(M_Q,\sigma)\}$ and $\phi\in\mathfrak{B}_{Q,\chi}.$ Then
\begin{equation}\label{166.}
\Psi_{\lambda}(s,\phi,\phi')=\int_{[\overline{G'}]}\phi'(x)\sigma_2'(x)e_v(x,\phi,s+1/2-\lambda/n)dx.
\end{equation}
In particular, $\Psi_{\lambda}(s,\phi,\phi')$ is holomorphic if $\sigma_2\not\simeq\widetilde{\pi}',$ and when $\sigma_2\simeq\widetilde{\pi}',$ $\Psi_{\lambda}(s,\phi,\phi')$ is meromorphic with (at most simple) poles on $s-\lambda/n=1/2$ and $s-\lambda/n=-1/2.$

\end{prop}
\begin{proof}
Substituting $\phi^*(g)=\phi^*_{\Phi,\phi_2}(g,{\lambda})$ into the Jacquet integral 
\begin{equation}\label{164}
	W_{\lambda}^*(g)=\int_{N(\mathbb{A}_F)}\phi^*(w_n^lug)\overline{\theta}(u)du,
\end{equation}
where $w_n^l$ is the long element in the Weyl group of $(G,B).$ We can write $w_n^l=\begin{pmatrix}
	&1\\
	w_{n-1}^l\\
\end{pmatrix},$ where $w_{n-1}^l$ is the long Weyl element relative to the pair $(G',B').$ Write $N(\mathbb{A}_F)\ni u=\begin{pmatrix}
	u'&\\
	&1
\end{pmatrix}\begin{pmatrix}
	I_n&\mathfrak{u}\\
	&1
\end{pmatrix}$ Then after a change of variable $g'\mapsto w_{n-1}^lu'g'$ to obtain 
\begin{align*}
	W_{\lambda}^*(g)=\frac{\sigma_{1,\lambda}(\det g)|\det g|^{\frac{n}{2}}}{L(1+(1+n^{-1})\lambda,\sigma_{1}\times\widetilde{\sigma}_{2})}&\int_{G'(\mathbb{A}_F)}W_{2,-\lambda}(g')\sigma_{1,\lambda}(\det g'^{-1})|\det g'|^{-\frac{n+1}{2}}dg'\\
	&\qquad \qquad \times \int_{\mathbb{A}_F^n}\Phi[(g'^{-1},g'^{-1}\mathfrak{u})g]\overline{\theta}(\eta \mathfrak{u})d\mathfrak{u},
\end{align*}
where 
$$
W_{2,-\lambda}(g'):=\int_{[N']}\phi_{2,-\lambda}(w_{n-1}u'g')\overline{\theta}(u')du'
$$
is a vector in $\mathcal{W}_{\sigma_{2,-\lambda}},$ the Whittaker model associated to $\sigma_{2,-\lambda}=\sigma_2\otimes |\det(\cdot)|^{-\lambda/n}.$  

Define $\widehat{\Phi}_2(x,\mathfrak{c})=\int_{\mathbb{A}_F^n}\Phi(x,\mathfrak{u})\overline{\theta}(\transp{\mathfrak{u}\mathfrak{c}})d\mathfrak{u},$ where $x\in M_{n\times n}(\mathbb{A}_F)$ and $\mathfrak{c}\in M_{n\times 1}(\mathbb{A}_F).$ Then by \eqref{164} the Rankin-Selberg integral defined by \eqref{r-s}  becomes 
\begin{align*}
	\Psi_{\lambda}&(s,\phi^*,\phi')=\frac{1}{L(1+(1+n^{-1})\lambda,\sigma_{1}\times\widetilde{\sigma}_{2})}\int_{N'(\mathbb{A}_F)\backslash G'(\mathbb{A}_F)}W'(x)\sigma_{1,\lambda}(\det x)\\
	& |\det x|^{s+\frac{n}{2}}dx\int_{G'(\mathbb{A}_F)}\widehat{\Phi}_2(g'x,\eta g'^{-1})W_{2,-\lambda}(g'^{-1})\sigma_{1,\lambda}(\det g')|\det g'|^{\frac{n-1}{2}}dg',
\end{align*}
which converges absolutely when $\Re(s)$ and $\Re(\lambda)$ are large. 

Change variable $g'\mapsto g'x^{-1}$ to yield 
\begin{align*}
	\Psi_{\lambda}(s,\phi^*,\phi')=&\frac{1}{L(1+(1+n^{-1})\lambda,\sigma_{1}\times\widetilde{\sigma}_{2})}\int_{N'(\mathbb{A}_F)\backslash G'(\mathbb{A}_F)}W'(x)|\det x|^{s+\frac{1}{2}}dx\\
	&\qquad \int_{G'(\mathbb{A}_F)}\widehat{\Phi}_2(g',\eta xg'^{-1})W_{2,-\lambda}(xg'^{-1})\sigma_{1,\lambda}(\det g')|\det g'|^{\frac{n-1}{2}}dg'.
\end{align*}

Change $x$ to $xg'$ then $L(1+(1+n^{-1})\lambda,\sigma_{1}\times\widetilde{\sigma}_{2})\Psi_{\lambda}(s,\phi^*,\phi')$ becomes
\begin{equation}\label{165}
	\int W_{2,-\lambda}(x)\int_{G'(\mathbb{A}_F)}W'(xy)\sigma_{1,\lambda}(\det y)|\det y|^{s+\frac{n}{2}}\widehat{\Phi}_2(y,\eta x)dy|\det x|^{s+\frac{1}{2}}dx,
\end{equation}
where $x$ ranges over $N'(\mathbb{A}_F)\backslash G'(\mathbb{A}_F).$ Following \cite{JPSS83} and \cite{Jac09} there exists a fundamental idempotent $\xi$ on $G'(\mathbb{A}_F)$ such that $R(\xi)(\Phi(y,\cdot)\sigma_{1,\lambda}(\det y))=\Phi(y,\cdot)\sigma_{1,\lambda}(\det y)$ for all $y.$ Then we can rewrite the above expression as 
\begin{align*}
	\int W_{2,-\lambda}(x)|\det x|^{s+\frac{1}{2}}dx\int_{G'(\mathbb{A}_F)}\widehat{\Phi}_2(y,\eta x)\sigma_{1,\lambda}(\det y)|\det y|^{s+\frac{n}{2}}R(\xi)\pi'(y)W'(x)dy,
\end{align*}
where $x$ ranges over $N'(\mathbb{A}_F)\backslash G'(\mathbb{A}_F).$ 

Note that the operator $R(\xi)$ is a projection onto the finite dimensional subspace of  $\mathcal{W}(\pi'),$ the Whittaker model of $\pi'.$ Let $\big\{W_j':\ j\in J\big\}$ be an orthonormal basis of this subspace, where $J$ is a finite index set. Then 
\begin{align*}
R(\xi)\pi'(y)W'=\sum_{j\in J}\langle R(\xi)\pi'(y)W', W_j'\rangle W_j'=\sum_{j\in J}\langle \pi'(y)W', W_j'\rangle W_j',
\end{align*}
where $y\in\supp \Phi.$ So 
\begin{align*}
	\Psi_{\lambda}(s,\phi^*,\phi')&=\frac{1}{L(1+(1+n^{-1})\lambda,\sigma_{1}\times\widetilde{\sigma}_{2})}\sum_{j\in J}\int_{N'(\mathbb{A}_F)\backslash G'(\mathbb{A}_F)}W_j'(x)W_{2,-\lambda}(x)\\
	&|\det x|^{s+\frac{1}{2}}dx\int_{G'(\mathbb{A}_F)}\widehat{\Phi}_2(y,\eta x)\langle \pi'(y)W', W_j'\rangle \sigma_{1,\lambda}(\det y)|\det y|^{s+\frac{n}{2}}dy.
\end{align*}

The first integral is a Rankin-Selberg convolution for $L(s+1/2,\pi'\times\sigma_{2,-\lambda})=L(s+1/2-\lambda/n,\pi'\times\sigma_2),$ and the second one is an Godement-Jacquet integral for $L(s+1/2,\pi'\times\sigma_{1,\lambda})=L(s+1/2+\lambda,\pi'\times\sigma_{1}),$ which is entire. 

By Rankin-Selberg theory, \eqref{165} converges absolutely when $\Re(s)-\Re(\lambda)/n>1/2,$ and admits a meromorphic continuation to $\lambda\in\mathbb{C}$ with possible (simple) poles at $\lambda=n(s-1/2)$ and $\lambda=n(s+1/2)$ when $\pi'\simeq\widetilde{\sigma}_2.$ 

We now proceed to derive the meromorphic continuation of $\Psi_{\lambda}(s,\phi^*,\phi').$ Define
\begin{equation}\label{187}
\widehat{\Phi}_2^{\sharp}(y,\eta g'):=\int_{\mathbb{A}_F^n}\Phi(y,\mathfrak{u})\overline{\theta}(\eta g'\mathfrak{u})d\mathfrak{u}
\end{equation}
for all $g'\in G'(\mathbb{A}_F).$ Define the Eisenstein series 
\begin{align*}
E(x,\widehat{\Phi}_2(y,\cdot),\lambda):=\sum_{\delta \in P'(F)\backslash G'(F)}\int_{Z'(\mathbb{A}_F)}\widehat{\Phi}_2^{\sharp}(y,\eta\delta zx)|\det zx|^{s+\frac{1}{2}-\frac{\lambda}{n}}d^{\times}z,
\end{align*}
which converges absolutely when $\Re(s)-\Re(\lambda)/n>1/2.$ Moreover, by Poisson summation $E(x,\widehat{\Phi}_2(y,\cdot),\lambda)$ admits a meromorphic continuation to $(s,\lambda)\in\mathbb{C}^2$ with at most simple poles on the line $s-\lambda/n=1/2$ and $s-\lambda/n=-1/2.$ Let $\Re(s)-\Re(\lambda)/n>1/2.$ Then $L(1+(1+n^{-1})\lambda,\sigma_{1}\times\widetilde{\sigma}_{2})\Psi_{\lambda}(s,\phi^*,\phi')$ is equal to 
\begin{align*}
\int_{G'(\mathbb{A}_F)}\sigma_{1,\lambda}(\det y)|\det y|^{s+\frac{n}{2}}\bigg[\int_{[\overline{G'}]}\phi_{2}(x)(\pi'(y)\phi')(x)E(x,\widehat{\Phi}_2(y,\cdot),\lambda)dx\bigg]dy.
\end{align*}

Hence $\int_{[\overline{G'}]}\phi_{2}(x)(\pi'(y)\phi')(x)E(x,\widehat{\Phi}_2(y,\cdot),\lambda)dx$ is equal to 
\begin{align*}
&\int_{[G']}\phi_{2}(x)(\pi'(y)\phi')(x)\sum_{\delta \in P'(F)\backslash G'(F)}\int_{\mathbb{A}_F^n}\Phi(y,\mathfrak{u})\overline{\theta}(\eta \delta x\mathfrak{u})d\mathfrak{u}|\det x|^{s+\frac{1}{2}-\frac{\lambda}{n}}dx\\
=&\int_{[G']}\phi_{2}(x)(\pi'(y)\phi')(x)\sum_{\delta \in P'(F)\backslash G'(F)}\int_{\mathbb{A}_F^n}\Phi(y,x^{-1}\mathfrak{u})\overline{\theta}(\eta \delta \mathfrak{u})d\mathfrak{u}|\det x|^{s-\frac{1}{2}-\frac{\lambda}{n}}dx.
\end{align*}

In conjunction with  $\Phi(y,\mathfrak{u})=\Phi[(y,\mathbf{0})\iota(y^{-1})u\iota(y)]=\Phi[(\mathbf{0},y)w\iota(y^{-1})u\iota(y)]$ and a change of variables $x\mapsto xy^{-1}$ and $y\mapsto yx,$ the above integral becomes 
\begin{align*}
&\int_{[G']}\int_{G'(\mathbb{A}_F)}\sigma_{1,\lambda}(\det yx)|\det yx|^{s+\frac{n}{2}}\phi_{2}(y^{-1})\phi'(x)\\
&\sum_{\delta \in P'(F)\backslash G'(F)}\int_{N_P(\mathbb{A}_F)}\Phi[(\mathbf{0},y)wu\iota(x)]\overline{\theta}(\eta \delta \mathfrak{u})d\mathfrak{u}|\det y|^{-s+\frac{1}{2}+\frac{\lambda}{n}}dydx,
\end{align*}
which, in terms of $\phi^*_{\Phi,\phi_2}(g,{\lambda})$ defined by \eqref{163}, turns out to be 
\begin{align*}
\int_{[G']}\phi'(x)
\sum_{\delta \in P'(F)\backslash G'(F)}\int_{N_P(\mathbb{A}_F)}\phi_{\Phi,\phi_2}^*(w\iota(x)u,\lambda)\overline{\theta}(\eta \delta x\mathfrak{u})d\mathfrak{u}|\det x|^{s+1}dx.
\end{align*}

Let $\phi\in\mathfrak{B}_{Q,\chi}.$ Making use of the decomposition $\phi=\sum F(\lambda)\phi^*_{\Phi,\phi_2}(\cdot,\lambda)$ gives
\begin{align*}
\Psi_{\lambda}(s,\phi,\phi')=\int_{[G']}\phi'(x)\sum_{\delta \in P'(F)\backslash G'(F)}\int_{N_P(\mathbb{A}_F)}\phi(w\iota(x)u)\overline{\theta}(\eta \delta x\mathfrak{u})d\mathfrak{u}|\det x|^{s+1-\frac{\lambda}{n}}dx.
\end{align*}

Note that $\phi(w\iota(x)u)=|\det x|^{-\frac{1}{2}}(\sigma_2(x)\phi)(wu).$ Then $\Psi_{\lambda}(s,\phi,\phi')$ becomes
\begin{equation}\label{171..}
\int_{[G']}\phi'(x)\sigma_2(x)\sum_{\delta \in P'(F)\backslash G'(F)}\int_{N_P(\mathbb{A}_F)}\phi(wu)\overline{\theta}(\eta \delta x\mathfrak{u})d\mathfrak{u}|\det x|^{s+\frac{1}{2}-\frac{\lambda}{n}}dx.
\end{equation}

So \eqref{166.} follows from \eqref{171..}.
\end{proof}

\begin{cor}\label{prop39}
	Let notation be as before. Let $\chi=\{(M_Q,\sigma)\}$ with $\pi'\simeq\widetilde{\sigma}_2.$ Let $\phi\in\mathfrak{B}_{R,\chi}.$ Denote by $w=\begin{pmatrix}
		&1\\
		I_{n}&
	\end{pmatrix}.$ Then 
	\begin{equation}\label{166}
		\underset{\lambda=n(s-1/2)}{\Res}\Psi_{\lambda}(s,\phi,\phi')=-\int_{[\overline{G'}]}\phi'(x)\phi(w\iota(x))|\det x|^{\frac{1}{2}}dx.
	\end{equation}
\end{cor}

\begin{cor}\label{cor40}
Let notation be as before. Let $\chi=\{(M_Q,\sigma)\}$ with $\pi'\simeq\widetilde{\sigma}_2.$ Let $\phi\in\mathfrak{B}_{R,\chi}.$ Then $\underset{\lambda=n(s-1/2)}{\Res}\Psi_{\lambda}(s,\mathcal{I}(\lambda,f)R(w^{-1})\phi,\phi')$ is equal to 
\begin{align*}
-\int_{[\overline{G'}]}\phi'(x)\int_{\mathbb{A}_F^n}\int_{G'(\mathbb{A}_F)}\int_{K}f(u(\mathbf{x})\iota(y)k)\phi(w\iota(xy)kw^{-1})|\det x|^{\frac{1}{2}}|\det y|^{\frac{3}{2}-s}dkdydx.
\end{align*}
In particular, $\underset{\lambda=n(s-1/2)}{\Res}\Psi_{\lambda}(s,\mathcal{I}(\lambda,f)R(w^{-1})\phi,\phi')$ is an entire function of $s.$
\end{cor}
\begin{proof}
By \eqref{172} and Corollary \ref{prop39} $\underset{\lambda=n(s-1/2)}{\Res}\Psi_{\lambda}(s,\mathcal{I}(\lambda,f)R(w^{-1})\phi,\phi')$ is equal to 
\begin{equation}\label{173}
-\int_{\overline{G}(\mathbb{A}_F)}f(y)\int_{[\overline{G'}]}\phi'(x)\phi(w\iota(x)yw^{-1})e^{\langle\boldsymbol{\lambda}, H_Q(w\iota(x)yw^{-1})\rangle}|\det x|^{s}dxdy.
\end{equation}	

Let $wyw^{-1}=umk$ be the Iwasawa decomposition such that $um\in \overline{Q}(\mathbb{A}_F).$ Identify the Levi component of $\overline{Q}$ with $\diag(1,G').$ So 
$$
H_Q(w\iota(x)yw^{-1})=H_Q(w\iota(x)w^{-1}um)=H_Q(w\iota(x)w^{-1}m)=|\det xm|^{\frac{1}{n}}=|\det xy|^{\frac{1}{n}}.
$$ 
Thus \eqref{173} becomes, after a further change of variable $y\mapsto w^{-1}yw,$ that 
\begin{equation}\label{174}
-\int_{[\overline{G'}]}\phi'(x)\int_{\overline{G}(\mathbb{A}_F)}f(w^{-1}yw)\phi(w\iota(x)w^{-1}y)|\det x|^{\frac{1}{2}}|\det y|^{\frac{3}{2}-s}dydx.
\end{equation} 
Since $f$ is compact supported on $\overline{G}(\mathbb{A}_F),$ then \eqref{174} converges absolutely for all $s\in\mathbb{C},$ defining an entire function of $s.$ So Corollary \ref{cor40} follows.
\end{proof}

With the above preparation we can prove Theorem \ref{thm49} now. 
\begin{proof}[Proof of Theorem \ref{thm49}]
Note that one can replace $\phi$ with $R(w^{-1})\phi$ in the sum
$$
\sum_{\phi\in\mathfrak{B}_{Q,\chi}}\Psi_{\lambda}(s_1,\mathcal{I}(\lambda,f)\phi,\phi_1')\overline{\Psi_{\lambda}(\overline{s_2},\phi,\phi_2')}.
$$
So $\mathcal{G}_{\chi}(s_1,s_2,\phi_1',\phi_2')$ is equal to 
\begin{align*}
	\sum_{\phi\in\mathfrak{B}_{Q,\chi}}\overline{\Psi_{n(s_1-1/2)}(\overline{s_2},R(w^{-1})\phi,\phi_2')}\underset{\lambda=n(s_1-1/2)}{\Res}\Psi_{\lambda}(s_1,\mathcal{I}(\lambda,f)R(w^{-1})\phi,\phi_1').
\end{align*}
By Proposition \ref{prop46} and Corollary \ref{cor40}, $\mathcal{G}_{\chi}(s_1,s_2,\phi_1',\phi_2')$ is
\begin{align*}
&-\sum_{\phi\in\mathfrak{B}_{Q,\chi}}\int_{[\overline{G'}]}\int_{N_Q(\mathbb{A}_F)}\int_{G'(\mathbb{A}_F)}\int_{K}f\left(w^{-1}u'\begin{pmatrix}
	1\\
	&y
\end{pmatrix}kw\right)\\
&\quad \phi\left(\begin{pmatrix}
	1\\
	&x
\end{pmatrix}u'\begin{pmatrix}
	1\\
	&y
\end{pmatrix}k\right)\phi_1'(x)|\det x|^{\frac{1}{2}}|\det y|^{\frac{3}{2}-s_1}dkdu'dydx\\
&\int_{[G']}\overline{\phi_2'(h)}\sum_{\delta \in P'(F)\backslash G'(F)}\int_{N_P(\mathbb{A}_F)}\overline{\phi(w\iota(h)uw^{-1})}\theta(\eta \delta h\mathfrak{u})du|\det h|^{s_2+s_1+\frac{1}{2}}dh.
\end{align*}

For $g\in{G}(\mathbb{A}_F),$ we define $\mathfrak{F}(g;s_1)$ by 
\begin{equation}\label{189}
\int_{N_Q(\mathbb{A}_F)}\int_{G'(\mathbb{A}_F)}f\left(w^{-1}u'\begin{pmatrix}
	1\\
	&y
\end{pmatrix}gw\right)\pi'(y^{-1})\phi_1'(I_n)|\det y|^{1-s_1}dydu'.
\end{equation}
Write $g=zu''\begin{pmatrix}
	1&\\
	&x
\end{pmatrix}k\in{G}(\mathbb{A}_F),$ where $z\in Z(\mathbb{A}_F),$ $u''\in N_Q(\mathbb{A}_F).$ Then 
$$
\mathfrak{F}(g;s_1)=\omega^{-1}(z)|\det x|^{s_1-1}\pi'(x)\mathfrak{F}(k;s_1),
$$ 
indicating that $\mathfrak{F}(\cdot;s_1)\in \Pi_{s_1}:=\Int_{Q(\mathbb{A}_F)}^{G(\mathbb{A}_F)}\widetilde{\sigma}\otimes e^{\langle \left(n(1/2-s_1), n(s_1-1/2)\right),H_Q(\cdot)\rangle},$ where $\sigma=\sigma_1\otimes\sigma_2$ with $\sigma_1=\omega\omega'^{-1}$ and $\sigma_2=\widetilde{\pi}'.$ Denote by 
$$
\phi(g; s_1):=\phi\left(zu'\begin{pmatrix}
1\\
&x
\end{pmatrix}k\right)|\det x|^{3/2-s_1}.
$$ 
Then $\phi(g; s_1)=\omega(z)|\det x|^{1-s_1}(\textbf{1}\otimes\widetilde{\pi}'(x))\phi(k; 1-s_1).$ So $\phi(\cdot;s_1)\in \widetilde{\Pi}_{s_1},$ the contragredient of the representation $\Pi_{s_1}.$

With the above notation the function $-\mathcal{G}_{\chi}(s_1,s_2,\phi_1',\phi_2')$ becomes 
\begin{align*}
\sum_{\phi}\langle\langle \mathfrak{F}(\cdot;s_1),\phi(\cdot;s_1)\rangle\rangle\int_{[G']}\overline{\phi_2'(h)}\sum_{\delta}\int\overline{\phi(w\iota(h)uw^{-1})}\theta(\eta \delta h\mathfrak{u})du|\det h|^{s_2+s_1+\frac{1}{2}}dh,
\end{align*}
where $\phi\in\mathfrak{B}_{Q,\chi},$ $\delta$ (resp. $u$) ranges through $P'(F)\backslash G'(F)$ (resp. $N_P(\mathbb{A}_F)$), and   
\begin{align*}
\langle\langle \mathfrak{F}(\cdot;s_1),\phi(\cdot;s_1)\rangle\rangle:=\int_{[\overline{G'}]}\int_{K}\mathfrak{F}\left(\begin{pmatrix}
	1\\
	&x
\end{pmatrix}k;s_1\right)\phi\left(\begin{pmatrix}
	1\\
	&x
\end{pmatrix}k;s_1\right)dkdx
\end{align*}
is the pair for $\Pi_{s_1}$ and $\widetilde{\Pi}_{s_1}.$ By orthogonality we have  
\begin{equation}\label{188}
\sum_{\phi\in\mathfrak{B}_{Q,\chi}}\langle\langle \mathfrak{F}(\cdot;s_1),\phi(\cdot;s_1)\rangle\rangle\overline{\phi\left(\begin{pmatrix}
		1\\
		&h
	\end{pmatrix}\right)}|\det h|^{-\frac{1}{2}+s_1}=\mathfrak{F}\left(\begin{pmatrix}
1\\
&h
\end{pmatrix};s_1\right).
\end{equation}

As a consequence of \eqref{188} the function $\mathcal{G}_{\chi}(s_1,s_2,\phi_1',\phi_2')$ becomes 
\begin{align*}
& -\int_{[G']}\overline{\phi_2'(h)}\sum_{\delta \in P'(F)\backslash G'(F)}\int_{N_P(\mathbb{A}_F)}\mathfrak{F}(w\iota(h)uw^{-1};s_1)
\theta(\eta \delta h\mathfrak{u})d\mathfrak{u}|\det h|^{s_2+1}dh.
\end{align*}

Expand $\mathfrak{F}(w\iota(h)uw^{-1};s_1)$ according to \eqref{189} to write  $\mathcal{G}_{\chi}(s_1,s_2,\phi_1',\phi_2')$ as
\begin{align*}
&-\int_{[G']}\overline{\phi_2'(h)}\sum_{\delta \in P'(F)\backslash G'(F)}\int_{N_P(\mathbb{A}_F)}
\theta(\eta \delta h\mathfrak{u})du|\det h|^{s_2+1}dh\\
&\int_{N_Q(\mathbb{A}_F)}\int_{G'(\mathbb{A}_F)}f\left(w^{-1}u'\begin{pmatrix}
	1\\
	&yh
\end{pmatrix}wu\right)\pi'(y^{-1})\phi_1'(I_n)|\det y|^{1-s_1}dydu'.
\end{align*}

Change the variable $y\mapsto yh^{-1}$ and $h\mapsto hy$ to transform the above integral into 
\begin{align*}
& -\int_{[G']}\overline{\phi_2'(hy)}\sum_{\delta \in P'(F)\backslash G'(F)}\int_{N_P(\mathbb{A}_F)}
\theta(\eta \delta hy\mathfrak{u})du|\det hy|^{s_1+s_2}dh\\
&\int_{N_Q(\mathbb{A}_F)}\int_{G'(\mathbb{A}_F)}f\left(w^{-1}u'\begin{pmatrix}
	1\\
	&y
\end{pmatrix}wu\right)\phi_1'(h)|\det y|^{1-s_1}dydu',
\end{align*}
which becomes, after a further change of variable $\mathfrak{u}\mapsto y^{-1}\mathfrak{u},$ that 
\begin{align*}
& -\int_{[G']}\overline{\phi_2'(hy)}\sum_{\delta \in P'(F)\backslash G'(F)}\int_{N_P(\mathbb{A}_F)}
\theta(\eta \delta h\mathfrak{u})du|\det h|^{s_1+s_2}dh\\
&\int_{N_Q(\mathbb{A}_F)}\int_{G'(\mathbb{A}_F)}f\left(w^{-1}u'wu\iota(y)\right)\phi_1'(h)|\det y|^{s_2}dydu'.
\end{align*}

After an unfolding we then see the above integral becomes \eqref{194}. Moreover, we have the functional equation  $E_P^{\dagger}(x,s;\widehat{f}_P,y)=E_{\gamma}(x,1-s;f_P^{\dag}(\cdot;z_2,y)),$ which is defined by \eqref{274} in \text\textsection \ref{sec5.2}.

Apply Poisson summation to the Eisenstein series $E_P^{\dagger}(x,s_1+s_2;\widehat{f}_P,y)$ and follow the proof of Propositions \ref{prop22}, \ref{prop19} and Lemma \ref{prop21.} we then see that $\mathcal{G}_{\chi}(s_1,s_2,\phi_1',\phi_2')$ admits a meromorphic continuation to $\mathcal{R}=\big\{\mathbf{s}=(s_1,s_2)\in\mathbb{C}^2:\  \Re(s_2)>-1/(n+1)\big\}.$ 
The analytic behaviors of $\widetilde{\mathcal{G}}_{\chi}(s_1,s_2,\phi_1',\phi_2')$ can be proved  similarly.
\end{proof}

\subsection{Singularity Matching of Geometric and Spectral Sides}\label{sec7.2}
By Proposition \ref{prop11'}, Lemma \ref{prop21}, and Theorems \ref{thm22}, \ref{thm24} and \ref{Red}, the geometric side 
\begin{align*}
J_{\Geo}^{\Reg}(f,\textbf{s})=J^{\Reg}_{\Geo,\sm}(f,\textbf{s})-&\mathcal{F}_{1,0}J^{\bi}_{\Geo}(f,\textbf{s})-\mathcal{F}_{0,1}J^{\bi}_{\Geo}(f,\textbf{s})\\
&\quad +J_{\Geo,\du}^{\bi}(f,\textbf{s})+J^{\Reg,\RNum{2}}_{\Geo,\bi}(f,\textbf{s})
\end{align*}
converges absolutely in $\Re(s_1)+\Re(s_2)>1$ and admits a meromorphic continuation to the region $\mathcal{R}$ with at most simple poles at $s_1+s_2\in \{0,1\},$ where 
\begin{equation}\label{271}
\mathcal{R}=\big\{\mathbf{s}=(s_1,s_2)\in\mathbb{C}^2:\ \Re(s_1),\ \Re(s_2)>-1/(n+1)\big\}.
\end{equation} 

By Theorem \ref{thm40} the spectral side 
\begin{equation}\label{186}
J_{\Spec}^{\Reg}(f,\mathbf{s})=J_{0}(f,\mathbf{s})+J_{\Eis,\reg}^{\Reg}(f,\mathbf{s})+J_{\Eis,\semi}^{\Reg}(f,\mathbf{s})+J_{\Eis,\sing}^{\Reg}(f,\mathbf{s})
\end{equation}
converges absolutely in $\Re(s_1), \Re(s_2)\gg 1$ and admits a meromorphic continuation to the region $\mathcal{R}$ with at most simple poles at $s_1+s_2\in \{0,1\}.$ In particular, only $J_{\Eis,\sing}^{\Reg}(f,\mathbf{s})$ on the RHS may be meromorphic, and others are holomorphic.

Note that the meromorphic terms $J^{\Reg}_{\Geo,\sm}(f,\textbf{s}),$ $\mathcal{F}_{1,0}J^{\bi}_{\Geo}(f,\textbf{s}),$ $\mathcal{F}_{0,1}J^{\bi}_{\Geo}(f,\textbf{s}),$ $J_{\Geo,\du}^{\bi}(f,\textbf{s})$ and $J_{\Eis,\sing}^{\Reg}(f,\mathbf{s})$ can be represented by integrals involving an Eisenstein series  $E(s,x,y;\Phi,\omega'')$ of the form \eqref{277} (e.g., $E_{\gamma}(x,1-s;f_P^{\dag}(\cdot;z_2,y))$ defined by \eqref{274} in \text\textsection \ref{sec5.2}, or $E_P^{\dagger}(x,s;\widehat{f}_P,y)$ defined by \eqref{276} in \text\textsection \ref{sec10.}), parametrized by $y$ continuously, and are of the form 
\begin{equation}\label{279}
J(s)=\iint h(x,y)E(s,x,y;\Phi,\omega'')dxdy
\end{equation}
for some continuous function $h.$ Following the decomposition \eqref{278} we can decompose $E(s,x,y;\Phi,\omega'')$ into three parts and thereby obtain a decomposition of $J(s):$ 
\begin{equation}\label{280}
J(s)=J_+(s)+J_+^{\wedge}(x,s)+J_{\Res}(s),
\end{equation}
where $J_+(s)$ and $J_+^{\wedge}(x,s)$ are holomorphic, and $J_{\Res}(s)$ may be just meromorphic. 

Define the meromorphic part $J_{\Geo,\Res}^{\Reg}(f,\textbf{s})$ of $J_{\Geo}^{\Reg}(f,\textbf{s})$ to be 
\begin{align*}
J^{\Reg}_{\Geo,\sm,\Res}(f,\textbf{s})-\mathcal{F}_{1,0}J^{\bi}_{\Geo,\Res}(f,\textbf{s})-\mathcal{F}_{0,1}J^{\bi}_{\Geo,\Res}(f,\textbf{s})+J_{\Geo,\gamma(0),\Res}^{\bi}(f,\textbf{s}).
\end{align*}
Denote by $J_{\Eis,\sing,\Res}^{\Reg}(f,\mathbf{s})$ the meromorphic part of $J_{\Spec}^{\Reg}(f,\mathbf{s}).$

We now compute $J_{\Geo,\Res}^{\Reg}(f,\textbf{s})$ and $J_{\Eis,\sing,\Res}^{\Reg}(f,\mathbf{s}),$ and show they match with each other. The main result in this section is the following.
\begin{prop}\label{prop45}
Let notation be as above. Then $J_{\Geo,\Res}^{\Reg}(f,\textbf{s})-J_{\Eis,\sing,\Res}^{\Reg}(f,\mathbf{s})$ is holomorphic in $\mathcal{R}$. In particular, in the region $\mathbf{s}\in \big\{(s_1,s_2)\in\mathcal{R}:\ \Re(s_1)<1/2,\ \Re(s_2)<1/2\big\},$ the meromorphic function 
\begin{align*}
&J^{\Reg}_{\Geo,\sm,\Res}(f,\textbf{s})+J_{\Geo,\du,\Res}^{\bi}(f,\textbf{s})-\mathcal{F}_{0,1}J_{\Geo,\Res}^{\bi}(f,\textbf{s})-\mathcal{F}_{1,0}J_{\Geo,\Res}^{\bi}(f,\textbf{s})\\
&+\mathcal{G}_{\chi,\Res}(s_1,s_2,\phi_1',\phi_2')-\widetilde{\mathcal{G}}_{\chi,\Res}(s_1,s_2,\phi_1',\phi_2')
\end{align*}
is equal to 
\begin{align*}
-J^{\Reg,0}_{\Geo,\sm,\Res}(f,\textbf{s})-J_{\Geo,\du,\Res}^{\bi,1}(f,\textbf{s}),
\end{align*}
which is holomorphic in the region $\Re(s_1+s_2)<1.$ Here $J_{\Geo,\du,\Res}^{\bi,1}(f,\textbf{s})$ is  
\begin{align*}
\frac{1}{n(s_1+s_2-1)}\int_{G'(\mathbb{A}_F)}\int_{\mathbb{A}_F^n}
	f(u(\textbf{x})\iota(y))\int_{[\overline{G'}]}\phi_1'(x)\overline{\phi_2'(xy)}dx|\det y|^{s_2}d\textbf{x}dy,
\end{align*}
and  $J^{\Reg,0}_{\Geo,\sm,\Res}(f,\textbf{s})$ is equal to 
\begin{align*}
-\frac{1}{n(s_1+s_2+1)}\int_{G'(\mathbb{A}_F)}\int_{N_P(\mathbb{A}_F)}
	f(u\iota(y))\int_{[\overline{G'}]}\phi_1'(x)\overline{\phi_2'(xy)}dx|\det y|^{s_2}dudy.
\end{align*}
\end{prop}
The proof of this Proposition will be given in \text\textsection \ref{7.2.6} below.

\subsubsection{Meromorphic Part of $J_{\Geo,\du}^{\bi}(f,\textbf{s})$}\label{7.2.1}
By Proposition \ref{prop18'}, $J_{\Geo,\du}^{\bi}(f,\textbf{s})$ is
\begin{align*}
\iint\int_{[\overline{G'}]}\phi_1'(x)\overline{\phi_2'(xy)}E_{\du}(x,s_1+s_2;f^{\dag}(\cdot,z_2,y))dx\overline{\omega}'(z_2)|\det z_2y|^{s_2}dyd^{\times}z_2,
\end{align*}
where $z_2$ (resp. $y$) ranges over $Z'(\mathbb{A}_F)$ (resp. $\overline{G'}(\mathbb{A}_F)$), and the Eisenstein series $E_{\du}(x,s_1+s_2;f^{\dag}(\cdot,z_2,y))$ is defined in \text\textsection \ref{sec9.2}. 

Execute Poisson summation to $E_{\du}(x,s_1+s_2;f^{\dag}(\cdot,z_2,y))$ implies that 
\begin{equation}\label{160}
J_{\Geo,\du,\Res}^{\bi}(f,\textbf{s})=J_{\Geo,\du,\Res}^{\bi,0}(f,\textbf{s})+J_{\Geo,\du,\Res}^{\bi,1}(f,\textbf{s}),
\end{equation} 
where $J_{\Geo,\du,\Res}^{\bi,1}(f,\textbf{s})$ is equal to 
\begin{align*}
\frac{1}{n(s_1+s_2-1)}\int_{G'(\mathbb{A}_F)}\int_{\mathbb{A}_F^n}
	f(u(\textbf{x})\iota(y))\int_{[\overline{G'}]}\phi_1'(x)\overline{\phi_2'(xy)}dx|\det y|^{s_2}d\textbf{x}dy,
\end{align*}
and $J_{\Geo,\du,\Res}^{\bi,0}(f,\textbf{s})$ is equal to 
\begin{equation}\label{157}
-\frac{1}{n(s_1+s_2)}\int_{{G'}(\mathbb{A}_F)}
	f(\iota(y))\int_{[\overline{G'}]}\phi_1'(x)\overline{\phi_2'(xy)}dx|\det y|^{s_2}dy.
\end{equation}

\subsubsection{Meromorphic Part of $J^{\Reg}_{\Geo,\sm}(f,\textbf{s})$}\label{7.2.2}
Fix $s_2\in\mathbb{C}.$ By Proposition \ref{prop14} and \eqref{34} the function $J^{\Reg}_{\Geo,\sm}(f,\textbf{s})$ has a simple pole at $s_1=-s_2,$ inherited from the Eisenstein series $E(x,s_1+s_2+1;\check{f}_P(\cdot;z_2,y)).$ By Poisson summation, $E(x,s_1+s_2+1;\check{f}_P(\cdot;z_2,y))$ has residue $n^{-1}\omega'(z_1)\overline{\omega_2(z_2)}|\det z_2y|^{s_2}\widehat{	\check{f}_P}(\mathbf{0};z_2,y).$ Note that $\check{f}_P$ is inverse Fourier transform and thus the Fourier transform $\widehat{\check{f}_P}$ gives back to $f.$ Hence, similar to the argument in \text\textsection \ref{7.2.1} we derive
\begin{equation}\label{158}
J^{\Reg}_{\Geo,\sm,\Res}(f,\textbf{s})=J^{\Reg,0}_{\Geo,\sm,\Res}(f,\textbf{s})+J^{\Reg,1}_{\Geo,\sm,\Res}(f,\textbf{s}),
\end{equation}
where $J^{\Reg,1}_{\Geo,\sm,\Res}(f,\textbf{s})$ is equal to 
\begin{equation}\label{158.}
\frac{1}{n(s_1+s_2)}\int_{{G'}(\mathbb{A}_F)}
	f(\iota(y))\int_{[\overline{G'}]}\phi_1'(x)\overline{\phi_2'(xy)}dx|\det y|^{s_2}dy,
\end{equation}
and $J^{\Reg,0}_{\Geo,\sm,\Res}(f,\textbf{s})$ is equal to 
\begin{align*}
-\frac{1}{n(s_1+s_2+1)}\int_{G'(\mathbb{A}_F)}\int_{N_P(\mathbb{A}_F)}
	f(u\iota(y))\int_{[\overline{G'}]}\phi_1'(x)\overline{\phi_2'(xy)}dx|\det y|^{s_2}dudy,
\end{align*}

As a consequence of \eqref{157} with \eqref{158.} we conclude the following cancellation.
\begin{lemma}\label{185}
Let notation be as before. Then 
\begin{align*}\label{159}
J_{\Geo,\du,\Res}^{\bi,0}(f,\textbf{s})+J^{\Reg,1}_{\Geo,\sm,\Res}(f,\textbf{s})=0.
\end{align*}
\end{lemma}
\subsubsection{Meromorphic Part of $\mathcal{F}_{0,1}J^{\bi}_{\Geo}(f,\textbf{s})$}\label{8.2.3}
Denote by $\mathcal{F}_{0,1}J^{\bi}_{\Geo,\Res}(f,\textbf{s})$ the meromorphic part of  $\mathcal{F}_{0,1}J^{\bi}_{\Geo}(f,\textbf{s})$, inherited from the Eisenstein series $E_{\gamma}(x,s_1+s_2+1;f_P^{\dag}(\cdot;z_2,y)).$ According to Proposition \ref{prop22}, we have 
\begin{equation}\label{27}
\mathcal{F}_{0,1}J_{\Geo,\Res}^{\bi}(f,\textbf{s})=\mathcal{F}_{0,1}J^{\bi,0}_{\Geo,\Res}(f,\textbf{s})+\mathcal{F}_{0,1}J^{\bi,1}_{\Geo,\Res}(f,\textbf{s}),
\end{equation}
where $\mathcal{F}_{0,1}J^{\bi,1}_{\Geo,\Res}(f,\textbf{s})$ is equal to
\begin{align*}
\frac{1}{n(s_1+s_2)}\int_{G'(\mathbb{A}_F)}\int_{\mathbb{A}_F^n}\int_{N_P(\mathbb{A}_F)}f(u(\mathbf{x})n\iota(y))\int_{[\overline{G'}]}\phi_1'(x)\overline{\phi_2'(xy)}dx|\det y|^{s_2}dnd\mathbf{x}dy,
\end{align*} 
and  $\mathcal{F}_{0,1}J^{\bi,0}_{\Geo,\Res}(f,\textbf{s})$ is equal to
\begin{align*}
-\frac{1}{n(s_1+s_2+1)}\int_{G'(\mathbb{A}_F)}\int_{N_P(\mathbb{A}_F)}f(n\iota(y))\int_{[\overline{G'}]}\phi_1'(x)\overline{\phi_2'(xy)}dx|\det y|^{s_2}dndy.
\end{align*}

\subsubsection{Meromorphic Part of $\mathcal{F}_{1,0}J^{\bi}_{\Geo}(f,\textbf{s})$}\label{7.2.4}
Denote by $\mathcal{F}_{1,0}J^{\bi}_{\Geo,\Res}(f,\textbf{s})$ the meromorphic part of  $\mathcal{F}_{1,0}J^{\bi}_{\Geo}(f,\textbf{s})$, inherited from the Eisenstein series $E_{\gamma}(x,s_1+s_2+1;{^{\dag}f_P}(\cdot;z_2,y)).$ According to Proposition \ref{prop23}, we have 
\begin{equation}\label{28}
\mathcal{F}_{1,0}J_{\Geo,\Res}^{\bi}(f,\textbf{s})=\mathcal{F}_{1,0}J^{\bi,0}_{\Geo,\Res}(f,\textbf{s})+\mathcal{F}_{1,0}J^{\bi,1}_{\Geo,\Res}(f,\textbf{s}),
\end{equation}
where $\mathcal{F}_{1,0}J^{\bi,1}_{\Geo,\Res}(f,\textbf{s})$ is equal to 
\begin{align*}
\frac{1}{n(s_1+s_2)}\int_{G'(\mathbb{A}_F)}\int_{\mathbb{A}_F^n}\int_{N_P(\mathbb{A}_F)}f(nu(\mathbf{x})\iota(y))\int_{[\overline{G'}]}\phi_1'(x)\overline{\phi_2'(xy)}dx|\det y|^{s_2}dnd\mathbf{x}dy,
\end{align*} 
and  $\mathcal{F}_{1,0}J^{\bi,0}_{\Geo,\Res}(f,\textbf{s})$ is equal to 
\begin{align*}
-\frac{1}{n(s_1+s_2+1)}\int_{G'(\mathbb{A}_F)}\int_{N_P(\mathbb{A}_F)}f(n\iota(y))\int_{[\overline{G'}]}\phi_1'(x)\overline{\phi_2'(xy)}dx|\det y|^{s_2}dndy.
\end{align*} 

\begin{lemma}\label{lem47}
Let notation be as before. Then 
\begin{align*}
J^{\Reg,0}_{\Geo,\sm,\Res}(f,\textbf{s})=\mathcal{F}_{0,1}J^{\bi,0}_{\Geo,\Res}(f,\textbf{s})=\mathcal{F}_{1,0}J^{\bi,0}_{\Geo,\Res}(f,\textbf{s})
\end{align*}
is holomorphic in the region $\Re(s_1+s_2)>-1.$
\end{lemma}

\subsubsection{Meromorphic Part of $J_{\Eis,\sing}^{\Reg}(f,\mathbf{s})$}\label{7.2.5}
Denote by $\mathcal{G}_{\chi,\Res}(s_1,s_2,\phi_1',\phi_2')$ the meromorphic part of $\mathcal{G}_{\chi}(s_1,s_2,\phi_1',\phi_2').$ Applying Poisson summation to the Eisenstein series $E_P^{\dagger}(x,s_1+s_2;\widehat{f}_P,y)$ defined in Theorem \ref{thm49}, similar to the calculation in \text\textsection \ref{8.2.3}, we have
\begin{equation}\label{168}
\mathcal{G}_{\chi,\Res}(s_1,s_2,\phi_1',\phi_2')=\mathcal{G}_{\chi,\Res}^1(s_1,s_2,\phi_1',\phi_2')+\mathcal{G}_{\chi,\Res}^0(s_1,s_2,\phi_1',\phi_2'),
\end{equation}
where $\mathcal{G}_{\chi,\Res}^1(s_1,s_2,\phi_1',\phi_2')$ is defined by 
\begin{align*}
-\frac{1}{n(s_1+s_2-1)}\int_{{G'}(\mathbb{A}_F)}\int_{\mathbb{A}_F^n}
	f(u(\textbf{x})\iota(y))\int_{[\overline{G'}]}\phi_1'(x)\overline{\phi_2'(xy)}dx|\det y|^{s_2}d\textbf{x}dy,
\end{align*}
and $\mathcal{G}_{\chi,\Res}^0(s_1,s_2,\phi_1',\phi_2')$ is defined by 
\begin{align*}
\frac{1}{n(s_1+s_2)}\int_{G'(\mathbb{A}_F)}\int_{\mathbb{A}_F^n}\int_{N_P(\mathbb{A}_F)}f(u(\mathbf{x})n\iota(y))\int_{[\overline{G'}]}\phi_1'(x)\overline{\phi_2'(xy)}dx|\det y|^{s_2}dnd\mathbf{x}dy.
\end{align*}

Let $\widetilde{\mathcal{G}}_{\chi,\Res}(s_1,s_2,\phi_1',\phi_2')$ be the meromorphic part of $\widetilde{\mathcal{G}}_{\chi}(s_1,s_2,\phi_1',\phi_2').$ Then
\begin{equation}\label{26}
\widetilde{\mathcal{G}}_{\chi,\Res}(s_1,s_2,\phi_1',\phi_2')=\widetilde{\mathcal{G}}_{\chi,\Res}^1(s_1,s_2,\phi_1',\phi_2')+\widetilde{\mathcal{G}}_{\chi,\Res}^0(s_1,s_2,\phi_1',\phi_2'),
\end{equation}
where $\widetilde{\mathcal{G}}_{\chi,\Res}^1(s_1,s_2,\phi_1',\phi_2')$ is defined by 
\begin{align*}
\frac{1}{n(s_1+s_2-1)}\int_{{G'}(\mathbb{A}_F)}\int_{\mathbb{A}_F^n}
	f(u(\textbf{x})\iota(y))\int_{[\overline{G'}]}\phi_1'(x)\overline{\phi_2'(xy)}dx|\det y|^{s_2}d\textbf{x}dy,
\end{align*}
and $\widetilde{\mathcal{G}}_{\chi,\Res}^0(s_1,s_2,\phi_1',\phi_2')$ is defined by
\begin{align*}
-\frac{1}{n(s_1+s_2)}\int_{G'(\mathbb{A}_F)}\int_{\mathbb{A}_F^n}\int_{N_P(\mathbb{A}_F)}f(nu(\mathbf{x})\iota(y))\int_{[\overline{G'}]}\phi_1'(x)\overline{\phi_2'(xy)}dx|\det y|^{s_2}dnd\mathbf{x}dy.
\end{align*}

Since the cuspidal datum $(\chi, M_Q)$ is isomorphic to $\{(\chi',M_P)\},$ where $\chi'$ is associated to $\Ind_{P(\mathbb{A}_F)}^{G(\mathbb{A}_F)}\widetilde{\pi}'\otimes\omega\omega'^{-1}.$ Then by \eqref{197} in \text\textsection \ref{9.3}  the function  $J_{\Eis,\sing}^{\Reg}(f,\mathbf{s})+\mathcal{G}_{\chi}(s_1,s_2,\phi_1',\phi_2')$ is holomorphic when $-1/4\leq \Re(s_1)\leq 1/2,$ and  
$$
J_{\Eis,\sing}^{\Reg}(f,\mathbf{s})+\mathcal{G}_{\chi}(s_1,s_2,\phi_1',\phi_2')-\widetilde{\mathcal{G}}_{\chi}(s_1,s_2,\phi_1',\phi_2')
$$ 
is holomorphic in the region $-1/4\leq \Re(s_1)<1/2,$ and $-1/4\leq \Re(s_2)<1/2.$ Hence, by \eqref{168} and \eqref{26} we obtain 
\begin{equation}\label{24}
J_{\Eis,\sing}^{\Reg}(f,\mathbf{s})=J_{\Eis,\sing}^{\Reg,0}(f,\mathbf{s})+J_{\Eis,\sing}^{\Reg,1}(f,\mathbf{s})+J_{\Eis,\sing}^{\Reg,\hol}(f,\mathbf{s}),
\end{equation}
where $J_{\Eis,\sing}^{\Reg,1}(f,\mathbf{s})=-\mathcal{G}_{\chi}^1(s_1,s_2,\phi_1',\phi_2'),$  $J_{\Eis,\sing}^{\Reg,0}(f,\mathbf{s})=-\mathcal{G}_{\chi}^0(s_1,s_2,\phi_1',\phi_2')+\widetilde{\mathcal{G}}_{\chi}^0(s_1,s_2,\phi_1',\phi_2'),$ and $J_{\Eis,\sing}^{\Reg,\hol}(f,\mathbf{s})$ is holomorphic in the region $\mathcal{R}.$

\begin{lemma}\label{167}
Let notation be as before. Let $s\in\mathcal{R}.$ Then
\begin{align*}
J_{\Eis,\sing}^{\Reg,1}(f,\mathbf{s})&=J_{\Geo,\gamma(0),\Res}^{\bi,1}(f,\textbf{s})=\widetilde{\mathcal{G}}_{\chi,\Res}^1(s_1,s_2,\phi_1',\phi_2'),\\
J_{\Eis,\sing}^{\Reg,0}(f,\mathbf{s})&=-\mathcal{F}_{0,1}J^{\bi,1}_{\Geo}(f,\textbf{s})-\mathcal{F}_{1,0}J^{\bi,1}_{\Geo}(f,\textbf{s}).
\end{align*}
\end{lemma}
\begin{proof}
Lemma \ref{167} follows from the calculations in \text\textsection \ref{7.2.1}, \text\textsection \ref{8.2.3}, \text\textsection \ref{7.2.4} and \text\textsection \ref{7.2.5}. 
\end{proof}

\subsubsection{Proof of Proposition \ref{prop45}}\label{7.2.6}
In the following proof, we stick the notation in the preceding subsections \text\textsection \ref{7.2.1}--\ref{7.2.5}. 
\begin{proof}
Note that for $\mathbf{s}\in\mathcal{R},$ $\mathcal{F}_{0,1}J^{\bi,0}_{\Geo,\Res}(f,\textbf{s})$ (cf. \textsection\ref{7.2.1}), $\mathcal{F}_{0,1}J^{\bi,0}_{\Geo,\Res}(f,\textbf{s})$ (cf. \textsection\ref{8.2.3}), and $\mathcal{F}_{1,0}J^{\bi,0}_{\Geo,\Res}(f,\textbf{s})$ (cf. \textsection\ref{7.2.4}) are holomorphic. Therefore, Proposition \ref{prop45} follows from decomposition \eqref{160}, \eqref{158}, \eqref{27}, \eqref{28}, \eqref{24}, Lemmas \ref{185}, \ref{lem47} and \ref{167}.
\end{proof}

\section{The Meromorphic Relative Trace Formula}\label{7.3}

Let $\pi_1', \pi_2'$ be unitary cuspidal representations of $\mathrm{GL}(n)$ over $F.$ Let $\phi_i'\in\pi_i',$ $i=1, 2.$  Let $\mathbf{s}=(s_1,s_2)\in\mathbb{C}^2$ with $\Re(s_1)>1/2,$ $\Re(s_2)>1/2.$  Let $f\in \mathcal{F}_S(\mathfrak{N},\omega^{-1})$ (cf. \textsection\ref{2.6}).  Consider
\begin{align*}
J^{\Reg}(f,\textbf{s}):=\iint J_{\Kuz}\left(\begin{pmatrix}
		x&\\
		&1
	\end{pmatrix},\begin{pmatrix}
		y&\\
		&1
	\end{pmatrix}\right)\phi_1'(x)\overline{\phi_2'(y)}|\det x|^{s_1}|\det y|^{s_2}dxdy,
\end{align*} 
where $x$ and $y$ range through  $N'(\mathbb{A}_F)\backslash G'(\mathbb{A}_F),$ and for $g_1,$ $g_2\in G(\mathbb{A}_F),$ 
\begin{align*}
	J_{\Kuz}(g_1,g_2):=\int_{[N]}\int_{[N]}\K(u_1g_2,u_2g_2)\theta(u_1)\overline{\theta}(u_2)du_1du_2.
\end{align*}
Here $\theta$ is a fixed unramified generic character of $[N]$ and $\K=\K^f$ is the kernel function associated with $f.$  

\subsection{The Spectral Side $J_{\Spec}^{\Reg}(f,\mathbf{s})$}\label{sec8.1}
Denote by $\widehat{G(\mathbb{A}_F)}$ the set of the isomorphism class of irreducible unitary automorphic representations which appear in the spectral decomposition of $L^2([G],\omega).$ Let $\widehat{G(\mathbb{A}_F)}_{\gen}(\mathfrak{N})$ be  the subclass of \textit{generic} representations of level $\mathfrak{N}\subseteq \mathcal{O}_F$ in $\widehat{G(\mathbb{A}_F)}.$ Fix an automorphic Plancherel measure $d\mu_{\aut}$ on $\widehat{G(\mathbb{A}_F)}_{\gen}$ compatible with the invariant probability measure on $\widehat{G(\mathbb{A}_F)}.$ 

For $\pi\in \widehat{G(\mathbb{A}_F)}_{\gen}(\mathfrak{N}),$ there exists a standard parabolic subgroup $Q$ of $G$ such that $\pi$ corresponds to a cuspidal data $\chi=\{(M_Q,\sigma)\}$ (cf. \textsection\ref{7.1}), where $M_Q$ is the Levi component of $Q.$ Define  
$$
\mathfrak{B}_{\pi}=\big\{\phi\in \mathfrak{B}_{Q,\chi}:\ \text{$\phi$ is right-$K_{\infty}\otimes K_0(\mathfrak{N})$-invariant}\big\}.
$$ 
By Weyl law, $\mathfrak{B}_{\pi}$ is a finite set. 

Executing the spectral expansion of $\K$ and swapping integrals, we obtain the spectral side of $J^{\Reg}(f,\textbf{s})$:
\begin{equation}\label{6444}
	J_{\Spec}^{\Reg}(f,\mathbf{s})=\int_{\widehat{G(\mathbb{A}_F)}_{\gen}(\mathfrak{N})}\sum_{\phi\in\mathfrak{B}_{\pi}}\Psi(s_1,\pi(f)W_{\phi},W_{\phi_1'}')\Psi(s_2,\widetilde{W}_{\phi},\widetilde{W}_{\phi_2'}')d\mu_{\pi},
\end{equation}
where $\mathbf{s}=(s_1,s_2)\in\mathbb{C}^2$ is such that $\Re(s_1), \Re(s_2)\gg1,$ and $\Psi(s,W_{\phi},W_{\phi_i'}')$ is the Rankin-Selberg period in the Whittaker form:
\begin{align*}
\Psi(s,W_{\phi},W_{\phi_i'}')=\int_{N'(\mathbb{A}_F)\backslash G'(\mathbb{A}_F)}W_{\phi}\left(\begin{pmatrix}
		x\\
		&1
	\end{pmatrix}\right)W_{\phi_i'}'(x)|\det x|^sdx,\ \ \Re(s)\gg 1.
\end{align*}
Here $W_{\phi}(g):=\int_{[N]}\phi(ng)\overline{\theta}(n)dn$ is the Whittaker function associated to $\phi.$  

In \textsection\ref{sec.spec} and \textsection\ref{sec10.} we  show that $J_{\Spec}^{\Reg}(f,\mathbf{s})$ converges absolutely in the cone $\Re(s_1)>1/2,$ $\Re(s_2)>1/2,$ and it admits a meromorphic continuation to $\textbf{s}\in\mathcal{R}$ (cf. \eqref{R} in \textsection\ref{notation}). In particular, the continuation is described in detail in \textsection\ref{9.3} and the singular part is analyzed in \textsection\ref{10.2}.

\subsection{The Geometric Side $J_{\Geo}^{\Reg}(f,\mathbf{s})$}
Substituting the geometric expansion of $\K(x,y)$ into $J^{\Reg}(f,\mathbf{s})$ we then obtain its geometric side $J_{\Geo}^{\Reg}(f,\mathbf{s}).$ According to the results in \textsection\ref{sec6}--\textsection\ref{sec5}, we obtain the decomposition
\begin{align*}
	J^{\Reg}_{\Geo}(f,\textbf{s})=J^{\Reg}_{\Geo,\sm}(f,\textbf{s})+^{\Reg}_{\Geo,\du}(f,\textbf{s})-J^{\Reg,\RNum{1}}_{\Geo,\bi}(f,\textbf{s})+J^{\Reg,\RNum{2}}_{\Geo,\bi}(f,\textbf{s}).
\end{align*}
See the notation and definitions therein.

By Propositions \ref{prop14}, \ref{prop22}, \ref{prop23}, \ref{prop18'}, Theorems \ref{thm22}, \ref{thm24}, \ref{Red}, and \ref{thm40}, the geometric side $J^{\Reg}_{\Geo}(f,\textbf{s})$ converges absolutely in $\Re(s_1)>1/2,$ $\Re(s_2)>1/2,$ and it admits a meromorphic continuation to $\mathcal{R}.$

\subsection{The Relative Trace Formula: $\textbf{s}\in\mathcal{R}$}\label{sec8.3}
Combining \textsection\ref{sec6}--\textsection\ref{sec10.}, the relative trace formula could be summarized as follows.
\begin{thmx}\label{thm51}
Let notation be as before. Let $f\in \mathcal{F}_S(\mathfrak{N},\omega^{-1}).$ Then 
\begin{equation}\label{bb}
J_{\Spec}^{\Reg}(f,\mathbf{s})=J^{\Reg}_{\Geo}(f,\textbf{s}),
\end{equation}
which converge absolutely in the region $\Re(s_1),\ \Re(s_2)\gg1.$ Moreover, \eqref{bb} admits a meromorphic continuation to $\Re(s_1),\ \Re(s_2)>-1/(n+1).$ We have the following description of the above integrals.
\begin{itemize}

\item The integral $J^{\Reg}_{\Geo,\sm}(f,\textbf{s})$ is a Rankin-Selberg convolution converges absolutely in $\Re(s_1)+\Re(s_2)>0,$ and represents  $L(1+s_1+s_2,\pi_1'\times\widetilde{\pi}_2').$

\item The integral $J^{\Reg}_{\Geo,\du}(f,\textbf{s})$ is a Rankin-Selberg convolution converges absolutely in $\Re(s_1)+\Re(s_2)>1,$ and represents  $L(s_1+s_2,\pi_1'\times\widetilde{\pi}_2').$ 

\item The integral $J^{\Reg,\RNum{1}}_{\Geo,\bi}(f,\textbf{s})$ converges absolutely in $\Re(s_1),\ \Re(s_2)\gg1,$ and admit a meromorphic continuation to $\Re(s_1),\ \Re(s_2)>-1/(n+1).$

\item The integral $J^{\Reg,\RNum{2}}_{\Geo,\bi}(f,\textbf{s})$ converges absolutely for all $\mathbf{s}\in\mathbb{C}^2.$

\item  $J^{\Reg}_{\Geo}(f,\textbf{s};\phi_1',\phi_2')$ and $J_{\Spec}^{\Reg}(f,\mathbf{s};\phi_1',\phi_2')$ have at most simple poles on $s_1+s_2=\delta,$ $\delta\in\{-1,0,1\},$ and are holomorphic elsewhere. Moreover, 
\begin{align*}
&\underset{s_1+s_2=\delta}{\Res}J_{\Spec}^{\Reg}(f,\mathbf{s})=\underset{s_1+s_2=\delta}{\Res}J^{\Reg}_{\Geo}(f,\textbf{s}),\ \ \delta\in\{-1,0,1\}.
\end{align*}
In particular, singular parts on both sides of \eqref{bb} cancel with each other.  See \textsection\ref{sec7.2} for details.

\item The geometric side has stability in the level aspect in the sense of \cite{MR12} and \cite{FW09}. Roughly, the regular orbital integrals vanish identically when the level is large enough. See Theorem \ref{Red} in \textsection\ref{sec10}.
\end{itemize}

\end{thmx}
\begin{remark}
	Note that both sides of \eqref{bb} depend on the cusp forms $\phi_1'$ and $\phi_2'.$ We write $J_{\Spec}^{\Reg}(f,\mathbf{s})$ (resp. $J^{\Reg}_{\Geo}(f,\textbf{s})$) for $J_{\Spec}^{\Reg}(f,\mathbf{s};\phi_1',\phi_2')$ (resp. $J^{\Reg}_{\Geo}(f,\textbf{s};\phi_1',\phi_2')$) to simplify the notation.
\end{remark}

\subsection{The Relative Trace Formula: $\textbf{s}\in\mathcal{R}^*$}\label{8.2}
In many crucial applications we hope to take $\textbf{s}=(0,0)$ in Theorem \ref{thm51}. Recall that (cf. Theorem \ref{thm40}) 
$$
J_{\Spec}^{\Reg}(f,\mathbf{s})=J_{0}(f,\mathbf{s})+J_{\Eis,\reg}^{\Reg}(f,\mathbf{s})+J_{\Eis,\semi}^{\Reg}(f,\mathbf{s})+J_{\Eis,\sing}^{\Reg}(f,\mathbf{s}),
$$ 
which is not of the form \eqref{6444} anymore when $\textbf{s}$ is near $(0,0).$ 

In this section we provide a modification of \eqref{bb},  which is more convenient to handle when $\textbf{s}$ is close to $(0,0).$  Define 
\begin{align*}
	\mathcal{R}^*:=\big\{\mathbf{s}=(s_1,s_2)\in\mathbb{C}^2:\ |\Re(s_1)|<1/(n+1),\ |\Re(s_2)|<1/(n+1)\big\}.
\end{align*}

For $\textbf{s}\in\mathcal{R}^*,$ let 
\begin{align*}
	J_{\Spec}^{\Reg,\heartsuit}(f,\textbf{s}):=J_{\Spec}^{\Reg}(f,\textbf{s})+\mathcal{G}_{\chi}(\textbf{s},\phi_1',\phi_2')-\widetilde{\mathcal{G}}_{\chi}(\textbf{s},\phi_1',\phi_2'),
\end{align*}
where $\mathcal{G}_{\chi}(\textbf{s},\phi_1',\phi_2')$ and $\widetilde{\mathcal{G}}_{\chi}(\textbf{s},\phi_1',\phi_2')$ are defined by \eqref{183} and \eqref{198}, respectively. 

According to the results in \textsection\ref{sec6}--\textsection\ref{sec5}, we obtain the decomposition
\begin{align*}
	J^{\Reg,\heartsuit}_{\Geo}(f,\textbf{s})=J^{\Reg}_{\Geo,\sm}(f,\textbf{s})+^{\Reg}_{\Geo,\du}(f,\textbf{s})-J^{\Reg,\RNum{1},\heartsuit}_{\Geo,\bi}(f,\textbf{s})+J^{\Reg,\RNum{2}}_{\Geo,\bi}(f,\textbf{s}),
\end{align*}
where the integral $J^{\Reg,\RNum{1},\heartsuit}_{\Geo,\bi}(f,\textbf{s})$ is defined by 
$$
\mathcal{F}_{0,1}J^{\bi}_{\Geo}(f,\textbf{s})+\mathcal{F}_{1,0}J^{\bi}_{\Geo}(f,\textbf{s})-\mathcal{G}_{\chi}(\textbf{s},\phi_1',\phi_2')+\widetilde{\mathcal{G}}_{\chi}(\textbf{s},\phi_1',\phi_2').
$$ 

Note that $\mathcal{G}_{\chi}(\textbf{s},\phi_1',\phi_2')$ and $\widetilde{\mathcal{G}}_{\chi}(\textbf{s},\phi_1',\phi_2')$ can be written into orbital integrals (cf. Theorem \ref{thm49}), regardless that they arise originally in the spectral side.

\begin{thmx}\label{thm52}
Let notation be as above. Let $\mathbf{s}=(s_1,s_2)\in\mathcal{R}^*.$ Then
\begin{equation}\label{70}
		J_{\Spec}^{\Reg,\heartsuit}(f,\textbf{s})=J_{\Geo}^{\Reg,\heartsuit}(f,\textbf{s}),
	\end{equation}
	as an equality of holomorphic functions in $\mathcal{R}^*.$ Moreover, \begin{equation}\label{67}
		J_{\Spec}^{\Reg,\heartsuit}(f,\textbf{s})=\int_{\widehat{G(\mathbb{A}_F)}_{\gen}(\mathfrak{N})}\sum_{\phi\in\mathfrak{B}_{\pi}}\Psi(s_1,\pi(f)W_{\phi},W_{\phi_1'}')\Psi(s_2,\widetilde{W}_{\phi},\widetilde{W}_{\phi_2'}')d\mu_{\pi},
	\end{equation}
	which converges absolutely. Here we identify $\Psi(s_1,\pi(f)W_{\phi},W_{\phi_1'}')\Psi(s_2,\widetilde{W}_{\phi},\widetilde{W}_{\phi_2'}')$ with its meromorphic continuation.  
\end{thmx}

 When $\phi_1'=\phi_2'$ and $f$ is of the form $h*h^*$ for some function $h,$  $J_{\Spec}^{\Reg,\heartsuit}(f,\mathbf{0})$ is a sum of \textit{nonnegative} terms, towards which one can drop some of them (e.g., the continuous spectrum) to derive an upper bound for the remaining terms. We will make use of this fact in \textsection\ref{11.2} below.

\section{Applications of the Relative Trace Formula}\label{sec9..}
\subsection{Moments of Rankin-Selberg $L$-functions}\label{11.2}
Now we apply Theorems \ref{thm51} and \ref{thm52} to compute the second moment of Rankin-Selberg $L$-values for $\mathrm{GL}(n+1)\times\mathrm{GL}(n)$ over number fields.
\subsubsection{Choice of Local and Global Data}\label{8.5.1}
Suppose that $\Sigma_{F,\infty}$ consists of $r_1$ real places and $r_2$ complex places.  

Let $\pi'\in \mathcal{A}_0([G'],\omega')$  be everywhere  unramified. Let $\phi'\in\pi'$ be a new vector whose   associated Whittaker function $W'=\otimes_vW_v'$ satisfies that $W'(I_n)=1.$ Take $\pi_1'=\pi_2'=\pi',$ and $\phi_1'=\phi_2'=\phi'.$ 

Let $S=\Sigma_{F,\infty}.$ Take $f=f_{\infty}\otimes f_{\fin}\in \mathcal{F}_S(\mathfrak{N},\omega^{-1})$ (cf. \textsection\ref{2.6}). Suppose that $f_{\infty}$ is bi-$K_{\infty}$-invariant. Denote by $\mathcal{H}_{f_{\infty}}$ the Harish-Chandra transform of $f_{\infty},$ i.e., 
 for an irreducible admissible unramified representation $\pi_{\infty}$ of $G(F_{\infty}),$ \begin{equation}\label{133..}
\int_{G'(F_{\infty})}f_{\infty}(g_{\infty})W_{\infty}(g'_{\infty}g_{\infty})=\mathcal{H}_{f_{\infty}}(\boldsymbol{\lambda}_{\pi_{\infty}})W_{\infty}(g_{\infty}')
\end{equation} 
for all $g_{\infty}'\in  G(F_{\infty}),$ where $W_{\infty}$ is a shperical vector in the Whittaker model of $\pi_{\infty},$ and $\boldsymbol{\lambda}_{\pi_{\infty}}$ is the Langlands parameter of $\pi_{\infty}.$

Since $f_{\infty}$ is bi-$K_{\infty}$-invariant, and $W_{\infty}'$ is spherical, then 
\begin{equation}\label{134..}
\int_{G'(F_{\infty})}f_{\infty}\left(\begin{pmatrix}
	y_{\infty}\\
	&1
\end{pmatrix}\right)W_{\infty}'(x_{\infty}y_{
\infty})=\mathcal{H}_{f_{\infty}}^{\heartsuit}(\boldsymbol{\lambda}_{\pi_{\infty}'})W_{\infty}'(x_{\infty})
\end{equation} 
for all $x_{\infty}\in  G'(F_{\infty}),$ where $\boldsymbol{\lambda}_{\pi_{\infty}}'$ is the Langlands parameters of $\pi_{\infty}'.$

For $\phi\in \mathfrak{B}_{\pi},$ let $W_{\phi}=\otimes_vW_{\phi,v}$ be the associated Whittaker vector. Suppose that $\pi=\pi_1\boxplus\cdots \boxplus\pi_m,$ with each $\pi_j$ being cuspidal, $1\leq j\leq m.$ Let 
\begin{equation}\label{137}
\mathbf{L}(\pi):=\prod_{j=1}^m\Lambda(1,\pi_j,\Ad)\cdot \prod_{1\leq j_1<j_2\leq m}|\Lambda(1,\pi_{j_1}\times\widetilde{\pi}_{j_2})|^2.
\end{equation}
Then by Rankin-Selberg theory (cf. \cite{JPSS83}) and the Langlands-Shahidi method, 
\begin{equation}\label{137.}
	\langle\phi,\phi\rangle=(n+1)\alpha_{\phi,\mathfrak{N}}|W(I_{n+1})|^2\cdot \mathbf{L}(\pi),
\end{equation}
where we denote by $\mathbf{L}_v(\pi_v)$ the $v$-th component of $\mathbf{L}(\pi),$ and 
\begin{equation}\label{136.}
	\alpha_{\phi,\mathfrak{N}}:=\prod_{v\mid\mathfrak{N}}\frac{\zeta_v(n+1)}{\mathbf{L}_v(\pi_v)}\int_{N'(F_v)\backslash G'(F_v)}\big|W_{\phi,v}(\iota(x_v))\big|^2dx_v.
\end{equation}

In particular, $\mathbf{L}(\pi)=\Lambda(1,\pi,\Ad)$ if $\pi$ is cuspidal. Similar to \eqref{137.}, we have 
\begin{equation}\label{138.}
\Psi(0,W_{\phi},W')=\beta_{\phi,\mathfrak{N}}W(I_{n+1})W'(I_n)\cdot \Lambda(1/2,\pi\times\pi'),
\end{equation}
where $\beta_{\phi,\mathfrak{N}}$ is defined by 
\begin{align*}
\lim_{s\rightarrow 0}\prod_{v\mid\mathfrak{N}}\frac{1}{L_v(1/2+s,\pi_v\times\pi_v')}\int_{N'(F_v)\backslash G'(F_v)}W_{\phi,v}(\iota(x_v))W'(x_v)|\det x_v|_v^sdx_v.
\end{align*}

\begin{thm}\label{M}
Let notation be as above. Then 
\begin{align*}
&\int_{\widehat{G(\mathbb{A}_F)}_{\gen}(\mathfrak{N})}\sum_{\phi\in\mathfrak{B}_{\pi}}\frac{|\beta_{\phi,\mathfrak{N}}|^2}{\alpha_{\phi,\mathfrak{N}}}\cdot \frac{|\Lambda(1/2,\pi\times\pi')|^2}{\mathbf{L}(\pi)}\cdot \mathcal{H}_{f_{\infty}}(\boldsymbol{\lambda}_{\pi_{\infty}})d\mu_{\pi}\\
=&c_{F,f_{\infty}}\cdot \mathcal{H}_{f_{\infty}}^{\heartsuit}(\boldsymbol{\lambda}_{\pi_{\infty}'})\cdot\Lambda(1,\pi',\Ad)|\mathfrak{N}|^n\log|\mathfrak{N}|+O(|\mathfrak{N}|^n),
\end{align*}
where $\mathbf{L}(\pi)$ is defined by \eqref{137}, $\mathcal{H}_{f_{\infty}}^{\heartsuit}(\boldsymbol{\lambda}_{\pi_{\infty}})$ (resp. $\mathcal{H}_{f_{\infty}}^{\heartsuit}(\boldsymbol{\lambda}_{\pi_{\infty}'})$)  is defined by \eqref{133..} (resp. \eqref{134..}),  
\begin{align*}
	c_{F,f_{\infty}}= n(n+1)2^{r_2(n-1)}\pi^{-r_2}|\mathfrak{D}_F|^{\frac{n}{2}}\underset{s=1}{\Res}\ \zeta_F(s),
\end{align*}	
and the implied constant in $O(|\mathfrak{N}|^n)$ depends on $F,$ $\pi'$ and $f_{\infty}.$ 
\end{thm}
\begin{remark}
Using the $\mathbf{c}$-function $f_{\infty}$ can be written explicitly by its Harish-Chandra transform $\mathcal{H}_{f_{\infty}}$ (cf. \cite{Wal92}). Hence, one can derive an explicit relation between $\mathcal{H}_{f_{\infty}}^{\heartsuit}(\boldsymbol{\lambda}_{\pi_{\infty}'})$ and $\mathcal{H}_{f_{\infty}}^{\heartsuit}(\boldsymbol{\lambda}_{\pi_{\infty}}).$ 
\end{remark}

\begin{defn}\label{defn57}
	Let $\mathcal{A}_0(\mathfrak{N},\omega)$ be the set of isomorphism classes of unitary cuspidal representations in  $\mathcal{A}_0([G],\omega)$ which are right invariant under $K_0(\mathfrak{N}).$  
\end{defn}

\begin{cor}\label{cor58}
Let notation be as before. Then 
\begin{align*}
\sum_{\substack{\pi\in \mathcal{A}_0(\mathfrak{N},\omega)\\ \boldsymbol{\lambda}_{\pi_{\infty}\in \mathcal{D}}}}\frac{|L(1,\pi\times\pi')|^2}{L(1,\pi,\Ad)}\ll |\mathfrak{N}|^n\log |\mathfrak{N}|,
\end{align*}
where $\mathcal{D}\subset \mathbb{C}^{n+1}$ is a compact subset, $\mathcal{A}_0(\mathfrak{N},\omega)$ is defined by Definition \ref{defn57}, and the implied constant depends on  $\mathcal{D},$ $\pi'$ and $f_{\infty}.$ 	
\end{cor}


\subsubsection{Asymptotic Behavior of the Geometric Side}

Let $\mathbf{s}=(s,0).$ Define
\begin{align*}
	J^{\Reg,\hol}_{\Geo,\sm}(f,\textbf{s}):=&J^{\Reg}_{\Geo,\sm}(f,\textbf{s})-s^{-1}\underset{s=0}{\Res}J^{\Reg}_{\Geo,\sm}(f,\textbf{s}),\\
	J^{\Reg,\hol}_{\Geo,\du}(f,\textbf{s}):=&J^{\Reg}_{\Geo,\du}(f,\textbf{s})-s^{-1}\underset{s=0}{\Res}J^{\Reg}_{\Geo,\du}(f,\textbf{s}).
\end{align*}
Then $J^{\Reg,\hol}_{\Geo,\sm}(f,\textbf{s})$ and $J^{\Reg,\hol}_{\Geo,\du}(f,\textbf{s})$ are holomorphic near  $s=0.$
\begin{lemma}\label{lem57}
	Let notation be as above. Then $J^{\Reg,\hol}_{\Geo,\sm}(f,\textbf{0})=O(|\mathfrak{N}|^n),$ where the implied constant depends on $F,$ $\pi'$ and $f_{\infty}.$ 
\end{lemma}
\begin{proof}
By Proposition \ref{prop11'} we have 
\begin{align*}
	J^{\Reg}_{\Geo,\sm}(f,\textbf{s})=\frac{|\mathfrak{D}_F|^{n(s+\frac{1}{2})}L(s+1,\pi'\times\widetilde{\pi}')}{\Vol(K_0(\mathfrak{N}))}\cdot \mathcal{I}_{\infty}(f_{\infty},s),
\end{align*}
where $|\mathfrak{D}_F|:=N_F(\mathfrak{D}_F),$ and 
\begin{align*}
\mathcal{I}_{\infty}(f_{\infty},s):=&\int_{N'(F_{\infty})\backslash G'(F_{\infty})}\int_{G'(F_{\infty})}\int_{F_{\infty}^n}f_{\infty}\left(\begin{pmatrix}
	y_{\infty}&u_{\infty}\\
	&1
\end{pmatrix}\right)\theta_{\infty}(\eta x_{\infty}u_{\infty})\\
&\qquad \qquad W_{\infty}'(x_{\infty})\overline{W_{\infty}'(x_{\infty}y_{\infty})}|\det x_{\infty}|_{\infty}^{1+s}du_{\infty}dy_{\infty}dx_{\infty}.
\end{align*}

Write $L(s+1,\pi'\times\widetilde{\pi}')=a_{-1}s^{-1}+a_0+O(s),$ and $|\mathfrak{D}_F|^{ns+\frac{n}{2}}\mathcal{I}_{\infty}(f_{\infty},s)=b_0+b_1s+O(s^2).$ Then 
\begin{align*}
	J^{\Reg,\hol}_{\Geo,\sm}(f,\textbf{0})=\frac{a_{-1}b_1+a_0b_0}{\Vol(K_0(\mathfrak{N}))}\ll |\mathfrak{N}|^n,
\end{align*}
where the implied constant relies on $F,$ $\pi'$ and $f_{\infty}.$
\end{proof}

\begin{lemma}\label{lem58}
	Let notation be as above. Then 
	\begin{equation}\label{135.}
	J^{\Reg,\hol}_{\Geo,\du}(f,\textbf{0})=c_{F,f_{\infty}}'\mathcal{H}_{f_{\infty}}^{\heartsuit}(\boldsymbol{\lambda}_{\pi_{\infty}'})\Lambda(1,\pi',\Ad)|\mathfrak{N}|^n\log|\mathfrak{N}|+O(|\mathfrak{N}|^n),
	\end{equation} 
	where $\mathcal{H}_{f_{\infty}}^{\heartsuit}(\boldsymbol{\lambda}_{\pi_{\infty}'})$ is defined by \eqref{134..},  and  
\begin{align*}
	c_{F,f_{\infty}}'= n2^{r_2(n-1)}\pi^{-r_2}|\mathfrak{D}_F|^{\frac{n}{2}}\underset{s=1}{\Res}\ \zeta_F(s),
\end{align*}	
and the implied constant in $O(|\mathfrak{N}|^n)$depends on $F,$ $\pi'$ and $f_{\infty}.$ 
\end{lemma}
\begin{proof}
Let $s<0.$ By Proposition \ref{prop18'} and the functional equation of the Eisenstein series $E_{\du}(x,s;f^{\dag}(\cdot,z_2,y))$ defined by \eqref{eisdu}, we have 
\begin{align*}
	J^{\Reg}_{\Geo,\du}(f,\textbf{s})=\prod_{v\in\Sigma_F}\mathcal{I}_{v}^{\heartsuit}(f_{v},s)
\end{align*} 
where 
\begin{align*}
\mathcal{I}_{v}^{\heartsuit}(f_{v},s):=&\int_{N'(F_{v})\backslash G'(F_{v})}\int_{G'(F_{v})}\int_{F_{v}^n}f_{v}\left(\begin{pmatrix}
	I_n&\\
	c_v&1
\end{pmatrix}\begin{pmatrix}
	y_{v}&\\
	&1
\end{pmatrix}\right)\theta_{v}(\eta x_{v}\transp{c_v})\\
&\qquad \qquad \overline{W_{v}'(x_{v})}W_{v}'(x_{v}\transp{y}_{v}^{-1})|\det x_{v}|_{v}^{1-s}dc_vdy_{v}dx_{v}.
\end{align*}

Similar to Proposition \ref{prop21} we have 
\begin{align*}
	J^{\Reg}_{\Geo,\du}(f,\textbf{s})=\frac{|\mathfrak{D}_F|^{n(-s+\frac{1}{2})}L(1-s,\pi'\times\widetilde{\pi}')}{\Vol(K_0(\mathfrak{N}))|\mathfrak{N}|^{ns}}\cdot \mathcal{I}_{\infty}^{\heartsuit}(f_{\infty},s),
\end{align*}
where $\mathcal{I}_{\infty}^{\heartsuit}(f_{\infty},s)=\prod_{v\in\Sigma_{F,\infty}}\mathcal{I}_{v}^{\heartsuit}(f_{\infty},s),$ namely, 
\begin{align*}
\mathcal{I}_{\infty}^{\heartsuit}(f_{\infty},s)=&\int_{N'(F_{\infty})\backslash G'(F_{\infty})}\int_{G'(F_{\infty})}\int_{F_{\infty}^n}f_{\infty}\left(\begin{pmatrix}
	I_n&\\
	c_{\infty}&1
\end{pmatrix}\begin{pmatrix}
	y_{\infty}&\\
	&1
\end{pmatrix}\right)\theta_{\infty}(\eta x_{\infty}\transp{c_{\infty}})\\
&\qquad \qquad \overline{W_{\infty}'(x_{\infty})}W_{\infty}'(x_{\infty}\transp{y}_{\infty}^{-1})|\det x_{\infty}|_{\infty}^{1-s}du_{\infty}dy_{\infty}dx_{\infty}.
\end{align*}

Write $L(1-s,\pi'\times\widetilde{\pi}')=-a_{-1}'s^{-1}+a_0'+O(s),$ and $|\mathfrak{D}_F|^{-ns+\frac{n}{2}}|\mathfrak{N}|^{-ns}\mathcal{I}_{\infty}^{\heartsuit}(f_{\infty},s)=b_0'-b_1's+O(s^2).$ Then 
\begin{align*}
	J^{\Reg,\hol}_{\Geo,\sm}(f,\textbf{0})=\frac{a_{-1}'b_1'+a_0'b_0'}{\Vol(K_0(\mathfrak{N}))}=\frac{a_{-1}'b_1'}{\Vol(K_0(\mathfrak{N}))}+O(|\mathfrak{N}|^n),
\end{align*}
where the implied constant relies on $F,$ $\pi'$ and $f_{\infty}.$ By definition,
\begin{align*}
	b_1'=&|\mathfrak{D}_F|^{\frac{n}{2}}\int_{N'(F_{\infty})\backslash G'(F_{\infty})}\int_{G'(F_{\infty})}\int_{F_{\infty}^n}f_{\infty}\left(\begin{pmatrix}
	I_n&\\
	c_{\infty}&1
\end{pmatrix}\begin{pmatrix}
	y_{\infty}&\\
	&1
\end{pmatrix}\right)\theta_{\infty}(\eta x_{\infty}\transp{c_{\infty}})\\
&\overline{W_{\infty}'(x_{\infty})}W_{\infty}'(x_{\infty}\transp{y}_{\infty}^{-1})|\det x_{\infty}|_{\infty}\log(|\det x_{\infty}|_{\infty}|\mathfrak{D}_F|^n|\mathfrak{N}|^n)dc_{\infty}dy_{\infty}dx_{\infty}.
\end{align*}

So $b_1'=b_1''+O(|\mathfrak{N}|^n),$ where 
\begin{align*}
	b_1''=&|\mathfrak{D}_F|^{\frac{n}{2}}\int_{N'(F_{\infty})\backslash G'(F_{\infty})}\int_{G'(F_{\infty})}\int_{F_{\infty}^n}f_{\infty}\left(\begin{pmatrix}
	I_n&\\
	c_{\infty}&1
\end{pmatrix}\begin{pmatrix}
	y_{\infty}&\\
	&1
\end{pmatrix}\right)\theta_{\infty}(\eta x_{\infty}\transp{c_{\infty}})\\
&\quad\overline{W_{\infty}'(x_{\infty})}W_{\infty}'(x_{\infty}\transp{y}_{\infty}^{-1})|\det x_{\infty}|_{\infty}\log(|\mathfrak{N}|^n)dc_{\infty}dy_{\infty}dx_{\infty}.
\end{align*}

Let $H=\diag(\mathrm{GL}(n-1),1)$ and $N_H=N'\cap H.$ By Iwasawa decomposition, \begin{align*}
	b_1''=&C_{n,F}\int_{G'(F_{\infty})}\int_{K_{\infty}'}\int_{Z'(F_{\infty})}\int_{F_{\infty}^n}f_{\infty}\left(\begin{pmatrix}
	I_n&\\
	c_{\infty}&1
\end{pmatrix}\begin{pmatrix}
	y_{\infty}&\\
	&1
\end{pmatrix}\right)\theta_{\infty}(\eta z_{\infty}k_{\infty}\transp{c_{\infty}})\\
&\int_{N_H(F_{\infty})\backslash H(F_{\infty})}\overline{W_{\infty}'(h_{\infty}k_{\infty})}W_{\infty}'(h_{\infty}k_{\infty}\transp{y}_{\infty}^{-1})dh_{\infty}dc_{\infty}|\det z_{\infty}|_{\infty}d^{\times}z_{\infty}dy_{\infty},
\end{align*}
where $C_{n,F}:=n|\mathfrak{D}_F|^{\frac{n}{2}}\log |\mathfrak{N}|.$ By Kirillov model theory, we have
\begin{equation}\label{143}
	\int \overline{W_{\infty}'(h_{\infty}k_{\infty})}W_{\infty}'(h_{\infty}k_{\infty}\transp{y}_{\infty}^{-1})dh_{\infty}=\int\overline{W_{\infty}'(h_{\infty})}W_{\infty}'(h_{\infty}\transp{y}_{\infty}^{-1})dh_{\infty},
\end{equation}
where $h_{\infty}$ ranges through $N_H(F_{\infty})\backslash H(F_{\infty}).$ 

Let $v\mid\infty.$ For $y_v\in G'(F_v),$ define $\mathcal{J}_{v}(y_{v})$ by
\begin{align*}
\int_{K_{v}'}\int_{Z'(F_{v})}\int_{F_{v}^n}f_{v}\left(\begin{pmatrix}
	I_n&\\
	c_{v}&1
\end{pmatrix}\begin{pmatrix}
	y_{v}&\\
	&1
\end{pmatrix}\right)\theta_{v}(\eta z_{v}k_{v}\transp{c_{v}})dc_{v}|\det z_{v}|_{v}d^{\times}z_{v}dk_{v}.
\end{align*}

For $v\mid \infty,$ $K_v'=\mathrm{O}(n)$ if $v$ is real, and $K_v'=\mathrm{U}(n)$ if $v$ is complex. Using the decomposition of Haar measures of $K_v'$ into measures on spheres, we then obtain 
\begin{align*}
	\mathcal{J}_v(y_v)=d_v\int_{\mathbb{S}_v}\int_{F_{v}^{\times}}\int_{F_{v}^n}f_{v}\left(\begin{pmatrix}
	I_n&\\
	c_{v}&1
\end{pmatrix}\begin{pmatrix}
	y_{v}&\\
	&1
\end{pmatrix}\right)\theta_{v}(\eta z_{v}k_{v}\transp{c_{v}})dc_{v}|z_{v}|_{v}^{n}d^{\times}z_{v}dk_{v},
\end{align*} 
where $\mathbb{S}_v=\mathbb{S}^{n-1},$ the $(n-1)$-sphere, if $v$ is real; and $\mathbb{S}_v=\mathbb{S}^{2n-1}$ if $v$ complex, and 
\begin{equation}\label{145}
	d_v=\begin{cases}
	\Vol(\mathbb{S}^{n-1})^{-1}=\frac{\pi^{-n/2}\Gamma(n/2)}{2}=\frac{\Gamma_v(n)}{2},\ &\text{if $v$ is real},\\
	\Vol(\mathbb{S}^{2n-1})^{-1}=\frac{\pi^{-n}\Gamma(n)}{2}=\frac{2^n\Gamma_v(n)}{4},\ &\text{if $v$ is complex}.
	\end{cases}
\end{equation}

Using Fourier inversion and polar coordinates we obtain that 
\begin{equation}\label{144}
	\mathcal{J}_v(y_v)=2^{([F_v: \mathbb{R}]-1)(n-1)}\Gamma_v(n)f_{v}\left(\begin{pmatrix}
	y_{v}&\\
	&1
\end{pmatrix}\right).
\end{equation} 

Therefore, $b_1''$ is equal to 
\begin{align*}
n2^{r_2(n-1)}|\mathfrak{D}_F|^{\frac{n}{2}}\log |\mathfrak{N}|\Lambda_{F,\infty}(n)\int_{G'(F_{\infty})}f_{v}\left(\begin{pmatrix}
	y_{v}&\\
	&1
\end{pmatrix}\right)\overline{\langle \pi_{\infty}'(y_v)W_{\infty}',W_{\infty}'\rangle}dy_v,
\end{align*}
where $\Lambda_{F,\infty}=\prod_{v\in \Sigma_{F,\infty}}\Gamma_v$ is the archimedean component of the complete Dedekind zeta function attached to the number field $F.$

In conjunction with the definition of $\mathcal{H}_{f_{\infty}}^{\heartsuit}(\boldsymbol{\lambda}_{\pi_{\infty}'}),$ we obtain that
\begin{align*}
b_1''=n2^{r_2(n-1)}\pi |\mathfrak{D}_F|^{\frac{n}{2}}\log |\mathfrak{N}|\Lambda_{F,\infty}(n)\mathcal{H}_{f_{\infty}}^{\heartsuit}(\boldsymbol{\lambda}_{\pi_{\infty}'})\overline{\langle W_{\infty}',W_{\infty}'\rangle}.
\end{align*}

By Stade's formula (cf. \cite{Sta01} and \cite{Sta02}), 
$$
\langle W_{\infty}',W_{\infty}'\rangle=L_{\infty}(1,\pi_{\infty}'\times\widetilde{\pi}_{\infty}')/\Lambda_{F,\infty}(n).
$$

On the other hand, we have $\Lambda_{F,\infty}(1)=\pi^{-r_2},$ and 
\begin{align*}
a_{-1}'=\underset{s=1}{\Res}\ L(s,\pi'\times\widetilde{\pi}')=L(1,\pi',\Ad)\cdot \underset{s=1}{\Res}\ \zeta_F(s).
\end{align*}

Therefore, \eqref{135.} follows.
\end{proof}

\subsubsection{Proof of Theorem \ref{M} and Corollary \ref{cor58}}
\begin{proof}[Proof of Theorem \ref{M}]
By the construction of $f,$ we have, in conjunction with \eqref{67} and \eqref{138.}, that 
	\begin{align*}
		J_{\Spec}^{\Reg,\heartsuit}(f,\textbf{0})=\int_{\widehat{G(\mathbb{A}_F)}_{\gen}(\mathfrak{N})}\sum_{\phi\in\mathfrak{B}_{\pi}}\frac{|\beta_{\phi,\mathfrak{N}}W(I_n)\cdot \Lambda(1/2,\pi\times\pi')|^2}{\langle\phi,\phi\rangle}\cdot \mathcal{H}_{f_{\infty}}(\boldsymbol{\lambda}_{\pi_{\infty}})d\mu_{\pi},
	\end{align*}
where we also make use of the assumption that $W'(I_n)=1.$

By Propositions \ref{prop22}, \ref{prop23}, \ref{prop45}, Theorems \ref{thm22}, \ref{thm24}, \ref{Red}, and Lemmas \ref{lem57} and \ref{lem58}, we obtain that 
 \begin{align*}
 	J^{\Reg,\heartsuit}_{\Geo}(f,\textbf{0})= c_{F,f_{
 	\infty}}'\mathcal{H}_{f_{\infty}}^{\heartsuit}(\boldsymbol{\lambda}_{\pi_{\infty}'})\Lambda(1,\pi',\Ad)|\mathfrak{N}|^n\log|\mathfrak{N}|+O(|\mathfrak{N}|^n) \end{align*}
where the implied constant relies on $\phi'$ and $f_{\infty}.$ 

Then Theorem \ref{M} follows from Theorem \ref{thm52} and \eqref{137.}.
\end{proof}

\begin{proof}[Proof of Corollary \ref{cor58}]
Let $\pi\in \mathcal{A}_0(\mathfrak{N},\omega).$ Since $W'$ is everywhere unramified, by  \cite{JPSS81} there exists $\phi\in \mathfrak{B}_{\pi}$ such that $\alpha_{\phi,\mathfrak{N}}=\beta_{\phi,\mathfrak{N}}=1.$ So Theorem \ref{M} yields 
\begin{align*}
\sum_{\substack{\pi\in \mathcal{A}_0(\mathfrak{N},\omega)\\ \boldsymbol{\lambda}_{\pi_{\infty}\in \mathcal{D}}}}\frac{|\Lambda(1,\pi\times\pi')|^2}{\Lambda(1,\pi,\Ad)}\ll |\mathfrak{N}|^n\log |\mathfrak{N}|.
\end{align*}

Notice that $\boldsymbol{\lambda}_{\pi_{\infty}}\in \mathcal{D},$ which is compact. By the bound towards the Langlands parameters (cf. \cite{KS03}) we have $\Lambda_{\infty}(1,\pi_{\infty}\times\pi_{\infty}')\asymp 1$ and $\Lambda_{\infty}(1,\pi_{\infty},\Ad)\asymp 1,$ where the implied constants rely on $\mathcal{D}$ and $\pi_{\infty}'.$ Hence  Corollary \ref{cor58} follows.
\end{proof}

\subsection{Simultaneously Nonvanishing}\label{sec8}
Now we apply Theorem \ref{thm52} to the simultaneously nonvanishing problem for Rankin-Selberg $L$-functions. 
\subsubsection{Choice of Local and Global Data}\label{9.2.1}
Let $\mathfrak{N}_j'\subseteq \mathcal{O}_F$ be an integral ideal, $j=1,2.$ Suppose $\mathfrak{N}_1'+\mathfrak{N}_2'=\mathcal{O}_F$ and $\mathfrak{N}_1'\mathfrak{N}_2'\subsetneq \mathcal{O}_F.$
Let $\pi_j'\in \mathcal{A}_0([G'],\omega')$ with arithmetic conductor $\mathfrak{N}_j',$ $j=1,2.$ Then $\pi_1'\not\simeq \pi_2'.$ Let $\phi_j'\in\pi_j'$ be the new form with the normalization that $W_{\phi_j'}'(I_n)=1,$ $j=1, 2.$ 

For $v\mid\mathfrak{N}_2',$ we define  
$$
I_v(\mathfrak{N}_2'):=\big\{g_v=(g_{i,j})_{1\leq i,j \leq n}\in G(\mathcal{O}_{F_v}):\ g_{i,j}\equiv 0\mod{q_v^{e_v(\mathfrak{N}_2')}},\ i>j\big\},
$$  
the Iwahori subgroup of level $e_v(\mathfrak{N}_2').$ Let $I(\mathfrak{N}_2')=\prod_{v}I_v(\mathfrak{N}_2').$

\begin{defn}
We say $\pi_1'$ and $\pi_2'$ are \textit{linked at $v\in\Sigma_F$} if $\pi_{1,v}'\simeq \pi_{2,v}'.$ 
\end{defn}

\begin{defn}\label{defn63}
Let $\mathfrak{N}\subset\mathcal{O}_F$ and $v_*\in \Sigma_F$ be a split place. Let $\varepsilon>0.$ Define 
\begin{align*}
	\mathcal{A}_0(\mathfrak{N};\mathfrak{N}_1',\mathfrak{N}_2',\sigma_{v_*},\omega,\varepsilon)=\big\{\pi\in\mathcal{A}_0([G],\omega)^{K^{\dagger}}:\ \|\boldsymbol{\lambda}_{\pi_{\infty}}\|\leq |\mathfrak{N}|^{\varepsilon},\ \pi_{v_*}\simeq \sigma_{v_*}\big\},
\end{align*}
where $K^{\dagger}=K_0(\mathfrak{N}\mathfrak{N}_1'\mathfrak{p}_{v_*}^{n+2})\cap I(\mathfrak{N}_2'),$ $\pi\in\mathcal{A}_0([G],\omega)^{K^{\dagger}}$ means $\pi$ is right-$K^{\dagger}$-invariant,  and $\|\boldsymbol{\lambda}_{\pi_{\infty}}\|$ is the Euclid norm of $\boldsymbol{\lambda}_{\pi_{\infty}}$ viewed as a point in $\mathbb{C}^{n+1}.$ Then 
\begin{equation}\label{set}
\sum_{\pi\in \mathcal{A}_0(\mathfrak{N};\mathfrak{N}_1',\mathfrak{N}_2',\sigma_{v_*},\omega,\varepsilon)}\sum_{\phi\in\mathfrak{B}_{\pi}}1\ll |\mathfrak{N}|^{n+(n+1)^2\varepsilon}|\mathfrak{N}_1'|^n|\mathfrak{N}_2'|^{\frac{n(n+1)}{2}}q_{v_*}^{n(n+2)},
\end{equation}
where the implied constant is absolute.
\end{defn}
\subsubsection{Statament of the Main Result}
\begin{thm}\label{thm53.}
Let notation be as above. Suppose that $\pi_1'$ and $\pi_2'$ are linked at a fixed split place $v_*\in\Sigma_F$ with $v_*\nmid \mathfrak{N}_1'\mathfrak{N}_2'.$ 
\begin{enumerate}
	\item There exists some $\pi\in\mathcal{A}_0([G],\omega)^{K^{\dagger}}$ with $\pi_{v_*}\simeq \sigma_{v_*}$  such that 
	$$
L(1/2,\pi\times\pi_1')L(1/2,\pi\times\pi_2')\neq 0.
$$
	\item Suppose further that $\pi_{1,\infty}'\simeq \pi_{2,\infty}'.$ Then $L(1/2,\pi\times\pi_1')L(1/2,\pi\times\pi_2')\neq 0$ for some $\pi\in\mathcal{A}_0(\mathfrak{N};\mathfrak{N}_1',\mathfrak{N}_2',\sigma_{v_*},\omega,\varepsilon).$ 
\end{enumerate}

\end{thm}

\subsubsection{Construction of the test function $f$}\label{9.2.3} 
Let $\sigma_{v_*}$ be a simple supercuspidal representation of $G(F_{v_*})$ constructed by \cite{KL15}. Let $W_{v_*}$ be a new vector in the Whittaker model of $\sigma_{v_*}.$ Define $f=\otimes_{v}f_v$ as follows: 
 \begin{itemize}
 	\item  $v\mid \mathfrak{N},$ $f_v$ is defined in \textsection\ref{2.6}; 
	\item $v=v_*,$ $f_v(g):=\langle \sigma_{v_*}(g)W_{v_*},W_{v_*}\rangle_{v_*},$ $g\in G(F_{v_*}),$ is the matrix coefficient;
 	\item  $v\nmid v_*\mathfrak{N}\mathfrak{N}_1'\mathfrak{N}_2'$ and $v\in\Sigma_{F,\fin},$ $f_v$ is defined in \textsection\ref{2.6};
 	\item $v\mid \mathfrak{N}_1'.$ Define the function 
\begin{align*}
 \widehat{\Phi}_v((b_1,\cdots, b_n))=\begin{cases}\overline{\omega}_{1,v}(b_n)\omega_{2,v}(b_n),\ &\text{if $b_1,\cdots, b_{n-1}\in \mathfrak{N}_1'\mathcal{O}_{F_v},$ $b_n\in\mathcal{O}_{F_v}^{\times},$}\\
 			0,\ & \text{otherwise}.
 		\end{cases}
 	\end{align*}
 	Let $\Phi_v$ be the Fourier inversion of $\widehat{\Phi}_v$ relative to the additive character $\psi_v$ (cf. \textsection\ref{notation}).  Define $f_v(g_v):=\int_{Z(F_v)}\tilde{f}_v(z_vg_v)\omega_v(z_v)d^{\times}z_v,$ $g_v\in G(F_v),$ where 
 	 \begin{align*}
 		\tilde{f}_v(g_v)=\frac{\textbf{1}_{M_{n,n}(\mathcal{O}_{F_v})}(A)\Phi_v(\mathfrak{b})\textbf{1}_{M_{1,n}(\mathcal{O}_{F_v})}(\mathfrak{c})\textbf{1}_{\mathcal{O}_{F_v}}(d)\textbf{1}_{G(\mathcal{O}_{F_v})}(g_v)}{\Vol(K_0'(q_v^{e_v(\mathfrak{N}_1')}))}
 	\end{align*}
 	for all $g_v=\begin{pmatrix}
 			A&\mathfrak{b}\\
 			\mathfrak{c}&d
 		\end{pmatrix}\in G(F_v).$ Here $K_0'(q_v^{e_v(\mathfrak{N}_1')})$ is the Hecke congruence subgroup of $K_v'$ of level $e_v(\mathfrak{N}_1').$ 
 		
 		In particular, $f_v$ is right $K_0(q_v)$-invariant, $v\mid \mathfrak{N}_1'$. 
 
 	\item $v\mid \mathfrak{N}_2'.$ Define $f_v(g_v):=\int_{Z(\mathcal{O}_{F_v})}\tilde{f}_v(z_vg_v)\omega_v(z_v)d^{\times}z_v,$ $g_v\in G(F_v),$ where 
 	 \begin{align*}
 		\tilde{f}_v(g_v)=\frac{\textbf{1}_{K_0'(q_v^{e_v(\mathfrak{N}_2')})}(A)\Phi_v(\mathfrak{b})\textbf{1}_{M_{1,n}(\mathcal{O}_{F_v})}(\mathfrak{c})\textbf{1}_{\mathcal{O}_{F_v}}(d)\textbf{1}_{G(\mathcal{O}_{F_v})}(g_v)}{\Vol(K_0'(q_v^{e_v(\mathfrak{N}_2')})}
 	\end{align*}
 	for all $g_v=\begin{pmatrix}
 			A&\mathfrak{b}\\
 			\mathfrak{c}&d
 		\end{pmatrix}\in G(F_v).$

In particular, $f_v$ is right $I_v(\mathfrak{N}_2')$-invariant, $v\mid \mathfrak{N}_2'.$

\item Define $f_{\infty}$ as follows:
\begin{itemize}
\item if $\pi_{1,\infty}'\simeq\pi_{2,\infty}'=\pi_{\infty}',$ let $f_{\infty}=h_{\infty}*h_{\infty}^**\omega_{\infty},$ where $h_{\infty}\in C_c(\overline{G}(F_{\infty}))$ is bi-$K_{\infty}$-invariant (cf. \textsection\ref{8.5.1})  with the property that $\mathcal{H}_{f_{\infty}}^{\heartsuit}(\boldsymbol{\lambda}_{\pi_{\infty}'})\neq 0.$
\item if $\pi_{1,\infty}'\not\simeq\pi_{2,\infty}',$ define $f_{\infty}\in C_c^{\infty}(F_{\infty})$ by setting 
\begin{align*}
f_{\infty}\left(z_{\infty}\begin{pmatrix}
	I_n\\
	c_{\infty}&1
\end{pmatrix}\begin{pmatrix}
	y_{\infty}&u_{\infty}\\
	&1
\end{pmatrix}\right)=\omega_{\infty}^{-1}(z_{\infty})h_0(y_{\infty})h_1(u_{\infty})h_2(\|c_{\infty}\|_{\infty}),
\end{align*}
where $h_0\in C_c^{\infty}(G'(F_{\infty}))$ is bi-$K_{\infty}'$-invariant with $\mathcal{H}_{h_0}(\boldsymbol{\lambda}_{\pi_{2,\infty}'})\neq 0,$ $h_1\in C_c^{\infty}(F_{\infty}^n)$ with the Fourier transform $\widehat{h}_1$ of $h_1$ satisfying that $\widehat{h}_1(u_{\infty})=\widehat{h}_1(\|u_{\infty}\|_{\infty}\eta),$ and 
$$
\widehat{h}_1(|z_{\infty}|_{\infty}\eta)\omega_{1,\infty}(z_{\infty})\omega_{2,\infty}^{-1}(z_{\infty})\geq 0
$$ 
for all $z_{\infty}\in F_{\infty}^{\times};$
$h_2\in C_c^{\infty}(\mathbb{R})$ with $h_2(0)=1,$ and $\|\cdot\|_{\infty}$ is the Euclid norm equipped with $F_{\infty}^n.$ Extend $f_{\infty}$ continuously to $G(F_{\infty}).$
\end{itemize}
\end{itemize}

 
\subsubsection{Proof of Theorem \ref{thm53.}}
Since $\pi_1'\not\simeq\pi_2',$ then $\Lambda(s,\pi_1'\times\widetilde{\pi}_2')$ is entire. By Proposition \ref{prop14}  the function $J^{\Reg}_{\Geo,\sm}(f,\textbf{s})$ is entire. In particular, $J^{\Reg}_{\Geo,\sm}(f,\textbf{0})$ is well defined. 
\begin{lemma}\label{thm53}
Let notation be as above. Let $n\geq 2.$ Let $f,$ $\phi_1'$ and $\phi_2'$ be constructed as above. Then 
\begin{align*}
	J^{\Reg}_{\Geo,\sm}(f,\textbf{0})\gg |\mathfrak{N}|^n,
\end{align*}
where the implied constants depend on $\pi_1',$ $\pi_2'$ and $f.$
\end{lemma}
\begin{proof}
By Proposition \ref{prop11'} we have 
\begin{align*}
	J^{\Reg}_{\Geo,\sm}(f,\textbf{0})=\frac{|\mathfrak{D}_F|^{\frac{n}{2}}\Lambda^S(1,\pi'\times\widetilde{\pi}')}{\Vol(K_0(\mathfrak{N}))}\cdot \prod_{v\in S}\mathcal{I}_{v}(f_{v},0),
\end{align*}
where $S=\Sigma_{F,\infty}\cup \{v:\ v\mid v_*\mathfrak{N}_1'\mathfrak{N}_2'\}$ and 
\begin{align*}
\mathcal{I}_{v}(f_{v},0):=&\int_{N'(F_{v})\backslash G'(F_{v})}\int_{G'(F_{v})}\int_{F_{v}^n}f_{v}\left(\begin{pmatrix}
	y_{v}&u_{v}\\
	&1
\end{pmatrix}\right)\theta_{v}(\eta x_{v}u_{v})\\
&\qquad \qquad W_{1,v}'(x_{v})\overline{W_{2,v}'(x_{v}y_{v})}|\det x_{v}|_{v}du_{v}dy_{v}dx_{v}.
\end{align*}

Note that $\mathfrak{N}_1'$ and $\mathfrak{N}_2'$ are coprime. By \cite{Kim10} (cf. Theorem 2.1.1) we have 
\begin{align*}
	\mathcal{I}_{v}(f_{v},0)=L_v(1,\pi_{1,v}'\times\widetilde{\pi}_{2,v}'),\ \ v\mid\mathfrak{N}_1'\mathfrak{N}_2'.
\end{align*}

For $v\in \Sigma_{F,\infty}\cup \{v_*\},$ by Iwasawa decomposition and Kirillov norm (cf. \eqref{143}),
\begin{align*}
\mathcal{I}_{v}(f_{v},0)=&\int_{G'(F_{v})}\int_{K_v'}\int_{F_v^{\times}}\int_{F_{v}^n}f_{v}\left(\begin{pmatrix}
	y_{v}&u_{v}\\
	&1
\end{pmatrix}\right)\theta_{v}(\eta z_{v}k_vu_{v})du_{v}|z_v|_v^{n}d^{\times}z_vdk_v\\
&\qquad \qquad \overline{\langle \pi_{2,v}'(y_v)W_{2,v}',W_{1,v}'\rangle_v} dy_{v},
\end{align*}
where $\langle \cdot,\cdot\rangle_v$ is the local Kirillov norm. By Fourier inversion and polar coordinates we obtain as \eqref{144}, for $v=v_*,$ and all  $v\in\Sigma_{F,\infty}$ if $\pi_{1,\infty}'\simeq\pi_{2,\infty}',$ that 
\begin{equation}\label{146}
\mathcal{I}_{v}(f_{v},0)=\zeta_{F_v}(1)d_v\int_{G'(F_{v})}f_{v}\left(\begin{pmatrix}
	y_{v}&\\
	&1
\end{pmatrix}\right)\overline{\langle \pi_{2,v}'(y_v)W_{2,v}',W_{1,v}'\rangle_v} dy_{v},
\end{equation}
where $d_v$ (e.g., cf. \eqref{145}) is certain positive constant depending only on $F_v$ and $n.$ 

Write $\pi_{1,v_*}'\simeq \pi_{2,v_*}'=\pi_{v_*}'.$  Note that $\sigma_{v_*}$ is tempered. Then $L_{v_*}(1/2,\sigma_{v_*}\times\pi_{v_*}')\neq 0.$ By local Rankin-Selberg theory, $\Hom_{G'(F_{v_*})}(\sigma_{v_*}\mid_{G'(F_{v_*})},\pi_{v_*}')\neq 0.$ Hence, 
\begin{align*}
	\mathcal{I}_{v}(f_{v},0)=&\zeta_{F_v}(1)d_v\int_{G'(F_{v})}\langle \sigma_{v}(y_{v})W_{v},W_{v}\rangle_{v}\overline{\langle \pi_{v}'(y_v)W_{v}',W_{v}'\rangle_v} dy_{v}\gg 1,
\end{align*}
for $v=v_*,$ where the implied constant depends on $\sigma_{v_*}$ and $\pi_{v_*}'.$ 

Now we consider $\mathcal{I}_{\infty}(f_{\infty},0):=\prod_{v\in\Sigma_{F,\infty}}\mathcal{I}_{v}(f_{v},0).$ 
\begin{itemize}
	\item Suppose that $\pi_{1,\infty}'\simeq\pi_{2,\infty}'.$ Then by \eqref{146} we have 
\begin{align*}
	\mathcal{I}_{\infty}(f_{\infty},0)\gg \mathcal{H}_{f_{\infty}}^{\heartsuit}(\boldsymbol{\lambda}_{\pi_{\infty}'})\langle W_{\infty}',W_{\infty}'\rangle_{\infty}\gg 1.
\end{align*}
\item Suppose that $\pi_{1,\infty}'\not\simeq\pi_{2,\infty}'.$ By the definition of $f_{\infty},$ 
\begin{align*}
\mathcal{I}_{\infty}(f_{\infty},0)=&\mathcal{H}_{h_0}(\boldsymbol{\lambda}_{\pi_{2,\infty}'})\int_{F_{\infty}^{\times}}\widehat{h}_1(|z_{\infty}|_{\infty}\eta )|z_{\infty}|_{\infty}^{n}\omega_{1,\infty}(z_{\infty})\omega_{2,\infty}^{-1}(z_{\infty})d^{\times}z_{\infty}\cdot \mathcal{J}_{\infty},
\end{align*}
where 
\begin{align*}
	\mathcal{J}_{\infty}:=\int_{N'(F_{\infty})\backslash \overline{G'}(F_{\infty})}W_{1,\infty}'(x_{\infty}) \overline{W_{2,\infty}'(x_{\infty})}|\det x_{\infty}|_{\infty}dx_{\infty}.\end{align*}

By Stade's formula, $\mathcal{J}_{\infty}=\Lambda_{\infty}(1,\pi_{1,\infty}'\times\widetilde{\pi}_{2,\infty}')/\Lambda_{\infty}(n).$ By Fourier inversion and polar coordinates transform, we have 
\begin{align*}
	\mathcal{I}_{\infty}(f_{\infty},0)\gg \mathcal{H}_{h_0}(\boldsymbol{\lambda}_{\pi_{2,\infty}'})\cdot \Lambda_{\infty}(1,\pi_{1,\infty}'\times\widetilde{\pi}_{2,\infty}')\gg 1.
\end{align*}
\end{itemize}

Therefore, Lemma \ref{thm53} follows.
\end{proof}

\begin{proof}[Proof of Theorem \ref{thm53.}]
Since $\mathfrak{N}_1'\mathfrak{N}_2'\subsetneq \mathcal{O}_F,$ by the construction of $f_v$ at $v\mid\mathfrak{N}_1'\mathfrak{N}_2',$ and  the integral expansion in Proposition \ref{prop18'}, we obtain that 
\begin{align*}
	J^{\Reg}_{\Geo,\du}(f,\textbf{s})\equiv 0
\end{align*}
for $\mathbf{s}=(s,0)$ with $\Re(s)>1.$

Moreover, since $f_{v_*}$ is the matrix coefficient of a supercuspidal representation, it has no constant in the Fourier expansion. By Propositions \ref{prop22} and \ref{prop23}, we have 
\begin{align*}
	\mathcal{F}_{0,1}J^{\bi}_{\Geo}(f,\textbf{s})=\mathcal{F}_{1,0}J^{\bi}_{\Geo}(f,\textbf{s})\equiv 0
\end{align*}
for $\mathbf{s}=(s,0)$ with $\Re(s)>0.$ In the spectral side we have $J^{\Reg,\heartsuit}_{\Spec}(f,\textbf{0})=J_0(f,\mathbf{0}),$ where $J_0(f,\textbf{0})$ is the cuspidal part defined by \eqref{j0} in \textsection\ref{sec4.2}.

By the uniqueness of meromorphic continuation, we have 
\begin{align*}
	J^{\Reg}_{\Geo,\du}(f,\textbf{0})=\mathcal{F}_{0,1}J^{\bi}_{\Geo}(f,\textbf{0})=\mathcal{F}_{1,0}J^{\bi}_{\Geo}(f,\textbf{0})=0.\end{align*}
	
Therefore, Theorems \ref{thm52} and \ref{Red} yield that
\begin{align*}
	J_0(f,\mathbf{0})=J^{\Reg}_{\Geo,\sm}(f,\textbf{0})+O(|\mathfrak{N}|).
\end{align*}

By the spectral expansion of $\K_0(x,y)$ and Lemma \ref{thm53}, we obtain 
\begin{equation}\label{147}
\sum_{\pi\in\mathcal{A}_0([G],\omega)}\sum_{\phi\in\mathfrak{B}_{\pi}}\Psi(0,\pi(f)W_{\phi},W_{1}')\overline{\Psi(0,W_{\phi},W_{2}')}\gg |\mathfrak{N}|^n.
\end{equation}

\begin{itemize}
	\item Suppose $\pi_{1,\infty}'\not\simeq\pi_{2,\infty}'.$ Then by \eqref{147} and the $K$-type of $f,$ there exists $\pi\in\mathcal{A}_0([G],\omega)^{K^{\dagger}}$ such that $\Psi(0,\pi(f)W_{\phi},W_{1}')\overline{\Psi(0,W_{\phi},W_{2}')}\neq 0,$ which implies that 
$L(1/2,\pi\times\pi_1')L(1/2,\pi\times\pi_2')\neq 0.$

\item Suppose $\pi_{1,\infty}'\simeq\pi_{2,\infty}'.$ Fix $f=f_{\infty}\otimes f_{\fin}$ as in \textsection\ref{9.2.3}. By the $K$-type of $f,$ 
\begin{align*}
J_0(f,\mathbf{0})=\sum_{\substack{\pi\in \mathcal{A}_0([G],\omega)^{K^{\dagger}}\\ \pi_{v_*}\simeq \sigma_{v_*}}}\sum_{\phi\in\mathfrak{B}_{\pi}^{K^{\dagger}}}\Psi(0,\pi(f)W_{\phi},W_{1}')\overline{\Psi(0,W_{\phi},W_{2}')} \cdot \mathcal{H}_{f_{\infty}}(\boldsymbol{\lambda}_{\pi_{\infty}}),
\end{align*}
where $\mathfrak{B}_{\pi}^{K^{\dagger}}$ consists of forms in $\mathfrak{B}_{\pi}$ which are right-$K^{\dagger}$-invariant. 

Since $\Psi(s,\pi(f)W_{\phi},W_{1}')\overline{\Psi(\overline{s},W_{\phi},W_{2}')}$ is entire, then by the maximum principle and the functional equation we have 
\begin{align*}
	|\Psi(0,\pi(f)W_{\phi},W_{1}')\overline{\Psi(\overline{0},W_{\phi},W_{2}')}|\leq |\Psi(s_0,\pi(f)W_{\phi},W_{1}')\overline{\Psi(\overline{s_0},W_{\phi},W_{2}')}|,
\end{align*}
where $\Re(s_0)=100.$ In conjunction with the sup-norm of $\phi$ (and thereby $W_{\phi}$) we obtain that $|\Psi(0,\pi(f)W_{\phi},W_{1}')\overline{\Psi(\overline{0},W_{\phi},W_{2}')}|$ is 
\begin{align*}
\ll (N_F(\mathfrak{N}\mathfrak{N}_1'\mathfrak{N_2'})q_{v_*}\|\boldsymbol{\lambda}_{\pi_{\infty}}\|)^{O(1)}|\Lambda_{\infty}(1,\pi_{\infty}\times\pi_{\infty}')|^2|W_{\phi}(I_{n+1})|^2|\mathcal{H}_{f_{\infty}}(\boldsymbol{\lambda}_{\pi_{\infty}})|,
\end{align*}
where the implied constant in $O(1)$ depends at most on $n$ and $F.$

Consider the Kuznetsov formula
\begin{align*}
\sum_{\substack{\pi}}\sum_{\phi\in\mathfrak{B}_{\pi}}\frac{|W_{\phi}(I_{n+1})|^2}{\langle \phi,\phi\rangle}\cdot \mathcal{H}_{f_{\infty}}(\boldsymbol{\lambda}_{\pi_{\infty}})=\int_{[N]}\int_{[N]}\K(u_1,u_2)\overline{\theta}(u_1)
\theta(u_2)du_1du_2,
\end{align*}
where $\pi\in \mathcal{A}_0([G],\omega)^{K^{\dagger}}$ with $ \pi_{v_*}\simeq \sigma_{v_*},$ $\K(\cdot,\cdot)$ is the kernel function associated with $f.$ 

Let $f^{\dagger}=\otimes_vf^{\dagger}_v$ be such that $f_{\infty}^{\dagger}=f_{\infty},$ $f_v^{\dagger}=f_v$ for all $v\nmid \mathfrak{N},$ and for $v\mid\mathfrak{N},$ $f_v^{\dagger}:=\|f_v\|_{\infty}\textbf{1}_{G(\mathcal{O}_{F_v})}*\omega_v.$ So $\supp f\subset \supp f^{\dagger}.$ Notice that (e.g., cf. \cite{Art05}) $\K(\cdot, \cdot)$ grows slowly, i.e., 
\begin{align*}
	\K(u_1,u_2)\ll\sum_{\gamma\in \overline{G}(F)}|f^{\dagger}(u_1^{-1}\gamma u_2)| \ll \|f\|_{\infty}\|u_1\|\|u_2\|,
\end{align*}
where the implied constant relies on $\supp f^{\dagger},$ which is independent of $\mathfrak{N}.$

By definition of $f_{\infty}$ (cf. \textsection\ref{9.2.3}), $0\leq \mathcal{H}_{f_{\infty}}(\boldsymbol{\lambda}_{\pi_{\infty}})\ll 1.$ Hence, 
\begin{align*}
\sum_{\substack{\pi\in \mathcal{A}_0([G],\omega)^{K^{\dagger}}\\ \pi_{v_*}\simeq \sigma_{v_*}}}\sum_{\phi\in\mathfrak{B}_{\pi}} |W_{\phi}(I_{n+1})|^2 \cdot \mathcal{H}_{f_{\infty}}(\boldsymbol{\lambda}_{\pi_{\infty}})\ll |\mathfrak{N}|^n,
\end{align*}
where the implied constant relies on $f_{\infty},$ $\mathfrak{N}_1',$ $\mathfrak{N}_2'$ and $\sigma_{v_*}.$

By definition, we have the decomposition 
\begin{align*}
	J_0(f,\mathbf{0})= \sum_{\pi\in\mathcal{A}_0(\mathfrak{N};\mathfrak{N}_1',\mathfrak{N}_2',\sigma_{v_*},\omega,\varepsilon)}\sum_{\phi\in\mathfrak{B}_{\pi}}\Psi(0,\pi(f)W_{\phi},W_{1}')\overline{\Psi(0,W_{\phi},W_{2}')}+J_0^{\dagger}(f,\mathbf{0}),
\end{align*}
where \begin{align*}
J_0^{\dagger}(f,\mathbf{0}):=\sum_{\substack{\pi\in \mathcal{A}_0([G],\omega)^{K^{\dagger}}\\ \pi_{v_*}\simeq \sigma_{v_*}\\
	\|\boldsymbol{\lambda}_{\pi_{\infty}}\|> |\mathfrak{N}|^{\varepsilon}}}\sum_{\phi\in\mathfrak{B}_{\pi}}\Psi(0,\pi(f)W_{\phi},W_{1}')\overline{\Psi(0,W_{\phi},W_{2}')}.
\end{align*}

Putting the above estimates together we then obtain that
\begin{align*}
	J_0^{\dagger}(f,\mathbf{0})\ll |\mathfrak{N}|^n\sum_{\substack{\pi\in \mathcal{A}_0([G],\omega)^{K^{\dagger}}\\ \pi_{v_*}\simeq \sigma_{v_*}\\
	\|\boldsymbol{\lambda}_{\pi_{\infty}}\|> |\mathfrak{N}|^{\varepsilon}}}(N_F(\mathfrak{N}\mathfrak{N}_1'\mathfrak{N_2'})q_{v_*}\|\boldsymbol{\lambda}_{\pi_{\infty}}\|)^{O(1)}|\Lambda_{\infty}(1,\pi_{\infty}\times\pi_{\infty}')|^2.
\end{align*}

By Stirling formula, $|\Lambda_{\infty}(1,\pi_{\infty}\times\pi_{\infty}')|^2\ll e^{-\pi \|\boldsymbol{\lambda}_{\pi_{\infty}}\|}<e^{-\pi|\mathfrak{N}|^{\varepsilon}}.$ So $J_0^{\dagger}(f,\mathbf{0})\ll |\mathfrak{N}|^{-100},$ where the implied constant relies on $f_{\infty},$ $\mathfrak{N}_1',$ $\mathfrak{N}_2'$ and $\sigma_{v_*}.$ As a consequence, we obtain 
\begin{align*}
	\sum_{\pi\in\mathcal{A}_0(\mathfrak{N};\mathfrak{N}_1',\mathfrak{N}_2',\sigma_{v_*},\omega,\varepsilon)}\sum_{\phi\in\mathfrak{B}_{\pi}}\Psi(0,\pi(f)W_{\phi},W_{1}')\overline{\Psi(0,W_{\phi},W_{2}')}\gg |\mathfrak{N}|^n,
\end{align*}
which implies that $\Psi(0,\pi(f)W_{\phi},W_{1}')\overline{\Psi(0,W_{\phi},W_{2}')}\neq 0$ for some $\phi\in \mathfrak{B}_{\pi}$ with $\pi\in\mathcal{A}_0(\mathfrak{N};\mathfrak{N}_1',\mathfrak{N}_2',\sigma_{v_*},\omega,\varepsilon).$ So $L(1/e,\pi\times\pi_1')L(1/2,\pi\times\pi_2')\neq 0.$
\end{itemize}

Therefore, Theorem \ref{thm53.} holds.
\end{proof}

\bibliographystyle{alpha}

\bibliography{RTF}

\end{document}